%% file: cartiermodules.tex
\documentclass[hidelinks,a4paper]{article}
\input{preamble}

\title{The $6\frac{1}{2}$-functor formalism of Cartier modules}

\cftpagenumbersoff{part} %Dont print part page numbers

\cftsetindents{section}{0em}{2em}
\cftsetindents{subsection}{0em}{2em}

\begin{document}

\maketitle 

\begin{abstract}
    Let $p$ be a prime. We construct a coherent formalism of geometric operations on Cartier modules with 
    coefficients in ind-coherent sheaves on noetherian derived $\mathbb F_p$-schemes. 
    For such a scheme $X$, we consider the stable presentable category $\Cart(X)$ of ind-coherent sheaves $M$ equipped with a structural morphism $F_{*,X}M \to M$, 
    where $F_X$ denotes the Frobenius endomorphism.

    We construct pushforward functors for arbitrary morphisms and exceptional pullback functors for separated morphisms almost of finite type, 
    together with coherent composition and base-change equivalences. 
    We single out conditions under which these functors admit further adjoints.

    On derived schemes separated and almost of finite type over $\mathbb F_p$, 
    we construct compatible closed symmetric monoidal structures given by shriek tensor products,
    and establish which functors are symmetric monoidal.

    The constructions arise from a general procedure for forming functor formalisms of Cartier objects, 
    using their description as lax limits and the universal properties of span $2$-categories developed by Cnossen--Lenz--Linskens \cite{cnossen2025universalityspan2categoriesconstruction}.
\end{abstract}

\newpage
\tableofcontents

\newpage
\input{introduction}

\section{Recollection on \texorpdfstring{$2$}{2}-categories} \label{sec:2cat}
\input{2-categories/cat2}
\input{2-categories/gray}
\input{2-categories/laxlim}
\input{2-categories/cartier-laxlim}

\section{The \texorpdfstring{$2$}{2}-functor formalism of Cartier modules} \label{sec:2ff}
\input{2FF/def-2FF}
\input{2FF/beck-chevalley}
\input{2FF/consequences}

\input{2FF/construction-of-2FF}

\input{2FF/frobenius-structures}

\section{\texorpdfstring{$2$}{2}-categories of lax symmetric monoidal functors} \label{sec:sym-mon}
\input{symmetric-monoidal/sym-mon-2cat}
\input{symmetric-monoidal/pointwise}

\input{symmetric-monoidal/marked-maps}
\input{symmetric-monoidal/oplax-lax-sym}
\input{symmetric-monoidal/calg}

\section{The \texorpdfstring{$3$}{3}-functor formalism of Cartier modules} \label{sec:3ff}
\input{3FF/def-3FF}
\input{3FF/construction-of-3FF}
\input{3FF/frobenius-structures}

\section{Cartier modules on derived schemes} \label{sec:derived-schemes}
\input{derived-schemes/animated-Fp-algebras}
\input{derived-schemes/derived-schemes}
\input{derived-schemes/cff-derived-schemes}

\input{derived-schemes/cff-cartier-modules-derived-schemes}
\input{derived-schemes/3ff-cart}
\input{derived-schemes/consequences}

\appendix

\input{appendix/mapping-objects} \label{sec:appendix-mapping-objects}

\newpage
\printindex[notation]
\newpage
{
    \footnotesize
    \bibliography{bibliography}
} 
\end{document}

%% file: preamble.tex
\usepackage[utf8]{inputenc}
\usepackage[english]{babel}

\newcommand{\tightunderline}[1]{%
  \underline{\smash{#1}}
}

\usepackage{tocloft}
\usepackage[margin=1in,marginpar=0.75in]{geometry}

\usepackage[colorinlistoftodos,textsize=scriptsize,obeyFinal]{todonotes}
\usepackage{xpatch}
\usepackage[explicit]{titlesec}

\usepackage[nottoc,notindex]{tocbibind}
\usepackage{imakeidx}

\makeindex[
    name=notation,
    title=Index of Notation,
    columns=2,
    intoc
]

\usepackage{amsmath, amsthm, amssymb, amsfonts}
\usepackage{mathtools}
\usepackage[a]{esvect}
\usepackage{bbm}

\usepackage{tikz}
\usepackage{tikz-cd}
\usepackage{graphicx}
\usetikzlibrary{nfold}
\usetikzlibrary{positioning}
\usetikzlibrary{arrows.meta,calc}
\usepackage{quiver}

\usepackage{verbatim}

\usepackage{enumitem}

\usepackage{thmtools}
\usepackage[dvipsnames, svgnames]{xcolor}
\definecolor{vert}{RGB}{15,120,5}
\definecolor{gris}{RGB}{128,128,128}
\definecolor{bleu}{RGB}{0,50,150}
\usepackage[hypertexnames=false]{hyperref}
\hypersetup{
    colorlinks=true,
    linkcolor={bleu},
    citecolor= {vert},
    urlcolor={blue!80!black},
    pdfborder={0 0 0}
}
\usepackage[capitalize,noabbrev,nameinlink]{cleveref}

\newcommand{\uses}[1]{}
\newcommand{\proofof}[1]{}

\theoremstyle{definition}
\newtheorem{defn}{Definition}[section]
\newtheorem{setting}[defn]{Setting}

\theoremstyle{plain}
\newtheorem{lem}[defn]{Lemma}
\newtheorem{notation}[defn]{Notation}
\newtheorem{constr}[defn]{Construction}
\newtheorem{variant}[defn]{Variant}
\newtheorem{cor}[defn]{Corollary}
\newtheorem{thm}[defn]{Theorem}
\newtheorem{prop}[defn]{Proposition}
\newtheorem*{prop*}{Proposition}
\newtheorem{exmpl}[defn]{Example}
\newtheorem{introthm}{Theorem}

\newtheorem{introprop}[introthm]{Proposition}

\theoremstyle{remark}
\newtheorem{rmk}[defn]{Remark}
\newtheorem*{rmk*}{Remark}

\crefname{lem}{Lemma}{Lemmas}
\crefname{prop}{Proposition}{Propositions}
\crefname{thm}{Theorem}{Theorems}
\crefname{introthm}{Theorem}{Theorems}
\crefname{defn}{Definition}{Definitions}
\crefname{cor}{Corollary}{Corollaries}
\crefname{section}{Section}{Sections}

\newcommand{\Fp}{\mathbb{F}_p}
\newcommand{\N}{\mathbb{N}}

\newcommand{\BN}{\vv{\mathrm{B}}\N}
\newcommand{\BNflat}{\BN^\flat}

\newcommand{\LEq}[2]{{\operatorname{LEq}({#1}, {#2})}}

\newcommand{\Fun}[2]{\operatorname{Fun}(#1, #2)}
\newcommand{\FunSegal}[2]{\operatorname{Fun}^{\operatorname{Segal}}(#1,#2)}
\newcommand{\FunGr}[2]{\operatorname{Fun}^{\operatorname{gr}}(#1,#2)}
\newcommand{\FunOpgr}[2]{\operatorname{Fun}^{\operatorname{opgr}}(#1,#2)}

\newcommand{\ulFun}{\tightunderline{\operatorname{Fun}}}

\newcommand{\TwoFunDecorated}[5]{\ulFun^{#1}_{#2}(#4,#5)_{#3}}

\newcommand{\TwoFun}[2]{\TwoFunDecorated{}{}{}{#1}{#2}}

\newcommand{\FunLax}[2]{\TwoFunDecorated{\operatorname{lax}}{}{}{#1}{#2}}
\newcommand{\FunOpLax}[2]{\TwoFunDecorated{\operatorname{oplax}}{}{}{#1}{#2}}
\newcommand{\FunEps}[2]{\TwoFunDecorated{\epsilon}{}{}{#1}{#2}}
\newcommand{\FunNEps}[2]{\TwoFunDecorated{\mathrm{op}(\epsilon)}{}{}{#1}{#2}}
\newcommand{\FunPartLax}[3]{\TwoFunDecorated{\operatorname{lax}, #1}{}{}{#2}{#3}}
\newcommand{\FunPartOpLax}[3]{\TwoFunDecorated{\operatorname{oplax}, #1}{}{}{#2}{#3}}
\newcommand{\FunPartEps}[3]{\TwoFunDecorated{\epsilon, #1}{}{}{#2}{#3}}

\newcommand{\TwoFunBadj}[4]{\TwoFunDecorated{}{}{(#3,#4)\operatorname{-badj}}{#1}{#2}}
\newcommand{\TwoFunBC}[2]{\TwoFunDecorated{}{}{(I,P)\operatorname{-BC}}{#1}{#2}}
\newcommand{\TwoFunVarBC}[4]{\TwoFunDecorated{}{}{(#3,#4)\operatorname{-BC}}{#1}{#2}}

\newcommand{\FunOpLaxIPBC}[2]{\TwoFunDecorated{\operatorname{oplax}}{}{(I,P)\operatorname{-BC}}{#1}{#2}}
\newcommand{\FunOpLaxBadj}[2]{\TwoFunDecorated{\operatorname{oplax}}{}{(I,P)\operatorname{-badj}}{#1}{#2}}
\newcommand{\FunEpsBadj}[2]{\TwoFunDecorated{\epsilon}{}{(I,P)\operatorname{-badj}}{#1}{#2}}
\newcommand{\FunEpsBC}[2]{\TwoFunDecorated{\epsilon}{}{(I,P)\operatorname{-BC}}{#1}{#2}}

\newcommand{\FunTensor}[2]{\TwoFunDecorated{}{\otimes}{}{#1}{#2}}
\newcommand{\FunLaxTensor}[2]{\TwoFunDecorated{}{\operatorname{lax-\hspace{-0.17em}}\otimes}{}{#1}{#2}}
\newcommand{\FunOplaxTensor}[2]{\TwoFunDecorated{}{\operatorname{oplax-\hspace{-0.17em}}\otimes}{}{#1}{#2}}

\newcommand{\FunLaxTensorIPBC}[2]{\TwoFunDecorated{}{\operatorname{lax-\hspace{-0.17em}}\otimes}{(I,P)\operatorname{-BC}}{#1}{#2}}
\newcommand{\FunLaxTensorBadj}[2]{\TwoFunDecorated{}{\operatorname{lax-\hspace{-0.17em}}\otimes}{(I,P)\operatorname{-badj}}{#1}{#2}}
\newcommand{\FunOplaxTensorBadj}[2]{\TwoFunDecorated{}{\operatorname{oplax-\hspace{-0.17em}}\otimes}{(I,P)\operatorname{-badj}}{#1}{#2}}

\newcommand{\FunOplaxLaxTensor}[2]{\TwoFunDecorated{\operatorname{oplax}}{\operatorname{lax-\hspace{-0.17em}}\otimes}{}{#1}{#2}}
\newcommand{\FunEpsLaxTensor}[2]{\TwoFunDecorated{\epsilon}{\operatorname{lax-\hspace{-0.17em}}\otimes}{}{#1}{#2}}
\newcommand{\FunOplaxOplaxTensor}[2]{\TwoFunDecorated{\operatorname{oplax}}{\operatorname{oplax-\hspace{-0.17em}}\otimes}{}{#1}{#2}}
\newcommand{\FunEpsOplaxTensor}[2]{\TwoFunDecorated{\epsilon}{\operatorname{oplax-\hspace{-0.17em}}\otimes}{}{#1}{#2}}
\newcommand{\FunEpsDeltaTensor}[2]{\TwoFunDecorated{\epsilon}{\operatorname{\delta-\hspace{-0.17em}}\otimes}{}{#1}{#2}}
\newcommand{\FunPartOplaxLaxTensor}[3]{\TwoFunDecorated{\operatorname{oplax}, #1}{\operatorname{lax-\hspace{-0.17em}}\otimes}{}{#2}{#3}}

\newcommand{\FunPartEpsDeltaTensor}[3]{\TwoFunDecorated{\epsilon, #1}{\operatorname{\delta-\hspace{-0.17em}}\otimes}{}{#2}{#3}}

\newcommand{\FunOplaxIPBCLaxTensor}[2]{\TwoFunDecorated{\operatorname{oplax}}{\operatorname{lax-\hspace{-0.17em}}\otimes}{(I,P)\operatorname{-BC}}{#1}{#2}}
\newcommand{\FunOplaxBadjOplaxTensor}[2]{\TwoFunDecorated{\operatorname{oplax}}{\operatorname{oplax-\hspace{-0.17em}}\otimes}{(I,P)\operatorname{-badj}}{#1}{#2}}
\newcommand{\FunEpsDeltaTensorIPBC}[2]{\TwoFunDecorated{\epsilon}{\operatorname{\delta-\hspace{-0.17em}}\otimes}{(I,P)\operatorname{-BC}}{#1}{#2}}
\newcommand{\FunEpsDeltaTensorBadj}[2]{\TwoFunDecorated{\epsilon}{\operatorname{\delta-\hspace{-0.17em}}\otimes}{(I,P)\operatorname{-badj}}{#1}{#2}}

\newcommand{\ptw}{{\otimes\operatorname{\hspace{-0.13em}-pt}}}

\newcommand{\oplaxvarset}{\{\,\mathrm{lax},\,\mathrm{oplax},\,\emptyset\,\}}

\newcommand{\gr}{\mathrm{gr}}

\newcommand{\Cat}[1]{\mathcal{#1}}
\newcommand{\TwoCat}[1]{\tightunderline{\Cat{#1}}}

\newcommand{\op}[1]{#1^{\operatorname{op}}}
\newcommand{\co}[1]{#1^{\operatorname{co}}}
\newcommand{\coop}[1]{#1^{\operatorname{coop}}}
\newcommand{\id}[1]{\operatorname{id}_{#1}}

\newcommand{\An}{\mathrm{An}}
\newcommand{\AnL}{\widehat{\An}}
\newcommand{\CatInf}{\operatorname{Cat}_{1}}
\newcommand{\TwoCatInf}{\tightunderline{\operatorname{Cat}}_{1}}
\newcommand{\CatInfTwo}{\operatorname{Cat}_{2}}
\newcommand{\TwoCatInfTwo}{\tightunderline{\operatorname{Cat}}_{2}}

\newcommand{\laxlim}{\operatorname{laxlim}}
\newcommand{\oplaxlim}{\operatorname{oplaxlim}}
\newcommand{\diag}{\operatorname{const}}

\newcommand{\Span}[5]{\tightunderline{\operatorname{Span}}(#1, #2, #3)_{(#4,#5)}}
\newcommand{\SpanStd}{\Span{\Cat C}{E}{\all}{I}{P}}
\newcommand{\all}{\operatorname{all}}

\newcommand{\QCoh}{\operatorname{QCoh}}
\newcommand{\IndCoh}{\operatorname{IndCoh}}
\newcommand{\Cart}{\operatorname{Cart}}

\newcommand{\Polyp}{\operatorname{Poly}_{\Fp}}
\newcommand{\PSig}[1]{\mathcal{P}_{\Sigma}(#1)}

\newcommand{\Ringp}{\operatorname{CAlg}_{\Fp}}

\newcommand{\AnRingp}{\operatorname{CAlg}_{\Fp}^{\operatorname{an}}}

\newcommand{\AffDSchp}{\operatorname{AffSch}_{\Fp}^{\operatorname{an}}}
\newcommand{\AffSchp}{\operatorname{AffSch}_{\Fp}}

\newcommand{\DSchp}{\operatorname{Sch}_{\Fp}^{\operatorname{an}}}
\newcommand{\DSchpnoeth}{\operatorname{Sch}_{\Fp}^{\operatorname{an,noeth}}}
\newcommand{\DSchsaft}{\operatorname{Sch}_{\Fp}^{\operatorname{an,saft}}}
\newcommand{\Schp}{\operatorname{Sch}_{\Fp}}

\newcommand{\StkZar}{\operatorname{Stk}_{\zar,\Fp}}
\newcommand{\PreStkp}{\operatorname{PreStk}_{\Fp}}

\newcommand{\Spec}[1]{\mathrm{Spec}(#1)}

\newcommand{\classical}[1]{{#1}^{\operatorname{cl}}}

\newcommand{\zar}{\operatorname{zar}}

\newcommand{\Shv}[3]{\operatorname{Shv}_{#1}(#2, #3)}

\newcommand{\cartsymb}{\arrow[dr, phantom,"\scalebox{1.5}{\color{black}$\lrcorner$}", near start, color=black]}
\newcommand{\cocartsymb}{\arrow[dr, phantom,"\scalebox{1.5}{\color{black}$\ulcorner$}", near end, color=black]}
\newcommand{\commsymb}{\arrow[dr, phantom,"\scalebox{1.5}{\color{black}$\circlearrowleft$}", color=black]}
\newcommand{\eqsymb}{\arrow[dr, phantom,"\scalebox{1.5}{\color{black}$=$}", color=black]}

\newcommand{\laxarrowur}[1]{\ar[ur, Rightarrow, shorten <= 10pt, shorten >= 10pt, "{#1}"]}
\newcommand{\laxarrowdl}[1]{\ar[dl, Rightarrow, shorten <= 10pt, shorten >= 10pt, "{#1}"]}

\newcommand{\Map}[3]{\operatorname{Map}_{#1}(#2, #3)}
\newcommand{\MapMarked}[4]{\operatorname{Map}^{#2}_{#1}(#3, #4)}
\newcommand{\Hom}[3]{\operatorname{Hom}_{#1}(#2, #3)}

\newcommand{\colim}{\operatorname{colim}}
\renewcommand{\lim}{\operatorname{lim}}

\newcommand{\Tw}[1]{\operatorname{Tw}(#1)}

\newcommand{\ev}[1]{\operatorname{ev}_{#1}}
\newcommand{\pr}{\operatorname{pr}}

\DeclareFontFamily{U}{min}{}
\DeclareFontShape{U}{min}{m}{n}{<-> udmj30}{}
\newcommand\yo{y}

\newcommand{\iHomL}[2]{\tightunderline{\operatorname{iHom}}^L(#1, #2)}
\newcommand{\iHomR}[2]{\tightunderline{\operatorname{iHom}}^R(#1, #2)}

\newcommand{\iHomRCat}[3]{\tightunderline{\operatorname{iHom}}_{#1}^R(#2, #3)}

\newcommand{\CAlg}[1]{\operatorname{CAlg}(#1)}
\newcommand{\CAlgLax}[1]{\operatorname{CAlg}_{\operatorname{lax}}(#1)}
\newcommand{\CAlgOpLax}[1]{\operatorname{CAlg}_{\operatorname{oplax}}(#1)}
\newcommand{\CAlgDelta}[1]{\operatorname{CAlg}_{\delta}(#1)}
\newcommand{\TwoCAlg}[1]{\tightunderline{\operatorname{CAlg}}(#1)}
\newcommand{\TwoCAlgLax}[1]{\tightunderline{\operatorname{CAlg}}_{\operatorname{lax}}(#1)}
\newcommand{\TwoCAlgOpLax}[1]{\tightunderline{\operatorname{CAlg}}_{\operatorname{oplax}}(#1)}
\newcommand{\TwoCAlgDelta}[1]{\tightunderline{\operatorname{CAlg}}_{\delta}(#1)}
\newcommand{\SpanFin}{\operatorname{Span}(\operatorname{Fin})}

\newcommand{\Fin}{\operatorname{Fin}}

\newcommand{\TwoCocartlax}[1]{\tightunderline{\operatorname{Cocart}}^{\operatorname{lax}}(#1)}

\newcommand{\TwoCocart}[1]{\tightunderline{\operatorname{Cocart}}(#1)}

\newcommand{\TwoCartOplax}[1]{\tightunderline{\operatorname{Cart}}^{\operatorname{oplax}}(#1)}

\newcommand{\shriektensor}{%
  \mathchoice
    {\mathbin{\overset{!}{\otimes}}}% displaystyle
    {\mathbin{\otimes^!}}% textstyle
    {\mathbin{\otimes^!}}% scriptstyle
    {\mathbin{\otimes^!}}% scriptscriptstyle
}

\newcommand{\nset}{\langle n \rangle}
\newcommand{\mset}{\langle m \rangle}

\newcommand\noloc{%
  \nobreak
  \mspace{6mu plus 1mu}
  {:}
  \nonscript\mkern-\thinmuskip
  \mathpunct{}
  \mspace{2mu}
}

\newcommand{\loccit}{\textit{loc.~cit.~}}

\let\OLDthebibliography\thebibliography
\renewcommand\thebibliography[1]{
  \OLDthebibliography{#1}
  \setlength{\parskip}{0pt}
  \setlength{\itemsep}{0pt plus 0.3ex}
}

\newcommand{\defemph}[1]{\textcolor{purple}{\emph{#1}}\index[notation]{#1}}
\newcommand{\defemphindex}[2]{\textcolor{purple}{\emph{#1}}\index[notation]{#2}}
\newcommand{\defemphnoindex}[1]{\textcolor{purple}{\emph{#1}}}
\newcommand{\defnot}[1]{\textcolor{purple}{#1}}

\date{\today}
\author{Klaus Mattis\footnote{\href{www.klaus-mattis.com}{www.klaus-mattis.com}}\qquad and \qquad Timo Wei{\ss}\footnote{\href{mailto:timo.weiss@uni-mainz.de}{timo.weiss@uni-mainz.de}}}

%% file: introduction.tex
\section{Introduction}
A distinguishing feature of algebraic geometry in positive characteristic $p > 0$ is the existence of the Frobenius endomorphism.
It gives rise to additional structures on sheaves on an $\Fp$-scheme, which has led to many deep structural results.

An example of such a structure is Cartier modules. 
Classically, a Cartier module on an $\Fp$-scheme $X$ is given by a (quasi-coherent) $\Cat O_X$-module $M$ together with an $\Cat O_X$-linear morphism $F_* M \to M$, where $F$ denotes the absolute Frobenius of $X$.
They first appeared in the work of Anderson \cite{Anderson} and were more thoroughly studied by Blickle and Böckle \cite{BBCartmod}, \cite{blickle2013cartiercrystals}, where they also proved strong finiteness results. 
Moreover, Cartier modules are related via their dual notion of Frobenius modules to Riemann--Hilbert correspondences in characteristic $p$, cf.\ \cite{emertonkisin}, \cite{bhatt2019riemann}.

In our previous articles \cite{us24cart}, \cite{us25frob} we established a higher categorical generalization of Cartier modules by defining them as certain pullbacks of categories. 
In particular, we are not bound to (quasi-coherent) $\Cat O_X$-modules nor the Frobenius endomorphism to study categorical properties of Cartier modules. 
In the present article, we will mainly be interested in Cartier modules in the stable $(\infty,1)$-category of ind-coherent $\Cat O_X$-modules over a derived $\Fp$-scheme $X$;
in particular we will study its functoriality in the base scheme $X$.

To do so, we will make use of our previous work \cite[Theorem 5.1]{us24cart} which shows that the derived $(\infty,1)$-category of Cartier modules can equivalently be described as an $(\infty,1)$-category of Cartier modules itself. 
Thus, it satisfies the universal property of a limit.
This perspective is particularly useful to define functors into the derived $(\infty,1)$-category of Cartier modules. 
In fact, we will construct functors like the exceptional pullback $f^!$ via this approach.

\subsection*{Functor formalisms of Cartier modules}

Six-functor formalisms were originally developed by Grothendieck \cite{SGA4-1} in the setting of étale cohomology. 
They describe the fundamental operations $f^*, f_*, f_!, f^!, \otimes$ and $\tightunderline{\mathrm{Hom}}$, and how they interact with each other, 
for example via adjunctions, base change or the projection formulas.
%In particular, they also capture Poincaré duality statements. \todo{Not so sure about that}
Since then, six-functor formalisms have been introduced in many other mathematical settings, for example for 
perverse sheaves \cite{deligne1983faisceaux} and $D$-modules \cite{kashiwara2013sheaves}, but also in motivic homotopy theory \cite{ayoub2007six} and condensed mathematics \cite{scholze2025six}.
For the present article, we use the six-functor formalism for ind-coherent sheaves from \cite{gaitsgory2012indcoherentsheaves,gaitsgory2019study}.

Although a six-functor formalism captures many functors and compatibilities at once, it can be encoded by a single lax symmetric monoidal functor 
$\Span{\Cat C}{\all}{E}{P}{I}^\otimes \to \TwoCatInf^\times$ out of a certain $(\infty, 2)$-category of spans, satisfying suitable additional adjointness conditions. 
In recent years, many authors, for example Liu--Zheng \cite{liu2015gluingrestrictednervesinftycategories}, \cite{liu2024enhancedoperationsbasechange}, Scholze \cite{scholze2025six}, Mann \cite{mann2022padic}, and Cnossen--Lenz--Linskens \cite{cnossen2025universalityspan2categoriesconstruction}, 
worked on the question of how to construct such a functor.
Rather than constructing each of the six functors and proving all of their compatibilities separately, 
Cnossen--Lenz--Linskens give a comparably small amount of data that is sufficient to construct a six-functor formalism. 
We use their approach to lift the existing six-functor formalism of ind-coherent sheaves on derived noetherian $\Fp$-schemes to a six-functor formalism of Cartier modules.
Our construction of the multiplicative enhancement requires us to restrict to derived schemes which are separated and almost of finite type over $\Fp$ to define the tensor product.

We write $\defnot{\DSchpnoeth}$ for the $(\infty,1)$-category of noetherian derived $\Fp$-schemes.
For $X \in \DSchpnoeth$ we write $\defnot{F_X} \colon X \to X$ for the absolute Frobenius,
and $\defnot{\Cart(X)} \coloneqq \Cart(\IndCoh(X), F_{*,X})$ for the $(\infty,1)$-category of \defemph{Cartier modules} in ind-coherent sheaves,
as defined in \cite[Definition 2.4]{us24cart}.
Our first result is the following:
\begin{introthm}[{\cref{lem:cartier-modules:main-thm-2ff,lem:consequences:composition,lem:derived-schemes-consequences:BC,lem:derived-schemes-consequences:adjoints-shriek}}] \label{intro:main-theorem-2FF}
    We exhibit the following functoriality of the stable presentable $(\infty,1)$-category $\Cart(X)$:
    \begin{enumerate}[label=(\arabic*)]
        \item For every morphism $f \colon X \to Y$ of derived noetherian $\Fp$-schemes there exists a \defemphindex{pushforward functor}{pushforward}
            $$\defnot{f_*} \colon \Cart(X) \to \Cart(Y).$$
        \item For every separated and almost of finite type morphism $g \colon X \to Y$ of derived noetherian $\Fp$-schemes there exists an \defemphindex{exceptional pullback functor}{pullback!exceptional}
            $$\defnot{g^!} \colon \Cart(Y) \to \Cart(X).$$
    \end{enumerate}
    These constructions satisfy the following:
    \begin{enumerate}[label=(\alph*)]
        \item \label{intro:unit-shriek-star} Let $X$ be a derived noetherian $\Fp$-scheme. Then $(\id{X})_* \cong \id{\Cart(X)} \cong (\id{X})^!$.
        \item \label{intro:star-composition} If $f \colon X \to Y$ and $g \colon Y \to Z$ are morphisms of derived noetherian $\Fp$-schemes, then there exists an equivalence 
            $g_* f_* \cong (gf)_*$ of functors $\Cart(X) \to \Cart(Z)$. 
        \item \label{intro:shriek-composition} If $f \colon X \to Y$ and $g \colon Y \to Z$ are separated and almost of finite type morphisms of derived noetherian $\Fp$-schemes, then there exists an equivalence 
            $f^! g^! \cong (gf)^!$ of functors $\Cart(Z) \to \Cart(X)$.
    \end{enumerate}
    The above composition and identity equivalences satisfy all higher associativity and unit coherences.
    \begin{enumerate}[resume,label=(\alph*)]
        \item \label{intro:basechange} Let
            \begin{center}
                \begin{tikzcd}
                    Y' \ar[r, "g'"] \ar[d, "f'"'] \cartsymb &X' \ar[d, "f"] \\
                    Y \ar[r, "g"'] & X,
                \end{tikzcd}
            \end{center}
            be a cartesian diagram of derived noetherian $\Fp$-schemes, where $f$ (and thus $f'$) is separated and almost of finite type.
            Then there is a canonical base change equivalence of functors
            \begin{equation*}
                f^! g_* \cong g'_* f'^! \colon \Cart(Y) \to \Cart(X'). 
            \end{equation*}
        \item \label{intro:adjoints} Let $f \colon X \to Y$ be a morphism of derived noetherian $\Fp$-schemes. Then the following adjunctions exist under the indicated hypotheses on $f$.
            Here, $\defnot{f^\times}$, $\defnot{f_\flat}$ and $\defnot{f^*}$ denote the corresponding adjoints.
            \begin{center}
                \begin{tabular}{l|l|l}
                    \textbf{\emph{Hypothesis on $f$}}&\textbf{\emph{Left Adjoint}} &\textbf{\emph{Right Adjoint}} \\\hline
                    \emph{arbitrary}                            &$f_*$&$\defnot{f^\times}$   \\
                    \emph{proper}                               &$f_*$&$f^!$        \\
                    \emph{separated and almost of finite type}  &$f^!$&$\defnot{f_\flat}$    \\
                    \emph{open immersion}                       &$f^!$&$f_*$        \\
                    \emph{eventually coconnective}              &$\defnot{f^*}$&$f_*$
                \end{tabular}
            \end{center}
            In particular, $f^! \cong f^\times$ for $f$ proper, and $f^! \cong f^*$ for $f$ an open immersion.
    \end{enumerate}
    The data of pushforward and exceptional pullback, with their composition, base change and open/proper adjunction compatibilities
    is packaged as a $2$-functor of $(\infty,2)$-categories 
    $$\defnot{\Cart(-)} \colon \Span{\DSchpnoeth}{E}{\all}{I}{P} \to \TwoCatInf,$$
    where $\defnot{I}$ denotes the class of open immersions, $\defnot{P}$ the class of proper morphisms, and $\defnot{E}$ the class 
    of separated and almost of finite type morphisms, cf.\ \cref{def:derived-schemes:nagata-subcategories}.
    In other words, we have defined a \textbf{\defemph{$4\frac{1}{2}$-functor formalism}} with respect to the class of eventually coconnective morphisms
    in the notation of \cref{def:2ff:defn}.
\end{introthm}

\begin{rmk*}
    In the above theorem, using the functor $\Cart(-)$, we define 
    $$f_* \coloneqq \Cart(X \xleftarrow{\id{X}} X \xrightarrow{f} Y) \qquad \text{and} \qquad f^! \coloneqq \Cart(Y \xleftarrow{f} X \xrightarrow{\id{X}} X).$$
\end{rmk*}

\begin{rmk*}
    Let us explain where the name \defemph{$4\frac{1}{2}$-functor formalism} comes from:
    We have the functors $f_*$ and $f^\times$ for all morphisms $f$, and for all 
    separated almost of finite type morphisms $g$, we have $g^!$ and $g_\flat$.
    This gives $4$ functors.
    For morphisms that are \defemph{eventually coconnective} 
    there exists an additional special adjoint $f^* \dashv f_*$, contributing another $\frac{1}{2}$ functor.
    An example of eventually coconnective morphisms are flat morphisms, so in particular open immersions, étale morphisms 
    and smooth morphisms, cf.\ \cref{lem:derived-schemes:flat-eventually-coconnective}.

    We explain further below why in our setting only the ``wrong'' adjoints exist.
\end{rmk*}

These operations are closely related to the familiar underlying functoriality of $\IndCoh$. 
In fact, we construct the $4\frac{1}{2}$-functor formalism $\Cart(-)$ as a lift of the corresponding $4\frac{1}{2}$-functor formalism of ind-coherent sheaves. 
In particular, we get explicit descriptions of the functors $f_*$ and $g^!$ for Cartier modules:

\begin{introprop}[{\Cref{lem:cartier-modules:main-thm-2ff}}]
    Let $X$ be a derived noetherian $\Fp$-scheme.
    Then $\Cart(X)$ fits into a pullback diagram of $(\infty,1)$-categories
    \begin{center}
        \begin{tikzcd}
            \Cart(X) \ar[r, "U"] \ar[d, "\kappa"] \cartsymb & \IndCoh(X) \ar[d, "{(F_{*,X}, \id{})}"] \\
            \Fun{\Delta^1}{\IndCoh(X)} \ar[r, "{(\ev{0},\ev{1})}"'] & \IndCoh(X) \times \IndCoh(X) \rlap{.}
        \end{tikzcd}
    \end{center}
    In particular, pointwise, the functor formalism of \cref{intro:main-theorem-2FF} recovers the $(\infty,1)$-category 
    of Cartier modules defined above.
    An object $M \in \Cart(X)$ is thus given by the tuple $(UM, \kappa(M) \colon F_{*,X} UM \to UM)$ 
    with $UM$ (resp.\ $\kappa(M)$) an ind-coherent sheaf (resp.\ a morphism of ind-coherent sheaves).
    Moreover, if $f \colon X \to Y$ is an arbitrary morphism of derived noetherian $\Fp$-schemes and $M \in \Cart(X)$ is a Cartier module, 
    then the pushforward $f_* M$ of $M$ is given by $U(f_* M) \cong f_* UM$ and
    $$\kappa(f_* M) \colon F_{*,Y} f_* UM \cong f_* F_{*,X} UM \xrightarrow{f_*\kappa(M)} f_* UM.$$
    For $g \colon X \to Y$ a separated and almost of finite type morphism of derived noetherian $\Fp$-schemes and $N \in \Cart(Y)$ a Cartier module, 
    the exceptional pullback $g^! N$ of $N$ is given by $U(g^! N) \cong g^! UN$ and 
    $$\kappa(g^! N) \colon F_{*,X} g^! UN \to g^! F_{*,Y} UN \xrightarrow{g^!\kappa(N)} g^! UN.$$ 
    Here, the first arrow in the structure morphism is a composition of the equivalence $F_{*,X} j^* \cong j^* F_{*,\overline{X}}$ for $j \colon X \to \overline{X}$ an open immersion 
    and the canonical map $F_{*,\overline{X}} p^\times \to p^\times F_{*,Y}$ defined via adjunction for $p \colon \overline{X} \to Y$ a proper map,
    coming from a compactification $g = pj$ as in \cref{intro:comapctification}.
\end{introprop}

\begin{rmk*}
    We make no claim whether $U$ commutes with any of the adjoints $f_\flat$, $f^*$ and $f^\times$, whenever they exist,
    except in trivial cases like $f$ proper, as then $f^\times = f^!$ 
    or $f$ an open immersion, as then $f^* = f^!$.
    We believe that one can extend the arguments in this paper to show that $U$ commutes with $f^*$ whenever
    $f$ is an étale morphism.
\end{rmk*}

We obtain the following expected formula for computing $f^!$ via compactifications. In particular,
note that the description of $f^!$ is \emph{independent} of the chosen compactification.
\begin{introprop}[{\cref{lem:derived-schemes-consequences:compactification,lem:derived-schemes:suitable-decomposition}}] \label{intro:comapctification}
    Let $f \colon X \to Y$ be a separated and almost of finite type morphism of derived noetherian $\Fp$-schemes.
    For any \defemph{compactification} $X \xrightarrow{j} \overline{X} \xrightarrow{p} Y$ of $f$
    (i.e.\, $\overline{X}$ is a derived noetherian scheme, $j$ is an open immersion and $p$ is proper),
    we have an equivalence $f^! \cong j^* p^\times$,
    where $p^\times$ is right adjoint to $p_*$, and $j^*$ is left adjoint to $j_*$.
    Moreover, such a compactification always exists.
\end{introprop}

Moreover, we have the following open base change result:
\begin{introprop}[Open base change, {\cref{lem:derived-schemes-consequences:open-basechange}}]
    Let $\square$ be a cartesian diagram of noetherian derived $\Fp$-schemes of the form
    \begin{center}
        \begin{tikzcd}
            Y' \ar[r, "g'"] \ar[d, "f'"'] \cartsymb &X' \ar[d, "f"] \\
            Y \ar[r, "g"'] & X,
        \end{tikzcd}
    \end{center}
    and suppose that $f$ is an open immersion and $g$ is proper.
    Then $\Cart_*(\square)$ satisfies the double Beck--Chevalley condition.
    In particular, there is a canonical equivalence $$f'^*g^\times \cong g'^\times f^* \colon \Cart(X) \to \Cart(Y').$$
\end{introprop}

As a special case of the base change result \cref{intro:main-theorem-2FF} \ref{intro:basechange},
we obtain the following base change result that holds for classical schemes.
\begin{introprop}[{\cref{lem:consequences:Tor-unital-basechange}}]
    Let
    \begin{center}
        \begin{tikzcd}
            &X' \ar[d, "f"] \\
            Y \ar[r, "g"'] & X,
        \end{tikzcd}
    \end{center}
    be a span of classical noetherian $\Fp$-schemes, where $f$ is separated and of finite type.
    Assume moreover that $f$ and $g$ are \emph{Tor-independent}, i.e.\ 
    for all compatible affine charts $\Spec{A} \subseteq X$, $\Spec{B} \subseteq Y$ and $\Spec{C} \subseteq X'$ we have
    $$\mathrm{Tor}^i_{A}(B, C) = 0$$
    for all $i > 0$.
    Then the derived pullback $X' \times_X Y$ is a classical scheme (in particular, it is the classical fiber product), and there is an equivalence of functors 
    \begin{equation*}
        f^! g_* \cong g'_* f'^! \colon \Cart(Y) \to \Cart(X'),
    \end{equation*}
    where $f' \colon X' \times_X Y \to Y$ and $g' \colon X' \times_X Y \to X'$ are the projections.
\end{introprop}

So far, we have mainly talked about two of the functors of a six-functor formalism, namely $f_*$ and $f^!$ (and their adjoints). 
Now we explain how to incorporate also the tensor product. 
For this, we have to change our setting a bit. 
Indeed, when constructing a symmetric monoidal structure on the $(\infty,1)$-category $\Cat D(X)$ out of the lax symmetric monoidal structure of a $2$-functor $\Cat D \colon \SpanStd^\otimes \to \TwoCatInf^\times$, 
we need that all diagonals lie in $E$. 
In fact, for each object $X \in \Cat C$ the symmetric monoidal structure on $\Cat D(X)$ is encoded by the functor 
$$\Cat D(X) \times \Cat D(X) \to \Cat D(X \times X) \xrightarrow{\Delta_X^!} \Cat D(X), $$
where the first arrow is given by the lax symmetric monoidal structure of $\Cat D$ and $\Delta_X \colon X \to X \times X$ denotes the diagonal.
Moreover, this also needs that the derived scheme $X \times X$ is again noetherian, which is false in general.
That is why we restrict ourselves in the following to $\defnot{\DSchsaft} \subseteq \DSchpnoeth$, by which we denote the full subcategory spanned by the \defemphnoindex{derived schemes 
that are separated and almost of finite type over $\Fp$}. 
\begin{introthm}[{\cref{lem:cartier-modules:main-thm-3ff,lem:consequences:shriek-symmetric-monoidal,lem:consequences:shriek-star-projection-formula}}] \label{intro:main-theorem-3FF}
    For every $X \in \DSchsaft$ the $(\infty,1)$-category $\Cart(X)$ admits 
    a symmetric monoidal structure, called the \defemph{shriek tensor product}, with tensor functor denoted by $\defnot{- \shriektensor -}$.
    Moreover, we have the following:
    \begin{enumerate}[label=(\alph*)]
        \item \label{intro:shriek-tensor-presentably-monoidal} The shriek tensor product $- \shriektensor -$ preserves colimits in each variable, i.e.\ $\Cart(X)$ is a \defemphnoindex{presentably symmetric monoidal} category.
            In particular, the symmetric monoidal category $\Cart(X)$ is \defemphnoindex{closed}, i.e.\ it admits an \defemph{internal hom}. 
        \item \label{intro:shriek-sym-mon} Let $f \colon X \to Y$ be a morphism in $\DSchsaft$. Then the exceptional pullback 
            $f^! \colon \Cart(Y) \to \Cart(X)$ admits a symmetric monoidal structure.
        \item \label{intro:shriek-projection-formula} Let $f \colon X \to Y$ be a morphism in $\DSchsaft$.
            Then the functor $f_* \colon \Cart(X) \to \Cart(Y)$ satisfies the projection formula: For two Cartier modules $M \in \Cart(X)$, $N \in \Cart(Y)$, there is an equivalence 
            $$f_*(M \shriektensor f^! N) \cong f_*(M) \shriektensor N.$$
    \end{enumerate}
    This data is packaged as a lax symmetric monoidal functor of $(\infty,2)$-categories 
    $$\defnot{\Cart^\otimes(-)} \colon \Span{\DSchsaft}{\all}{\all}{I}{P}^\otimes \to \TwoCatInf^\times,$$
    where $I$ denotes the class of open immersions, $P$ the class of proper morphisms.
    In other words, we have defined a \textbf{\defemph{$6\frac{1}{2}$-functor formalism}} with respect to the class of eventually coconnective morphisms
    in the notation of \cref{def:3ff:defn}.

    Forgetting the lax symmetric monoidal structure, the underlying functor is equivalent to the (restriction to $\DSchsaft$ of the) functor of \cref{intro:main-theorem-2FF}.
\end{introthm}
Again, the symmetric monoidal structure is closely related to the familiar underlying shriek tensor product of $\IndCoh$.

\begin{introprop}[{\cref{lem:derived-schemes-consequences:symm-mon-structure,lem:cartier-modules:main-thm-3ff}}]
    Let $X$ be a derived noetherian $\Fp$-scheme that is separated and almost of finite type over $\Fp$.
    The forgetful functor $U_X \colon \Cart(X) \to \IndCoh(X)$ is symmetric monoidal for the shriek tensor product on both sides.
    We can describe the unit and tensor functor of the shriek tensor product:
    \begin{enumerate}[label=(\arabic*)]
        \item The tensor unit $\mathbf{1}_{\Cart(X)}$ is given by 
            $$U(\mathbf{1}_{\Cart(X)}) = \omega_X$$
            and
            $$\kappa(\mathbf{1}_{\Cart(X)}) \colon F_{*,X} \omega_X \cong F_{*,X} F_X^! \omega_X \xrightarrow{\mathrm{counit}} \omega_X.$$
            Here, $\omega_X$ denotes the tensor unit of $\IndCoh(X)$ for the shriek tensor product.
        \item For two Cartier modules $M, N \in \Cart(X)$, their shriek tensor product $M \shriektensor N$ is given by 
            $$U(M \shriektensor N) \cong U(M) \shriektensor U(N)$$ 
            and
            \begin{equation*}
                \kappa(M \shriektensor N) \colon F_{*,X} (UM \shriektensor UN) \to F_{*,X} UM \shriektensor F_{*,X} UN \xrightarrow{\kappa(M) \shriektensor \kappa(N)} UM \shriektensor UN, 
            \end{equation*}
            where the first arrow is given by the oplax symmetric monoidal structure of $F_{*,X}$.
    \end{enumerate}
\end{introprop}

\subsection*{A short overview of the construction}

We give a brief outline of the construction of the underlying $3$-functor formalism of $\Cart^\otimes(-)$ which encodes the functors $f_*$, $g^!$ and $\shriektensor$. 
The extension to a $6\frac{1}{2}$-functor formalism then only requires showing that certain adjoints exist.
Let $X$ be a derived noetherian $\Fp$-scheme separated and almost of finite type over $\Fp$ and denote by $F_X \colon X \to X$ its absolute Frobenius.
Recall that the $(\infty,1)$-category of Cartier modules of ind-coherent sheaves is defined as the pullback in the diagram 
\begin{center}
    \begin{tikzcd}
        \Cart(X) \ar[r, "U"] \ar[d, "\kappa"'] \cartsymb & \IndCoh(X) \ar[d, "{(F_{*,X}, \id{})}"] \\
        \Fun{\Delta^1}{\IndCoh(X)} \ar[r, "{(\ev{0},\ev{1})}"'] & \IndCoh(X) \times \IndCoh(X) \rlap{.}
    \end{tikzcd}
\end{center}
Ideally, we would like to adapt this idea and just define $\Cart^\otimes(-)$ as the pullback of 
\begin{center}
    \begin{tikzcd}[column sep=huge]
        & \IndCoh^\otimes(-) \ar[d, "{(F_*, \id{})}"] \\
        \Fun{\Delta^1}{\IndCoh^\otimes(-)} \ar[r, "{(\ev{0},\ev{1})}"'] & \IndCoh^\otimes(-) \times \IndCoh^\otimes(-)
    \end{tikzcd}
\end{center}
in the $(\infty,1)$-category of $3$-functor formalisms.
But this would require $F_* \colon \IndCoh^\otimes(-) \to \IndCoh^\otimes(-)$ to be a natural transformation, 
in particular, it needs to commute with the functors $f_*$ and $f^!$. 
Unfortunately, for $f^!$ this is wrong in general.
Indeed, for $[p \colon X \to Y] \in P$, there is a canonical map $F_{*,X} p^! \to p^! F_{*,Y}$ but it is not an equivalence in general.
In other words, $F_*$ is an \emph{oplax natural transformation} which is strong over forward spans (encoding the $f_*$-functoriality) and over backward spans lying in $I$ (encoding the $j^*$-functoriality).
The strongness over some class of maps is captured by the notion of a \emph{partially oplax natural transformation}.
This fact naturally forces us to work with $(\infty,2)$-categories as $(\infty,1)$-categories serve to only encode strong natural transformations.

Having said this we would like to form the above pullback in the $(\infty,1)$-category of $3$-functor formalisms with partially oplax natural transformations between them. 
But it is unclear if pullbacks exist in this category. 
That is why we give an equivalent definition of Cartier modules in terms of a lax limit in \cref{lem:BN:Cart-as-fully-lax-lim}. 
In fact, we show in \Cref{lem:tensor-product:lax-sym-mon-lax-lim-existence}, generalizing \cite[Proposition 5.4.2]{abellán2026freefibrationslaxcolimits}, that lax limits exist in an $(\infty,2)$-categorical enhancement of this $(\infty,1)$-category and that they are even computed pointwise.

More concretely, $\Cart(X)$ is equivalent to the lax limit of the diagram $\BN \to \TwoCatInf$ given by the functor $F_{*,X} \colon \IndCoh(X) \to \IndCoh(X)$. 
To extend this to $3$-functor formalisms we have to extend the known natural transformation $F_* \colon \IndCoh_*(-) \to \IndCoh_*(-)$ (encoding the commutation with $f_*$) 
to a partially oplax natural transformation $F_* \colon \IndCoh^\otimes(-) \to \IndCoh^\otimes(-)$ between $3$-functor formalisms. 
To that end, we use the main theorem of Cnossen--Lenz--Linskens \cite[Theorem 6.13]{cnossen2025universalityspan2categoriesconstruction} about the construction of six-functor formalisms. 
In fact, Cnossen--Lenz--Linskens work with six-functor formalisms being defined via $f^*$ and $f_!$, whereas the six-functor formalisms we consider are defined via $f_*$ and $f^!$, 
so we first have to establish the somewhat dual result of \cite[Theorem 6.13]{cnossen2025universalityspan2categoriesconstruction}. 

Once we have the partially oplax natural transformation $F_* \colon \IndCoh^\otimes(-) \to \IndCoh^\otimes(-)$, it defines a $\BN$-diagram in the 
$(\infty,2)$-category of $3$-functor formalisms with partially oplax natural transformations. 
We define the $3$-functor formalism $\Cart^\otimes(-)$ to be the lax limit of this diagram. 
As this lax limit is computed pointwise by \Cref{lem:tensor-product:lax-sym-mon-lax-lim-existence} and \cite[Proposition 5.4.2]{abellán2026freefibrationslaxcolimits}, 
we see that for each $X\in \DSchsaft$, $\Cart(X)$ is indeed given by the $(\infty,1)$-category of Cartier modules on $X$.

\subsection*{Why do we use ind-coherent sheaves?}

Our choice of coefficients is motivated by Grothendieck duality.
For a proper morphism $f$, the exceptional pullback $f^!$ should
be right adjoint to the pushforward $f_*$.
In the derived category of quasi-coherent sheaves, this right
adjoint need not preserve colimits.
Our construction, however, requires exceptional pullback to
preserve colimits, so that it also admits a right adjoint.
Ind-coherent sheaves provide a setting in which these requirements
are compatible; see
\cite{gaitsgory2012indcoherentsheaves,gaitsgory2019study}.

The distinction is particularly relevant on singular schemes.
The compact objects of $\QCoh(X)$ are the perfect complexes,
whereas $\IndCoh(X)$ is obtained by adjoining filtered colimits
to the stable $(\infty,1)$-category of bounded coherent complexes.
Thus, bounded coherent complexes are compact in $\IndCoh(X)$,
even when they are not perfect.
Proper pushforward preserves bounded coherent complexes, and
this allows its right adjoint on ind-coherent sheaves to preserve
colimits.
Together with restriction to open subschemes, this gives the
exceptional pullback functors used in our formalism.

Nevertheless, for classical noetherian schemes, passing to
ind-coherent coefficients does not change the bounded derived
theory of Cartier modules:

\begin{prop*}[{\cref{lem:derived-schemes:bdd-above-cartier-modules}}]
    Let $X$ be a classical noetherian $\Fp$-scheme and denote by
    $F \colon X \to X$ its absolute Frobenius.
    Then there are equivalences of stable $(\infty,1)$-categories
    $$\Cart(\IndCoh(X),F_*)^b \cong \Cart(\QCoh(X),F_*)^b \cong D^b\bigl(\Cart(\QCoh(X)^\heartsuit,F_*)\bigr).$$
    Here, $\QCoh(X)^\heartsuit$ denotes the abelian category
    of quasi-coherent $\Cat O_X$-modules, $\QCoh(X)$ denotes
    the derived $(\infty,1)$-category of quasi-coherent complexes, and
    $(-)^b$ denotes the bounded part of the respective
    standard t-structures.
\end{prop*}

The comparison uses the canonical functor
$\IndCoh(X) \to \QCoh(X)$, which is an equivalence on bounded
objects and is compatible with Frobenius pushforward.
The final equivalence also uses
\cite[Theorem 5.1]{us24cart}, identifying the derived $(\infty,1)$-category
of Cartier modules with Cartier objects in the derived
coefficient category.
We therefore recover the familiar bounded derived $(\infty,1)$-category
of quasi-coherent Cartier modules while working in an
unbounded coefficient category suited to our geometric operations.

\subsection*{Why do we use derived schemes?}

The derived tensor product retains homological information that the ordinary tensor product can lose. 
For ring maps $A \to B$ and $A \to C$, the ordinary tensor product $B \otimes_A C$ records only $\operatorname{Tor}_0^A(B,C)$, 
whereas the derived tensor product $B \otimes_A^{\mathbf L} C$ also records the higher groups $\operatorname{Tor}_i^A(B,C)$. 
These groups measure the failure of the tensor product to behave as it would under flatness assumptions. 
Derived fiber products incorporate this additional information geometrically. 
When the higher Tor groups vanish, the derived fiber product agrees with the ordinary fiber product.

This distinction is essential for our functor formalism: a $3$-functor formalism encodes base-change equivalences for cartesian squares. 
For the Cartier-module operations considered here, base change with respect to ordinary fiber products of classical schemes is false in general. 
Consequently, these operations cannot assemble into a $3$-functor formalism on classical schemes with their ordinary fiber products. 
Passing to derived schemes replaces those fiber products by derived fiber products, for which our base-change theorem holds. 
For Tor-independent diagrams of classical schemes, it recovers the corresponding classical base-change equivalence.

\subsection*{Why has our functor formalism the ``wrong'' adjoints?}
In a usual six-functor formalism, one has adjunctions
$$f^* \dashv f_*, \qquad f_! \dashv f^!, \qquad - \otimes M \dashv \underline{\mathrm{Hom}}(M,-),$$
with the appropriate hypotheses on $f$. 
The pullback $f^*$ is symmetric monoidal, whereas the other geometric operations need not be; see \cite[Sections 2--3]{scholze2025six}. 
Our Cartier module formalism has a different orientation: 
In our situation $f^!$ is symmetric monoidal for the shriek tensor product, and the additional adjunctions constructed here are $f_* \dashv f^\times$ and $f^! \dashv f_\flat$.

A guiding explanation for this difference is the duality between Cartier modules and Frobenius modules. 
For noetherian $F$-finite schemes with affine diagonal and a unit dualizing complex, 
Baudin establishes a duality between bounded derived categories of coherent Cartier and Frobenius modules \cite[Theorem 4.2.7]{baudin2023duality}. 
Ren constructs an enhancement to stable $(\infty,1)$-categories and extends the duality to Ind- and Pro-completions \cite[Theorems 3.18 and 4.5, Corollary 8.1]{ren2026duality}. 
Duality reverses adjunctions, and Baudin's functorial compatibility statements for crystals exchange the star and shriek operations \cite[Theorem 5.1.13]{baudin2023duality}. 
This suggests viewing our formalism as one side of a dual picture, with the usual six-functor orientation expected on the Frobenius side. 
In future work, we will construct a six-functor formalism of Frobenius modules with coefficients in $\operatorname{Pro(Coh(X))}$, here $\operatorname{Coh}(X)$ denotes the stable $(\infty,1)$-category of
bounded coherent complexes, and lift the duality statements of 
Baudin and Ren to the level of six-functor formalisms, including their functorial and monoidal compatibilities.

\subsection*{Linear Outline} 
\paragraph{\cref{sec:2cat}: Recollection on $2$-categories.} 
    In this section, we first review the theory of $(\infty, 2)$-categories (which we from now on only call $2$-categories).
    We then discuss the Gray tensor product of $2$-categories, as well as its left and right mapping objects: the $2$-categories 
    of functors and op/lax natural transformations. We then define marked $2$-categories and partially op/lax limits.
    At the end of this section, we rewrite the category of generalized Cartier modules from \cite[Definition 2.4]{us24cart}
    as a lax limit along the universal walking endomorphism $\BN$.

\paragraph{\cref{sec:2ff}: The $2$-functor formalism of Cartier modules.}
    In this section, we first review the definition of $2$-functor formalisms, including the description of the span $2$-category we need for our purposes. 
    Next, we discuss the Beck--Chevalley conditions and explain how many compatibilities like adjunctions and base change results are encoded by a $2$-functor formalism.
    We then prove the dual version of \cite[Theorem 6.7]{cnossen2025universalityspan2categoriesconstruction} about the construction of $2$-functor formalisms.
    At the end of this section, we establish a general construction of functor formalisms of Cartier modules, which we apply in \cref{sec:derived-schemes} to prove \cref{intro:main-theorem-2FF}.

\paragraph{\cref{sec:sym-mon}: $2$-categories of lax symmetric monoidal functors.}
    We review symmetric monoidal $2$-categories and lax and oplax symmetric monoidal functors, including their interaction with adjunctions. 
    We then construct pointwise symmetric monoidal structures on functor $2$-categories. 
    Using marked mapping spaces, we define $2$-categories of monoidal functors with partially lax or oplax natural transformations, 
    establish the functoriality needed to preserve adjunctions, and study the existence and computation of lax limits. 
    Finally, we discuss commutative algebra objects and their relation to oplax symmetric monoidal functors. 
    These constructions provide the categorical input for the multiplicative formalism developed in \cref{sec:3ff}.
    Readers repulsed by the prospect of delicate $(\infty,2)$-categorical technicalities concerning oplax natural transformations 
    between lax symmetric monoidal functors are encouraged to skip this section on a first reading.

\paragraph{\cref{sec:3ff}: The $3$-functor formalism of Cartier modules.}
    In this section, we first define $3$-functor formalisms and describe how they encode symmetric monoidal structures and projection formulas. 
    We then prove how we can construct a $3$-functor formalism out of a lax symmetric monoidal functor encoding the $f_*$-functoriality, 
    which is another version of \cite[Theorem 6.13]{cnossen2025universalityspan2categoriesconstruction}.
    At the end of this section, we establish the multiplicative version of the main construction from \cref{sec:2ff}, which we apply in \cref{sec:derived-schemes} to prove \cref{intro:main-theorem-3FF}.

\paragraph{\cref{sec:derived-schemes}: Cartier modules on derived schemes.}
    In the last section, we first review the theory of animated $\Fp$-algebras and derived $\Fp$-schemes.
    We then discuss the formalism $\IndCoh$ of ind-coherent sheaves, developed by Gaitsgory \cite{gaitsgory2012indcoherentsheaves},
    as well as by Gaitsgory and Rozenblyum \cite{gaitsgory2019study}.
    In particular, we show that $\IndCoh$ admits a $3$-functor formalism.
    At the end, we use this result together with the main results of \cref{sec:2ff,sec:3ff}
    to show \cref{intro:main-theorem-2FF,intro:main-theorem-3FF}.

\paragraph{\cref{sec:appendix-mapping-objects}: Internal mapping objects.}
    In this short appendix, we review the behaviour of left and right mapping objects in an $\mathbb{E}_1$-monoidal category.

\subsection*{Notation}
We freely use the language of $(\infty, 1)$-categories, as outlined by Lurie in \cite{highertopoi,higheralgebra,sag,kerodon}.
Moreover, most of this paper is written in $(\infty, 2)$-categorical language. We will review the required notions in \cref{sec:2cat,sec:sym-mon}.
From now on, we will drop the $\infty$:
\begin{enumerate}
    \item The word \defemphnoindex{anima} means $(\infty, 0)$-category (or space, or $\infty$-groupoid).
    \item The word \defemphnoindex{category} means $(\infty, 1)$-category.
    \item The word \defemphnoindex{$2$-category} means $(\infty, 2)$-category.
\end{enumerate}
Generic categories will always be denoted by curly symbols like $\Cat C$ or $\Cat D$,
whereas $2$-categories will always be denoted with an underline like $\TwoCat{C}$ or $\TwoCat{D}$.
We fix three universes (ordered by size): \defemphnoindex{small}, \defemphnoindex{large} and \defemphnoindex{very large}.
We employ the following notation:

\vspace{10pt}

\begin{tabular}{l|l|l}
    Notation&Definition&Reference\\\hline
    $\An$ &The large category of small anima& \\
    $\AnL$ &The very large category of large anima& \\
    $\CatInf$ &The very large category of large categories \\
    $\TwoCatInf$ &The very large $2$-category of large categories&\cref{sec:recoll-2-cat} \\
    $\CatInfTwo$ &The very large category of large $2$-categories&\cref{sec:recoll-2-cat} \\
    $\TwoCatInfTwo$ &The very large $2$-category of large $2$-categories&\cref{sec:recoll-2-cat} \\
    $\Map{\Cat C}{-}{-}$ &Mapping \emph{anima} between two objects in a category $\Cat C$& \\
    $\Map{\TwoCat C}{-}{-}$ &Mapping \emph{category} between two objects in a $2$-category $\TwoCat C$& \\
    $\TwoFun{-}{-}$ &$2$-category of functors between $2$-categories &\cref{sec:recoll-2-cat}\\
    $\FunEps{-}{-}$ &$2$-category of functors between $2$-categories with&\cref{def:gray:fun-lax}\\&$\epsilon$-natural transformations for $\epsilon \in \{ \mathrm{lax}, \mathrm{oplax} \}$&\\
    $\FunPartEps{\dagger}{-}{-}$ &$2$-category of functors between $2$-categories with&\cref{def:lax-lim:partial-lax-functors}\\ &partially $\epsilon$-natural transformations for $\epsilon \in \{ \mathrm{lax}, \mathrm{oplax} \}$&\\
    $\FunPartEpsDeltaTensor{\dagger}{-}{-}$ &$2$-category of $\delta$-symmetric monoidal functors between &\cref{constr:2-functoriality-of-lax-tensor-oplax:1-functoriality} \\&symmetric monoidal $2$-categories with partially \\& $\epsilon$-natural transformations for $\epsilon, \delta \in \oplaxvarset$  \\
    $\TwoCAlgDelta{\TwoCatInfTwo}$ &$2$-category of symmetric monoidal $2$-categories with  &\cref{def:2sym:CAlgDelta} \\ &$\delta$-symmetric monoidal functors for $\delta \in \oplaxvarset$\\
    $\CAlgDelta{\TwoCatInfTwo}$ &Underlying category of $\TwoCAlgDelta{\TwoCatInfTwo}$ &\cref{def:2sym:CAlgDelta} \\
\end{tabular}

\vspace{10pt}

Since $\CatInf$ and $\CatInfTwo$ are locally large categories, we can talk about the Yoneda embeddings 
$$\CatInf \hookrightarrow \Fun{\op{\CatInf}}{\AnL} \qquad \text{and} \qquad \CatInfTwo \hookrightarrow \Fun{\op{\CatInfTwo}}{\AnL}.$$

In \cref{sec:derived-schemes} we use \emph{homological} indexing for t-structures, and write $\pi_n$ for the homotopy/homology objects.

\subsection*{Acknowledgments}
We thank our respective advisors Tom Bachmann and Manuel Blickle for many 
insightful discussions. We thank Sil Linskens for helpful conversations about 
partially lax morphisms between six functor formalisms, and Bastiaan Cnossen 
for explaining to us how to dualize the main result of \cite{cnossen2025universalityspan2categoriesconstruction}.
We moreover thank Julie Bannwart, Anton Engelmann, Dominik Kirstein, Jordan Levin, Marcus Nicolas and Luca Passolunghi 
for fruitful discussions about parts of this project. 
We thank Anton Engelmann and Dominik Kirstein for feedback on a draft of this article and for brewing amazing coffee.

The authors acknowledge support by the Deutsche Forschungsgemeinschaft 
(DFG, German Research Foundation) through the Collaborative Research 
Centre TRR 326 \textit{Geometry and Arithmetic of Uniformized 
Structures}, project number 444845124.

Part of this research was performed while the first author was visiting the Mathematical Sciences Research Institute (MSRI), 
now becoming the Simons Laufer Mathematical Sciences Institute (SLMath), 
which is supported by the National Science Foundation (Grant No. DMS-2424139).
The first author moreover acknowledges support from the DFG via Grant No. BA7141/2-1 714165, 
which funded his travel to MSRI.

\subsection*{AI disclosure}
We used AI models for proofreading, literature research, as well as mathematical discussions 
about parts of this document.

%% file: 2-categories/cat2.tex
In this section we review the theory of $2$-categories. Specifically,
in \cref{sec:recoll-2-cat} we recall the category of $2$-categories, and variants thereof.
We then discuss the involutions $\op{(-)}$, $\co{(-)}$ and $\coop{(-)}$.
Then, in \cref{sec:gray-tensor-product} we discuss the Gray tensor product of $2$-categories, as well as its left 
and right mapping objects, the $2$-categories of functors with op/lax natural transformations between them.
In \cref{sec:partially-lax-limits}, we introduce partially op/lax variants, and use them to define the notions 
of (partially) op/lax limits in $2$-categories.
Last, in \cref{sec:Cart-as-lax-limit}, we show that the category of generalized Cartier modules of \cite{us24cart} 
can be equivalently described as a lax limit over the walking endomorphism $\BN$.

\subsection{The category of \texorpdfstring{$2$}{2}-categories} \label{sec:recoll-2-cat}
We write $\defnot{\CatInfTwo}$ for the \defemphindex{category of $2$-categories}{$2$-category},
see e.g.\ \cite[Definition 8.2]{barwick2011unicity}. 
Because of \cite[Theorem 7.3]{barwick2011unicity},
it does not matter which exact model of $2$-categories we choose,
and in the following we will thus work as model-independent as possible.
Recall from \cite[Chapter 10, §2.2.2]{gaitsgory2019study} that there is an adjunction 
\begin{equation*}
    \defnot{\operatorname{incl}} \colon \CatInf \rightleftarrows \CatInfTwo \colon \defnot{\iota_1},
\end{equation*}
where the left adjoint is fully faithful,
and $\iota_1$ is given by sending a $2$-category to its largest sub-$1$-category.
We will thus view any category as a $2$-category, where all the $2$-morphisms are invertible.
Note that similarly, there is an adjunction
\begin{equation*}
    \defnot{\operatorname{incl}} \colon \AnL \rightleftarrows \CatInf \colon \defnot{\iota_0},
\end{equation*}
where the left adjoint is again fully faithful, and $\iota_0$ is given by sending a $1$-category 
to its largest subanima (or subgroupoid). We will also denote the composition 
$\iota_0 \iota_1 \colon \CatInfTwo \to \AnL$ by $\iota_0$.

Moreover, we write $\defnot{\TwoCatInf}$ for the \defemphindex{$2$-category of categories}{category}, see e.g.\ \cite[Chapter 10, Section 2.4]{gaitsgory2019study}, 
and $\defnot{\TwoCatInfTwo}$ for the \defemphindex{$2$-category of $2$-categories}{$2$-category}, see e.g.\ \cite[Chapter 10, Section 6.1]{gaitsgory2019study}.
In particular, $\iota_1 \TwoCatInf = \CatInf$ and $\iota_1 \TwoCatInfTwo = \CatInfTwo$, cf.\ \cite[Chapter 10, Proposition 2.4.2(b) and Section 6.1.1]{gaitsgory2019study}.

The category $\CatInfTwo$ is presentable, cf.\ \cite[Chapter 10, Lemma 2.1.5]{gaitsgory2019study}.
In particular, it has finite products and thus can be equipped with 
the cartesian symmetric monoidal structure, cf.\ \cite[Section 2.4.1]{higheralgebra}.
Note that $\CatInfTwo$ is cartesian closed (this follows from presentability and the 
fact that products commute with colimits in each variable, cf.\ \cite[Chapter 10, Theorem 2.5.3]{gaitsgory2019study}),
hence by \cref{lem:mapping-objects:ihom-existence} there exists an internal mapping object 
\begin{equation*}
    \defnot{\TwoFun{-}{-}} \colon \op{\CatInfTwo} \times \CatInfTwo \to \CatInfTwo,
\end{equation*}
called the \defemphindex{$2$-category of functors}{functors},
such that for every $2$-category $\TwoCat{C}$ there is an adjunction 
\begin{equation*}
    - \times \TwoCat{C} \dashv \TwoFun{\TwoCat C}{-}.
\end{equation*}

\begin{defn}
    Let $F \colon \TwoCat{M} \to \TwoCat{N}$ be a functor between $2$-categories.
    We say that $F$ is 
    \begin{itemize}
        \item a \defemphnoindex{locally full inclusion} (or that $\TwoCat{M}$ is a \defemphindex{$2$-full sub-$2$-category}{sub-$2$-category!$2$-full} of $\TwoCat{N}$)
            if the induced map of underlying anima $\iota_0 F \colon \iota_0 \TwoCat{M} \to \iota_0 \TwoCat{N}$ is a monomorphism and 
            the induced functors on mapping categories are fully faithful,
        \item a \defemphnoindex{full inclusion} (or that $\TwoCat M$ is a \defemphindex{full sub-$2$-category}{sub-$2$-category!full} of $\TwoCat N$)
            if $\iota_0 F$ is a monomorphism and the induced functors on mapping categories are equivalences, and
        \item a \defemphnoindex{wide inclusion} (or that $\TwoCat{M}$ is a \defemphindex{wide sub-$2$-category}{sub-$2$-category!wide} of $\TwoCat{N}$)
            if $\iota_0 F$ is an equivalence and the induced functors on mapping categories are subcategory inclusions
            in the sense of \cite[Section 1.2.11]{highertopoi}.
    \end{itemize}
\end{defn}

\begin{lem} \label{lem:two-cat:coop}
    There are autoequivalences $\defnot{\op{(-)}}, \defnot{\co{(-)}}, \defnot{\coop{(-)}} \colon \CatInfTwo \to \CatInfTwo$ 
    that send a 2-category $\TwoCat{M}$ to $\op{\TwoCat{M}}$ (the $2$-category 
    where we change the direction of all $1$-morphisms),
    to $\co{\TwoCat{M}}$ (the $2$-category where we change the direction of all $2$-morphisms),
    or to $\coop{\TwoCat{M}}$ (the $2$-category where we change the direction of all $1$- and $2$-morphisms).
    Moreover, we have canonical natural equivalences $$\co{(-)} \circ \op{(-)} \cong \coop{(-)} \cong \op{(-)} \circ \co{(-)}.$$
    In particular, if $\TwoCat{M}$ and $\TwoCat{N}$ are 2-categories,
    we get induced equivalences
    \begin{align*}
        &\TwoFun{\TwoCat{M}}{\TwoCat{N}} \xrightarrow[\op{(-)}]{\simeq} \op{\TwoFun{\op{\TwoCat{M}}}{\op{\TwoCat{N}}}}, \\
        &\TwoFun{\TwoCat{M}}{\TwoCat{N}} \xrightarrow[\co{(-)}]{\simeq} \co{\TwoFun{\co{\TwoCat{M}}}{\co{\TwoCat{N}}}}, \text{and } \\
        &\TwoFun{\TwoCat{M}}{\TwoCat{N}} \xrightarrow[\coop{(-)}]{\simeq} \coop{\TwoFun{\coop{\TwoCat{M}}}{\coop{\TwoCat{N}}}}
    \end{align*}
    which are natural in $\TwoCat{M}$ and $\TwoCat{N}$.
\end{lem}
\begin{proof}
    The first claim is \cite[Theorem 10.1]{barwick2011unicity}.
    Note that since $\op{(-)}$, $\co{(-)}$ and $\coop{(-)}$ are autoequivalences (which are their own inverses),
    they in particular preserve finite products and thus are symmetric monoidal 
    for the cartesian symmetric monoidal structure. We show the above equivalence for $\op{(-)}$, the proof for $\co{(-)}$ and $\coop{(-)}$ is similar.
    For every $2$-category $\TwoCat{K}$, we have natural equivalences of anima
    \begin{align*}
        \operatorname{Map}\bigl(
            \TwoCat{K},
            \TwoFun{\TwoCat{M}}{\TwoCat{N}}
        \bigr)
        &\simeq
        \operatorname{Map}\bigl(
            \TwoCat{K} \times \TwoCat{M},
            \TwoCat{N}
        \bigr)\\
        &\simeq
        \operatorname{Map}\bigl(
            \op{(\TwoCat{K} \times \TwoCat{M})},
            \op{\TwoCat{N}}
        \bigr)\\
        &\simeq
        \operatorname{Map}\bigl(
            \op{\TwoCat{K}} \times \op{\TwoCat{M}},
            \op{\TwoCat{N}}
        \bigr)\\
        &\simeq
        \operatorname{Map}\bigl(
            \op{\TwoCat{K}},
            \TwoFun{\op{\TwoCat{M}}}{\op{\TwoCat{N}}}
        \bigr)\\
        &\simeq
        \operatorname{Map}\bigl(
            \TwoCat{K},
            \op{\TwoFun{\op{\TwoCat{M}}}{\op{\TwoCat{N}}}}
        \bigr).
    \end{align*}
    Here the second and last equivalences use that $\op{(-)}$ is its own inverse (and hence its own left and right adjoint), 
    and the third uses that $\op{(-)}$ is symmetric monoidal for the cartesian monoidal structure.
    All the equivalences above are natural in $\TwoCat{M}$ and
    $\TwoCat{N}$, whence we conclude by the Yoneda lemma.
\end{proof}

These equivalences also admit a $2$-categorical enhancement:

\begin{lem} \label{lem:two-cat:two-cat-coop}
    The equivalences $\op{(-)}, \co{(-)}$ and $\coop{(-)}$ from \Cref{lem:two-cat:coop} upgrade to self-inverse equivalences $\defnot{\op{(-)}} \colon \TwoCatInfTwo \xrightarrow{\simeq} \co{\TwoCatInfTwo}$, 
    $\defnot{\co{(-)}} \colon \TwoCatInfTwo \xrightarrow{\simeq} \TwoCatInfTwo$ and $\defnot{\coop{(-)}} \colon \TwoCatInfTwo \xrightarrow{\simeq} \co{\TwoCatInfTwo}$.
    The statements of \Cref{lem:two-cat:coop} also hold for the upgraded versions.
\end{lem}
\begin{proof}
    Apply \cite[Observation 3.2.9, equation (3.1)]{AGH24} with base $\TwoCat B = \Delta^0$.
    For every $i,j \in \{0,1\}$, the $2$-category $\mathrm{FIB}_{(i,j)}(\Delta^0)$ identifies canonically
    with $\TwoCatInfTwo$: every $2$-category defines an $(i,j)$-fibration over a point, and every functor preserves the relevant cartesian and cocartesian morphisms.
    The three equivalences in that equation therefore give the asserted $2$-categorical enhancements.
    Their involutivity follows from involutivity of the corresponding opposite operations.
\end{proof}

\begin{notation}
    Let $\TwoCat{M}$ be a $2$-category.
    For every $1$-morphism $f$ and every $2$-morphism $\eta$ of $\TwoCat{M}$, we write
    \[
        \defnot{\op{f}},\ \defnot{\co{f}},\ \defnot{\coop{f}}
        \qquad\text{and}\qquad
        \defnot{\op{\eta}},\ \defnot{\co{\eta}},\ \defnot{\coop{\eta}}
    \]
    for the corresponding $1$- (or $2$-)morphisms in
    $\op{\TwoCat{M}}$, $\co{\TwoCat{M}}$, and $\coop{\TwoCat{M}}$.
\end{notation}

\begin{lem} \label{lem:two-cat:op-adjunctions}
    Let $\TwoCat{M}$ be a $2$-category and let $f\colon X \rightleftarrows Y \colon g$ be a pair of $1$-morphisms in $\TwoCat{M}$.
    The following are equivalent:
    \begin{enumerate}[label=(\roman*)]
        \item\label{point:two-cat:adjunction} The morphism $f \colon X \to Y$ is left adjoint to $g \colon Y \to X$ in $\TwoCat{M}$.
        \item\label{point:two-cat:op-adjunction} The morphism $\op{f} \colon Y \to X$ is right adjoint to $\op{g} \colon X \to Y$ in $\op{\TwoCat{M}}$.
        \item\label{point:two-cat:co-adjunction} The morphism $\co{f} \colon X \to Y$ is right adjoint to $\co{g} \colon Y \to X$ in $\co{\TwoCat{M}}$.
        \item\label{point:two-cat:coop-adjunction} The morphism $\coop{f} \colon Y \to X$ is left adjoint to $\coop{g} \colon X \to Y$ in $\coop{\TwoCat{M}}$.
    \end{enumerate}
\end{lem}
\begin{proof}
    As the operations $\op{(-)}$, $\co{(-)}$ and $\coop{(-)}$ are self-inverse, 
    it suffices to show that \ref{point:two-cat:adjunction} implies \ref{point:two-cat:op-adjunction} and \ref{point:two-cat:co-adjunction} 
    (both implications together yield that \ref{point:two-cat:adjunction} also implies \ref{point:two-cat:coop-adjunction}). 
    So let $f \colon X \to Y$ be left adjoint to $g \colon Y \to X$ in $\TwoCat{M}$.
    Let $u \colon \id{X} \to gf$ and $c \colon fg \to \id{Y}$ be a unit, respectively counit, of the adjunction.
    In particular, $u$ and $c$ satisfy the triangle identities, i.e.\ the compositions
    $$f \xrightarrow{fu} fgf \xrightarrow{cf} f \qquad \text{and} \qquad g \xrightarrow{ug} gfg \xrightarrow{gc} g$$ are each equivalent to the identity.
    
    To prove \ref{point:two-cat:op-adjunction}, note that passage to the opposite category turns $u$ and $c$ into $2$-morphisms
    $$\op{u} \colon \id{X} \to \op{(gf)} = \op{f} \op{g} \qquad \text{and} \qquad \op{c} \colon \op{g} \op{f} = \op{(fg)} \to \id{Y}.$$
    Moreover, the triangle identities translate to 
    $$\op{f} \xrightarrow{\op{u}\op{f}} \op{f} \op{g} \op{f} \xrightarrow{\op{f} \op{c}} \op{f} \qquad \text{and} \qquad  \op{g} \xrightarrow{\op{g} \op{u}} \op{g} \op{f} \op{g} \xrightarrow{\op{c} \op{g}} \op{g}$$
    being equivalent to the identity.
    This implies that $\op{f}$ is right adjoint to $\op{g}$.

    Similarly, one can show that the induced $2$-morphisms
    $$\co{c} \colon \id{Y} \to \co{(fg)} = \co{f} \co{g} \qquad \text{and} \qquad \co{u} \colon \co{g} \co{f} = \co{(gf)} \to \id{X}$$
    satisfy the triangle identities.
    Hence, they serve as a unit, respectively counit, of an adjunction $\co{g} \dashv \co{f}$,
    which proves \ref{point:two-cat:co-adjunction}.
\end{proof}

%% file: 2-categories/gray.tex
\subsection{The Gray tensor product and op/lax natural transformations} \label{sec:gray-tensor-product}
The category $\CatInfTwo$ has another monoidal structure, called the 
\defemphindex{($2$-truncated) Gray tensor product}{Gray tensor product}, cf.\ \cite[Corollary 2.15]{GagnaGray},
we write 
\begin{equation*}
\defnot{- \otimes^{\gr} -} \colon \CatInfTwo \times \CatInfTwo \to \CatInfTwo 
\end{equation*}
for the tensor product functor. Its tensor unit is given by the terminal category $\Delta^0$, cf.\ \cite[Remark 2.3]{GagnaGray}.
Note however that this monoidal structure is not symmetric (and in fact not even $\mathbb{E}_2$).
Recall from the discussion after \cite[Example 2.5]{GagnaGray} that the Gray monoidal structure is left and right closed.
Thus, we can define the following:
\begin{defn} \label{def:gray:fun-lax}
    We write via \Cref{lem:mapping-objects:ihom-existence}
    \begin{itemize}
        \item $\textcolor{purple}{\FunLax{-}{-}} \colon \op{\CatInfTwo} \times \CatInfTwo \to \CatInfTwo$
            for the right mapping object of the Gray tensor product, called the \defemphindex{$2$-category of functors with lax natural transformations}{functors!with lax natural transformations} and
        \item $\textcolor{purple}{\FunOpLax{-}{-}} \colon \op{\CatInfTwo} \times \CatInfTwo \to \CatInfTwo$
            for the left mapping object of the Gray tensor product, called the \defemphindex{$2$-category of functors with oplax natural transformations}{functors!with oplax natural transformations}.
    \end{itemize}
\end{defn}

\begin{rmk}
    In the notation of \cite[Section 2.1]{GagnaGray},
    we have 
    \begin{equation*}
        \FunLax{-}{-} \cong \FunOpgr{-}{-} \qquad \text{and} \qquad \FunOpLax{-}{-} \cong \FunGr{-}{-}.
    \end{equation*}
\end{rmk}

\begin{rmk}
    Let $\TwoCat{C}$ be a $2$-category.
    Then, by definition we know that there are adjunctions
    \begin{equation*}
        - \otimes^{\gr} \TwoCat{C} \dashv \FunLax{\TwoCat{C}}{-} \qquad \text{and} \qquad \TwoCat{C} \otimes^{\gr} - \dashv \FunOpLax{\TwoCat{C}}{-}.
    \end{equation*}
\end{rmk}

\begin{rmk} \label{rmk:lax-nat-transf:explicit-description}
    Let $\TwoCat C, \TwoCat D \in \CatInfTwo$ be $2$-categories. 
    Following the computations of \cite[Example 2.5]{GagnaGray}\footnote{Note that in \loccit they actually describe morphisms of $\FunGr{-}{-}$ when they say they compute morphisms of $\FunOpgr{-}{-}$.} we see that objects of $\FunLax{\TwoCat{C}}{\TwoCat{D}}$ are given by functors $F \colon \TwoCat{C} \to \TwoCat{D}$ 
    and morphisms are given by lax natural transformations $t \colon F \to G$. 
    In particular, such a lax natural transformation supplies the data of a $1$-morphism $t_X \colon FX \to GX$ in $\TwoCat{D}$ for each object $X \in \TwoCat{C}$ and a diagram 
    \begin{center}
        \begin{tikzcd}
            FX \ar[r, "{F(f)}"] \ar[d, "t_X"'] & FY \ar[d, "t_Y"] \\
            GX \ar[r, "{G(f)}"'] \laxarrowur{t_f} & GY
        \end{tikzcd}
    \end{center}
    for each $1$-morphism $f \colon X \to Y$ in $\TwoCat{C}$. 
    For an oplax natural transformation, the direction of the $2$-morphism $t_f$ is reversed.
\end{rmk}

\begin{lem} \label{lem:lax-nat-transf:coop-lax-oplax}
    If $\TwoCat{M}$ and $\TwoCat{N}$ are $2$-categories, then $\coop{(-)} \colon \CatInfTwo \to \CatInfTwo$ induces equivalences
    \begin{equation*}
        \FunLax{\TwoCat{M}}{\TwoCat{N}} \xrightarrow{\simeq} \coop{\FunLax{\coop{\TwoCat{M}}}{\coop{\TwoCat{N}}}}
    \end{equation*}
    and 
    \begin{equation*}
        \FunOpLax{\TwoCat{M}}{\TwoCat{N}} \xrightarrow{\simeq} \coop{\FunOpLax{\coop{\TwoCat{M}}}{\coop{\TwoCat{N}}}} \, .
    \end{equation*}
\end{lem}
\begin{proof}
    By \cite[Proposition 2.4.17]{masuda2026algebracategoricalspectra} \footnote{This result is for the Gray tensor product on $(\infty,\infty)$-categories.
    The version we need follows by truncating this result to the $2$-categorical level. See also \cite[Proposition~2.22]{LoubatonRuit2025Squares}.},
    the functor $\coop{(-)}$ is monoidal with respect to the
    Gray tensor product. Thus, naturally in $\TwoCat{A}$ and $\TwoCat{B}$,
    there is an equivalence
    \begin{equation} \label{eq:coop-gray-monoidal}
        \coop{\bigl(\TwoCat{A} \otimes^{\gr} \TwoCat{B}\bigr)}
        \cong
        \coop{\TwoCat{A}} \otimes^{\gr} \coop{\TwoCat{B}}.
    \end{equation}
    For every $2$-category $\TwoCat{K}$, we have natural equivalences of anima
    \begin{align*}
        \operatorname{Map}\bigl(
            \TwoCat{K},
            \FunLax{\TwoCat{M}}{\TwoCat{N}}
        \bigr)
        &\cong
        \operatorname{Map}\bigl(
            \TwoCat{K} \otimes^{\gr} \TwoCat{M},
            \TwoCat{N}
        \bigr)\\
        &\cong
        \operatorname{Map}\bigl(
            \coop{(\TwoCat{K} \otimes^{\gr} \TwoCat{M})},
            \coop{\TwoCat{N}}
        \bigr)\\
        &\cong
        \operatorname{Map}\bigl(
            \coop{\TwoCat{K}} \otimes^{\gr} \coop{\TwoCat{M}},
            \coop{\TwoCat{N}}
        \bigr)\\
        &\cong
        \operatorname{Map}\bigl(
            \coop{\TwoCat{K}},
            \FunLax{\coop{\TwoCat{M}}}{\coop{\TwoCat{N}}}
        \bigr)\\
        &\cong
        \operatorname{Map}\bigl(
            \TwoCat{K},
            \coop{\FunLax{\coop{\TwoCat{M}}}{\coop{\TwoCat{N}}}}
        \bigr).
    \end{align*}
    Here, the second and last equivalences use that $\coop{(-)}$ is its own inverse (and hence its own left and right adjoint), 
    and the third uses \eqref{eq:coop-gray-monoidal}.
    Similarly,
    \begin{align*}
        \operatorname{Map}\bigl(
            \TwoCat{K},
            \FunOpLax{\TwoCat{M}}{\TwoCat{N}}
        \bigr)
        &\cong
        \operatorname{Map}\bigl(
            \TwoCat{M} \otimes^{\gr} \TwoCat{K},
            \TwoCat{N}
        \bigr)\\
        &\cong
        \operatorname{Map}\bigl(
            \coop{(\TwoCat{M} \otimes^{\gr} \TwoCat{K})},
            \coop{\TwoCat{N}}
        \bigr)\\
        &\cong
        \operatorname{Map}\bigl(
            \coop{\TwoCat{M}} \otimes^{\gr} \coop{\TwoCat{K}},
            \coop{\TwoCat{N}}
        \bigr)\\
        &\cong
        \operatorname{Map}\bigl(
            \coop{\TwoCat{K}},
            \FunOpLax{\coop{\TwoCat{M}}}{\coop{\TwoCat{N}}}
        \bigr)\\
        &\cong
        \operatorname{Map}\bigl(
            \TwoCat{K},
            \coop{\FunOpLax{\coop{\TwoCat{M}}}{\coop{\TwoCat{N}}}}
        \bigr).
    \end{align*}
    All the equivalences above are natural in $\TwoCat{M}$, $\TwoCat{N}$ and
    $\TwoCat{K}$, whence we conclude by the Yoneda lemma.
\end{proof}

\begin{lem} \label{lem:lax-arrows:commutes-with-limits}
    The bifunctors $\TwoFun{-}{-}$, $\FunLax{-}{-}$ and $\FunOpLax{-}{-}$
    preserve limits in both variables\footnote{Since these functors are contravariant in the first variable,
    this of course means that they send colimits in $\CatInfTwo$ to limits in the first variable.}.
\end{lem}
\begin{proof}
    This is a special case of \Cref{lem:mapping-objects:ihom-preserves-limits-both-variables}.
\end{proof}

\begin{lem} \label{lem:lax-nat-trans:fun-to-lax-fun}
    There are natural transformations 
    $$\TwoFun{-}{-} \to \FunLax{-}{-} \qquad \text{and} \qquad \TwoFun{-}{-} \to \FunOpLax{-}{-}.$$
    Pointwise, these functors are wide sub-$2$-category inclusions.
\end{lem}
\begin{proof}
    This is a special case of \Cref{lem:mapping-objects:ihom-map-of-lax-functor},
    applied to the lax monoidal identity functor $\CatInfTwo^{\times} \xrightarrow{\id{}} \CatInfTwo^{\otimes^{\gr}}$
    (see \cite[Lemma 2.4.13]{masuda2026algebracategoricalspectra} for the construction of the lax monoidal structure,
    note that the Gray tensor product is semiCartesian (i.e.\ the unit is terminal) by \cite[Remark 2.3]{GagnaGray}).
    The last claim is \cite[Corollary 3.17]{haugseng2021laxtransformationsadjunctionsmonads}.
\end{proof}

\begin{notation}
    We will repeatedly state results for $\FunLax{-}{-}$, $\FunOpLax{-}{-}$ and $\TwoFun{-}{-}$. In order to avoid unnecessary 
    duplications, we will usually work with a variance $\epsilon \in \oplaxvarset$.
    We then write $\defnot{\FunEps{-}{-}}$ to mean $\FunLax{-}{-}$ if $\epsilon = \mathrm{lax}$,
    $\FunOpLax{-}{-}$ if $\epsilon = \mathrm{oplax}$ and $\TwoFun{-}{-}$ if $\epsilon = \emptyset$.

    Moreover, there will be formulas where a variance and its opposite variance appear.
    For this, we define the following function:
    $$\defnot{\mathrm{op}(-)} \colon \oplaxvarset \to \oplaxvarset; \qquad \mathrm{op}(\mathrm{lax}) = \mathrm{oplax}, \mathrm{op}(\mathrm{oplax}) = \mathrm{lax}, \mathrm{op}(\emptyset) = \emptyset.$$
    For example, if $\epsilon = \mathrm{lax}$, then $\FunNEps{-}{-} = \FunOpLax{-}{-}$.
\end{notation}

We have the following comparison functor:
\begin{prop} \label{lem:lax-arrow:lax-strong-comparison}
    Let $\TwoCat C, \TwoCat E$ and $\TwoCat N$ be $2$-categories, and let $\epsilon \in \oplaxvarset$.
    There exists a natural \defemph{currying} comparison functor
    $$\FunEps{\TwoCat E}{\TwoFun{\TwoCat C}{\TwoCat N}} \to \FunEps{\TwoCat C\times\TwoCat E}{\TwoCat N},$$
    natural in $\TwoCat C$, $\TwoCat E$ and $\TwoCat N$.
    This functor is an equivalence if $\epsilon = \emptyset$. The diagrams
    \begin{center}
        \begin{tikzcd}
            \TwoFun{\TwoCat E}{\TwoFun{\TwoCat C}{\TwoCat N}} \ar[d] \ar[r, "\simeq"] &\TwoFun{\TwoCat C\times\TwoCat E}{\TwoCat N} \ar[d]\\
            \FunLax{\TwoCat E}{\TwoFun{\TwoCat C}{\TwoCat N}} \ar[r] &\FunLax{\TwoCat C\times\TwoCat E}{\TwoCat N}
        \end{tikzcd}
    \end{center}
    and
    \begin{center}
        \begin{tikzcd}
            \TwoFun{\TwoCat E}{\TwoFun{\TwoCat C}{\TwoCat N}} \ar[d] \ar[r, "\simeq"] &\TwoFun{\TwoCat C\times\TwoCat E}{\TwoCat N} \ar[d]\\
            \FunOpLax{\TwoCat E}{\TwoFun{\TwoCat C}{\TwoCat N}} \ar[r] &\FunOpLax{\TwoCat C\times\TwoCat E}{\TwoCat N}
        \end{tikzcd}
    \end{center}
    commute, where the vertical maps are given by the natural transformations of \cref{lem:lax-nat-trans:fun-to-lax-fun}.
\end{prop}
\begin{proof}
    For $\epsilon = \mathrm{lax}$, the existence of the functor is
    \cite[Corollary 2.8.4]{AGH24}, and for
    $\epsilon = \mathrm{oplax}$ one can use the dual construction.
    For $\epsilon = \emptyset$, one uses that the currying adjunction
    holds on functor $2$-categories; cf.\
    \cref{lem:mapping-objects:tensor-hom-adjunction-internal}.
    To verify the compatibility diagrams, write
    $$q_{\TwoCat A, \TwoCat B}\colon \TwoCat A\otimes^\gr \TwoCat B \to \TwoCat A\times \TwoCat B$$
    for the canonical comparison. Under the tensor--hom adjunctions,
    precomposition with $q_{\TwoCat A, \TwoCat B}$ induces the inclusions of strong
    transformations into lax or oplax transformations.
    For every test $2$-category $\TwoCat X$, consider the diagrams
    \begin{center}
        \begin{tikzcd}[column sep=large, row sep=large]
            \TwoCat X\otimes^\gr(\TwoCat C\times\TwoCat E) \ar[r, "\lambda"] \ar[d, "q_{\TwoCat X,\TwoCat C\times\TwoCat E}"'] &\TwoCat C\times(\TwoCat X\otimes^\gr\TwoCat E) \ar[d, "\mathrm{id}_{\TwoCat C}\times q_{\TwoCat X,\TwoCat E}"] \\
            \TwoCat X\times(\TwoCat C\times\TwoCat E) \ar[r, "\simeq"'] &\TwoCat C\times(\TwoCat X\times\TwoCat E)
        \end{tikzcd}
    \end{center}
    and
    \begin{center}
        \begin{tikzcd}[column sep=large, row sep=large]
            (\TwoCat C\times\TwoCat E)\otimes^\gr\TwoCat X \ar[r, "\rho"] \ar[d, "q_{\TwoCat C\times\TwoCat E,\TwoCat X}"'] &\TwoCat C\times(\TwoCat E\otimes^\gr\TwoCat X) \ar[d, "\mathrm{id}_{\TwoCat C}\times q_{\TwoCat E,\TwoCat X}"] \\
            (\TwoCat C\times\TwoCat E)\times\TwoCat X \ar[r, "\simeq"'] & \TwoCat C\times(\TwoCat E\times\TwoCat X)\rlap{.}
        \end{tikzcd}
    \end{center}
    The bottom arrows are the canonical symmetry and associativity
    equivalences. The map $\lambda$ is induced by the pair
    $$\TwoCat X\otimes^\gr(\TwoCat C\times\TwoCat E) \longrightarrow \Delta^0\otimes^\gr\TwoCat C  \cong \TwoCat C, \qquad \mathrm{id}_{\TwoCat X}\otimes^\gr\mathrm{pr}_{\TwoCat E},$$
    and $\rho$ is induced by the pair
    $$(\TwoCat C\times\TwoCat E)\otimes^\gr\TwoCat X \longrightarrow \TwoCat C\otimes^\gr \Delta^0 \cong \TwoCat C, \qquad \mathrm{pr}_{\TwoCat E}\otimes^\gr\mathrm{id}_{\TwoCat X}.$$
    Both squares commute naturally: after projection to each cartesian
    factor, their commutativity follows from the naturality of $q$ and
    its compatibility with the unit.

    Applying
    $\Map{\CatInfTwo}{-}{\TwoCat N}$ and using the
    defining adjunctions, the top arrows induce respectively the
    lax currying comparison of \cite[Corollary 2.8.4]{AGH24} and its
    oplax analogue. The bottom arrows induce strong currying, and the
    vertical arrows induce the strong-to-lax and strong-to-oplax
    inclusions. Thus the resulting squares are the images of the two
    claimed compatibility diagrams under
    $\Map{\CatInfTwo}{\TwoCat X}{-}$.
    Since the constructions are natural in $\TwoCat X$, the Yoneda
    lemma proves the required commutativity.
\end{proof}

%% file: 2-categories/laxlim.tex
\subsection{Partially op/lax limits} \label{sec:partially-lax-limits}
In this section, we recall some results about partially op/lax limits.
\begin{defn} \label{def:lax-lim:marking}
    Let $\TwoCat{I}$ be a $2$-category.
    A \defemph{marking} $\dagger$ on $\TwoCat{I}$ 
    is a $2$-full sub-$2$-category $\dagger \subseteq \TwoCat{I}$ containing all equivalences.

    We write $\TwoCat{I}^\dagger$ for a $2$-category $\TwoCat{I}$
    equipped with a marking $\dagger$, and call it a \defemphindex{marked $2$-category}{marked!$2$-category}.
\end{defn}

\begin{defn} \label{def:lax-lim:partial-lax-functors}
    Let $\TwoCat{I}^\dagger$ be a marked $2$-category,
    and $\TwoCat{M}$ be a $2$-category.
    We define the $2$-category of \defemphindex{functors with partially lax natural transformations}{functors!with partially lax natural transformations} as
    the pullback in the following cartesian diagram
    \begin{center}
        \begin{tikzcd}
            \textcolor{purple}{\FunPartLax{\dagger}{\TwoCat{I}}{\TwoCat{M}}} \ar[r] \ar[d] \cartsymb & \FunLax{\TwoCat{I}}{\TwoCat{M}}\ar[d] \\
            \TwoFun{\dagger}{\TwoCat{M}} \ar[r] &\FunLax{\dagger}{\TwoCat{M}}\rlap{,}
        \end{tikzcd}
    \end{center}
    where the right vertical morphism is given by restriction,
    and the horizontal morphism is the inclusion from \cref{lem:lax-nat-trans:fun-to-lax-fun}.

    Similarly, we define the $2$-category of \defemphindex{functors with partially oplax natural transformations}{functors!with partially oplax natural transformations} as
    the pullback in the following cartesian diagram
    \begin{center}
        \begin{tikzcd}
            \textcolor{purple}{\FunPartOpLax{\dagger}{\TwoCat{I}}{\TwoCat{M}}} \ar[r] \ar[d] \cartsymb & \FunOpLax{\TwoCat{I}}{\TwoCat{M}}\ar[d] \\
            \TwoFun{\dagger}{\TwoCat{M}} \ar[r] &\FunOpLax{\dagger}{\TwoCat{M}}\rlap{,}
        \end{tikzcd}
    \end{center}
    where the right vertical morphism is given by restriction,
    and the horizontal morphism is the inclusion from \cref{lem:lax-nat-trans:fun-to-lax-fun}.
\end{defn}

\begin{rmk}
    Unravelling this definition we see that a partially op/lax natural transformation $t \colon F \to G$ is given by an op/lax natural transformation 
    whose $2$-morphisms $t_f$ are equivalences for $f \in \dagger$ (cf.\ \Cref{rmk:lax-nat-transf:explicit-description} for the notation).
\end{rmk}

\begin{defn} \label{def:lax-lim:partial-lax-lim}
    Let $\TwoCat{I}^\dagger$ be a marked $2$-category, and let $\TwoCat{M}$ be a $2$-category.
    \begin{enumerate}[label=(\alph*)]
        \item We say that a functor $F \colon \TwoCat{I} \to \TwoCat{M}$ \defemphindex{admits a partially lax limit}{lax limit!partially}
            if the \defemph{constant diagram functor}
            \begin{equation*}
                \defnot{\mathrm{const}} \colon \TwoCat{M} = \FunLax{\Delta^0}{\TwoCat{M}} \to \FunPartLax{\dagger}{\TwoCat{I}}{\TwoCat{M}}
            \end{equation*}
            has a partial right adjoint that is defined at $F$, i.e.\ there exists $L \in \TwoCat M$, and for every $X \in \TwoCat M$ an equivalence of mapping categories 
            $$\Map{\FunPartLax{\dagger}{\TwoCat I}{\TwoCat M}}{\mathrm{const}(X)}{F} \cong \Map{\TwoCat M}{X}{L},$$
            natural in $X$. 
            In this case we denote the \defemphindex{partially lax limit}{lax limit!partially} $L$ of $F$ by $\defnot{\laxlim^\dagger_{\TwoCat{I}} F}$.
        \item We say that $\TwoCat{M}$ \defemphnoindex{has all partially lax limits of shape $\TwoCat{I}^\dagger$} if each functor $F \colon \TwoCat{I} \to \TwoCat{M}$ admits a partially lax limit.
            In particular, we get an adjunction $\mathrm{const} \dashv \laxlim^\dagger_{\TwoCat{I}}$.
        \item Similarly, we say that a functor $F \colon \TwoCat{I} \to \TwoCat{M}$ \defemphindex{admits a partially oplax limit}{oplax limit!partially} 
            if the \defemph{constant diagram functor}
            \begin{equation*}
                \defnot{\mathrm{const}} \colon \TwoCat{M} = \FunOpLax{\Delta^0}{\TwoCat{M}} \to \FunPartOpLax{\dagger}{\TwoCat{I}}{\TwoCat{M}}
            \end{equation*}
            has a partial right adjoint that is defined at $F$, i.e.\ there exists $L \in \TwoCat M$, and for every $X \in \TwoCat M$ an equivalence of mapping categories 
            $$\Map{\FunPartOpLax{\dagger}{\TwoCat I}{\TwoCat M}}{\mathrm{const}(X)}{F} \cong \Map{\TwoCat M}{X}{L},$$
            natural in $X$. In this case we denote the \defemphindex{partially oplax limit}{oplax limit!partially} $L$ of $F$ by $\defnot{\oplaxlim^\dagger_{\TwoCat{I}} F}$.
        \item We say that $\TwoCat{M}$ \defemphnoindex{has all partially oplax limits of shape $\TwoCat{I}^\dagger$} if each functor $F \colon \TwoCat{I} \to \TwoCat{M}$ admits a partially oplax limit.
            In particular, we get an adjunction $\mathrm{const} \dashv \oplaxlim^\dagger_{\TwoCat{I}}$.
    \end{enumerate}
\end{defn}

\begin{exmpl}
    Let $\TwoCat{I}$ and $\TwoCat M$ be $2$-categories, and let $\defnot{\flat}$ be the marking on $\TwoCat{I}$
    where we mark only equivalences.
    Then there are canonical equivalences
    \begin{equation*}
        \FunPartLax{\flat}{\TwoCat{I}}{\TwoCat{M}} \cong \FunLax{\TwoCat{I}}{\TwoCat{M}}
        \qquad \text{and} \qquad
        \FunPartOpLax{\flat}{\TwoCat{I}}{\TwoCat{M}} \cong \FunOpLax{\TwoCat{I}}{\TwoCat{M}}.
    \end{equation*}
    Consequently, at least when $\TwoCat I$ is a $1$-category and $\TwoCat M = \TwoCatInf$, partially op/lax limits of shape $\TwoCat{I}^\flat$ recover 
    the notion of op/lax limits of shape $\TwoCat{I}$
    as discussed in \cite[Definition 2.9]{gepner2017lax}, cf.\ \cite[Remark 2.7]{gepner2024global}.
\end{exmpl}

The above example justifies the following definition of op/lax limits.

\begin{defn}
    Let $\TwoCat I$ and $\TwoCat M$ be $2$-categories. 
    We say that a functor $F \colon \TwoCat{I} \to \TwoCat{M}$ \defemphindex{admits a lax limit}{lax limit} if it admits a partially lax limit with respect to $\TwoCat{I}^\flat$.
    In this case we denote the lax limit of $F$ by $\defnot{\laxlim_{\TwoCat{I}} F}$.

    We say that $\TwoCat{M}$ \defemphnoindex{has all lax limits of shape $\TwoCat{I}$} if each functor $F \colon \TwoCat{I} \to \TwoCat{M}$ admits a lax limit.

    Similarly, we say that a functor $F \colon \TwoCat{I} \to \TwoCat{M}$ \defemphindex{admits an oplax limit}{oplax limit} if it admits a partially oplax limit with respect to $\TwoCat{I}^\flat$.
    In this case we denote the oplax limit of $F$ by $\defnot{\oplaxlim_{\TwoCat{I}} F}$.

    We say that $\TwoCat{M}$ \defemphnoindex{has all oplax limits of shape $\TwoCat{I}$} if each functor $F \colon \TwoCat{I} \to \TwoCat{M}$ admits an oplax limit.
\end{defn}

\begin{exmpl}
    Let $\TwoCat{I}$ and $\TwoCat M$ be $2$-categories, and let $\defnot{\sharp}$ the marking on $\TwoCat I$ 
    where we mark everything. Then there are canonical equivalences
    \begin{equation*}
        \FunPartLax{\sharp}{\TwoCat{I}}{\TwoCat{M}} \cong \TwoFun{\TwoCat{I}}{\TwoCat{M}} \cong \FunPartOpLax{\sharp}{\TwoCat{I}}{\TwoCat{M}}.
    \end{equation*}
    Consequently, if $\TwoCat{M}$ has all limits of shape $\TwoCat{I}$,
    then it has all partially lax and oplax limits of shape $\TwoCat{I}^\sharp$,
    and there are canonical equivalences 
    \begin{equation*}
        \lim_{\TwoCat{I}} \cong \laxlim_{\TwoCat{I}}^\sharp \cong \oplaxlim_{\TwoCat{I}}^\sharp.
    \end{equation*}
    Here, the limit over a $2$-category is defined similarly to the partially lax limit:
    As the partial right adjoint to the constant diagram functor $\TwoCat{M} \to \TwoFun{\TwoCat{I}}{\TwoCat{M}}$.
\end{exmpl}

The previous two examples show in particular 
that statements about partially op/lax limits
also specialize to the corresponding statements about op/lax limits and limits.

\begin{lem} [{\cite[Remark 2.7]{gepner2024global}}] \label{lem:lax-lim:partial-lax-lim-descr}
    Let $\Cat{I}^\dagger$ be a marked category,
    and let $F \colon \Cat{I} \to \TwoCatInf$ be a functor.
    Then the partially lax limit of 
    $F$ in the sense of \cite[Definition 3.3]{berman2020lax} is a partially lax limit of $F$ in the sense of \Cref{def:lax-lim:partial-lax-lim}.
    Moreover, the partially oplax limit of $F$
    in the sense of \cite[Definition 3.3]{berman2020lax} is a partially oplax limit in the sense of \Cref{def:lax-lim:partial-lax-lim}.
\end{lem}

We have the following existence result:
\begin{prop}[{\cite[Corollary 5.3.10]{abellán2026freefibrationslaxcolimits}}] \label{lem:lax-lim:partial-lax-lim-existence}
    Let $\TwoCat{I}^\dagger$ be a marked $2$-category.
    Then $\TwoCatInf$ has all partially lax and oplax limits of shape $\TwoCat{I}^\dagger$.
\end{prop}

%% file: 2-categories/cartier-laxlim.tex
\subsection{Generalized Cartier modules as lax limits} \label{sec:Cart-as-lax-limit}
We now specialize the discussion about (partially) op/lax limits to the case of lax equalizers. 
We then use this to rewrite the category of generalized Cartier modules from \cite{us24cart}
as a lax limit along the walking endomorphism $\BN$.

For this, we first need the following category:
\begin{defn} \label{def:lax-lim:parallel-arrow-cat}
    Write $\defnot{{(\rightrightarrows)}}$ for the \defemph{category of two parallel arrows},
    i.e.\ the category that fits in the following pushout diagram:
    \begin{center}
        \begin{tikzcd}
            \partial \Delta^1 \amalg \partial \Delta^1 \ar[r, "\nabla"] \ar[d] \cocartsymb & \partial \Delta^1 \ar[d, "{\defnot{0} \amalg \defnot{1}}"]\\
            \Delta^1 \amalg \Delta^1 \ar[r, "{\defnot{u} \amalg \defnot{d}}"']&\defnot{{(\rightrightarrows)}}\rlap{.}
        \end{tikzcd}
    \end{center}
    Informally, ${(\rightrightarrows)}$ looks like the following:
    \begin{center}
        \begin{tikzcd}
            0 \ar[rr, bend left, "u"] \ar[rr, bend right, "d"'] && 1\rlap{.}
        \end{tikzcd}
    \end{center}
\end{defn}

\begin{defn} \label{def:lax-lim:parallel-arrow-cat-marking}
    Write $u \subseteq {(\rightrightarrows)}$ for the marking on ${(\rightrightarrows)}$,
    given by the ``top''-morphism, i.e.\ we mark the identities and the blue arrow in the following diagram:
    \begin{center}
        \begin{tikzcd}
            0 \ar[rr, bend left, blue, "u"] \ar[rr, bend right, "d"'] && 1\rlap{.}
        \end{tikzcd}
    \end{center}
    In more formal language, the marking $u$ is the subcategory given by the image of $u \colon \Delta^1 \to {(\rightrightarrows)}$.
\end{defn}

\begin{notation}
    For a category $\Cat J$, we denote by $\Tw{\Cat J}$ the \defemph{twisted arrow category} of $\Cat J$, following \cite[Construction 5.2.1.1]{higheralgebra}.
    It comes equipped with source and target projections $s \colon \Tw{\Cat J} \to \Cat J$ 
    and $t \colon \Tw{\Cat J} \to \op{\Cat J}$, whose product is a right fibration by \cite[Proposition 5.2.1.3]{higheralgebra}.
\end{notation}

\begin{defn}
    Let $(\Cat I, \dagger)$ be a marked category, and $X \in \Cat I$ an object.
    We write $\defnot{\Cat I^\dagger_{/X}}$ for the \defemphindex{marked slice}{marked!slice},
    which is a marked category with underlying category $\Cat I_{/X}$, where a morphism 
    $f$ over $X$ is marked if it is marked in $\Cat I$.
\end{defn}

\begin{prop} \label{lem:lax-lim:LEq-as-partially-lax-lim}
    Let $F,G \colon \Cat C \to \Cat D$ be functors of categories. Then the \defemphindex{lax equalizer}{lax limit!equalizer} $\defnot{\LEq{F}{G}}$ of \cite[Definition II.1.4]{nikolaus2018topological} is equivalent to the partially lax limit of the functor $\Phi \colon {(\rightrightarrows)}^u \to \TwoCatInf$ given by the diagram
    \begin{center}
        \begin{tikzcd}
            \Cat C \ar[rr, bend left, blue, "G"] \ar[rr, bend right, "F"'] && \Cat D\rlap{.}
        \end{tikzcd}
    \end{center}
\end{prop}
\begin{proof} 
    By \Cref{lem:lax-lim:partial-lax-lim-descr}, we have an equivalence $\laxlim^u_{{(\rightrightarrows)}} \Phi \cong \lim \Phi'$,
    where $\Phi'$ is the functor given as
    \begin{equation*}
        \op{\Tw{{(\rightrightarrows)}}} \xrightarrow{s \times t} \op{{(\rightrightarrows)}} \times {(\rightrightarrows)} \xrightarrow{|{(\rightrightarrows)}^u_{/-}| \times \Phi} \op{\TwoCatInf} \times \TwoCatInf \xrightarrow{\Fun{-}{-}} \TwoCatInf,
    \end{equation*}
    where $|{(\rightrightarrows)}^u_{/-}|$ denotes \defemphindex{marked geometric realization}{marked!geometric realization},
    i.e.\ the Dwyer--Kan localization of the marked slice $({(\rightrightarrows)}^u_{/-})$ at marked morphisms.
    Unwinding the definition, the category $\Tw{{(\rightrightarrows)}}$ can be interpreted as a diagram
    \begin{center}
        \begin{tikzcd}
            & u \ar[dl] \ar[dr] & \\
            \id{0} && \id{1} \\
            & d \ar[ul] \ar[ur]
        \end{tikzcd}
    \end{center}
    We compute $|{(\rightrightarrows)}^u_{/0}| \cong \Delta^0$ because $0 \in {(\rightrightarrows)}$ is the only object of ${(\rightrightarrows)}$ lying over $0$. Moreover, $|{(\rightrightarrows)}^u_{/1}|$ can be seen as a diagram
    \begin{center}
        \begin{tikzcd}
            u \ar[r, "\simeq"] & \id{1} & d \ar[l]\rlap{.}
        \end{tikzcd}
    \end{center}
    Putting everything together we get that $\laxlim^u_{{(\rightrightarrows)}} \Phi$ is given by the limit of the diagram
    \begin{center}
        \begin{tikzcd}
            & \Cat D = \Phi(1) & \\
            \Cat C = \Phi(0) \ar[ur, "G"] \ar[dr, "F"'] && \Fun{|{(\rightrightarrows)}^u_{/1}|}{\Cat D} \ar[ul, "\ev{u}"'] \ar[dl, "\ev{d}"]\rlap{.}\\
            & \Cat D = \Phi(1) &
        \end{tikzcd}
    \end{center}
    Using that $\Fun{|{(\rightrightarrows)}^u_{/1}|}{\Cat D} \cong \Fun{\Delta^1}{\Cat D}$ and that under this equivalence the maps $\ev{d}$ and $\ev{u}$ correspond to the source and target map, respectively, we see that $\laxlim^u_{{(\rightrightarrows)}} \Phi$ is equivalent to the limit of the diagram
    \begin{center}
        \begin{tikzcd}
            \Cat{C} \ar[r, "{(F, G)}"] & \Cat{D} \times \Cat{D} & \Fun{\Delta^1}{\Cat{D}} \ar[l, "{(s, t)}"']\rlap{,}
        \end{tikzcd}
    \end{center}
    which is by definition the lax equalizer $\LEq{F}{G}$.
\end{proof}

\begin{defn}
    Write $\defnot{\BN}$ for the $(1,1)$-category consisting of a single object $\defnot{*}$, with endomorphism monoid given by 
    $\Map{\BN}{*}{*} = (\N,+,0)$.
    We denote by $\defnot{e} \colon * \to *$ the morphism given by $1 \in \N$.
\end{defn}

This category has the following $1$-categorical universal property:
\begin{prop} \label{lem:BN:univ-property-1-cat}
    For any category $\Cat C$ there is an equivalence of anima, natural in $\Cat C$:
    $$\Map{\CatInf}{\BN}{\Cat C} \xrightarrow{\simeq} \Map{\CatInf}{\Delta^1}{\Cat C} \times_{\iota_0\Cat C \times \iota_0\Cat C} \iota_0\Cat C.$$
    Here, the maps defining the pullback on the right hand side are given by the source, the target, and the diagonal.
    Moreover, the map is induced by evaluating a functor $F \colon \BN \to \Cat C$ at $e$.
\end{prop}
\begin{proof}
    The canonical map $\Delta^1 /\partial \Delta^1 \to \BN$ selecting $e$ 
    is a categorical equivalence of simplicial sets, see \cite[\href{https://kerodon.net/tag/00J4}{Theorem 00J4}]{kerodon},
    as $\BN$ is the path category of the graph $\Delta^1 /\partial \Delta^1$.
    Hence, the claim follows by applying $\Map{}{-}{\Cat C}$ to this morphism.
    See also \cite[§3.3]{blumberg2013ktheoryendomorphismsnoncommutativemotives}.
\end{proof}

We can immediately improve this to a $2$-categorical universal property:
\begin{prop} \label{lem:BN:univ-property-2-cat}
    For any $2$-category $\TwoCat C$ there is an equivalence of anima, natural in $\TwoCat C$:
    $$\Map{\CatInfTwo}{\BN}{\TwoCat C} \xrightarrow{\simeq} \Map{\CatInfTwo}{\Delta^1}{\TwoCat C} \times_{\iota_0\TwoCat C \times \iota_0\TwoCat C} \iota_0\TwoCat C.$$
    Here, the maps defining the pullback on the right hand side are given by the source, the target, and the diagonal.
    Moreover, the map is induced by evaluating a functor $F \colon \BN \to \TwoCat C$ at $e$.
\end{prop}
\begin{proof}
    This follows immediately from \cref{lem:BN:univ-property-1-cat},
    using the adjunction $\mathrm{incl} \colon \CatInf \rightleftarrows \CatInfTwo \noloc \iota_1$,
    and the fact that $\iota_0\TwoCat{C} \cong \iota_0\iota_1 \TwoCat{C}$.
\end{proof}

\begin{rmk}
    In other words, supplying a $2$-functor $F \colon \BN \to \TwoCat C$ 
    is the same as supplying an arrow $\Delta^1 \to \TwoCat C$, together with an identification 
    of its source and target; i.e.\ an endomorphism.
\end{rmk}

\begin{lem} \label{lem:BN:I-to-BN-partially-lax-lim-cofinal}
    The functor ${(\rightrightarrows)}^u \to \BNflat$ of marked categories 
    given by sending $u$ to $\id{*}$ and $d$ to $e$
    is partially-lax-limit-cofinal (cf.\ the dual notion of \cite[§1.2]{abellan2023cofinalfunctorsinftybicategories}).
\end{lem}
\begin{proof}
    We use (the dual version of) Quillen's Theorem A in the $(\infty,2)$-categorical context from \cite[Theorem 3.31]{abellan2023cofinalfunctorsinftybicategories}:
    In other words, we have to show that in the pullback square (where we define the top left object as the pullback)
    \begin{center}
        \begin{tikzcd}
            \defnot{{(\rightrightarrows)}^u_{/*}} \cartsymb \ar[d] \ar[r]& \BNflat_{/*} \ar[d]\\
            {(\rightrightarrows)}^u  \ar[r] &\BNflat
        \end{tikzcd}
    \end{center}
    the upper horizontal map induces an equivalence after marked geometric realization $|-|$
    (i.e.\ the Dwyer--Kan localization at all \emph{marked} morphisms).

    For this, let us first compute the marked $(1,1)$-category ${(\rightrightarrows)}^u_{/*}$. We will describe it by giving the set of objects,
    and the morphism sets between them.
    The set of objects is simply given by the product $\{0, 1\} \times \{e^N \vert N \in \N \}$,
    using that $\BN$ has a single object and no nontrivial automorphisms.
    We now compute the morphism set from $(i, e^M)$ to $(j, e^N)$, 
    which is given by $$\Map{{(\rightrightarrows)}_{/*}}{(i, e^M)}{(j, e^N)} \cong \Map{{(\rightrightarrows)}}{i}{j} \times_{\Map{\BN}{*}{*}} \Map{\BN_{/*}}{e^M}{e^N}.$$
    We have
    $$\Map{\BN_{/*}}{e^M}{e^N} = \begin{cases}
        * &\text{if } M \ge N \\
        \emptyset &\text{else}.
    \end{cases}$$
    If $M \ge N$, then this single morphism is sent to $e^{M - N} \in \Map{\BN}{*}{*}$ under the canonical projection.
    We consider four cases:

    \underline{\smash{$i = 0$ and $j = 0$}}: Here, $\Map{{(\rightrightarrows)}}{i}{j} = \Map{{(\rightrightarrows)}}{0}{0} = \{ \id{0} \}$.
    Now, $\id{0}$ is sent to $\id{*} = e^0$ in $\BN$. Whence, by the above, the fiber product contains 
    a single element if $N-M = 0$ (i.e.\ $N = M$), and no element otherwise.
    
    \underline{\smash{$i = 0$ and $j = 1$}}: Here, $\Map{{(\rightrightarrows)}}{i}{j} = \Map{{(\rightrightarrows)}}{0}{1} = \{ u, d \}$.
    Now, by definition $u$ is sent to $\id{*} = e^0$, and $d$ is sent to $e^1$.
    Whence, by the above, the fiber product contains 
    a single element if $M-N = 0$ (i.e.\ $N = M$) or if $M - N = 1$ (i.e. $M = N+1$), and no element otherwise.

    \underline{\smash{$i = 1$ and $j = 0$}}: Here, $\Map{{(\rightrightarrows)}}{i}{j} = \Map{{(\rightrightarrows)}}{1}{0} = \emptyset$,
    whence we get $\Map{{(\rightrightarrows)}_{/*}}{(i, e^M)}{(j, e^N)} \cong \emptyset$.

    \underline{\smash{$i = 1$ and $j = 1$}}: Here, $\Map{{(\rightrightarrows)}}{i}{j} = \Map{{(\rightrightarrows)}}{1}{1} = \{ \id{1} \}$.
    Now, $\id{1}$ is sent to $\id{*} = e^0$ in $\BN$. Whence, by the above, the fiber product contains 
    a single element if $M-N = 0$ (i.e.\ $N = M$), and no element otherwise. 

    In particular, we see that ${(\rightrightarrows)}_{/*}$ is a poset, with relation 
    $(i, e^M) \le (j, e^N)$ if and only if $i = j$ and $M = N$; or $i < j$ and ($M = N$ or $M = N+1$).

    We now investigate the marking. Since the categories on the right carry the flat markings,
    we see by inspection that a morphism $(i, e^M) \le (j, e^N)$ is marked if and only if $M = N$.

    We picture this in the below diagram, where blue arrows are marked. We omit identity morphisms (which are of course also marked).
    \begin{center}
        \begin{tikzcd}
            \vdots \ar[dr] &\vdots \\
            (0, e^3) \ar[blue,r] \ar[dr] &(1, e^3) \\
            (0, e^2) \ar[blue,r] \ar[dr] &(1, e^2) \\
            (0, e^1) \ar[blue,r] \ar[dr] &(1, e^1) \\
            (0, e^0) \ar[blue,r] &(1, e^0) 
        \end{tikzcd}
    \end{center}
    The top horizontal functor in the square at the beginning of the proof is given by sending $(i, e^M)$ to $e^M$.
    Using this description, we see that the marked geometric realization $|{(\rightrightarrows)}^u_{/*}|$ is equivalent to $\BN_{/*}$
    via this functor.
\end{proof}

\begin{lem} \label{lem:BN:Cart-as-fully-lax-lim}
    Let $\Cat C$ be a category and let $G \colon \Cat C \to \Cat C$ be an endofunctor.
    This data defines via \cref{lem:BN:univ-property-2-cat} a functor $\Phi \colon \BN \to \TwoCatInf$.
    Then $\laxlim \Phi \simeq \Cart(\Cat C, G) = \LEq{G}{\id{\Cat C}}$.
\end{lem}
\begin{proof}
    By \Cref{lem:BN:I-to-BN-partially-lax-lim-cofinal}, the functor $\xi \colon {(\rightrightarrows)}^u \to \BNflat$ is partially-lax-limit-cofinal, i.e.\ 
    $$\laxlim_{\BN}(\Phi) \simeq \laxlim_{{(\rightrightarrows)}}^u(\Phi \circ \xi).$$
    The functor $\Phi \circ \xi \colon {(\rightrightarrows)} \to \TwoCatInf$ maps $0$ and $1$ to $\Cat C$, the morphism $u$ to $\id{\Cat C}$ and the morphism $d$ to $G$.
    Thus, it follows from \Cref{lem:lax-lim:LEq-as-partially-lax-lim} that $\laxlim_{{(\rightrightarrows)}}^{u}(\Phi \circ \xi) \simeq \LEq{G}{\id{\Cat C}}$.
\end{proof}

%% file: 2FF/def-2FF.tex
We start this section by introducing in \cref{sec:2FF-def} one of the main actors of this article: a $2$-functor formalism,
and its variant, a $4\frac{1}{2}$-functor formalism. For this, we also introduce $2$-categories of spans,
following \cite{cnossen2025universalityspan2categoriesconstruction}.
In \cref{sec:2FF-BC} we discuss Beck--Chevalley conditions, that are needed to construct $2$-functor formalisms.
Then, in \cref{sec:2FF-consequences}, we deduce some consequences of the existence of a $2$-functor formalism,
such as base change and compactification results.
The heart of the section is in \cref{sec:constr-2FF}: We discuss how to construct a $2$-functor formalism 
from a surprisingly small amount of data, essentially dualizing the construction of \cite{cnossen2025universalityspan2categoriesconstruction}.
Last, in \cref{sec:Cart-2FF} we explain how to construct the $2$-functor formalism of Cartier modules,
given a $2$-functor formalism $\Cat D$, as well as Frobenius operations for every geometric object.

\subsection{Definition of \texorpdfstring{$2$}{2}-functor formalisms} \label{sec:2FF-def}
Recall the following definitions:
\begin{defn} [{\cite[Definition A.5.1]{mann2022padic}}]
    Let $\Cat C$ be a category and $E \subseteq \Cat C$ a wide subcategory.
    We say that $E$ is \defemph{stable under base change} (or that $(\Cat C, E)$ is a \defemph{geometric setup})
    if for any $[X \to Y] \in E$, and any $[Z \to Y] \in \Cat C$, the pullback $X \times_Y Z$ exists and
    the projection $X \times_Y Z \to Z$ is again in $E$.
\end{defn}

\begin{defn} \label{def:cff:suitable-decomposition}
    Let $(\Cat C, E)$ be a geometric setup.
    Let $I, P \subseteq E$ be two wide subcategories.
    We say that $(I, P)$ is a \defemph{suitable decomposition of $E$}
    if the following holds 
    \begin{enumerate}[label=\textbf{(\alph*)}]
        \item $I$ and $P$ are stable under base change in $\Cat C$.
        \item Every $f \in E$ is of the form $p \circ i$
            for some $p \in P$ and $i \in I$.
        \item Every morphism $f \in P \cap I$ is $n$-truncated for some $n$.
        \item $I$ and $P$ are \defemph{left cancellable}, i.e.\ if $f \colon Y \to X$ and $g \colon Z \to Y$ are composable 
            morphisms such that $f \in P$ (resp.\ $f \in I$),
            then $g \in P$ (resp.\ $g \in I$) if and only if $f \circ g \in P$
            (resp.\ $f \circ g \in I$).
    \end{enumerate}
\end{defn}

\begin{rmk} \label{rmk:why-suitable-decomposition-different}
    \cref{def:cff:suitable-decomposition} is essentially \cite[Definition A.5.9]{mann2022padic},
    except that Mann furthermore requires the existence of pullbacks in $\Cat C$.
    In our applications, for the $2$-functor formalism we will work in arbitrary derived noetherian $\Fp$-schemes.
    This category does not even admit finite products, therefore we cannot assume this for the $2$-functor formalism.
\end{rmk}

\begin{defn}
    A \defemph{Nagata setup} is a quadruple $(\Cat C, E, I, P)$
    such that $(\Cat C, E)$ is a geometric setup, and 
    $(I, P)$ is a suitable decomposition of $E$.
\end{defn}

\begin{rmk} \label{rmk:why-nagata-different}
    This definition is slightly weaker than the definition given in \cite[Definition 3.1.2]{Kuijper_2026}:
    Kuijper additionally requires that any morphism in $P$ and $I$ is truncated, not just the morphisms 
    in the intersection $P \cap I$. Moreover, she requires the existence of finite products.
    We will assume the existence of finite products later for the multiplicative theory in \cref{sec:3ff}.
    For the $2$-functor formalism we cannot assume this, see \cref{rmk:why-suitable-decomposition-different}.
\end{rmk}

\begin{setting}    
    For the rest of this section we will fix wide subcategories $I, P \subseteq E \subseteq \Cat C$
    that are all stable under base change in $\Cat C$.
\end{setting}

Recall that in \cite[Construction 4.12]{cnossen2025universalityspan2categoriesconstruction} 
the authors define a $2$-category $\defnot{\Span{\Cat C}{\all}{E}{P}{I}}$ of spans.
We will be mainly interested in the $2$-category $\SpanStd$, which is a dual variant of it.

\begin{defn}
    We define a $2$-category of spans by $\defnot{\SpanStd} \coloneqq \coop{\Span{\Cat C}{\all}{E}{P}{I}}$.
\end{defn}

Unravelling the $\coop{(-)}$ appearing in the definition we can describe the objects, $1$-morphisms, $2$-morphisms and compositions of $\SpanStd$ explicitly.

\begin{rmk}
    The objects of $\SpanStd$ are the same as the objects of $\Cat C$. 

    A $1$-morphism $X \to Y$ in $\SpanStd$ is given by a span 
    \begin{center}
        \begin{tikzcd}
            & Z \ar[dl, "f"'] \ar[dr, "g"] & \\
            X && Y\rlap{,}
        \end{tikzcd}
    \end{center}
    where $f \in E$ and $g$ is some morphism in $\Cat C$.

    A $2$-morphism in $\SpanStd$ from $X \xleftarrow{f} Z \xrightarrow{g} Y$ to $X \xleftarrow{f'} Z' \xrightarrow{g'} Y$ is given by a diagram of the form 
    \begin{center}
        \begin{tikzcd}
            & Z \ar[dl, "f"'] \ar[dr, "g"] & \\
            X & Z'' \ar[u, "i"'] \ar[d, "p"] \ar[r, "g''"] \ar[l, "f''"'] & Y \\
            & Z'\rlap{,} \ar[ul, "f'"] \ar[ur, "g'"']
        \end{tikzcd}
    \end{center}
    where $i \in I$, $p \in P$,  $f, f', f'' \in E$ and $g$, $g'$ and $g''$ are some morphisms in $\Cat C$.

    Moreover, composition of $1$-morphisms is defined via pullback, 
    and composition of $2$-morphisms is defined via composing the (vertical or horizontal, depending on which composition of $2$-morphisms is considered) inner spans in the above diagram.
\end{rmk}

The $2$-category of spans also comes equipped with two functors
\begin{equation*}
    \Cat C \to \SpanStd, \qquad [f \colon X \to Y] \mapsto [X \xleftarrow{\id{X}} X \xrightarrow{f} Y]
\end{equation*}
and
\begin{equation*}
    \op{E} \to \SpanStd, \qquad [f \colon X \to Y] \mapsto [Y \xleftarrow{f} X \xrightarrow{\id{X}} X].
\end{equation*}
They are the coop-duals of the corresponding functors for $\Span{\Cat C}{\all}{E}{P}{I}$.

\begin{defn}
    Let $\TwoCat{M}$ be a $2$-category.
    A \defemphindex{$2$-functor formalism (with values in $\TwoCat{M}$)}{$2$-functor formalism} is
    a functor $\Cat D \colon \Span{\Cat C}{E}{\all}{I}{P} \to \TwoCat{M}$.
    The \defemphnoindex{$2$-category of $2$-functor formalisms with values in $\TwoCat{M}$}
    is given by the usual functor-$2$-category $\TwoFun{\SpanStd}{\TwoCat{M}}$.
\end{defn}

\begin{defn}
    Let $\TwoCat{M}$ be a $2$-category and 
    $\Cat D \colon \Span{\Cat C}{E}{\all}{I}{P} \to \TwoCat{M}$
    be a $2$-functor formalism.
    We write $\defnot{\Cat D_*}$ for the composition
    \begin{equation*}
        \Cat C \to \Span{\Cat C}{E}{\all}{I}{P} \xrightarrow{\Cat D} \TwoCat{M}
    \end{equation*}
    and $\defnot{\Cat D^!}$ for the composition
    \begin{equation*}
        \op{E} \to \Span{\Cat C}{E}{\all}{I}{P} \xrightarrow{\Cat D} \TwoCat{M}.
    \end{equation*}
\end{defn}

\begin{defn} \label{def:2ff:defn}
    Let $\Cat D \colon \Span{\Cat C}{E}{\all}{I}{P} \to \TwoCatInf$ be a $2$-functor formalism.
    We say that $\Cat D$ is a \defemphindex{$4\frac{1}{2}$-functor formalism
    with respect to a collection of morphisms $T \subseteq \Cat C$}{$4\frac{1}{2}$-functor formalism} if 
    \begin{enumerate}
        \item[\textbf{(FF1)}] \label{axiom:FF1} for every object $X \in \Cat C$ the category $\Cat D(X)$ is presentable stable.
        \item[\textbf{(FF2)}] \label{axiom:FF2} for every morphism $f$ in $\Cat C$ the functor $\Cat D_*(f)$ admits a right adjoint $\defnot{f^\times}$.
        \item[\textbf{(FF3)}] \label{axiom:FF3} for every morphism $f$ in $E$ the functor $\Cat D^!(f)$ admits a right adjoint $\defnot{f_\flat}$.
        \item[\textbf{(FF4)}] \label{axiom:FF4} for every morphism $f$ in $T$ the functor $\Cat D_*(f)$ admits a left adjoint $\defnot{f^*}$. 
    \end{enumerate}
\end{defn}

%% file: 2FF/beck-chevalley.tex
\subsection{Beck--Chevalley conditions} \label{sec:2FF-BC}
In this subsection we discuss various Beck--Chevalley conditions in a $2$-category,
following \cite[Appendix F]{cnossen2024phd}.
\begin{defn} [Left Beck--Chevalley condition, {\cite[Definition F.2]{cnossen2024phd}}] \label{def:BC:left-BC}
    Let $\TwoCat{M}$ be a $2$-category. Assume that there is a diagram $(\square)$
    \begin{center}
        \begin{tikzcd}
            Y' \ar[r, "g'"] \ar[d, "f'"'] & X' \ar[d, "f"] \\
            Y \ar[r, "g"'] \laxarrowur{\eta} & X
        \end{tikzcd}
    \end{center}
    in $\TwoCat{M}$.
    Suppose that $f$ and $f'$ have left adjoint morphisms 
    $f_L$ and $f'_L$, respectively.
    The \defemph{mate} of $\eta$ is given by
    \begin{equation*}
        \defnot{\mathrm{Ex}^{\mathrm{L}}(\square)} \colon f_L g \xrightarrow{u'} f_L g f' f'_L \xrightarrow{\eta} f_L f g' f'_L \xrightarrow{c} g' f'_L,
    \end{equation*}
    where $u'$ and $c$ denote the unit, resp.\ counit, of the corresponding adjunction.
    We say that $\square$ \defemphindex{satisfies the left Beck--Chevalley condition}{Beck--Chevalley condition!left}
    if the mate $\mathrm{Ex}^{\mathrm{L}}(\square)$ of $\eta$ is an equivalence.
\end{defn}

\begin{defn} [Right Beck--Chevalley condition, {\cite[Definition F.2]{cnossen2024phd}}]
    Let $\TwoCat{M}$ be a $2$-category. Assume that there is a diagram $(\square)$
    \begin{center}
        \begin{tikzcd}
            Y' \ar[r, "g'"] \ar[d, "f'"'] & X' \ar[d, "f"] \laxarrowdl{\varepsilon}\\
            Y \ar[r, "g"'] & X
        \end{tikzcd}
    \end{center}
    in $\TwoCat{M}$.
    Suppose that $f$ and $f'$ have right adjoint morphisms 
    $f_R$ and $f'_R$, respectively.
    The \defemph{mate} of $\varepsilon$ is given by
    \begin{equation*}
        \defnot{\mathrm{Ex}^{\mathrm{R}}(\square)} \colon g' f'_R \xrightarrow{u} f_R f g' f'_R \xrightarrow{\varepsilon} f_R g f' f'_R \xrightarrow{c'} f_R g,
    \end{equation*}
    where $u$ and $c'$ denote the unit, resp.\ counit, of the corresponding adjunction.
    We say that $\square$ \defemphindex{satisfies the right Beck--Chevalley condition}{Beck--Chevalley condition!right}
    if the mate
    $\mathrm{Ex}^{\mathrm{R}}(\square)$ of $\varepsilon$ is an equivalence.
\end{defn}

\begin{defn} [Double Beck--Chevalley condition, {\cite[Definition F.12]{cnossen2024phd}}]
    Let $\TwoCat{M}$ be a $2$-category. Assume that there is a diagram $(\square)$
    \begin{center}
        \begin{tikzcd}
            Y' \ar[r, "g'"] \ar[d, "f'"'] & X' \ar[d, "f"] \\
            Y \ar[r, "g"'] \laxarrowur{\eta} & X
        \end{tikzcd}
    \end{center}
    in $\TwoCat{M}$.
    Suppose that $f$ and $f'$ have left adjoint morphisms 
    $f_L$ and $f'_L$, respectively,
    and that $g$ and $g'$ have right adjoint morphisms 
    $g_R$ and $g'_R$, respectively.
    Suppose moreover that $\square$ satisfies the left Beck--Chevalley condition.
    We define the \defemphnoindex{double Beck--Chevalley transformation} as 
    \begin{equation*}
        \defnot{\mathrm{Ex}^{\mathrm{L,R}}(\square)} \colon f'_L g_R \xrightarrow{u'} g'_R g' f'_L g_R \xrightarrow{(\mathrm{Ex}^{\mathrm{L}}(\square))^{-1}} g'_R f_L g g_R \xrightarrow{c} g'_R f_L,
    \end{equation*}  
    where $u'$ and $c$ denote the unit, resp.\ counit, of the corresponding adjunction.
    We say that $\square$ \defemphindex{satisfies the double Beck--Chevalley condition}{Beck--Chevalley condition!double} if the double Beck--Chevalley transformation $\mathrm{Ex}^{L,R}(\square)$ is an equivalence.
\end{defn}

\begin{rmk}
    Let $\TwoCat{M}$ be a $2$-category. Assume that there is a diagram $(\square)$
    \begin{center}
        \begin{tikzcd}
            Y' \ar[r, "g'"] \ar[d, "f'"'] & X' \ar[d, "f"] \\
            Y \ar[r, "g"'] \laxarrowur{\eta} & X
        \end{tikzcd}
    \end{center}
    in $\TwoCat{M}$.
    Suppose that $f$ and $f'$ have left adjoint morphisms $f_L$ and $f'_L$, respectively,
    and that $g$ and $g'$ have right adjoint morphisms $g_R$ and $g'_R$, respectively.
    Moreover, assume that $\square$ satisfies the left Beck--Chevalley condition,
    and that the transposed square $\square^t$ 
    \begin{center}
        \begin{tikzcd}
            Y' \ar[r, "f'"] \ar[d, "g'"'] & Y \ar[d, "g"] \laxarrowdl{\eta} \\
            X' \ar[r, "f"'] & X
        \end{tikzcd}
    \end{center}
    satisfies the right Beck--Chevalley condition.
    We can also define a second double Beck--Chevalley transformation via 
    \begin{equation*}
        \defnot{\mathrm{Ex}'^{\mathrm{L,R}}(\square)} \colon f'_L g_R \xrightarrow{u} f'_L g_R f f_L \xrightarrow{(\mathrm{Ex}^{\mathrm{R}}(\square^t))^{-1}} f'_L f' g'_R f_L \xrightarrow{c'} g'_R f_L.
    \end{equation*}
    Then $\mathrm{Ex}^{\mathrm{L,R}} \cong \mathrm{Ex}'^{\mathrm{L,R}}$, cf.\ \cite[Lemma F.11]{cnossen2024phd} (the proof there works in any $2$-category).
\end{rmk}

\begin{lem} \label{lem:Beck--Chevalley:op-BC}
    Let $\TwoCat{M}$ be a $2$-category. Assume that there is a diagram $(\square)$
    \begin{center}
        \begin{tikzcd}
            Y' \ar[r, "g'"] \ar[d, "f'"'] & X' \ar[d, "f"] \\
            Y \ar[r, "g"'] \laxarrowur{\eta} & X
        \end{tikzcd}
    \end{center}
    in $\TwoCat{M}$.
    The following are equivalent:
    \begin{enumerate}[label=(\roman*)]
        \item\label{point:Beck--Chevalley:BC} The square $\square$ in $\TwoCat{M}$ satisfies the left Beck--Chevalley condition.
        \item The square $\op{\square}$ in $\op{\TwoCat{M}}$ satisfies the right Beck--Chevalley condition.
        \item The square $\co{\square}$ in $\co{\TwoCat{M}}$ satisfies the right Beck--Chevalley condition.
        \item\label{point:Beck--Chevalley:coop-BC} The square $\coop{\square}$ in $\coop{\TwoCat{M}}$ satisfies the left Beck--Chevalley condition.
    \end{enumerate}
    In the case that $\eta$ maps in the opposite direction, the analogous statements where we exchange ``left'' and ``right'' are equivalent.

    Moreover, $\square$ satisfies the double Beck--Chevalley condition if and only if $\coop{\square}$ does.
\end{lem}
\begin{proof}
    Using the descriptions of unit and counit from the proof of \Cref{lem:two-cat:op-adjunctions}, 
    it follows immediately that $\op{(\mathrm{Ex}^{\mathrm{L}}(\square))} \cong \mathrm{Ex}^{\mathrm{R}}(\op{\square})$ and $\co{(\mathrm{Ex}^{\mathrm{L}}(\square))} \cong \mathrm{Ex}^{\mathrm{R}}(\co{\square})$ 
    (and similarly for exchanged roles of $L$ and $R$).
    Combining these, we also get $\coop{(\mathrm{Ex}^{\mathrm{L}}(\square))} \cong \op{\mathrm{Ex}^{\mathrm{R}}(\co{\square})} \cong \mathrm{Ex}^{\mathrm{L}}(\coop{\square})$.
    This shows the equivalence of \ref{point:Beck--Chevalley:BC} through \ref{point:Beck--Chevalley:coop-BC}.
    The last statement about the double Beck--Chevalley condition follows analogously using that $\coop{(\mathrm{Ex}^{\mathrm{L}}(\square)^{-1})} = (\coop{\mathrm{Ex}^{\mathrm{L}}(\square)})^{-1} = \mathrm{Ex}^{\mathrm{L}}(\coop{\square})^{-1}$.
\end{proof}

%% file: 2FF/consequences.tex
\subsection{Consequences of the existence of a \texorpdfstring{$2$}{2}-functor formalism} \label{sec:2FF-consequences}
In this section, we collect some consequences of the existence of a \texorpdfstring{$2$}{2}-functor formalism.
The most important result will be a base change result for cartesian squares with one edge in $E$.

First, for now and the rest of the paper, we will fix the following notation:
\begin{notation}
    Suppose we are given a Nagata setup $(\Cat C, E, I, P)$, a 2-category $\TwoCat{M}$,
    and a $2$-functor formalism $\Cat D \colon \SpanStd \to \TwoCat{M}$.
    %We will adopt the following notation:
    %\begin{align*}
    %    \defnot{\Cat D_*} \colon \Cat C \to &\SpanStd \xrightarrow{\Cat D} \TwoCat{M}, \text{ and}\\
    %    \defnot{\Cat D^!} \colon \op{E} \to &\SpanStd \xrightarrow{\Cat D} \TwoCat{M}.
    %\end{align*}
    Suppose that $f \in \Cat C$ and $g \in E$ are two morphisms.
    We then write 
    \begin{itemize}
        \item $\defnot{f_*} \coloneqq \Cat D_*(f)$,
        \item $\defnot{f^{\times}}$ for a right adjoint of $f_*$ (if it exists),
        \item $\defnot{f^{*}}$ for a left adjoint of $f_*$ (if it exists),
        \item $\defnot{g^{!}} \coloneqq \Cat D^{!}(g)$, and
        \item $\defnot{g_\flat}$ for a right adjoint of $g^!$ (if it exists).
    \end{itemize}
\end{notation}

\begin{setting}
    We fix a Nagata setup $(\Cat C, E, I, P)$, a 2-category $\TwoCat{M}$,
    and a $2$-functor formalism $$\Cat D \colon \Span{\Cat C}{E}{\all}{I}{P} \to \TwoCat{M}.$$
\end{setting}

\begin{lem} \label{lem:consequences:adjoints-of-upper-shriek}
    Let $f \colon X \to Y$ be a morphism in $\Cat C$.
    \begin{itemize}
        \item If $f \in I$, then $f^!$ is left adjoint to $f_*$, i.e.\ $f^! \cong f^*$.
        \item If $f \in P$, then $f^!$ is right adjoint to $f_*$, i.e.\ $f^! \cong f^\times$.
    \end{itemize}
\end{lem}
\begin{proof}
    As functors between $2$-categories preserve adjunctions, 
    it suffices to show that the span $$Y \xleftarrow{f} X \xrightarrow{\id{X}} X$$
    is left, respectively right, adjoint to the span $$X \xleftarrow{\id{X}} X \xrightarrow{f} Y$$
    in the $2$-category $\SpanStd$ for $f \in I$, respectively $f \in P$.
    It was shown in the proof of \cite[Lemma 12.3]{Haugseng_2017} that this is the case in the $2$-category $\Span{\Cat C}{\all}{\all}{\all}{\all}$\footnote{For $\Span{\Cat C}{\all}{\all}{\all}{\all}$ to be well-defined one has to assume that $\Cat C$ admits all pullbacks, which we do not do here.
    However, the proof of \loccit is only giving the unit and counit of the adjunction and computing that they satisfy the triangle identities. 
    Thus, the computational argument for the triangle identities in the proof of \loccit remains valid in our situation.}.
    Thus, it is enough to show that unit and counit of the adjunctions lie in the sub-$2$-category $\SpanStd$.
    We only prove this for the case $f \in I$, the other case being similar.
    By the proof of \cite[Lemma 12.3]{Haugseng_2017} the unit of the adjunction
    \begin{equation*}
        [Y \xleftarrow{f} X \xrightarrow{\id{X}} X] \dashv [X \xleftarrow{\id{X}} X \xrightarrow{f} Y]
    \end{equation*}
    is given by the $2$-morphism
    \begin{center}
        \begin{tikzcd}
            & Y \ar[dl, "\id{Y}"'] \ar[dr, "\id{Y}"] & \\
            Y & X \ar[u, "f"'] \ar[l, "f"'] \ar[r, "f"] \ar[d, "\id{X}"] & Y \\
            & X\rlap{,} \ar[ul, "f"] \ar[ur, "f"'] &
        \end{tikzcd}
    \end{center}
    whereas the counit is given by the $2$-morphism
    \begin{center}
        \begin{tikzcd}
            & X \times_{Y} X \ar[dl, "\pr_1"'] \ar[dr, "\pr_2"] & \\
            X & X \ar[u, "\Delta"'] \ar[l, "\id{X}"'] \ar[r, "\id{X}"] \ar[d, "\id{X}"] & X \\
            & X\rlap{,} \ar[ul, "\id{X}"] \ar[ur, "\id{X}"'] &
        \end{tikzcd}
    \end{center}
    where $\Delta$ denotes the diagonal.
    It is clear that the unit is a $2$-morphism in $\SpanStd$ as we assumed $f \in I$.
    It is left to show that $\Delta \in I$.
    For this, note that $\pr_1 \in I$ as $I$ is stable under base change.
    Thus, $\Delta \in I$ follows from $\pr_1 \circ \Delta = \id{X}$ and the left cancellation property of $I$. 
\end{proof}

\begin{lem} \label{lem:consequences:composition}
    Let $f \colon X \to Y$ and $g \colon Y \to Z$ be composable morphisms in $\Cat C$.
    Then the following holds:
    \begin{enumerate} [label=\textbf{(\alph*)}, ref=\textbf{(\alph*)}]
        \item \label{point:consequences:composition-point-1} $g_* f_* \cong (gf)_*$.
        \item \label{point:consequences:composition-point-2} If $f_*$ and $g_*$ have left adjoints, then $f^* g^* \cong (gf)^*$.
        \item \label{point:consequences:composition-point-5} If $f_*$ and $g_*$ have right adjoints, then $f^\times g^\times \cong (gf)^\times$.
        \item \label{point:consequences:composition-point-3} For $f, g \in E$ we have $f^! g^! \cong (gf)^!$.
        \item \label{point:consequences:composition-point-4} For $f, g \in E$, if $f^!$ and $g^!$ have right adjoints, then $g_\flat f_\flat \cong (gf)_\flat$.
    \end{enumerate}
    Note that if $\Cat D$ is a $4\frac{1}{2}$-functor formalism with respect to a class of morphisms $T \subseteq \Cat C$,
    then the condition in \ref{point:consequences:composition-point-2} is satisfied if $f, g \in T$,
    the condition in \ref{point:consequences:composition-point-5} is satisfied for arbitrary $f$ and $g$,
    and the condition in \ref{point:consequences:composition-point-4} is satisfied whenever $f, g \in E$.
\end{lem}
\begin{proof}
    For \ref{point:consequences:composition-point-1} we compute
    \begin{align*}
        g_*f_* &= \Cat D (Y \xleftarrow{\id{Y}} Y \xrightarrow{g} Z) \Cat D (X \xleftarrow{\id{X}} X \xrightarrow{f} Y) \\
        &\cong \Cat D (X \xleftarrow{\id{X}} X \xrightarrow{gf} Z ) \\
        &= (gf)_*,
    \end{align*}
    where we used the composition law in $\Span{\Cat C}{E}{\all}{I}{P}$ in the second isomorphism.

    For \ref{point:consequences:composition-point-2} note that $f^*g^*$ is left adjoint to $g_*f_*$ and $(gf)^*$ is left adjoint to $(gf)_*$. Thus, the claim follows from \ref{point:consequences:composition-point-1} and the uniqueness of left adjoints up to isomorphism.

    The other points \ref{point:consequences:composition-point-5}, \ref{point:consequences:composition-point-3} and \ref{point:consequences:composition-point-4} can be proven analogously to \ref{point:consequences:composition-point-1} and \ref{point:consequences:composition-point-2}.
\end{proof}

Using the above lemmas, we can give a simple description of the image of a $2$-morphism under $\Cat D$
in terms of units and counits of adjunctions.

\begin{lem}
    Let 
    \begin{center}
        \begin{tikzcd}
            & Z \ar[dl, "f"'] \ar[dr, "g"] & \\
            X & Z'' \ar[u, "i"'] \ar[d, "p"] \ar[r, "g''"] \ar[l, "f''"'] & Y \\
            & Z' \ar[ul, "f'"] \ar[ur, "g'"']
        \end{tikzcd}
    \end{center}
    be a $2$-morphism in $\SpanStd$ (in particular, $i \in I$, $p \in P$ and $f, f', f'' \in E$).
    The image of this $2$-morphism under $\Cat D$ is given by the composition
    \begin{equation*}
        g_* f^! \xrightarrow{u} g_* i_* i^! f^! \simeq (gi)_* (fi)^! \simeq (g'p)_* (f'p)^! \simeq g'_* p_* p^! f'^! \xrightarrow{c} g'_* f'^!, 
    \end{equation*}
    where $u$ denotes the unit of the adjunction $i^! \dashv i_*$, 
    $c$ denotes the counit of the adjunction $p_* \dashv p^!$ (for both see \cref{lem:consequences:adjoints-of-upper-shriek}), 
    and the equivalences follow from \cref{lem:consequences:composition} 
    and the commutativity of the diagram.
\end{lem}
\begin{proof}
    The composition of functors described in the lemma is obtained by applying $\Cat D$
    to a certain $2$-morphism of $\SpanStd$. 
    We show that this $2$-morphism is the one given in the lemma. 
    Recall from the proof of \cref{lem:consequences:adjoints-of-upper-shriek} 
    that the unit of the adjunction $i^! \dashv i_*$ is obtained by applying $\Cat D$
    to the $2$-morphism
    \begin{center}
        \begin{tikzcd}
            & Z \ar[dl, "\id{Z}"'] \ar[dr, "\id{Z}"] & \\
            Z & Z'' \ar[u, "i"'] \ar[d, "\id{Z''}"] \ar[r, "i"] \ar[l, "i"'] & Z \\
            & Z''\rlap{.} \ar[ul, "i"] \ar[ur, "i"']
        \end{tikzcd}
    \end{center}
    Similarly, the counit of the adjunction $p_* \dashv p^!$ is obtained by applying $\Cat D$ to the $2$-morphism
    \begin{center}
        \begin{tikzcd}
            & Z'' \ar[dl, "p"'] \ar[dr, "p"] & \\
            Z' & Z'' \ar[u, "\id{Z''}"'] \ar[d, "p"] \ar[r, "p"] \ar[l, "p"'] & Z' \\
            & Z'\rlap{.} \ar[ul, "\id{Z'}"] \ar[ur, "\id{Z'}"']
        \end{tikzcd}
    \end{center}
    Using this, we compute that the composition of functors described in the lemma is obtained by the
    following composition of $2$-morphisms:
    \begin{center}
        \begin{tikzcd}
            & Z \ar[dl, "f"'] \ar[dr, "g"] & \\
            X & Z'' \ar[u, "i"'] \ar[d, "\id{Z''}"] \ar[r, "gi"] \ar[l, "fi"'] & Y \\
            & Z'' \ar[ul, "fi"] \ar[ur, "gi"']  \ar[dl, "f'p"'] \ar[dr, "g'p"] & \\
            X & Z'' \ar[u, "\id{Z''}"'] \ar[d, "p"] \ar[r, "g'p"] \ar[l, "f'p"'] & Y \\
            & Z'\rlap{.} \ar[ul, "f'"] \ar[ur, "g'"']
        \end{tikzcd}
    \end{center}
    Note that $fi=f'p=f''$ and $gi=g'p=g''$ by assumption.
    Hence, simplifying the above composition by computing the composition of the two vertical spans,
    we obtain the $2$-morphism we started with.
\end{proof}

\begin{defn}
    Let $f \colon X \to Y \in E$ be a morphism.
    A \defemphindex{compactification of $f$}{compactification} is a factorization of $f$ as
    \begin{equation*}
        X \xrightarrow{j} \overline{X} \xrightarrow{p} Y,
    \end{equation*}
    with $j \in I$ and $p \in P$. Any morphism $f \in E$ 
    has a compactification, since by assumption $(I,P)$
    is a suitable decomposition of $E$.
\end{defn}

\begin{lem} \label{lem:consequences:compactification}
    Let $f \colon X \to Y \in E$ be a morphism  
    with a compactification $f = pj$.
    Then there is a canonical equivalence 
    \begin{equation*}
        f^! \cong j^* p^\times,
    \end{equation*}
    where $j^*$ is left adjoint to $j_*$,
    and $p^\times$ is right adjoint to $p_*$.
\end{lem}
\begin{proof}
    From \cref{lem:consequences:composition},
    we know that $f^! \cong j^! p^!$.
    But now $j^!$ is a left adjoint to $j_*$ by \Cref{lem:consequences:adjoints-of-upper-shriek}, and hence $j^! \cong j^*$.
    Analogously, we also get $p^! \cong p^\times$ by \loccit.
\end{proof}

\begin{lem}[Base change for $I$] \label{lem:consequences:BC-for-I}
     Let $\square$ be a cartesian diagram in $\Cat C$ of the form
    \begin{center}
        \begin{tikzcd}
            Y' \ar[r, "g'"] \ar[d, "j'"'] \cartsymb &X' \ar[d, "j"] \\
            Y \ar[r, "g"'] & X,
        \end{tikzcd}
    \end{center}
    and suppose that $j \in I$.
    Then $\Cat D_*(\square)$ satisfies the left Beck--Chevalley condition.
\end{lem}
\begin{proof}
    First note that $j_*$ and $j'_*$ admit left adjoints $j^!$ and $j'^!$, respectively, by \cref{lem:consequences:adjoints-of-upper-shriek}. 
    We have to show that the map
    \begin{equation*}
        \mathrm{Ex}^{\mathrm{L}}(\Cat D_*(\square))\colon j^! g_* \xrightarrow{u} j^! g_* j'_* j'^! \cong j^! j_* g'_* j'^! \xrightarrow{c} g'_* j'^!
    \end{equation*}
    is an equivalence, where $u$, respectively $c$, denote the unit, respectively counit, of the corresponding adjunction.
    We work directly in the category $\SpanStd$, i.e.\ before applying the functor $\Cat D$.
    This is possible since $2$-functors preserve adjunctions and thus in particular the mate of a square.
    Hence, if this mate is an equivalence before applying $\Cat D$, the same is true afterwards.
    The lemma follows from a similar calculation as the one in the proof of \cite[Proposition 4.14]{cnossen2025universalityspan2categoriesconstruction}:
    Unwinding the definition of the unit and counit (which we gave in the proof of \cref{lem:consequences:adjoints-of-upper-shriek}), 
    and of the composition of $2$-morphisms, 
    we see that $\mathrm{Ex}^{\mathrm{L}}(\Cat D_*(\square))$ is given by the following composition of $2$-morphisms:
    \begin{center}
        \begin{tikzcd}
            && Y \times_{X} X' \ar[ddll, "\pr_1"'] \ar[ddrr, "\pr_2"] && \\
            &&&& \\
            Y && Y' \times_{X'} X' \times_{X} X' \ar[uu, "{(j'\pr_1, \pr_3)}"'] \ar[rr, "\pr_3"] \ar[dd, "\id{}"] \ar[ll, "j'\pr_1"'] && X' \\
            &&&& \\
            && Y' \times_{X'} X' \times_{X} X' \ar[uull, "j'\pr_1"] \ar[uurr, "\pr_3"'] \ar[ddll, "j'\pr_1"'] \ar[ddrr, "\pr_3"] && \\
            &&&& \\
            Y && Y' \ar[uu, "{(\id{Y'}, \Delta g')}"'] \ar[rr, "g'"] \ar[dd, "\id{Y'}"] \ar[ll, "j'"'] && X' \\
            &&&& \\
            && Y'\rlap{,} \ar[uull, "j'"] \ar[uurr, "g'"'] && 
        \end{tikzcd}
    \end{center}
    where $\Delta \colon X' \to X' \times_{X} X'$ denotes the diagonal.
    This composition is in turn given by the $2$-morphism
    \begin{center}
        \begin{tikzcd}
            && Y \times_{X} X' \ar[ddll, "\pr_1"'] \ar[ddrr, "\pr_2"] && \\
            &&&& \\
            Y && Y' \ar[uu, "{(j', g')}"'] \ar[rr, "g'"] \ar[dd, "\id{Y'}"] \ar[ll, "j'"'] && X' \\
            &&&& \\
            && Y'\rlap{,} \ar[uull, "j'"] \ar[uurr, "g'"'] && 
        \end{tikzcd}
    \end{center}
    where the morphism $(j',g')$ is exactly the morphism exhibiting $Y'$ as the pullback $Y \times_{X} X'$.
    In particular, it is an equivalence.
\end{proof}

\begin{lem}[Base change for $P$] \label{lem:consequences:BC-for-P}
     Let $\square$ be a cartesian diagram in $\Cat C$ of the form
    \begin{center}
        \begin{tikzcd}
            Y' \ar[r, "g'"] \ar[d, "p'"'] \cartsymb &X' \ar[d, "p"] \\
            Y \ar[r, "g"'] & X,
        \end{tikzcd}
    \end{center}
    and suppose that $p \in P$.
    Then $\Cat D_*(\square)$ satisfies the right Beck--Chevalley condition.
\end{lem}
\begin{proof}
    First note that $p_*$ and $p'_*$ admit right adjoints by \cref{lem:consequences:adjoints-of-upper-shriek}.
    The proof is now dual to the one of the previous lemma.
\end{proof}

\begin{cor}[Base change] \label{lem:consequences:BC-for-E}
    Let $\square$ be a cartesian diagram in $\Cat C$ of the form
    \begin{center}
        \begin{tikzcd}
            Y' \ar[r, "g'"] \ar[d, "f'"'] \cartsymb &X' \ar[d, "f"] \\
            Y \ar[r, "g"'] & X,
        \end{tikzcd}
    \end{center}
    and suppose that $f \in E$.
    Then there is an equivalence
    \begin{equation*}
        f^! g_* \cong g'_* f'^!.
    \end{equation*}
\end{cor}
\begin{proof}
    Choose a compactification $f = p \circ j$ with $p \in P$ and $j \in I$.
    Hence, the above square decomposes as 
    \begin{center}
        \begin{tikzcd}
            Y' \ar[r, "g'"] \ar[d, "j'"'] \cartsymb &X' \ar[d, "j"] \\
            Y \times_X Z \ar[r, "g''"] \ar[d, "p'"'] \cartsymb &Z \ar[d, "p"] \\
            Y \ar[r, "g"'] & X,
        \end{tikzcd}
    \end{center}
    In particular, $j'$ and $p'$ give a compactification of $f'$.
    By \cref{lem:consequences:BC-for-I,lem:consequences:BC-for-P} the upper square satisfies the left Beck--Chevalley condition,
    and the lower square the right Beck--Chevalley condition.
    In particular, we get 
    \begin{equation*}
        f^! g_* \cong j^* p^\times g_* \cong j^* g''_* p'^\times \cong g'_* j'^* p'^\times \cong g'_* f'^!, 
    \end{equation*}
    where we used \cref{lem:consequences:compactification} twice.
\end{proof}

\begin{lem}[Double Beck–Chevalley for $(I,P)$-squares] \label{lem:consequences:contravariant-exchange}
     Let $\square$ be a cartesian diagram in $\Cat C$ of the form
    \begin{center}
        \begin{tikzcd}
            Y' \ar[r, "g'"] \ar[d, "f'"'] \cartsymb &X' \ar[d, "f"] \\
            Y \ar[r, "g"'] & X,
        \end{tikzcd}
    \end{center}
    and suppose that $f \in I$ and $g \in P$.
    Then $\Cat D_*(\square)$ satisfies the double Beck--Chevalley condition.
\end{lem}
\begin{proof}
    The proof of this is analogous to the proofs of \cref{lem:consequences:BC-for-I,lem:consequences:BC-for-P}.
\end{proof}

%% file: 2FF/construction-of-2FF.tex
\subsection{Construction of \texorpdfstring{$2$}{2}-functor formalisms} \label{sec:constr-2FF}
In this section, we recall how to construct $2$-functor formalisms 
from a manageable amount of data,
using the universal property of the $2$-categories of spans,
as shown by Cnossen--Lenz--Linskens in \cite{cnossen2025universalityspan2categoriesconstruction}.

\begin{setting}
    In this subsection, we fix a 2-category $\TwoCat{M}$,
    a category $\Cat C$, and two wide subcategories $I, P \subseteq \Cat C$
    which are closed under base change.
\end{setting}

In \cite[Definition 2.4]{cnossen2025universalityspan2categoriesconstruction}, the authors define the notion
of \defemphindex{$(I,P)$-biadjointability}{biadjointable} for a contravariant functor $\op{\Cat C} \to \TwoCat{M}$.
We need the following analogous notion for covariant functors:

\begin{defn} \label{def:cff:suitable-decomposition-all-BC}
    Let $\Cat D_* \colon \Cat C \to \TwoCat{M}$ be a functor.
    We say that $\Cat D_*$ \defemphindex{satisfies all Beck--Chevalley conditions with respect to $(I, P)$}{Beck--Chevalley condition!satisfies all}
    if the following conditions hold:
    \begin{enumerate}[label=\textbf{(BC\arabic{enumi})}, ref=\textbf{(BC\arabic{enumi})}]
        \item \label{axiom:all-BC:I-left-adjoint} For $j \in I$, the morphism $\Cat D_*(j)$ admits a left adjoint morphism $j^*$.
        \item \label{axiom:all-BC:I-left-BC} For every cartesian square $\square$ of the form 
            \begin{center}
                \begin{tikzcd}
                    U' \ar[r, "g'"] \ar[d, "j'"'] \cartsymb &U \ar[d, "j"] \\
                    X' \ar[r, "g"'] &X\rlap{,}
                \end{tikzcd}
            \end{center}
            with $j \in I$,
            the square $\Cat D_*(\square)$ satisfies the 
            left Beck--Chevalley condition. 
        \item \label{axiom:all-BC:P-right-adjoint} For $f \in P$, the morphism $\Cat D_*(f)$ admits a right adjoint morphism $f^!$.
        \item \label{axiom:all-BC:P-right-BC} For every cartesian square $\square$ of the form 
            \begin{center}
                \begin{tikzcd}
                    Y' \ar[r, "g'"] \ar[d, "f'"'] \cartsymb &Y \ar[d, "f"] \\
                    X' \ar[r, "g"'] &X\rlap{,}
                \end{tikzcd}
            \end{center} 
            with $f \in P$,
            the square $\Cat D_*(\square)$ satisfies the 
            right Beck--Chevalley condition. 
        \item \label{axiom:all-BC:I-P-double-BC} For every cartesian diagram $\square$ of the form 
            \begin{center}
                \begin{tikzcd}
                    U' \ar[r, "f'"] \ar[d, "j'"'] \cartsymb &U \ar[d, "j"] \\
                    X' \ar[r, "f"'] &X\rlap{,}
                \end{tikzcd}
            \end{center}  
            with $j \in I$ and $f \in P$, the square $\Cat D_*(\square)$ 
            satisfies the double Beck--Chevalley condition.
    \end{enumerate}
    Moreover, if $\Cat D_*, \Cat D'_* \colon \Cat C \to \TwoCat{M}$ are two functors satisfying all Beck--Chevalley conditions,
    then a natural transformation $\alpha \colon \Cat D_* \to \Cat D'_*$
    \defemphindex{satisfies all Beck--Chevalley conditions with respect to $(I,P)$}{Beck--Chevalley condition!satisfies all}
    if for each $f \colon X \to Y$ in $\Cat C$, the commutative diagram
    \begin{center}
        \begin{tikzcd}
            \Cat D_*(X) \ar[r, "\alpha_X"] \ar[d,"\Cat D_*(f)"'] \commsymb & \Cat D'_*(X) \ar[d,"\Cat D'_*(f)"] \\
            \Cat D_*(Y) \ar[r, "\alpha_Y"'] & \Cat D'_*(Y)
        \end{tikzcd}
    \end{center}    
    satisfies the left Beck--Chevalley condition if $f \in I$ 
    and the right Beck--Chevalley condition if $f \in P$.

    If $\alpha \colon \Cat D_* \to \Cat D'_*$ is an oplax natural transformation, then it 
    \defemphindex{satisfies all Beck--Chevalley conditions with respect to $(I,P)$}{Beck--Chevalley condition!satisfies all}
    if it is a strong natural transformation with respect to $P$ and 
    if for each $f \colon X \to Y$ in $I$, the diagram
    \begin{center}
        \begin{tikzcd}
            \Cat D_*(X) \ar[r, "\alpha_X"] \ar[d,"\Cat D_*(f)"'] & \Cat D'_*(X) \ar[d,"\Cat D'_*(f)"] \\
            \Cat D_*(Y) \ar[r, "\alpha_Y"'] \laxarrowur{} & \Cat D'_*(Y)
        \end{tikzcd}
    \end{center}    
    satisfies the vertical left Beck--Chevalley condition.

    We denote by $\defnot{\TwoFunBC{\Cat C}{\TwoCat{M}}} \subseteq \TwoFun{\Cat C}{\TwoCat{M}}$ (resp.\ $\defnot{\FunOpLaxIPBC{\Cat C}{\TwoCat{M}}} \subseteq \FunOpLax{\Cat C}{\TwoCat{M}}$)
    the $2$-full sub-$2$-category 
    spanned by the functors satisfying all Beck--Chevalley conditions and (oplax) natural transformations satisfying all Beck--Chevalley conditions.
\end{defn}

The above notion of satisfying all Beck--Chevalley conditions and the notion of $(I,P)$-biadjointability
are related in the following way:

\begin{lem} \label{lem:constr-2cat:allBC-is-IPbadj}
    A functor $F \colon \Cat C \to \TwoCat{M}$ satisfies all Beck--Chevalley conditions with respect to $(I,P)$
    if and only if the functor $\coop{F} \colon \coop{\Cat C} \to \coop{\TwoCat{M}}$ is $(I,P)$-biadjointable
    in the sense of \cite[Definition 2.4]{cnossen2025universalityspan2categoriesconstruction} 
    (note that $\coop{\Cat C} = \op{\Cat C}$).

    More generally, the functor $\coop{(-)} \colon \CatInfTwo \to \CatInfTwo$ induces 
    equivalences of $2$-categories
    $$\TwoFunBC{\Cat C}{\TwoCat{M}} \simeq \coop{\TwoFunBadj{\coop{\Cat C}}{\coop{\TwoCat{M}}}{I}{P}}$$ and  
    $$\FunOpLaxIPBC{\Cat C}{\TwoCat{M}} \simeq \coop{\FunOpLaxBadj{\coop{\Cat C}}{\coop{\TwoCat{M}}}},$$
    where the category $\defnot{\TwoFunBadj{\coop{\Cat C}}{\coop{\TwoCat{M}}}{I}{P}}$ is the one from \cite[Definition 2.4]{cnossen2025universalityspan2categoriesconstruction} 
    and the category $\defnot{\FunOpLaxBadj{\coop{\Cat C}}{\coop{\TwoCat{M}}}}$ is the $2$-full sub-$2$-category of $\FunOpLax{\coop{\Cat C}}{\coop{\TwoCat{M}}}$ spanned by the 
    $(I,P)$-biadjointable functors and $(I,P)$-biadjointable oplax natural transformations in the sense of
    \cite[Remark 6.8]{cnossen2025universalityspan2categoriesconstruction}.
\end{lem}
\begin{proof}
    Since $\coop{(-)}$ induces an equivalence $\TwoFun{\Cat C}{\TwoCat{M}} \xrightarrow{\simeq} \coop{\TwoFun{\coop{\Cat C}}{\coop{\TwoCat{M}}}}$ by \cref{lem:two-cat:coop}
    and an equivalence $\FunOpLax{\Cat C}{\TwoCat{M}} \xrightarrow{\simeq} \coop{\FunOpLax{\coop{\Cat C}}{\coop{\TwoCat{M}}}}$ by \Cref{lem:lax-nat-transf:coop-lax-oplax},
    it suffices to show that the imposed conditions on objects and morphisms correspond to each other
    under these equivalences.
    It suffices to note the following: 
    A morphism $f$ in $\TwoCat{M}$ has a left adjoint (resp.\ right adjoint)
    if and only if the morphism $\coop{f}$ in $\coop{\TwoCat{M}}$ has a left adjoint (resp.\ right adjoint) by \Cref{lem:two-cat:op-adjunctions}.
    Similarly, a square $(\square)$ in $\TwoCat{M}$ satisfies the left (resp.\ right, resp.\ double) Beck--Chevalley 
    condition if and only if the square $\coop{(\square)}$ in $\coop{\TwoCat{M}}$ satisfies 
    the left (resp.\ right, resp.\ double) Beck--Chevalley condition by \Cref{lem:Beck--Chevalley:op-BC}.
    Moreover, an oplax square $(\square)$ is commutative if and only if the oplax square $\coop{(\square)}$ is commutative.
\end{proof}

\begin{setting}
    For the rest of this subsection, we fix a Nagata setup $(\Cat C, E, I, P)$.
\end{setting}
We will compare the $2$-category $\TwoFunBC{\Cat C}{\TwoCat{M}}$
with the $2$-category of $2$-functor formalisms $\TwoFun{\SpanStd}{\TwoCat{M}}$.

\begin{lem} \label{lem:constr-2cat:D_*-allBC}
    Let $\Cat D, \Cat D' \colon \SpanStd \to \TwoCat{M}$ be $2$-functor formalisms with values in $\TwoCat{M}$.
    The functor $\Cat C \to \SpanStd \xrightarrow{\Cat D} \TwoCat{M}$ 
    satisfies all Beck--Chevalley conditions with respect to $(I,P)$.

    Similarly, if $\epsilon \in \{\,\operatorname{oplax},\, \emptyset\, \}$ and $\varphi \colon \Cat D \to \Cat D'$ is an $\epsilon$-natural transformation of $2$-functor formalisms,
    then the induced $\epsilon$-natural transformation 
    $\varphi_* \colon \Cat D_* \to \Cat D'_*$ satisfies all Beck--Chevalley conditions with respect to $(I,P)$.
\end{lem}
\begin{proof}
    By \cite[Theorems 4.20 and 6.7]{cnossen2025universalityspan2categoriesconstruction}, precomposition with $h \colon \op{\Cat C} \to \Span{\Cat C}{\all}{E}{P}{I}$ induces
    $$\FunEps{\Span{\Cat C}{\all}{E}{P}{I}}{\coop{\TwoCat{M}}} \xrightarrow{- \circ h} \FunEpsBadj{\op{\Cat C}}{\coop{\TwoCat{M}}}.$$
    We apply $\coop{(-)}$ to this functor.
    As $\coop{h} \colon \Cat C \to \SpanStd$ is the canonical functor, we get that precomposition with it defines the top horizontal functor in the commutative diagram
    \begin{center}
        \begin{tikzcd}[column sep=huge]
            \FunEps{\SpanStd}{\TwoCat{M}} \ar[r, "- \circ \coop{h}"] \ar[d, "\simeq"'] & \FunEpsBC{\Cat C}{\TwoCat{M}} \ar[d, "\simeq"] \\
            \coop{\FunEps{\Span{\Cat C}{\all}{E}{P}{I}}{\coop{\TwoCat{M}}}} \ar[r, "\coop{(- \circ h)}"'] & \coop{\FunEpsBadj{\op{\Cat C}}{\coop{\TwoCat{M}}}} \rlap{.}
        \end{tikzcd}
    \end{center}
    Here, for the equivalence on the left we use \Cref{lem:two-cat:coop} if $\epsilon = \emptyset$, or 
    \cref{lem:lax-nat-transf:coop-lax-oplax} if $\epsilon = \mathrm{oplax}$.
    For the equivalence on the right we use \Cref{lem:constr-2cat:allBC-is-IPbadj}.
\end{proof}

The previous lemma tells us that there are restriction functors
\begin{equation*}
        \TwoFun{\Span{\Cat C}{E}{\all}{I}{P}}{\TwoCat{M}} \to \TwoFunBC{\Cat C}{\TwoCat{M}}
\end{equation*}   
and
\begin{equation*}
    \FunOpLax{\Span{\Cat C}{E}{\all}{I}{P}}{\TwoCat{M}} \to \FunOpLaxIPBC{\Cat C}{\TwoCat{M}}
\end{equation*} 
which are both given by precomposition with the functor $\Cat C \to \SpanStd$. 
The next theorem shows that these restriction functors are equivalences.
In particular, extensions of functors satisfying all Beck--Chevalley conditions to $2$-functor formalisms exist and are unique.

\begin{thm} [Cnossen--Lenz--Linskens] \label{lem:constr-2cat:ext-and-unique} \uses{lem:constr-2cat:D_*-allBC}
    The restriction functors 
    \begin{equation*}
        \TwoFun{\Span{\Cat C}{E}{\all}{I}{P}}{\TwoCat{M}} \to \TwoFunBC{\Cat C}{\TwoCat{M}}
    \end{equation*}    
    and 
    \begin{equation*}
        \FunOpLax{\Span{\Cat C}{E}{\all}{I}{P}}{\TwoCat{M}} \to \FunOpLaxIPBC{\Cat C}{\TwoCat{M}}
    \end{equation*}
    are equivalences of $2$-categories.
\end{thm}
\begin{proof}
    We only prove that the first restriction functor is an equivalence.
    The proof for the second one works exactly the same by replacing every $\tightunderline{\operatorname{Fun}}$ below by $\tightunderline{\operatorname{Fun}}^{\operatorname{oplax}}$.
    Consider the following commutative diagram where the horizontal functors are given by restriction:
    \begin{center}
        \begin{tikzcd}
            \TwoFun{\Span{\Cat C}{E}{\all}{I}{P}}{\TwoCat{M}} \ar[d, equal] \ar[r] &\TwoFun{\Cat C}{\TwoCat{M}} \ar[d, equal] \\
            \TwoFun{\coop{\Span{\Cat C}{\all}{E}{P}{I}}}{\TwoCat{M}} \ar[r] \ar[d, "\simeq"'] &\TwoFun{\Cat C}{\TwoCat{M}} \ar[d, "\simeq"] \\
            \coop{\TwoFun{\Span{\co{\Cat C}}{\all}{E}{P}{I}}{\coop{\TwoCat{M}}}} \ar[r] &\coop{\TwoFun{\coop{\Cat C}}{\coop{\TwoCat{M}}}} \rlap{.}
        \end{tikzcd}
    \end{center}
    The upper square is given by the definition of $\SpanStd$,
    whereas the lower square is given by \cref{lem:two-cat:coop} (resp.\ \Cref{lem:lax-nat-transf:coop-lax-oplax} for the oplax case),
    where we also use that $\coop{(-)}$ is self-inverse, and that $\co{\Cat C} = \Cat C$ 
    since $\Cat C$ is a $1$-category.
    Now, the right vertical functor restricts by \cref{lem:constr-2cat:allBC-is-IPbadj} to an equivalence 
    \begin{equation*}
        \TwoFunBC{\Cat C}{\TwoCat{M}} \xrightarrow{\simeq} \coop{\TwoFunBadj{\coop{\Cat C}}{\coop{\TwoCat{M}}}{I}{P}}.
    \end{equation*}
    The bottom horizontal functor is a locally full inclusion, with essential image $\coop{\TwoFunBadj{\coop{\Cat C}}{\coop{\TwoCat{M}}}{I}{P}}$ 
    by \cite[Theorem 4.20]{cnossen2025universalityspan2categoriesconstruction} (resp.\ \cite[Theorem 6.7]{cnossen2025universalityspan2categoriesconstruction} for the oplax case).
    Thus, we conclude that also the top horizontal functor 
    is a locally full inclusion, with essential image $\TwoFunBC{\Cat C}{\TwoCat{M}}$, 
    which is exactly what we wanted to show.
\end{proof}

%% file: 2FF/frobenius-structures.tex
\subsection{Frobenius structures and \texorpdfstring{$2$}{2}-functor formalisms of Cartier modules} \label{sec:Cart-2FF}
In this section, we explain how one gets out of an existing $2$-functor formalism $\Cat D$ 
a new functor formalism $\Cart(\Cat D, F)$ of Cartier-modules.
This needs a geometric input: Cartier modules on $X$ should be objects of the form $F_* M \to M$,
where $M \in \Cat D(X)$, and $F_*$ is pushforward along a \defemphindex{Frobenius endomorphism}{Frobenius} that is already 
defined on the geometry. This geometric input will be formalized in the following way:
We will assume that the category $\Cat C$ admits a Frobenius structure (cf.\ \cref{def:constr-cart:frob-structure}).
Informally, a Frobenius structure encodes for every geometric object $X \in \Cat C$
an endomorphism $F_X \colon X \to X$, that commutes with every morphism $f \colon X \to Y$ in $\Cat C$
in the following sense: $F_Y f \cong f F_X$.
As an example, one should think of $\Cat C$ as the category of $\mathbb{F}_p$-schemes for a prime $p$,
and of $F_X$ as the absolute Frobenius endomorphism on $X$.

\begin{defn} \label{def:constr-cart:frob-structure}
    Let $\Cat C$ be a category.
    A \defemph{Frobenius structure} on $\Cat C$ is 
    a natural transformation 
    \begin{equation*}
        F \colon \id{\Cat C} \to \id{\Cat C}.
    \end{equation*}
\end{defn}

\begin{defn}
    Let $\Cat C$ be a category with a Frobenius structure $F$.
    Let $f \colon X \to Y$ be a morphism in $\Cat C$.
    We write $\defnot{\square^F_f}$ for the following commutative square,
    given by naturality of $F$:
    \begin{center}
        \begin{tikzcd}
            X \ar[d, "f"'] \ar[r, "F_X"] \commsymb &X \ar[d, "f"] \\
            Y \ar[r, "F_Y"'] &Y\rlap{.}
        \end{tikzcd}
    \end{center}
\end{defn}

\begin{defn} \label{def:constr-cart:def-of-F_*}
    Let $\Cat C$ be a category with a Frobenius structure $F$.
    Let $\Cat D_* \colon \Cat C \to \TwoCatInf$ be a functor.
    We get a natural transformation $\defnot{F_*} \colon \Cat D_* \to \Cat D_*$ by horizontal composition of $F$ with $\Cat D_*$,
    i.e.\ it is given by the diagram 
    \begin{center}
        \begin{tikzcd}
            &\phantom{}\ar[dd, "F", shorten >=1.8em, shorten <=1.8em, Rightarrow] \\
            \Cat C \ar[rr, bend left, "\id{}"] \ar[rr, bend right, "\id{}"'] && \Cat C \ar[r, "\Cat D_*"] &\TwoCatInf \rlap{.}\\
            &\phantom{}
        \end{tikzcd}
    \end{center}
\end{defn}

\begin{rmk}
    Concretely, if $\Cat D_* \colon \Cat C \to \TwoCatInf$ is a functor, and $X \in \Cat C$, then the natural transformation 
    $F_* \colon \Cat D_* \to \Cat D_*$ is given on the object $X$ by 
    $$\defnot{F_{*, X}} \coloneqq \Cat D_*(F_X) \colon \Cat D_*(X) \to \Cat D_*(X).$$
\end{rmk}

\begin{setting}
    For this section let $(\Cat C, E, I, P)$ be a Nagata setup, 
    let $\Cat D \colon \Span{\Cat C}{E}{\all}{I}{P} \to \TwoCatInf$ be a $2$-functor formalism, and let $F$ be a Frobenius structure on $\Cat C$.
    We assume moreover that for every $j \in I$, the square $\Cat D_*(\square^F_j)$ satisfies the left Beck--Chevalley condition.
\end{setting}

\begin{notation}
    We will equip $\Span{\Cat C}{E}{\all}{I}{P}$ with the marking given by $\defnot{\Span{\Cat C}{I}{\all}{I}{P}}$,
    the $2$-full wide sub-$2$-category given by those spans with the left leg in $I$.
    We denote the marking by $\defnot{(I,\all)}$.
\end{notation}

\begin{lem} \label{lem:BN:F-star-partially-oplax-nat-trans}
    The natural transformation $F_* \colon \Cat D_* \to \Cat D_*$
    extends to a morphism 
    $$(F_* \colon \Cat D \to \Cat D) \in \FunPartOpLax{(I,\all)}{\Span{\Cat C}{E}{\all}{I}{P}}{\TwoCatInf},$$
    i.e.\ $F_*$ is a partially oplax natural transformation which is a natural transformation on the forward morphisms and on the backward morphisms lying in $I$.
\end{lem}
\begin{proof}
    By \Cref{lem:constr-2cat:ext-and-unique}, there is an equivalence of categories
    \begin{equation*}
        \FunOpLax{\Span{\Cat C}{E}{\all}{I}{P}}{\TwoCatInf} \xrightarrow{\simeq} \FunOpLaxIPBC{\Cat C}{\TwoCatInf} \, .
    \end{equation*}
    The natural transformation $F_* \colon \Cat D_* \to \Cat D_*$ satisfies all Beck--Chevalley conditions with respect to $(I,P)$ when seen as an oplax natural transformation:
    Indeed, as it is a natural transformation, it is in particular strong with respect to $P$, 
    and for $j \in I$ the square $\Cat D_*(\square^F_j)$ satisfies the left Beck--Chevalley condition by assumption. 
    Thus, $F_* \colon \Cat D_* \to \Cat D_*$ extends to an oplax natural transformation $F_* \colon \Cat D \to \Cat D$ via the above equivalence.
    It remains to show that this extension is actually a partially oplax natural transformation with respect to the $(I,\all)$-marking.
    For this, note that on forward morphisms the extension restricts to the strong natural transformation $F_* \colon \Cat D_* \to \Cat D_*$.
    Moreover, for $(j \colon Y \to X) \in I$ the oplax naturality square 
    \begin{center}
        \begin{tikzcd}
            \Cat D(X) \ar[r, "j^!"] \ar[d, "F_{*,X}"'] & \Cat D(Y) \ar[d, "F_{*,Y}"] \laxarrowdl{}\\
            \Cat D(X) \ar[r, "j^!"'] & \Cat D(Y)
        \end{tikzcd}
    \end{center}
    commutes by virtue of the assumption that $\Cat D_*(\square^F_j)$ satisfies the left Beck--Chevalley condition, noting that $j^! \cong j^*$ by \Cref{lem:consequences:adjoints-of-upper-shriek}.
\end{proof}

\begin{defn} \label{def:BN:Cart-Frob-pCFF}
    Consider the partially oplax natural transformation $F_* \colon \Cat D \to \Cat D$ from \cref{lem:BN:F-star-partially-oplax-nat-trans}.
    This data defines a functor $\defnot{\Psi} \colon \BN \to \FunPartOpLax{(I,\all)}{\Span{\Cat C}{E}{\all}{I}{P}}{\TwoCatInf}$ by \Cref{lem:BN:univ-property-2-cat}.
    We denote by $\defnot{\Cart(\Cat D, F_*)}$ the $2$-functor formalism $\laxlim(\Psi)$,
    which exists by \cite[Proposition 5.4.2 (i)]{abellán2026freefibrationslaxcolimits} and \Cref{lem:lax-lim:partial-lax-lim-existence}.
    The projection to the unique object of $\BN$ defines a morphism $$\defnot{U}\colon \Cart(\Cat D, F_*) \to \Cat D$$ in $\FunPartOpLax{(I,\all)}{\SpanStd}{\TwoCatInf}$. 
    Moreover, the lax limit comes equipped with a diagram 
    \begin{center}
        \begin{tikzcd}
            & \Cart(\Cat D, F_*) \ar[dl, "U"'] \ar[dr, "U"] & \rlap{} \\
            \Cat D \ar[rr, "F_*"'] \ar[urr, Rightarrow, shorten <= 35pt, shorten >= 50pt, "\defnot{\kappa}"] && \Cat D \rlap{,}
        \end{tikzcd}
    \end{center}
    i.e.\ $\kappa \colon F_* U \to U$ is a $2$-morphism in $\FunPartOpLax{(I,\all)}{\SpanStd}{\TwoCatInf}$.
\end{defn}

The name $\Cart(\Cat D, F_*)$ for this $2$-functor formalism is justified by the following lemmas.

\begin{lem} \label{lem:constr-cart:CartD-on-objects}
    Consider the $2$-functor formalism $\Cart(\Cat D, F_*)$ from \cref{def:BN:Cart-Frob-pCFF}.
    For each object $X$ of $\Cat C$, there is a pullback square
    \begin{center}
        \begin{tikzcd}
            \Cart(\Cat D, F_*)(X) \ar[r, "\defnot{U}"] \ar[d, "\defnot{\kappa}"'] \cartsymb & \Cat D(X) \ar[d, "{(F_{*,X}, \id{\Cat D(X)})}"] \\
            \Fun{\Delta^1}{\Cat D(X)} \ar[r, "{(s,t)}"'] & \Cat D(X) \times \Cat D(X)
        \end{tikzcd}
    \end{center}
    in $\TwoCatInf$.
    In particular, $\Cart(\Cat D, F_*)(X) = \Cart(\Cat D(X), F_{*,X})$.
\end{lem}
\begin{proof}
    By \cite[Proposition 5.4.2 (i)]{abellán2026freefibrationslaxcolimits}, lax limits in $\FunPartOpLax{(I,\all)}{\Span{\Cat C}{E}{\all}{I}{P}}{\TwoCatInf}$
    are computed pointwise, i.e.\ for any object $X$ of $\Span{\Cat C}{E}{\all}{I}{P}$ we have
    $\laxlim(\Psi)(X) \simeq \laxlim(\ev{X} \circ \Psi)$. But $\ev{X} \circ \Psi$ is the functor
    $\BN \to \TwoCatInf$ which is defined by the datum of $F_{*,X} \colon \Cat D(X) \to \Cat D(X)$ via \Cref{lem:BN:univ-property-2-cat}.
    Hence, the claim follows from \Cref{lem:BN:Cart-as-fully-lax-lim}.
\end{proof}

\begin{lem} \label{lem:constr-cart:CartD-compatibility-U-and-kappa}
    Consider the $2$-functor formalism $\Cart(\Cat D, F_*)$ from \cref{def:BN:Cart-Frob-pCFF}.
    Let $f \colon X \to Y$ be a morphism in $\Span{\Cat C}{E}{\all}{I}{P}$.
    There are commutative diagrams
    \begin{center}
        \begin{tikzcd}[column sep=huge]
            \Cart(\Cat D(X), F_{*,X}) \ar[r, "{\Cart(\Cat D, F_*)(f)}"] \ar[d, "U"'] & \Cart(\Cat D(Y),F_{*,Y}) \ar[d, "U"] \\
            \Cat D(X) \ar[r, "{\Cat D(f)}"'] & \Cat D(Y)     
        \end{tikzcd}
    \end{center}
    and
    \begin{center}
        \begin{tikzcd}
            F_{*,Y}U \Cart(\Cat D,F_*)(f) \ar[rr, "{\kappa \Cart(\Cat D,F_*)(f)}"] \ar[d, "\simeq"'] && U \Cart(\Cat D,F_*)(f) \ar[d, "\simeq"] \\
            F_{*,Y} \Cat D(f)U \ar[r] & \Cat D(f)F_{*,X}U \ar[r, "\Cat D(f)\kappa"'] & \Cat D(f)U \rlap{,} 
        \end{tikzcd}
    \end{center}
    where the vertical equivalences in the second diagram are given by the commutativity of the first diagram and the bottom left morphism is given by the oplax naturality of $F_*$.
    In particular, $U \colon \Cart(\Cat D, F_*) \to \Cat D$ is a natural transformation.
\end{lem}
\begin{proof}
    Let $\Psi \colon \BN \to \FunPartOpLax{(I,\all)}{\Span{\Cat C}{E}{\all}{I}{P}}{\TwoCatInf}$ denote the functor given by $F_*\colon \Cat D \to \Cat D$ via \Cref{lem:BN:univ-property-2-cat}, 
    so that $\Cart(\Cat D,F_*)=\laxlim(\Psi)$.
    Consider the marked functor $$\tilde{f} \colon (\Delta^1,\flat) \to (\Span{\Cat C}{E}{\all}{I}{P}, (I,\all))$$ which is given by choosing $f$.
    By \cite[Proposition 5.4.2 (ii)]{abellán2026freefibrationslaxcolimits}, the restriction
    \[-\circ \tilde{f} \colon \FunPartOpLax{(I,\all)}{\Span{\Cat C}{E}{\all}{I}{P}}{\TwoCatInf} \to \FunOpLax{\Delta^1}{\TwoCatInf}\]
    along this functor preserves lax limits.
    In particular, $\Cart(\Cat D, F_*)(f) = \laxlim(\Psi)(f) \cong \laxlim((-\circ \tilde{f}) \circ \Psi)$.
    Unwinding the definitions, we see that the functor $(-\circ\tilde{f}) \circ \Psi \colon \BN \to \FunPartOpLax{\flat}{\Delta^1}{\TwoCatInf}$ maps the 
    unique object of $\BN$ to $\Cat D(f)\colon \Cat D(X) \to \Cat D(Y)$ and the morphism $e$ to the diagram
    \begin{center}
        \begin{tikzcd}
            \Cat D(X) \ar[r, "\Cat D(f)"] \ar[d, "F_{*,X}"'] & \Cat D(Y) \ar[d, "F_{*,Y}"] \laxarrowdl{} \\
            \Cat D(X) \ar[r, "\Cat D(f)"'] & \Cat D(Y)
        \end{tikzcd}
    \end{center}
    given by the oplax naturality of $F_*$.
    We have a commutative diagram 
    \begin{center}
        \begin{tikzcd}[column sep=tiny]
            \FunLax{\BN}{\FunOpLax{\Delta^1}{\TwoCatInf}} \ar[rr, "\simeq"] \ar[dr, "\laxlim_{\BN}"'] && \FunOpLax{\Delta^1}{\FunLax{\BN}{\TwoCatInf}} \ar[dl, "\laxlim_{\BN} \circ -"] \\
            & \FunOpLax{\Delta^1}{\TwoCatInf} \rlap{.} &
        \end{tikzcd}
    \end{center}
    Indeed, the upper equivalence holds by e.g.\ \cite[Lemma 2.2.10]{AGH24}, and the commutation of the diagram 
    may be checked on left adjoints, where it is clear.
    In particular, the functor $(- \circ \widetilde{f}) \circ \Psi$ corresponds to a lax natural transformation $\widetilde{\Psi} \colon \Phi_0 \to \Phi_1$ under the top equivalence in the diagram. 
    Here, $\Phi_0$ maps the unique object of $\BN$ to $\Cat D(X)$ and the morphism $e$ to the functor $F_{*,X}$, and $\Phi_1$ similarly if we replace $X$ by $Y$.
    Moreover, $\laxlim((- \circ \widetilde{f}) \circ \Psi)$ is given by $\laxlim(\widetilde{\Psi}) \colon \laxlim_{\BN}(\Phi_0) \to \laxlim_{\BN}(\Phi_1)$. 
    By \Cref{lem:BN:Cart-as-fully-lax-lim}, $\laxlim_{\BN}(\Phi_0) \cong \Cart(\Cat D(X), F_{*,X})$ and $\laxlim_{\BN}(\Phi_1) \cong \Cart(\Cat D(Y), F_{*,Y})$.
    Consider the counit of the adjunction
    \[ \diag \colon \TwoCatInf \leftrightarrows \FunLax{\BN}{\TwoCatInf} \noloc \laxlim_{\BN}, \]
    i.e.\ the natural transformation $c \colon \diag \circ \laxlim_{\BN} \to \id{\FunLax{\BN}{\TwoCatInf}}$.
    Applying naturality to the lax natural transformation $\widetilde{\Psi} \colon \Phi_0 \to \Phi_1$ yields a commutative diagram
    \begin{center}
        \begin{tikzcd}[column sep=huge]
            \diag(\Cart(\Cat D(X), F_{*,X})) \ar[r, "\diag \laxlim_{\BN} \widetilde{\Psi}"] \ar[d, "c_{\Phi_0}"'] & \diag(\Cart(\Cat D(Y), F_{*,Y})) \ar[d, "c_{\Phi_1}"] \\
            \Phi_0 \ar[r, "\widetilde{\Psi}"'] & \Phi_1
        \end{tikzcd}
    \end{center}
    of morphisms in $\FunLax{\BN}{\TwoCatInf}$.
    Evaluating this at the unique object of $\BN$ gives the commutative diagram
    \begin{center}
        \begin{tikzcd}[column sep=huge]
            \Cart(\Cat D(X), F_{*,X}) \ar[r, "{\Cart(\Cat D, F_*)(f)}"] \ar[d, "U"'] & \Cart(\Cat D(Y),F_{*,Y}) \ar[d, "U"] \\
            \Cat D(X) \ar[r, "{\Cat D(f)}"'] & \Cat D(Y) \rlap{.}        
        \end{tikzcd}
    \end{center}
    Moreover, evaluating at the morphism $e$ yields the commutative diagram
    \begin{center}
        \begin{tikzcd}
            F_{*,Y}U \Cart(\Cat D,F_*)(f) \ar[rr, "{\kappa \Cart(\Cat D,F_*)(f)}"] \ar[d, "\simeq"'] && U \Cart(\Cat D,F_*)(f) \ar[d, "\simeq"] \\
            F_{*,Y} \Cat D(f)U \ar[r] & \Cat D(f)F_{*,X}U \ar[r, "\Cat D(f)\kappa"'] & \Cat D(f)U \rlap{.} 
        \end{tikzcd}
    \end{center}
    Indeed, the laxity $2$-morphism of $c_{\Phi_0}$ evaluated at $e$ is a diagram of the form 
    \begin{center}
        \begin{tikzcd}
            \Cart(\Cat D(X), F_{*,X}) \ar[r, "\id{}"] \ar[d, "U"'] & \Cart(\Cat D(X), F_{*,X}) \ar[d, "U"] \\
            \Cat D(X) \ar[r, "F_{*,X}"'] \laxarrowur{\kappa} & \Cat D(X) \rlap{,}
        \end{tikzcd}
    \end{center}
    where the $2$-morphism is given by the natural transformation corresponding to the functor $$\kappa \colon \Cart(\Cat D(X), F_{*,X}) \to \TwoFun{\Delta^1}{\Cat D(X)},$$ and similarly for $c_{\Phi_1}$.
    Furthermore, evaluating $\widetilde{\Psi}$ at $e$ recovers precisely the natural transformation $F_{*,Y} \Cat D(f) \to \Cat D(f) F_{*,X}$ given by the oplax naturality of $F_*$.
\end{proof}

\begin{prop} \label{lem:constr-cart:CartD-coherent}
    Consider the $2$-functor formalism $\Cart(\Cat D, F_*)$ from \Cref{def:BN:Cart-Frob-pCFF}.
    Assume that $\Cat D$ is a $4\frac{1}{2}$-functor formalism with respect to a collection of morphisms $T \subseteq \Cat C$.
    Then $\Cart(\Cat D, F_*)$ is a $4\frac{1}{2}$-functor formalism with respect to $T$.
\end{prop}
\begin{proof}
    We will repeatedly use that for every $X \in \Cat C$ the functor $F_{*,X} = \Cat D_*(F_X)$ preserves colimits:
    Indeed, by assumption $\Cat D_*(F_X)$ is a left adjoint.
    Also, we use without mention the fact that $\Cart(\Cat D, F_*)(X) \cong \Cart(\Cat D(X), F_{*,X})$, cf.\ \cref{lem:constr-cart:CartD-on-objects}.
    We now show that $\Cart(\Cat D, F_*)$ satisfies the axioms of \cref{def:2ff:defn}:
    \begin{enumerate}[label=Proof of \textbf{(FF\arabic{enumi})}:]
        \item Let $X \in \Cat C$. By assumption $\Cat D(X)$ is presentable stable.
            Hence, since $\Cat D_*(F_X)$ preserves colimits, we see from \cite[Corollary 2.7 (e) and (f)]{us24cart} that $\Cart(\Cat D(X), F_{*,X})$ is stable presentable.
        \item Let $f$ be a morphism in $\Cat C$. By the adjoint functor theorem, we must show that $\Cart(\Cat D, F_*)_*(f)$ 
            preserves colimits. This follows from the commutative diagram of \cref{lem:constr-cart:CartD-compatibility-U-and-kappa},
            the assumption that $\Cat D_*(f)$ preserves colimits (as it is a left adjoint), and the fact that $U$ preserves 
            colimits and is conservative, cf.\ \cite[Corollary 2.7 (b) and (d)]{us24cart} (here we use again that $\Cat D_*(F_X)$ preserves colimits).
        \item Let $f$ be a morphism in $E$. By the adjoint functor theorem, we must show that $\Cart(\Cat D, F_*)^!(f)$ 
            preserves colimits. This follows from the commutative diagram of \cref{lem:constr-cart:CartD-compatibility-U-and-kappa},
            the assumption that $\Cat D^!(f)$ preserves colimits (as it is a left adjoint), and the fact that $U$ preserves 
            colimits and is conservative, cf.\ \cite[Corollary 2.7 (b) and (d)]{us24cart} (here we use again that $\Cat D_*(F_X)$ preserves colimits).
        \item Let $f$ be a morphism in $T$. By the adjoint functor theorem, we must show that $\Cart(\Cat D, F_*)_*(f)$ 
            preserves limits and is accessible. Since we already know that it preserves colimits, it is accessible.
            The claim about limits follows from the commutative diagram of \cref{lem:constr-cart:CartD-compatibility-U-and-kappa},
            the assumption that $\Cat D_*(f)$ preserves limits (as it is a right adjoint), and the fact that $U$ preserves 
            limits and is conservative, cf.\ \cite[Corollary 2.7 (b) and (c)]{us24cart}. \qedhere
    \end{enumerate}
\end{proof}

For easier reference we summarize our results of this section:

\begin{thm} \label{lem:constr-cart:main-thm}
    Let $(\Cat C, E, I, P)$ be a Nagata setup.
    Let $\Cat D \colon \SpanStd \to \TwoCatInf$ be a $2$-functor formalism and let $F \colon \id{\Cat C} \to \id{\Cat C}$ be a Frobenius structure on $\Cat C$.
    Assume that for every $j \in I$, the square $\Cat D_*(\square_j^F)$ satisfies the left Beck--Chevalley condition.
    Then there is a $2$-functor formalism $\Cart(\Cat D, F_*)$ satisfying the following:
    \begin{enumerate}[label=(\arabic*)]
        \item If $X \in \Cat C$ is an object, $\Cart(\Cat D, F_*)(X) \cong \Cart(\Cat D(X), F_{*,X})$ is the category of Cartier modules in the category $\Cat D(X)$ with respect to the functor $F_{*,X}$.
        \item There is a natural transformation $U \colon \Cart(\Cat D, F_*) \to \Cat D$ such that for each object $X \in \Cat C$ the functor $U_X$ is equivalent to the forgetful functor $\Cart(\Cat D(X), F_{*,X}) \to \Cat D(X)$. 
        In particular, for each morphism $f \colon X \to Y$ in $\SpanStd$ there is a commutative diagram 
        \begin{center}
            \begin{tikzcd}[column sep=huge]
                \Cart(\Cat D(X), F_{*,X}) \ar[r, "{\Cart(\Cat D, F_*)(f)}"] \ar[d, "U_X"'] & \Cart(\Cat D(Y), F_{*,Y}) \ar[d, "U_Y"] \\
                \Cat D(X) \ar[r, "{\Cat D(f)}"'] & \Cat D(Y) \rlap{.}
            \end{tikzcd}
        \end{center}
        \item  If $f \colon X \to Y$ is a morphism in $\SpanStd$, there is a commutative diagram 
        \begin{center}
            \begin{tikzcd}
                F_{*,Y}U \Cart(\Cat D,F_*)(f) \ar[rr, "{\kappa \Cart(\Cat D,F_*)(f)}"] \ar[d, "\simeq"'] && U \Cart(\Cat D,F_*)(f) \ar[d, "\simeq"] \\
                F_{*,Y}\Cat D(f)U \ar[r] & \Cat D(f)F_{*,X}U \ar[r, "{\Cat D(f)\kappa}"'] & \Cat D(f)U \rlap{,} 
            \end{tikzcd}
        \end{center}
        where the vertical equivalences are given by the previous property and the bottom left arrow is given by (partially) oplax naturality of $F_*$.
    \end{enumerate}
    Moreover, if $\Cat D$ is a $4\frac{1}{2}$-functor formalism with respect to a class of morphisms $T \subseteq \Cat C$, 
    then $\Cart(\Cat D, F_*)$ also is a $4\frac{1}{2}$-functor formalism with respect to $T$.
\end{thm}
\begin{proof}
    By \Cref{lem:BN:F-star-partially-oplax-nat-trans}, the natural transformation $F_* \colon \Cat D_* \to \Cat D_*$ extends to a partially oplax natural transformation $F_* \colon \Cat D \to \Cat D$.
    Now we define $\Cart(\Cat D, F_*)$ as in \Cref{def:BN:Cart-Frob-pCFF}.
    It has the described properties by \Cref{lem:constr-cart:CartD-on-objects,lem:constr-cart:CartD-compatibility-U-and-kappa}.
    The claim about $4\frac{1}{2}$-functor formalisms was shown in \Cref{lem:constr-cart:CartD-coherent}.
\end{proof}

%% file: symmetric-monoidal/sym-mon-2cat.tex
In this section, we lay the necessary groundwork for the constructions in \cref{sec:3ff}.
First, we review the theory of symmetric monoidal $2$-categories in \cref{sec:sym-sym}.
In \cref{sec:sym-pointwise}, we explain the pointwise symmetric monoidal structure on various $2$-categories of functors;
as well as their interaction with other $2$-categorical constructions. In particular, we show that the pointwise structure 
is the cotensor by categories in the $2$-category of symmetric monoidal $2$-categories.
Then, in \cref{sec:sym-marked} we define the category of marked $2$-categories as well as 
the category of marked symmetric monoidal $2$-categories. We use this, to talk about marked mapping anima 
between marked symmetric monoidal $2$-categories, an essential ingredient to \cref{sec:sym-oplax-lax}.
Here, in this subsection, we finally define the $2$-category of lax symmetric monoidal functors with partially oplax 
natural transformations between them. It is this $2$-category in which the Frobenius endofunctor 
of the $3$-functor formalism of ind-coherent sheaves will live. 
We go on by showing that this construction behaves well with respect to adjoints; 
a feat that will be accomplished by constructing a (incoherent) $2$-functoriality on those functor categories.
This result will then be used to show that the $2$-category in question admits all partially lax limits,
a fact crucial for the construction of the Cartier module $3$-functor formalism.
Last, in \Cref{sec:sym-calg}, we define the $2$-category of commutative algebras in a symmetric monoidal $2$-category,
and explore some basic properties.

\subsection{Symmetric monoidal \texorpdfstring{$2$}{2}-categories} \label{sec:sym-sym}
In this subsection we discuss symmetric monoidal $2$-categories, and the $2$-category of 
symmetric monoidal $2$-categories with lax (resp.\ oplax, resp.\ strong) symmetric monoidal functors.

\begin{notation}
    We employ the following notation:
    \begin{enumerate}[label=(\alph*)]
        \item We write $\defnot{\Fin_*}$ for the $(1,1)$-category of finite pointed sets.
        \item For $n \ge 0$, we write $\defnot{\nset} \in \Fin_*$ for the finite set $\{ *, 1, \dots, n \}$.
        \item A morphism $f \colon \nset \to \mset$ is called \defemph{inert} if $f^{-1}(i)$ is a singleton for every $* \neq i \in \mset$.
        \item We write $\defnot{\SpanFin}$ for the $(2,1)$-category of spans of finite sets.
        \item For $n \ge 0$, we write $\defnot{[n]} \in \SpanFin$ for the finite set $\{ 1, \dots, n \}$.
    \end{enumerate}
\end{notation}

\begin{defn}
    Recall the notions of \defemph{$2$-(co)cartesian fibrations}, \defemphnoindex{$2$-cocartesian $1$-morphisms} and \defemphnoindex{cartesian $2$-morphisms} from for example \cite[Definition 5.1]{cnossen2025universalityspan2categoriesconstruction} or \cite[Definition 3.2.5]{AGH24}
    (note that in the latter reference, $2$-cocartesian fibrations are called \defemphnoindex{$(0,1)$-fibrations} and $2$-cartesian fibrations are called \defemphnoindex{$(1,0)$-fibrations}).
    \begin{enumerate} [label=(\arabic*)]
        \item We denote by $\textcolor{purple}{\TwoCocartlax{\Fin_*}} \subseteq (\TwoCatInfTwo)_{/ \Fin_*}$ the full sub-$2$-category spanned by $2$-cocartesian fibrations. 
        \item We denote by $\textcolor{purple}{\TwoCocart{\Fin_*}} \subseteq \TwoCocartlax{\Fin_*}$ the wide $2$-full sub-$2$-category whose morphisms preserve $2$-cocartesian $1$-morphisms and cartesian $2$-morphisms. 
        \item We denote by $\textcolor{purple}{\TwoCartOplax{\Fin_*}} \subseteq (\TwoCatInfTwo)_{/ \op{\Fin_*}}$ the full sub-$2$-category spanned by $2$-cartesian fibrations. 
    \end{enumerate}
\end{defn}

\begin{rmk}
    Note that any functor over a $1$-category between $2$-(co)cartesian fibrations automatically preserves (co)cartesian $2$-morphisms.
    Indeed, the mapping categories of the base are anima, so the (co)cartesian morphisms of the induced functors on mapping categories are precisely the equivalences by \cite[Proposition 2.4.1.5]{highertopoi} and its dual.
    Thus the (co)cartesian $2$-morphisms are precisely the invertible $2$-morphisms, which every functor preserves.
\end{rmk}

\begin{defn}
    We say that a $2$-cocartesian fibration $\TwoCat{C} \to \Fin_*$ \defemphindex{satisfies the Segal condition}{Segal condition} if for every $\langle n \rangle \in \Fin_*$, the inert maps $\rho_i \colon \langle n \rangle \to \langle 1 \rangle$ sending $i$ to $1$ and everything else to $*$
    assemble into an equivalence $\TwoCat{C}_{\langle n \rangle} \cong \prod_{i=1}^{n} \TwoCat{C}_{\langle 1 \rangle}$.
    We define the Segal condition for $2$-cartesian fibrations similarly.
\end{defn}

\begin{defn} \label{def:2sym:CAlgDelta}
    We define the \defemphindex{$2$-category of symmetric monoidal $2$-categories, symmetric monoidal (resp.\ lax symmetric monoidal, resp.\ oplax symmetric monoidal) functors, 
    and symmetric monoidal transformations}{$2$-category!symmetric monoidal} as follows: 
    \begin{enumerate}[label=(\arabic*)]
        \item $\textcolor{purple}{\TwoCAlg{\TwoCatInfTwo}} \subseteq \TwoCocart{\Fin_*}$ is the $2$-full sub-$2$-category  
            spanned by $2$-cocartesian fibrations satisfying the Segal condition 
            and by morphisms which preserve $2$-cocartesian $1$-morphisms.
        \item $\textcolor{purple}{\TwoCAlgLax{\TwoCatInfTwo}} \subseteq \TwoCocartlax{\Fin_*}$ is the $2$-full sub-$2$-category 
            spanned by $2$-cocartesian fibrations that satisfy the Segal condition, 
            and by morphisms which preserve $2$-cocartesian $1$-morphisms lying over inert morphisms.
        \item $\textcolor{purple}{\TwoCAlgOpLax{\TwoCatInfTwo}} \subseteq \TwoCartOplax{\Fin_*}$ is the $2$-full sub-$2$-category 
            spanned by $2$-cartesian fibrations that satisfy the Segal condition, 
            and by morphisms which preserve $2$-cartesian $1$-morphisms lying over opposites of inert morphisms.
    \end{enumerate}
    We also write $\textcolor{purple}{\CAlg{\CatInfTwo}}$, $\textcolor{purple}{\CAlgLax{\TwoCatInfTwo}}$ and $\textcolor{purple}{\CAlgOpLax{\TwoCatInfTwo}}$ for the corresponding underlying categories.
\end{defn}

The preceding definition is justified by the following:

\begin{rmk}
    By \cite[Chapter 9, Sections 1.4.1, 1.4.2]{gaitsgory2019study}, a $2$-cocartesian fibration $\TwoCat{C} \to \Fin_*$ (resp.\ a $2$-cartesian fibration $\TwoCat{C} \to \op{\Fin_*}$) satisfying the Segal condition encodes precisely a commutative monoid in $\TwoCatInfTwo$, 
    which is in turn a symmetric monoidal $2$-category.
    Indeed, the tensor product on $\TwoCat{C}_{\langle 1 \rangle}$ can be defined by choosing $2$-cocartesian lifts over the map $\mu \colon \langle 2 \rangle \to \langle 1 \rangle$ mapping $1$ and $2$ to $1$ (using the Segal condition).

    Moreover, by \cite[Definition 5.12 and the beginning of \S5]{cnossen2025universalityspan2categoriesconstruction} (where it is explained that the definitions of symmetric monoidal categories and lax symmetric monoidal functors via $\SpanFin$ and via $\Fin_*$ agree), 
    a morphism of $2$-cocartesian fibrations which preserves $2$-cocartesian $1$-morphisms lying over inert morphisms is a lax symmetric monoidal functor $F$. 
    The corresponding maps $F(X) \otimes F(Y) \to F(X \otimes Y)$ are obtained by applying $F$ to a $2$-cocartesian lift of $\mu$ starting in $(X,Y)$ and factoring it through a $2$-cocartesian lift of $\mu$ starting in $(F(X), F(Y))$.
    Note that these maps are not equivalences in general as $\mu$ is not inert.
    Dually (also cf.\ \cite[Definitions 3.4.1, 3.4.3]{haugseng2024laxmonoidaladjunctionstwovariable} for the $1$-categorical analogue), a morphism of $2$-cartesian fibrations which preserves 
    $2$-cartesian $1$-morphisms over opposites of inert morphisms is an oplax symmetric monoidal functor. 
    In \loccit it is also explained that a morphism preserving all $2$-cocartesian $1$-morphisms is a symmetric monoidal functor.

    The $2$-morphisms in all of the above $2$-categories are given by symmetric monoidal natural transformations.
    This follows directly from applying naturality, for example to maps $F(X) \otimes F(Y) \to F(X \otimes Y)$, and similarly for $G$, where $F$ and $G$ are lax symmetric monoidal functors. 
\end{rmk}

\begin{rmk} \label{rmk:2sym:forgetf-functor}
    There is a canonical functor $(\TwoCatInfTwo)_{/ \Fin_*} \to (\TwoCatInfTwo)_{/ *} \cong \TwoCatInfTwo$ given by pulling back along the functor $* \to \Fin_*$ that maps $*$ to $\langle 1 \rangle$.
    In particular, it restricts to forgetful functors 
    $\CAlg{\TwoCatInfTwo} \to \CatInfTwo$ and $\CAlgLax{\TwoCatInfTwo} \to \CatInfTwo$. Dually, one gets a forgetful functor $\CAlgOpLax{\TwoCatInfTwo} \to \CatInfTwo$. 
    We call the image of a symmetric monoidal $2$-category or of a functor their \defemphnoindex{underlying $2$-category} or \defemphnoindex{underlying functor}, respectively.
\end{rmk}

\begin{lem} \label{lem:2-sym:oplax-lax-CAlg-coop-equivalence}
    The autoequivalence $\coop{(-)} \colon \CatInfTwo \to \CatInfTwo$ upgrades to an equivalence 
    $$\coop{(-)} \colon \CAlgLax{\TwoCatInfTwo} \xrightarrow{\simeq} \CAlgOpLax{\TwoCatInfTwo}$$
    with inverse also given by $\coop{(-)}$. 
\end{lem}
\begin{proof}
    Consider the equivalence $\coop{(-)} \colon \TwoCatInfTwo \xrightarrow{\simeq} \co{\TwoCatInfTwo}$ from \Cref{lem:two-cat:two-cat-coop}.
    Slicing over $\Fin_*$, this induces an equivalence 
    $$\coop{(-)} \colon (\TwoCatInfTwo)_{/\Fin_*} \xrightarrow{\simeq} \co{(\TwoCatInfTwo)}_{/\op{\Fin_*}},$$
    noting that $\coop{\Fin_*} = \op{\Fin_*}$, since $\Fin_*$ is a category.

    We show that it restricts to an equivalence
    $$\coop{(-)} \colon \TwoCAlgLax{\TwoCatInfTwo} \xrightarrow{\simeq} \co{\TwoCAlgOpLax{\TwoCatInfTwo}}.$$
    Then the claim follows by applying $\iota_1$ to this equivalence.
    By \cite[Observation 3.2.9]{AGH24} we only have to show that $\coop{(-)}$ preserves the property of satisfying the Segal condition.
    But this follows directly from the fact that $\coop{(-)}$ preserves finite products (as it is an equivalence).
\end{proof}

\begin{prop} \label{lem:2-sym:adjunction-in-CAlgLax}
    Let $F \colon \TwoCat C^\otimes \to \TwoCat D^\otimes \in \TwoCAlgLax{\TwoCatInfTwo}$ be a lax symmetric monoidal functor.
    Assume that $F$ is actually strong monoidal, and that the underlying functor $F \colon \TwoCat C \to \TwoCat D$ admits a right adjoint.
    Then $F$ admits a lax monoidal right adjoint $G \colon \TwoCat D^\otimes \to \TwoCat C^\otimes$
    (i.e.\ a right adjoint in the $2$-category $\TwoCAlgLax{\TwoCatInfTwo}$),
    such that the underlying functor $G \colon \TwoCat D \to \TwoCat C$ is right adjoint to $F$.
\end{prop}
\begin{proof}
    Choose an adjunction $F \dashv \tilde{G}$.
    By the Segal condition, for every $n \ge 0$, the functor on the fiber
    $F_{\nset} \colon \TwoCat C^\otimes_{\nset} \to \TwoCat D^\otimes_{\nset}$
    identifies with
    $F^{\times n} \colon \TwoCat C^{\times n} \to \TwoCat D^{\times n}$.
    It therefore admits a right adjoint equivalent to $\tilde{G}^{\times n}$.

    Since $F$ is strong symmetric monoidal, it preserves $2$-cocartesian $1$-morphisms.
    Hence, the fiberwise criterion for adjunctions, namely the opposite
    version of \cite[Proposition 4.2.5]{AGH24}, gives an adjunction
    $$F \dashv G \qquad\text{in}\qquad (\TwoCatInfTwo)_{/\Fin_*}.$$
    Here, the projections to $\Fin_*$ are $2$-cocartesian fibrations, and the
    cartesian-enrichment condition in that criterion is automatic because
    $\Fin_*$ is an ordinary category.
    On each fiber, this adjunction identifies with
    $F^{\times n} \dashv \tilde{G}^{\times n}$.
    Write $\eta_{\nset}$ and $\varepsilon_{\nset}$ for its fiberwise unit
    and counit. Under the Segal equivalences, these identify with the products
    of the unit and counit of $F \dashv \tilde{G}$.

    We claim that $G$ preserves $2$-cocartesian $1$-morphisms lying over inert
    morphisms. Let $\alpha \colon \nset \to \mset$ be inert.
    Write $\alpha_! \colon \TwoCat C^\otimes_{\nset} \to \TwoCat C^\otimes_{\mset}$
    and similarly $\alpha_! \colon \TwoCat D^\otimes_{\nset} \to \TwoCat D^\otimes_{\mset}$
    for the induced cocartesian transport functors.
    Since $F$ preserves $2$-cocartesian $1$-morphisms, there is a natural equivalence
    $$\alpha_!F_{\nset} \simeq F_{\mset}\alpha_!.$$
    Its mate under the fiberwise adjunctions is the natural transformation
    $\kappa_\alpha \colon \alpha_!G_{\nset} \to G_{\mset}\alpha_!$
    given by the composite
    $$\alpha_!G_{\nset} \xrightarrow{\eta_{\mset}} G_{\mset}F_{\mset}\alpha_!G_{\nset} \simeq G_{\mset}\alpha_!F_{\nset}G_{\nset} \xrightarrow{G_{\mset}\alpha_!\varepsilon_{\nset}} G_{\mset}\alpha_!.$$

    Since $\alpha$ is inert, for each $1 \le j \le m$ there is a unique
    $i_j$ with $\alpha(i_j)=j$.
    Under the Segal equivalences, the transport functor $\alpha_!$
    is the coordinate projection
    $$(x_1,\ldots,x_n) \longmapsto (x_{i_1},\ldots,x_{i_m}).$$
    Thus, at an object $y=(y_1,\ldots,y_n)$, both the source and target of
    $\kappa_\alpha(y)$ identify with
    $(\tilde{G}y_{i_1},\ldots,\tilde{G}y_{i_m})$.
    In the $j$th coordinate, $\kappa_\alpha(y)$ identifies with
    $$\tilde{G}y_{i_j} \xrightarrow{\eta_{\tilde{G}y_{i_j}}} \tilde{G}F\tilde{G}y_{i_j} \xrightarrow{\tilde{G}\varepsilon_{y_{i_j}}} \tilde{G}y_{i_j},$$
    which is equivalent to the identity by the triangle identity.
    Consequently, $\kappa_\alpha$ is an equivalence.
    If $m=0$, the same conclusion follows because the target fibers
    are terminal.

    Now let $f \colon y \to y'$ be a cocartesian morphism over $\alpha$.
    Using $y' \simeq \alpha_!y$, the morphism $G(f)$ factors as
    $$G_{\nset}y \longrightarrow \alpha_!G_{\nset}y \xrightarrow{\kappa_\alpha(y)} G_{\mset}\alpha_!y \simeq G_{\mset}y',$$
    where the first morphism is cocartesian.
    Since $\kappa_\alpha(y)$ is an equivalence, $G(f)$ is cocartesian.
    This proves that $G$ is lax symmetric monoidal.

    Finally, the unit and counit of the adjunction over $\Fin_*$
    are monoidal natural transformations: they are transformations
    over $\Fin_*$ between functors preserving $2$-cocartesian $1$-morphisms over inert maps.
    Their triangle identities therefore also hold in
    $\TwoCAlgLax{\TwoCatInfTwo}$.
    Hence $F \dashv G$ in that $2$-category.
    Restriction to the fiber over $\langle 1\rangle$ recovers
    the adjunction $F \dashv \tilde{G}$.
\end{proof}

We end this subsection by supplying a source of symmetric monoidal $2$-categories:
\begin{exmpl} \label{exmpl:2-sym:cartesian-monoidal}
    Let $\TwoCat C$ be a $2$-category which admits finite products\footnote{In the $2$-categorical sense, i.e.\ $2$-categorical limits indexed by finite sets (or equivalently, $2$-categorical lax limits indexed over finite sets since there are no non-identity maps in the indexing categories)}.
    Then there exists by \cite[Example 2.2.5]{blans2025chainrule} a symmetric monoidal structure on $\TwoCat C$, denoted by $\defnot{\TwoCat C^\times}$,
    called the \defemphindex{cartesian symmetric monoidal structure}{$2$-category!cartesian symmetric monoidal},
    such that the tensor functor is given by the product functor $- \times - \colon \TwoCat C \times \TwoCat C \to \TwoCat C$.
\end{exmpl}

Specializing this, we obtain:
\begin{exmpl}
    When we consider $\TwoCatInf$ and $\TwoCatInfTwo$ as symmetric monoidal $2$-categories,
    we will always do this via the cartesian symmetric monoidal structure of \cref{exmpl:2-sym:cartesian-monoidal}.
\end{exmpl}

%% file: symmetric-monoidal/pointwise.tex
\subsection{Pointwise symmetric monoidal structures} \label{sec:sym-pointwise}
In this subsection we discuss the pointwise symmetric monoidal structure on 
$\TwoFun{\TwoCat C}{\TwoCat D^\otimes}$, $\FunLax{\TwoCat C}{\TwoCat D^\otimes}$ and $\FunOpLax{\TwoCat C}{\TwoCat D^\otimes}$,
its functoriality, and its interaction with other constructions.

\begin{prop} \label{lem:2-sym:pointwise-structure-functoriality}
    Let $\epsilon, \delta \in \oplaxvarset$.
    There exists a functor
    $$\textcolor{purple}{\FunEps{-}{-}^{\ptw}} \colon \op{\CatInfTwo} \times \CAlgDelta{\TwoCatInfTwo} \to \CAlgDelta{\TwoCatInfTwo},$$
    such that the diagram
    \begin{center}
        \begin{tikzcd}
            \op{\CatInfTwo} \times \CAlgDelta{\TwoCatInfTwo} \ar[d] \ar[rr, "{\FunEps{-}{-}^{\ptw}}"] &&\CAlgDelta{\TwoCatInfTwo} \ar[d] \\
            \op{\CatInfTwo} \times \CatInfTwo \ar[rr, "{\FunEps{-}{-}}"'] &&\CatInfTwo
        \end{tikzcd}
    \end{center}
    commutes.
    Here, the vertical functors are the forgetful functors of \cref{rmk:2sym:forgetf-functor}.
\end{prop}
\begin{proof}
    We prove the statement for $\delta = \mathrm{lax}$, the proof for $\delta = \mathrm{oplax}$ and $\delta = \emptyset$ is analogous.
    For every $2$-category $\TwoCat C$,
    consider the functor $\FunEps{\TwoCat C}{-} \colon \TwoCatInfTwo \to \TwoCatInfTwo$.
    Slicing over $\Fin_*$, we obtain a functor 
    $$\FunEps{\TwoCat C}{-} \colon (\TwoCatInfTwo)_{/\Fin_*} \to (\TwoCatInfTwo)_{/\FunEps{\TwoCat C}{\Fin_*}}.$$
    Consider the canonical diagonal functor $\Fin_* \to \FunEps{\TwoCat C}{\Fin_*}$. Pulling back along this functor gives 
    $$\Fin_* \times_{\FunEps{\TwoCat C}{\Fin_*}} \FunEps{\TwoCat C}{-} \colon (\TwoCatInfTwo)_{/\Fin_*} \to (\TwoCatInfTwo)_{/\Fin_*}.$$
    We claim that if we restrict this functor to $\CAlgLax{\TwoCatInfTwo}$, we land again in $\CAlgLax{\TwoCatInfTwo}$.
    Indeed, if $p \colon \TwoCat{D} \to \Fin_*$ is a $2$-cocartesian fibration then $p_* \colon \FunEps{\TwoCat{C}}{\TwoCat{D}} \to \FunEps{\TwoCat{C}}{\Fin_*}$ is also a $2$-cocartesian fibration by \cite[Corollary 3.5.1 and Remark 3.5.2]{abellán2026freefibrationslaxcolimits} for $\epsilon = \operatorname{lax}$, 
    the co-dual version of \loccit for $\epsilon = \operatorname{oplax}$ (using that $\co{\Fin_*} = \Fin_*$), and \cite[Corollary 3.5.6]{abellán2026freefibrationslaxcolimits} for $\epsilon = \emptyset$.
    As $2$-cocartesian fibrations are stable under base change, $\overline{p} \coloneqq \Fin_* \times_{\FunEps{\TwoCat{C}}{\Fin_*}} \FunEps{\TwoCat{C}}{p}$ is a $2$-cocartesian fibration as well.
    The fiber of $\overline{p}$ over $\langle n \rangle$ is equivalent to $\FunEps{\TwoCat{C}}{\TwoCat{D}_{\langle n \rangle}}$. 
    As $\FunEps{\TwoCat{C}}{-}$ preserves finite products by \Cref{lem:lax-arrows:commutes-with-limits}, we see that $\overline{p}$ satisfies the Segal condition if $p$ does so.
    It remains to show that if $F^\otimes \colon \TwoCat{D}^\otimes \to \TwoCat{E}^\otimes$ is a lax symmetric monoidal functor, so is $\Fin_* \times_{\FunEps{\TwoCat{C}}{\Fin_*}} \FunEps{\TwoCat{C}}{F}$. 
    Denote the unstraightened functor over $\Fin_*$ by $F \colon \TwoCat{D} \to \TwoCat{E}$. 
    By assumption, $F$ preserves $2$-cocartesian $1$-morphisms over inert morphisms. 
    Consider $\FunEps{\TwoCat{C}}{F} \colon \FunEps{\TwoCat{C}}{\TwoCat{D}} \to \FunEps{\TwoCat{C}}{\TwoCat{E}}$. 
    By \cite[Remark 3.5.2]{abellán2026freefibrationslaxcolimits}, a $1$-morphism $t \colon \Phi \to \Psi$ in $\FunEps{\TwoCat{C}}{\TwoCat{D}}$ (resp. $\FunEps{\TwoCat{C}}{\TwoCat{E}}$) is $2$-cocartesian if and only if 
    for each object $X \in \TwoCat{C}$ the $1$-morphism $t_X \colon \Phi(X) \to \Psi(X)$ is $2$-cocartesian in $\TwoCat{D}$ (resp. $\TwoCat{E}$) and 
    all $\epsilon$ naturality squares contain cartesian $2$-morphisms. 
    As $F$ preserves $2$-cocartesian $1$-morphisms over inert morphisms, we see that in this case for each object $X \in \TwoCat{C}$, the $1$-morphism $F(t_X) = (F t)_X \colon F\Phi(X) \to F\Psi(X)$ is $2$-cocartesian.
    Moreover, after pulling back to the $1$-category $\Fin_*$ the condition of being a cartesian $2$-morphism is equivalent to being an equivalence. 
    This property is preserved by every functor.
    Thus, the functor $\Fin_* \times_{\FunEps{\TwoCat{C}}{\Fin_*}} \FunEps{\TwoCat{C}}{F}$ again preserves $2$-cocartesian $1$-morphisms over inert morphisms.
    This finishes the proof of the restriction of the functor $\Fin_* \times_{\FunEps{\TwoCat{C}}{\Fin_*}} \FunEps{\TwoCat{C}}{-}$ to a functor $\CAlgLax{\TwoCatInfTwo} \to \CAlgLax{\TwoCatInfTwo}$.

    Hence, by restricting (and since the above construction was functorial in $\TwoCat C$), we obtain a functor 
    $$\FunEps{-}{-}^{\ptw} \colon \op{\CatInfTwo} \times \CAlgLax{\TwoCatInfTwo} \to \CAlgLax{\TwoCatInfTwo}.$$
    It is left to show that the underlying $2$-category is given by $\FunEps{-}{-}$.
    One can recover the underlying $2$-category of a symmetric monoidal $2$-category by further pullback along the functor $\Delta^0 \to \Fin_*$
    selecting $\langle1\rangle$.
    Hence, we get that the underlying category is 
    \begin{equation*}
        \Delta^0 \times_{\Fin_*} \Fin_* \times_{\FunEps{\TwoCat C}{\Fin_*}} \FunEps{\TwoCat C}{\TwoCat D^\otimes} \cong \FunEps{\TwoCat C}{\Delta^0 \times_{\Fin_*} \TwoCat D^\otimes} = \FunEps{\TwoCat C}{\TwoCat D},
    \end{equation*}
    where we used that $\Fun{\TwoCat C}{\Delta^0} = \Delta^0$, and that $\FunEps{-}{-}$
    preserves limits in the second variable, cf.\ \cref{lem:lax-arrows:commutes-with-limits}.
\end{proof}

\begin{rmk}
    We suppress the variance $\delta$ in the notation of $\FunEps{-}{-}^{\ptw}$.
    This is harmless, as one can show that these functors are compatible with the inclusions 
    $\CAlg{\TwoCatInfTwo} \hookrightarrow \CAlgLax{\TwoCatInfTwo}$ and $\CAlg{\TwoCatInfTwo} \hookrightarrow \CAlgOpLax{\TwoCatInfTwo}$.
    Most of the time we will only consider the functoriality of $\FunEps{-}{-}^{\ptw}$ in $\CAlgLax{\TwoCatInfTwo}$.
\end{rmk}

\begin{lem} \label{lem:pointwise:cartesian-pointwise}
    Let $\epsilon \in \oplaxvarset$, and let $\TwoCat C$ and $\TwoCat I$ be $2$-categories, and assume that $\TwoCat C$ admits finite products.
    Then the pointwise symmetric monoidal structure $\FunEps{\TwoCat I}{\TwoCat C^\times}^\ptw$ is the cartesian symmetric monoidal structure.
\end{lem}
\begin{proof}
    Finite products in $\FunEps{\TwoCat I}{\TwoCat C}$ exist and are computed pointwise, with projections obtained by applying $\FunEps{\TwoCat I}{-}$ to the product projections in $\TwoCat C$.
    By the construction in \cref{lem:2-sym:pointwise-structure-functoriality}, these are precisely the $n$-ary tensor operations and projections of $\FunEps{\TwoCat I}{\TwoCat C^\times}^{\ptw}$.
    This includes $n=0$, so its tensor unit is terminal.
    The coherence maps agree with the cartesian ones by the universal property of products, proving the claim.
\end{proof}

We have the following interaction of $\coop{(-)}$ with the pointwise symmetric monoidal structure:
\begin{lem} \label{lem:2-sym:pointwise-coop-interchange-law}
    Let $\TwoCat C$ be a $2$-category, $\TwoCat D^\otimes$ be a symmetric monoidal $2$-category,
    and $\epsilon \in \oplaxvarset$.
    Then there is a natural equivalence 
    $$\coop{(\FunEps{\TwoCat C}{\TwoCat D^\otimes}^{\ptw})} \cong \FunEps{\coop{\TwoCat C}}{\coop{(\TwoCat{D}^\otimes)}}^{\ptw},$$
    such that on underlying $2$-categories it recovers the equivalence of \Cref{lem:lax-nat-transf:coop-lax-oplax}. More formally,
    the following diagram commutes:
    \begin{center}
        \begin{tikzcd}[column sep=huge]
            \op{\CatInfTwo} \times \CAlgLax{\TwoCatInfTwo} \ar[r, "\FunEps{-}{-}^\ptw"] \ar[d, "\coop{(-)} \times \coop{(-)}"'] & \CAlgLax{\TwoCatInfTwo} \ar[d, "\coop{(-)}"] \\
            \op{\CatInfTwo} \times \CAlgOpLax{\TwoCatInfTwo} \ar[r, "\FunEps{-}{-}^\ptw"'] & \CAlgOpLax{\TwoCatInfTwo} \rlap{.}
        \end{tikzcd}
    \end{center}
\end{lem}
\begin{proof}
    We use the description of $\FunEps{\TwoCat{C}}{-}^\ptw$ as pullback $\Fin_* \times_{\FunEps{\TwoCat C}{\Fin_*}} \FunEps{\TwoCat C}{-}$ from the proof of \Cref{lem:2-sym:pointwise-structure-functoriality}. 
    We get a chain of equivalences 
    \begin{align*}
        \coop{(\FunEps{\TwoCat C}{-}^\ptw)} &\cong \coop{(\Fin_* \times_{\FunEps{\TwoCat C}{\Fin_*}} \FunEps{\TwoCat{C}}{-})} \\
        &\cong \op{\Fin_*} \times_{\coop{(\FunEps{\TwoCat C}{\Fin_*})}} \coop{(\FunEps{\TwoCat{C}}{-})} \\
        &\cong \op{\Fin_*} \times_{\FunEps{\coop{\TwoCat C}}{\op{\Fin_*}}} \FunEps{\coop{\TwoCat{C}}}{\coop{(-)}} \\
        &\cong \FunEps{\coop{\TwoCat{C}}}{\coop{(-)}}^\ptw .
    \end{align*}
    Here, we used that $\coop{(-)} \colon \CatInfTwo \to \CatInfTwo$ preserves limits (as it is an equivalence by \Cref{lem:two-cat:coop}) in the second equivalence, 
    \Cref{lem:lax-nat-transf:coop-lax-oplax} in the third equivalence, and 
    the definition of $\FunEps{-}{-}^\ptw$ as a functor $\op{\CatInfTwo} \times \CAlgOpLax{\TwoCatInfTwo} \to \CAlgOpLax{\TwoCatInfTwo}$ in the last equivalence.
    Since all of the above is functorial in $\TwoCat{C}$ we get a commutative diagram
    \begin{center}
        \begin{tikzcd}[column sep=huge]
            \op{\CatInfTwo} \times \CAlgLax{\TwoCatInfTwo} \ar[r, "\FunEps{-}{-}^\ptw"] \ar[d, "\coop{(-)} \times \coop{(-)}"'] & \CAlgLax{\TwoCatInfTwo} \ar[d, "\coop{(-)}"] \\
            \op{\CatInfTwo} \times \CAlgOpLax{\TwoCatInfTwo} \ar[r, "\FunEps{-}{-}^\ptw"'] & \CAlgOpLax{\TwoCatInfTwo} \rlap{,}
        \end{tikzcd}
    \end{center}
    where we used \Cref{lem:2-sym:oplax-lax-CAlg-coop-equivalence} for the vertical functors.
\end{proof}

We have the following comparison functor, which lifts the comparison of \cref{lem:lax-arrow:lax-strong-comparison} to pointwise structures:
\begin{prop} \label{lem:2-sym:pointwise-comparison}
    Let $\TwoCat C$ and $\TwoCat E$ be $2$-categories, let $\TwoCat N^\otimes$ be a symmetric monoidal $2$-category,
    and let $\epsilon, \delta \in \oplaxvarset$.
    There exists a symmetric monoidal comparison functor
    $$\FunEps{\TwoCat E}{\TwoFun{\TwoCat C}{\TwoCat N^\otimes}^{\ptw}}^\ptw \to \FunEps{\TwoCat C\times\TwoCat E}{\TwoCat N^\otimes}^{\ptw},$$
    natural in $\TwoCat E, \TwoCat C \in \CatInfTwo$ and $\TwoCat N^\otimes \in \CAlgDelta{\TwoCatInfTwo}$.
    This functor is an equivalence if $\epsilon = \emptyset$.
    On underlying $2$-categories, it agrees with the comparison functor from \cref{lem:lax-arrow:lax-strong-comparison}.
    Moreover, the squares
    \begin{center}
        \begin{tikzcd}
            \TwoFun{\TwoCat E}{\TwoFun{\TwoCat C}{\TwoCat N^\otimes}^{\ptw}}^\ptw \ar[d] \ar[r, "\simeq"] &\TwoFun{\TwoCat C\times\TwoCat E}{\TwoCat N^\otimes}^{\ptw} \ar[d]\\
            \FunLax{\TwoCat E}{\TwoFun{\TwoCat C}{\TwoCat N^\otimes}^{\ptw}}^\ptw \ar[r] &\FunLax{\TwoCat C\times\TwoCat E}{\TwoCat N^\otimes}^{\ptw}
        \end{tikzcd}
    \end{center}
    and
    \begin{center}
        \begin{tikzcd}
            \TwoFun{\TwoCat E}{\TwoFun{\TwoCat C}{\TwoCat N^\otimes}^{\ptw}}^\ptw \ar[d] \ar[r, "\simeq"] &\TwoFun{\TwoCat C\times\TwoCat E}{\TwoCat N^\otimes}^{\ptw} \ar[d]\\
            \FunOpLax{\TwoCat E}{\TwoFun{\TwoCat C}{\TwoCat N^\otimes}^{\ptw}}^\ptw \ar[r] &\FunOpLax{\TwoCat C\times\TwoCat E}{\TwoCat N^\otimes}^{\ptw}
        \end{tikzcd}
    \end{center}
    commute.
\end{prop}
\begin{proof}
    From \cref{lem:lax-arrow:lax-strong-comparison} we have a canonical comparison functor 
    $$\FunEps{\TwoCat E}{\TwoFun{\TwoCat C}{\TwoCat N}} \to \FunEps{\TwoCat C \times \TwoCat E}{\TwoCat N},$$
    natural in $\TwoCat E$, $\TwoCat C$ and $\TwoCat N$, which is an equivalence if $\epsilon = \emptyset$.

    Since $\Fin_*$ is a $1$-category, and since $\FunEps{-}{K} \cong \TwoFun{-}{K}$ for any $1$-category $K$,
    we obtain that the comparison
    $$\Cat F \coloneqq \FunEps{\TwoCat E}{\TwoFun{\TwoCat C}{\Fin_*}} \xrightarrow{\simeq} \FunEps{\TwoCat C \times \TwoCat E}{\Fin_*}$$
    is an equivalence.

    Now, consider the natural transformation $\FunEps{\TwoCat E}{\TwoFun{\TwoCat C}{-}} \to \FunEps{\TwoCat C \times \TwoCat E}{-} \colon \TwoCatInfTwo \to \TwoCatInfTwo$.
    Slicing over $\Fin_*$ and using the above equivalence, we obtain a natural transformation of functors $(\TwoCatInfTwo)_{/\Fin_*} \to (\TwoCatInfTwo)_{/\Cat F}$.

    Pulling back along the constant functor $\Fin_* \to \Cat F$ and restricting to $\CAlgLax{\TwoCatInfTwo}$ yields 
    the required comparison, the verification that this restricted functor lands again in $\CAlgLax{\TwoCatInfTwo}$ is similar 
    to the proof of \Cref{lem:2-sym:pointwise-structure-functoriality}.
    To get the commutative squares, one proceeds in the same way, using the commutative squares 
    of \cref{lem:lax-arrow:lax-strong-comparison}.
\end{proof}

We also have the following lift of \cref{lem:mapping-objects:tensor-hom-adjunction-internal} for $\FunEps{-}{-}$ to the level of pointwise symmetric monoidal structures:
\begin{lem} \label{lem:pointwise:internal-currying-pointwise-structure}
    Let $\epsilon, \delta \in \oplaxvarset$, let $\TwoCat C^\otimes \in \CAlgDelta{\TwoCatInfTwo}$ and let $\TwoCat K, \TwoCat I \in \CatInfTwo$.
    Then there is an equivalence in $\CAlgDelta{\TwoCatInfTwo}$
    $$\FunEps{\TwoCat K}{\FunEps{\TwoCat I}{\TwoCat C^\otimes}^\ptw}^\ptw \cong \FunEps{\TwoCat K \otimes^\epsilon \TwoCat I}{\TwoCat C^\otimes}^\ptw,$$
    natural in $\TwoCat C^\otimes$, $\TwoCat K$ and $\TwoCat I$.
    Here, $\otimes^\epsilon$ denotes the Gray tensor product if $\epsilon = \mathrm{lax}$, the reverse of the Gray tensor product if $\epsilon = \mathrm{oplax}$,
    and the cartesian product if $\epsilon = \emptyset$. In particular, $- \otimes^\epsilon \TwoCat I$ is left adjoint to $\FunEps{\TwoCat I}{-}$.
\end{lem}
\begin{proof}
    We prove the statement for $\delta = \mathrm{lax}$, the other cases are similar.
    By the pullback description of pointwise monoidal structures, preservation
    of limits by $\FunEps{\TwoCat K}{-}$, cf.\ \cref{lem:mapping-objects:ihom-preserves-limits-both-variables}, and \cref{lem:mapping-objects:tensor-hom-adjunction-internal}, we obtain
    \begin{align*}
        &\FunEps{\TwoCat K}{\FunEps{\TwoCat I}{\TwoCat C^\otimes}^{\ptw}}^{\ptw} \\
        &\cong \Fin_* \times_{\FunEps{\TwoCat K}{\Fin_*}} \FunEps{\TwoCat K}{\Fin_* \times_{\FunEps{\TwoCat I}{\Fin_*}} \FunEps{\TwoCat I}{\TwoCat C^\otimes}} \\
        &\cong \Fin_* \times_{\FunEps{\TwoCat K}{\FunEps{\TwoCat I}{\Fin_*}}} \FunEps{\TwoCat K}{\FunEps{\TwoCat I}{\TwoCat C^\otimes}} \\
        &\cong \Fin_* \times_{\FunEps{\TwoCat K \otimes^\epsilon \TwoCat I}{\Fin_*}} \FunEps{\TwoCat K \otimes^\epsilon \TwoCat I}{\TwoCat C^\otimes} \\
        &\cong \FunEps{\TwoCat K \otimes^\epsilon \TwoCat I}{\TwoCat C^\otimes}^{\ptw}.
    \end{align*}
    The maps from $\Fin_*$ are the constant-diagram maps. These equivalences are over $\Fin_*$ and hence give an equivalence of symmetric monoidal $2$-categories.
\end{proof}

\begin{lem} \label{lem:2-sym:category-indexed-cotensors}
    Let $\delta \in \oplaxvarset$. The $2$-category
    $\TwoCAlgDelta{\TwoCatInfTwo}$ admits cotensors by
    categories. They are given by the pointwise symmetric
    monoidal structure:
    $$(\TwoCat C^\otimes)^K \cong \TwoFun{K}{\TwoCat C^\otimes}^{\ptw}.$$
    More precisely, there is an equivalence of categories
    $$\Map{\TwoCAlgDelta{\TwoCatInfTwo}}{\TwoCat D^\otimes}{\TwoFun{K}{\TwoCat C^\otimes}^{\ptw}} \cong \Map{\TwoCatInf}{K}{\Map{\TwoCAlgDelta{\TwoCatInfTwo}}{\TwoCat D^\otimes}{\TwoCat C^\otimes}},$$
    natural in $K \in \CatInf$ and
    $\TwoCat C^\otimes,\TwoCat D^\otimes
    \in \TwoCAlgDelta{\TwoCatInfTwo}$.
    In particular, taking maximal subgroupoids gives
    $$\Map{\CAlgDelta{\TwoCatInfTwo}}{\TwoCat D^\otimes}{\TwoFun{K}{\TwoCat C^\otimes}^{\ptw}} \cong \Map{\CatInf}{K}{\Map{\TwoCAlgDelta{\TwoCatInfTwo}}{\TwoCat D^\otimes}{\TwoCat C^\otimes}}.$$
\end{lem}
\begin{proof}
    We first treat $\delta=\mathrm{lax}$.
    Recall from the construction of the pointwise structure
    in \cref{lem:2-sym:pointwise-structure-functoriality}
    that its total $2$-category is
    $$\TwoFun{K}{\TwoCat C^\otimes}^{\ptw} \cong \Fin_* \times_{\TwoFun{K}{\Fin_*}} \TwoFun{K}{\TwoCat C^\otimes},$$
    where $\Fin_* \to \TwoFun{K}{\Fin_*}$ is the constant
    diagram functor.

    The cartesian exponential law, together with this
    pullback description, gives an equivalence of mapping
    categories in the slice:
    $$\Map{(\TwoCatInfTwo)_{/\Fin_*}}{\TwoCat D^\otimes}{\Fin_* \times_{\TwoFun{K}{\Fin_*}} \TwoFun{K}{\TwoCat C^\otimes}} \cong \Fun{K}{\Map{(\TwoCatInfTwo)_{/\Fin_*}}{\TwoCat D^\otimes}{\TwoCat C^\otimes}}.$$
    Explicitly, a functor over $\Fin_*$ into the pointwise
    structure corresponds to a $K$-indexed family of functors
    $\TwoCat D^\otimes \to \TwoCat C^\otimes$ over $\Fin_*$,
    and transformations correspond under the same currying
    equivalence.

    We check that this equivalence restricts to lax symmetric
    monoidal functors. Cocartesian transport in
    $\TwoFun{K}{\TwoCat C^\otimes}^{\ptw}$ is obtained by
    applying $\TwoFun{K}{-}$ to cocartesian transport in
    $\TwoCat C^\otimes$. Consequently, a $1$-morphism over
    a fixed arrow of $\Fin_*$ is $2$-cocartesian if and only
    if its evaluation at every object of $K$ is
    $2$-cocartesian. Indeed, after factoring through a
    cocartesian lift, this assertion reduces to the fact
    that equivalences in a functor $2$-category are detected
    pointwise.

    It follows that a functor
    $\TwoCat D^\otimes \to
    \TwoFun{K}{\TwoCat C^\otimes}^{\ptw}$
    preserves $2$-cocartesian $1$-morphisms over inert maps
    if and only if each functor in the corresponding
    $K$-indexed family does so.
    Since the mapping categories of
    $\TwoCAlgLax{\TwoCatInfTwo}$ are full subcategories of
    the mapping categories in the slice, the preceding
    equivalence therefore restricts to
    $$\Map{\TwoCAlgLax{\TwoCatInfTwo}}{\TwoCat D^\otimes}{\TwoFun{K}{\TwoCat C^\otimes}^{\ptw}} \cong \Fun{K}{\Map{\TwoCAlgLax{\TwoCatInfTwo}}{\TwoCat D^\otimes}{\TwoCat C^\otimes}}.$$
    This is the universal property of the cotensor by $K$.

    For $\delta=\emptyset$, the same argument applies,
    requiring preservation of all $2$-cocartesian
    $1$-morphisms instead of only those over inert maps.
    For $\delta=\mathrm{oplax}$, use the identical argument
    in the slice over $\op{\Fin_*}$, replacing cocartesian
    transport by cartesian transport.

    All constructions are natural in $K$ and in the
    symmetric monoidal $2$-categories. The final assertion
    follows by taking maximal subgroupoids.
\end{proof}

%% file: symmetric-monoidal/marked-maps.tex
\subsection{Marked mapping spaces} \label{sec:sym-marked}
We now discuss the categories of marked $2$-categories and marked symmetric monoidal $2$-categories.

\begin{defn}
    We define the \defemphindex{category of marked $2$-categories}{marked!$2$-category} as the full subcategory 
    $$\textcolor{purple}{\CatInfTwo^\dagger} \subseteq \Fun{\Delta^1}{\CatInfTwo}$$ spanned by wide 2-full sub-$2$-category inclusions.
\end{defn}

There is also a symmetric monoidal version of a marked $2$-category:
\begin{defn}
    A \defemphindex{marked symmetric monoidal $2$-category}{marked!symmetric monoidal $2$-category} $(\TwoCat C^\otimes, \dagger)$ 
    is a symmetric monoidal $2$-category $\TwoCat C^\otimes$, together with a marking $\dagger \subseteq \TwoCat C$ 
    on the underlying $2$-category.
\end{defn}

\begin{rmk}
    It may come as a shock to the reader that we do \emph{not} require a compatibility of the marking with the tensor product.
    This is intentional, as later we will care about symmetric monoidal oplax natural transformations between symmetric monoidal functors,
    whose underlying oplax transformation is strong over some marking. Since this is just a condition on the underlying non-multiplicative object,
    we require no compatibility of the marking with the symmetric monoidal structure.
\end{rmk}
Again, marked symmetric monoidal $2$-categories fit together into a category:
\begin{defn}
    Let $\delta \in \oplaxvarset$.
    We define the \defemphindex{category of marked symmetric monoidal $2$-categories with $\delta$-symmetric monoidal functors}{marked!symmetric monoidal $2$-category} as the top left category in the pullback diagram 
    \begin{center}
        \begin{tikzcd}
            \textcolor{purple}{\CAlgDelta{\TwoCatInfTwo}^\dagger} \ar[r] \ar[d] \cartsymb &\CAlgDelta{\TwoCatInfTwo} \ar[d] \\
            \CatInfTwo^\dagger \ar[r] &\CatInfTwo,
        \end{tikzcd}
    \end{center}
    where the bottom and right functors are the canonical forgetful functors.
\end{defn}

\begin{notation}
    Let $(\TwoCat C, \dagger)$ and $(\TwoCat D, \ddagger)$ be two marked $2$-categories.
    We write 
    $$\textcolor{purple}{\MapMarked{\CatInfTwo}{\dagger}{\TwoCat C}{\TwoCat D}} \coloneqq \Map{\CatInfTwo^\dagger}{(\TwoCat C, \dagger)}{(\TwoCat D, \ddagger)},$$
    suppressing the marking on the categories.

    Similarly, if $\delta \in \oplaxvarset$ and $(\TwoCat C^\otimes, \dagger)$ and $(\TwoCat D^\otimes, \ddagger)$ are marked symmetric monoidal $2$-categories 
    $$\textcolor{purple}{\MapMarked{\CAlgDelta{\TwoCatInfTwo}}{\dagger}{\TwoCat C^\otimes}{\TwoCat D^\otimes}} \coloneqq \Map{\CAlgDelta{\TwoCatInfTwo}^\dagger}{(\TwoCat C^\otimes, \dagger)}{(\TwoCat D^\otimes, \ddagger)}.$$
\end{notation}

\begin{rmk} \label{rmk:marked:mapping-space-pullback}
    Let $(\TwoCat C, \dagger)$ and $(\TwoCat D, \ddagger)$ be two marked $2$-categories.
    Then, by the formula for mapping anima in arrow categories (see e.g.\ \cite[\href{https://kerodon.net/tag/06BZ}{Theorem 06BZ}]{kerodon} with $F_+ = F_- = \id{}$), there is a pullback square of anima 
    \begin{center}
        \begin{tikzcd}
            \MapMarked{\CatInfTwo}{\dagger}{\TwoCat C}{\TwoCat D} \ar[r] \ar[d] \cartsymb & \Map{\CatInfTwo}{\TwoCat C}{\TwoCat D} \ar[d]\\
            \Map{\CatInfTwo}{\dagger}{\ddagger} \ar[r]& \Map{\CatInfTwo}{\dagger}{\TwoCat D}\rlap{.} 
        \end{tikzcd}
    \end{center}
\end{rmk}

\begin{rmk} \label{rmk:marked:symm-monoidal-mapping-space-pullback}
    Let $\delta \in \oplaxvarset$ and $(\TwoCat C^\otimes, \dagger)$ and $(\TwoCat D^\otimes, \ddagger)$ be two marked symmetric monoidal $2$-categories.
    Then, since mapping anima in pullbacks of categories are computed as the pullback of the respective mapping anima, 
    there is a cartesian diagram of anima 
    \begin{center}
        \begin{tikzcd}
            \MapMarked{\CAlgDelta{\TwoCatInfTwo}}{\dagger}{\TwoCat C^\otimes}{\TwoCat D^\otimes} \ar[d] \ar[r] \cartsymb &\Map{\CAlgDelta{\TwoCatInfTwo}}{\TwoCat C^\otimes}{\TwoCat D^\otimes} \ar[d] \\
            \MapMarked{\CatInfTwo}{\dagger}{\TwoCat C}{\TwoCat D} \ar[r] &\Map{\CatInfTwo}{\TwoCat C}{\TwoCat D} \rlap{,}
        \end{tikzcd}
    \end{center}
    natural in $\TwoCat C^\otimes, \TwoCat D^\otimes \in \CAlgDelta{\TwoCatInfTwo}^\dagger$.
\end{rmk}

\begin{prop} \label{lem:marking:fun-marking}
    Let $\epsilon \in \oplaxvarset$.
    The functor $$\FunEps{-}{-} \colon \op{\CatInfTwo} \times \CatInfTwo \to \CatInfTwo$$
    naturally upgrades to a functor $$\FunEps{-}{-} \colon \op{\CatInfTwo} \times \CatInfTwo \to \CatInfTwo^\dagger,$$
    such that the marking is given by the natural transformation $\TwoFun{-}{-} \to \FunEps{-}{-}$ from \cref{lem:lax-nat-trans:fun-to-lax-fun}.
\end{prop}
\begin{proof}
    This follows from currying. The only thing left to show is that the natural transformation 
    is pointwise a wide sub-$2$-category inclusion. This was shown in \cref{lem:lax-nat-trans:fun-to-lax-fun}.
\end{proof}

\begin{rmk}
    If we view $\FunEps{\TwoCat C}{\TwoCat D}$ as a marked $2$-category, we will always do this via the marking from \cref{lem:marking:fun-marking}.
    We call this the \defemphindex{strong marking}{marking!strong}.
\end{rmk}

\begin{prop} \label{lem:marking:fun-symm-marking}
    Let $\epsilon, \delta \in \oplaxvarset$.
    The functor $$\FunEps{-}{-}^\ptw \colon \op{\CatInfTwo} \times \CAlgDelta{\TwoCatInfTwo} \to \CAlgDelta{\TwoCatInfTwo}$$
    naturally upgrades to a functor $$\FunEps{-}{-}^\ptw \colon \op{\CatInfTwo} \times \CAlgDelta{\TwoCatInfTwo} \to \CAlgDelta{\TwoCatInfTwo}^\dagger,$$
    such that the marking is given by the natural transformation $\TwoFun{-}{-} \to \FunEps{-}{-}$ from \cref{lem:lax-nat-trans:fun-to-lax-fun}.
\end{prop}
\begin{proof}
    Follows immediately from \cref{lem:marking:fun-marking} and the construction of $\CAlgDelta{\TwoCatInfTwo}^\dagger$ as a pullback.
\end{proof}

\begin{rmk}
    If we view $\FunEps{\TwoCat C}{\TwoCat D^\otimes}^\ptw$ as a marked symmetric monoidal $2$-category, we will always do this via the marking from \cref{lem:marking:fun-symm-marking}.
    We again call this the \defemphindex{strong marking}{marking!strong}.
\end{rmk}

\begin{prop} \label{lem:marked:marked-mapping-space-UP}
    Let $\epsilon \in \oplaxvarset$.
    Let $(\TwoCat C, \dagger)$ be a marked $2$-category, and let $\TwoCat D$ and $\TwoCat E$ be two further $2$-categories.
    We have an equivalence 
    $$\Map{\CatInfTwo}{\TwoCat E}{\FunPartEps{\dagger}{\TwoCat C}{\TwoCat D}} \cong \MapMarked{\CatInfTwo}{\dagger}{\TwoCat C}{\FunNEps{\TwoCat E}{\TwoCat D}},$$
    natural in $(\TwoCat C, \dagger)$, $\TwoCat D$ and $\TwoCat E$. 
\end{prop}
\begin{proof}
    By \cref{rmk:marked:mapping-space-pullback} we have a pullback square 
    \begin{center}
        \begin{tikzcd}
            \MapMarked{\CatInfTwo}{\dagger}{\TwoCat C}{\FunNEps{\TwoCat E}{\TwoCat D}} \ar[r] \ar[d] \cartsymb & \Map{\CatInfTwo}{\TwoCat C}{\FunNEps{\TwoCat E}{\TwoCat D}} \ar[d]\\
            \Map{\CatInfTwo}{\dagger}{\TwoFun{\TwoCat E}{\TwoCat D}} \ar[r]& \Map{\CatInfTwo}{\dagger}{\FunNEps{\TwoCat E}{\TwoCat D}}\rlap{.} 
        \end{tikzcd}
    \end{center}
    On the other hand, by definition of $\FunPartEps{\dagger}{\TwoCat C}{\TwoCat D}$, we obtain a pullback square 
    \begin{center}
        \begin{tikzcd}
            \Map{\CatInfTwo}{\TwoCat E}{\FunPartEps{\dagger}{\TwoCat C}{\TwoCat D}} \ar[r] \ar[d] \cartsymb & \Map{\CatInfTwo}{\TwoCat E}{\FunEps{\TwoCat C}{\TwoCat D}} \ar[d]\\
            \Map{\CatInfTwo}{\TwoCat E}{\TwoFun{\dagger}{\TwoCat D}} \ar[r]& \Map{\CatInfTwo}{\TwoCat E}{\FunEps{\dagger}{\TwoCat D}}\rlap{.} 
        \end{tikzcd}
    \end{center}
    Now, double currying gives equivalences on the bottom and right objects of the two diagrams, compatible with the restriction morphisms.
    Hence, we obtain the required equivalence of the pullbacks.
    The required naturality follows since all the double currying equivalences were already natural.
\end{proof}

\begin{variant}
    Let $\epsilon \in \oplaxvarset$.
    One can show that the construction $\FunPartEps{\dagger}{\TwoCat C}{\TwoCat D}$
    is functorial, i.e.\ upgrades to a functor 
    $$\defnot{\FunPartEps{\dagger}{-}{-}} \colon \op{(\CatInfTwo^\dagger)} \times \CatInfTwo \to \CatInfTwo.$$
\end{variant}

\begin{variant} \label{var:pointwise:partially-lax-pointwise}
    Let $\epsilon, \delta \in \oplaxvarset$.
    Using a similar construction as in \cref{lem:2-sym:pointwise-structure-functoriality}, one can define a functor
    $$\textcolor{purple}{\FunPartEps{\dagger}{-}{-}^{\ptw}} \colon \op{(\CatInfTwo^\dagger)} \times \CAlgDelta{\TwoCatInfTwo} \to \CAlgDelta{\TwoCatInfTwo},$$
    such that the diagram
    \begin{center}
        \begin{tikzcd}
            \op{(\CatInfTwo^\dagger)} \times \CAlgDelta{\TwoCatInfTwo} \ar[d] \ar[rr, "{\FunPartEps{\dagger}{-}{-}^{\ptw}}"] &&\CAlgDelta{\TwoCatInfTwo} \ar[d] \\
            \op{(\CatInfTwo^\dagger)} \times \CatInfTwo \ar[rr, "{\FunPartEps{\dagger}{-}{-}}"] &&\CatInfTwo
        \end{tikzcd}
    \end{center}
    commutes.
\end{variant}

Using this functoriality, we obtain:
\begin{prop} \label{lem:lax-lim:partial-lax-lim-lax-sym-mon}
    Let $\TwoCat{I}^\dagger$ be a marked $2$-category, and let $\TwoCat{M}^\otimes$ be a symmetric monoidal $2$-category,
    so that the underlying $2$-category $\TwoCat{M}$ admits all partially lax limits of shape $\TwoCat{I}^\dagger$.
    Then the partially lax limit functor $\laxlim_{\TwoCat{I}}^\dagger \colon \FunPartLax{\dagger}{\TwoCat{I}}{\TwoCat{M}} \to \TwoCat{M}$
    upgrades to a lax symmetric monoidal functor 
    $$\laxlim_{\TwoCat{I}}^\dagger \colon \FunPartLax{\dagger}{\TwoCat{I}}{\TwoCat{M}^\otimes}^\ptw \to \TwoCat{M}^\otimes.$$
    Moreover, the adjunction $\mathrm{const} \dashv \laxlim_{\TwoCat{I}}^\dagger$ upgrades to an adjunction in $\TwoCAlgLax{\TwoCatInfTwo}$.
\end{prop}
\begin{proof}
    The constant diagram functor $\mathrm{const} \colon \TwoCat{M} \to \FunPartLax{\dagger}{\TwoCat{I}}{\TwoCat{M}}$ 
    is symmetric monoidal: Indeed, it is induced by the marked functor $\TwoCat I^\dagger \to (\Delta^0)^\sharp$,
    and hence, using the functoriality of \cref{var:pointwise:partially-lax-pointwise},
    upgrades to a symmetric monoidal functor $\mathrm{const} \colon \TwoCat{M}^\otimes \to \FunPartLax{\dagger}{\TwoCat{I}}{\TwoCat{M}^\otimes}^\ptw$.
    Thus, since $\laxlim_{\TwoCat{I}}^\dagger$ is right adjoint to $\mathrm{const}$ by definition,
    the result follows from \cref{lem:2-sym:adjunction-in-CAlgLax}.
\end{proof}

%% file: symmetric-monoidal/oplax-lax-sym.tex
\subsection{Op/Lax symmetric monoidal functors with (partially) op/lax natural transformations} \label{sec:sym-oplax-lax}
In this section, following \cite[Definition 6.10]{cnossen2025universalityspan2categoriesconstruction} we construct 
the $2$-category $\FunPartOplaxLaxTensor{\dagger}{\TwoCat C^\otimes}{\TwoCat D^\otimes}$ of lax symmetric monoidal functors between $\TwoCat C^\otimes$ and $\TwoCat D^\otimes$,
together with partially oplax symmetric monoidal natural transformations between them.
Unfortunately, this is not the mapping $2$-category of some $3$-category of symmetric monoidal $2$-categories with the correct notion of functors and natural transformations 
between them. Thus, following \loccit we will define it via a universal property.
This universal property will furthermore imply that $\FunPartOplaxLaxTensor{\dagger}{\TwoCat C^\otimes}{-}$ is a $1$-functor $\CAlgLax{\TwoCatInfTwo} \to \CatInfTwo$.
Later, we will need that this functor preserves adjunctions. Thus, it is natural to ask if it upgrades to a $2$-functor $\TwoCAlgLax{\TwoCatInfTwo} \to \TwoCatInfTwo$.
This is true, but the construction is tedious. Therefore, in this section, we will construct enough of the $2$-functoriality to be able to show that it preserves adjoints,
which will be the datum of a functor $\TwoCAlgLax{\TwoCatInfTwo} \to \underline{h} \, \TwoCatInfTwo$ into the homotopy $(2,2)$-category of $2$-categories, extending the 
$1$-functor given by the universal property. Moreover, we will do this construction for all possible variances, both in the 
symmetric monoidal functor direction, and the natural transformation direction.

For the rest of this section, we fix two variances $\epsilon, \delta \in \oplaxvarset$.

Recall that we can consider $\FunEps{-}{-}^\ptw$ as a functor to marked symmetric monoidal categories
with $\delta$-symmetric monoidal functors, cf.\ \cref{lem:marking:fun-symm-marking}.

\begin{lem} \label{constr:2-functoriality-of-lax-tensor-oplax:factor-Yoneda}
    The functor 
    $$\Phi \colon \op{(\CAlgDelta{\TwoCatInfTwo}^\dagger)} \times \CAlgDelta{\TwoCatInfTwo} \to \Fun{\op{\CatInfTwo}}{\AnL}$$ given by 
    currying the functor 
    $$\MapMarked{\CAlgDelta{\TwoCatInfTwo}}{\dagger}{-}{\FunNEps{-}{-}^\ptw} \colon \op{(\CAlgDelta{\TwoCatInfTwo}^\dagger)} \times \op{\CatInfTwo} \times \CAlgDelta{\TwoCatInfTwo} \to \AnL,$$
    i.e.\ 
    $$\Phi((\TwoCat C^\otimes,\dagger), \Cat D^\otimes) = (\TwoCat E \mapsto \MapMarked{\CAlgDelta{\TwoCatInfTwo}}{\dagger}{\TwoCat C^\otimes}{\FunNEps{\TwoCat E}{\TwoCat D^\otimes}^{\ptw}}),$$
    lands in the essential image of the Yoneda embedding.
\end{lem}
\begin{proof}
    By \cref{rmk:marked:symm-monoidal-mapping-space-pullback}, there is a cartesian square of functors $\op{\CatInfTwo} \to \AnL$.
    \begin{center}
        \begin{tikzcd}
            \MapMarked{\CAlgDelta{\TwoCatInfTwo}}{\dagger}{\TwoCat C^\otimes}{\FunNEps{-}{\Cat D^\otimes}^\ptw} \ar[d] \ar[r] \cartsymb &\Map{\CAlgDelta{\TwoCatInfTwo}}{\TwoCat C^\otimes}{\FunNEps{-}{\TwoCat D^\otimes}^\ptw} \ar[d] \\
            \MapMarked{\CatInfTwo}{\dagger}{\TwoCat C}{\FunNEps{-}{\TwoCat D}} \ar[r]&\Map{\CatInfTwo}{\TwoCat C}{\FunNEps{-}{\TwoCat D}}\rlap{.}
        \end{tikzcd}
    \end{center}
    Since the Yoneda embedding preserves limits \cite[Proposition 5.1.3.2]{highertopoi},
    it suffices to show that the three other corners of this diagram lie in the essential image.

    The bottom corners are representable: Indeed, 
    by definition we have 
    $$\Map{\CatInfTwo}{\TwoCat C}{\FunNEps{-}{\TwoCat D}} \cong \Map{\CatInfTwo}{-}{\FunEps{\TwoCat C}{\TwoCat D}},$$
    for the bottom right corner, and for the bottom left corner we have 
    $$\Map{\CatInfTwo^\dagger}{\TwoCat C}{\FunNEps{-}{\TwoCat D}} \cong \Map{\CatInfTwo}{-}{\FunPartEps{\dagger}{\TwoCat C}{\TwoCat D}}$$
    by \cref{lem:marked:marked-mapping-space-UP}.

    In order to show that the top right corner is representable, it suffices to show that it preserves limits,
    cf.\ \cite[Proposition 5.5.2.2]{highertopoi}.
    Since the mapping anima preserves limits in the second variable,
    it suffices to show that the functor
    $$\FunNEps{-}{\TwoCat D^\otimes}^\ptw \colon \op{\CatInfTwo} \to \CAlgDelta{\TwoCatInfTwo}$$
    preserves limits.
    This was shown in the proof of \cite[Lemma 6.11]{cnossen2025universalityspan2categoriesconstruction},
    its dual proof if $\delta = \mathrm{oplax}$, or its adaptation to strong monoidal functors 
    or strong natural transformations if $\delta = \emptyset$ or $\epsilon = \emptyset$, respectively.
\end{proof}

\begin{constr} \label{constr:2-functoriality-of-lax-tensor-oplax:1-functoriality}
    In view of \cref{constr:2-functoriality-of-lax-tensor-oplax:factor-Yoneda} and the full faithfulness of the Yoneda embedding,
    we see that the functor 
    $$\Phi \colon \op{(\CAlgDelta{\TwoCatInfTwo}^\dagger)} \times \CAlgDelta{\TwoCatInfTwo} \to \Fun{\op{\CatInfTwo}}{\AnL}$$
    factors uniquely over a functor 
    $$\textcolor{purple}{\FunPartEpsDeltaTensor{\dagger}{-}{-}} \colon \op{(\CAlgDelta{\TwoCatInfTwo}^\dagger)} \times \CAlgDelta{\TwoCatInfTwo} \to \CatInfTwo.$$
\end{constr}

By universal property and the Yoneda lemma, we obtain:
\begin{lem} \label{lem:2-functoriality-of-lax-tensor-oplax:1-functoriality-univ-prop}
    \uses{constr:2-functoriality-of-lax-tensor-oplax:1-functoriality}
    Let $(\TwoCat C^\otimes, \dagger)$ be a marked symmetric monoidal $2$-category, $\TwoCat D^\otimes$ be a symmetric monoidal $2$-category and let $\TwoCat E$ be a $2$-category.
    Then there is an equivalence, natural in $(\TwoCat C^\otimes, \dagger)$, $\TwoCat D^\otimes$ and $\TwoCat E$:
    $$\Map{\CatInfTwo}{\TwoCat E}{\FunPartEpsDeltaTensor{\dagger}{\TwoCat C^\otimes}{\TwoCat D^\otimes}} \cong \MapMarked{\CAlgDelta{\TwoCatInfTwo}}{\dagger}{\TwoCat C^\otimes}{\FunNEps{\TwoCat E}{\TwoCat D^\otimes}^\ptw}.$$
\end{lem}

\begin{rmk} \label{lem:2-functoriality-of-lax-tensor-oplax:explicit-morphisms}
    Applying \cref{lem:2-functoriality-of-lax-tensor-oplax:1-functoriality-univ-prop} with $\TwoCat E = \Delta^0$,
    we see that objects of $\FunPartEpsDeltaTensor{\dagger}{\TwoCat C^\otimes}{\TwoCat D^\otimes}$
    are given by $\delta$-symmetric monoidal functors $\TwoCat C^\otimes \to \TwoCat D^\otimes$.

    Similarly, applying \loccit with $\TwoCat E = \Delta^1$ we see that morphisms of $\FunPartEpsDeltaTensor{\dagger}{\TwoCat C^\otimes}{\TwoCat D^\otimes}$ 
    are given by partially $\epsilon$-natural transformations $t \colon F \to G$ between $\delta$-symmetric monoidal functors $F, G \colon \TwoCat{C}^\otimes \to \TwoCat{D}^\otimes$ together 
    with a coherent collection of $2$-morphisms of the form (we write it for $\delta=\mathrm{lax}$ and $\epsilon=\mathrm{oplax}$)
    \begin{center}
        \begin{tikzcd}
            F(X) \otimes F(Y) \ar[r] \ar[d, "t_X \otimes t_Y"'] & F(X \otimes Y) \ar[d, "t_{X \otimes Y}"] \laxarrowdl{} \\
            G(X) \otimes G(Y) \ar[r] & G(X \otimes Y) \rlap{,}
        \end{tikzcd}
    \end{center}
    where the direction of the $2$-morphism depends on $\mathrm{op}(\epsilon)$ and the horizontal maps are given by the $\delta$-symmetric monoidal structures of $F$ and $G$, respectively.
    Indeed, by forgetting the $\delta$-symmetric monoidal structure and applying \Cref{lem:marked:marked-mapping-space-UP}, we see that the underlying morphisms are exactly partially $\epsilon$-natural transformations. 
    Moreover, by forgetting the marking and postcomposing with the symmetric monoidal functors $\ev{0}, \ev{1} \colon \FunNEps{\Delta^1}{\TwoCat{D}^\otimes}^\ptw \to \TwoCat{D}^\otimes$, 
    we see that $F$ and $G$ have to be $\delta$-symmetric monoidal. 
    The data of the above diagrams is induced by the $\delta$-symmetric monoidal structure of the corresponding functor $\TwoCat{C}^\otimes \to \FunNEps{\Delta^1}{\TwoCat{D}^\otimes}^\ptw$.

    In particular, this description shows that if $t \colon F \to G$ is a \emph{symmetric monoidal} partially $\epsilon$-natural transformation between $\delta$-symmetric monoidal functors, 
    then it can be regarded as a morphism of $\FunPartEpsDeltaTensor{\dagger}{\TwoCat{C}^\otimes}{\TwoCat{D}^\otimes}$.
\end{rmk}

There is a natural comparison functor that forgets about the lax symmetric monoidal structure of a functor:
\begin{lem} \label{lem:2-functoriality-of-lax-tensor-oplax:forgetful-map}
    There is a forgetful map
    $$\FunPartEpsDeltaTensor{\dagger}{\TwoCat C^\otimes}{\TwoCat D^\otimes} \to \FunPartEps{\dagger}{\TwoCat C}{\TwoCat D},$$
    natural in $(\TwoCat C^\otimes, \dagger)$ and $\TwoCat D^\otimes$.
\end{lem}
\begin{proof}
    Let $\TwoCat E$ be a $2$-category, $(\TwoCat C^\otimes, \dagger) \in \CAlgDelta{\TwoCatInfTwo}^\dagger$ and $\TwoCat D^\otimes \in \CAlgDelta{\TwoCatInfTwo}$. It suffices to provide this natural transformation on Yoneda.
    Hence, we will construct a map 
    $$\Map{\CatInfTwo}{\TwoCat E}{\FunPartEpsDeltaTensor{\dagger}{\TwoCat C^\otimes}{\TwoCat D^\otimes}} \to \Map{\CatInfTwo}{\TwoCat E}{\FunPartEps{\dagger}{\TwoCat C}{\TwoCat D}},$$
    natural in $\TwoCat E$, $(\TwoCat C^\otimes,\dagger)$ and $\TwoCat D^\otimes$.
    By \cref{lem:2-functoriality-of-lax-tensor-oplax:1-functoriality-univ-prop}, the left hand side is equivalent to 
    $$\MapMarked{\CAlgDelta{\TwoCatInfTwo}}{\dagger}{\TwoCat C^\otimes}{\FunNEps{\TwoCat E}{\TwoCat D^\otimes}^\ptw}.$$
    On the other hand, by \cref{lem:marked:marked-mapping-space-UP}, the right hand side is equivalent to 
    $$\MapMarked{\CatInfTwo}{\dagger}{\TwoCat C}{\FunNEps{\TwoCat E}{\TwoCat D}}.$$
    Hence, we can choose the map induced by the forgetful functor $\CAlgDelta{\TwoCatInfTwo}^\dagger \to \CatInfTwo^\dagger$.
\end{proof}

From now on, fix a marked symmetric monoidal $2$-category $(\TwoCat C^\otimes, \dagger)$.
Our goal is to show that the functor $\FunPartEpsDeltaTensor{\dagger}{\TwoCat C^\otimes}{-}$ preserves adjunctions.

\begin{lem} \label{lem:2-functoriality-of-lax-tensor-oplax:cotensor-comparison}
    Let $K$ be a category and let $\TwoCat N^\otimes \in \CAlgDelta{\TwoCatInfTwo}$.
    There is a canonical comparison
    $$\sigma_{K,\TwoCat N^\otimes}\colon \FunPartEpsDeltaTensor{\dagger}{\TwoCat C^\otimes}{\TwoFun{K}{\TwoCat N^\otimes}} \longrightarrow \TwoFun{K}{\FunPartEpsDeltaTensor{\dagger}{\TwoCat C^\otimes}{\TwoCat N^\otimes}},$$
    natural in $K$ and $\TwoCat N^\otimes$ (with $\delta$-symmetric monoidal functors, i.e.\ in $\CAlgDelta{\TwoCatInfTwo}$).
    These comparisons satisfy
    $$\sigma_{*,\TwoCat N^\otimes}\simeq\mathrm{id} \colon \FunPartEpsDeltaTensor{\dagger}{\TwoCat C^\otimes}{\TwoCat N^\otimes} \to \FunPartEpsDeltaTensor{\dagger}{\TwoCat C^\otimes}{\TwoCat N^\otimes}$$
    and
    $$\sigma_{K\times L,\TwoCat N^\otimes} \simeq \TwoFun{K}{\sigma_{L,\TwoCat N^\otimes}} \circ \sigma_{K,\TwoFun{L}{\TwoCat N^\otimes}}.$$
    These equivalences are compatible with associativity and units.
\end{lem}
\begin{proof}
    We will construct this comparison again after applying Yoneda.
    Hence, we will construct for every $2$-category $\TwoCat E$ a morphism 
    $$\Map{\CatInfTwo}{\TwoCat E}{\FunPartEpsDeltaTensor{\dagger}{\TwoCat C^\otimes}{\TwoFun{K}{\TwoCat N^\otimes}^\ptw}} \to  \Map{\CatInfTwo}{\TwoCat E}{\TwoFun{K}{\FunPartEpsDeltaTensor{\dagger}{\TwoCat C^\otimes}{\TwoCat N^\otimes}}},$$
    natural in $\TwoCat E$, $K$ and $\TwoCat N^\otimes$.
    By \cref{lem:2-functoriality-of-lax-tensor-oplax:1-functoriality-univ-prop}, the left hand side is equivalent to 
    $$\MapMarked{\CAlgDelta{\TwoCatInfTwo}}{\dagger}{\TwoCat C^\otimes}{\FunNEps{\TwoCat E}{\TwoFun{K}{\TwoCat N^\otimes}^\ptw}^{\ptw}}.$$
    Similarly, by \loccit and the adjunction $- \times K \dashv \TwoFun{K}{-}$, the right hand side is equivalent to 
    $$\MapMarked{\CAlgDelta{\TwoCatInfTwo}}{\dagger}{\TwoCat C^\otimes}{\FunNEps{K\times\TwoCat E}{\TwoCat N^\otimes}^{\ptw}}.$$
    Recall that we constructed in \cref{lem:2-sym:pointwise-comparison} a comparison functor 
    $$\FunNEps{\TwoCat E}{\TwoFun{K}{\TwoCat N^\otimes}^{\ptw}}^\ptw \to \FunNEps{K\times\TwoCat E}{\TwoCat N^\otimes}^{\ptw},$$
    together with its compatibility with the strong marking.
    Since $\MapMarked{\CAlgDelta{\TwoCatInfTwo}}{\dagger}{-}{-}$ is functorial in $\CAlgDelta{\TwoCatInfTwo}^\dagger$ (since it is the mapping anima in that category), 
    we obtain a natural transformation 
    $$\MapMarked{\CAlgDelta{\TwoCatInfTwo}}{\dagger}{\TwoCat C^\otimes}{\FunNEps{\TwoCat E}{\TwoFun{K}{\TwoCat N^\otimes}^\ptw}^{\ptw}} \to \MapMarked{\CAlgDelta{\TwoCatInfTwo}}{\dagger}{\TwoCat C^\otimes}{\FunNEps{K\times\TwoCat E}{\TwoCat N^\otimes}^{\ptw}}.$$
    Thus, via the Yoneda lemma we obtain 
    $$\sigma_{K,\TwoCat N^\otimes}\colon \FunPartEpsDeltaTensor{\dagger}{\TwoCat C^\otimes}{\TwoFun{K}{\TwoCat N^\otimes}} \longrightarrow \TwoFun{K}{\FunPartEpsDeltaTensor{\dagger}{\TwoCat C^\otimes}{\TwoCat N^\otimes}},$$
    natural in $K$ and $\TwoCat N^\otimes$.

    For every test $2$-category $\TwoCat X$, the comparison of \cref{lem:2-sym:pointwise-comparison} is induced, under the tensor--hom adjunctions, by
    $$\lambda_{K,\TwoCat E,\TwoCat X}\colon\TwoCat X\otimes^{\mathrm{op}(\epsilon)}(K\times\TwoCat E)\xrightarrow{(!\otimes\mathrm{pr}_K,\,\mathrm{id}\otimes\mathrm{pr}_{\TwoCat E})}(*\otimes^{\mathrm{op}(\epsilon)}K)\times(\TwoCat X\otimes^{\mathrm{op}(\epsilon)}\TwoCat E)\simeq K\times(\TwoCat X\otimes^{\mathrm{op}(\epsilon)}\TwoCat E).$$
    Here $\otimes^{\mathrm{op}(\epsilon)}$ is the Gray tensor product for $\epsilon=\mathrm{oplax}$, its reverse for $\epsilon=\mathrm{lax}$, and the cartesian product for $\epsilon=\emptyset$.

    For $K=*$, the following diagram commutes, where the vertical maps are the canonical unit equivalences:
    $$\begin{tikzcd}[column sep=large] \TwoCat X\otimes^{\mathrm{op}(\epsilon)}(*\times\TwoCat E) \ar[r,"\lambda_{*,\TwoCat E,\TwoCat X}"] \ar[d,"\simeq"'] & *\times(\TwoCat X\otimes^{\mathrm{op}(\epsilon)}\TwoCat E) \ar[d,"\simeq"] \\ \TwoCat X\otimes^{\mathrm{op}(\epsilon)}\TwoCat E \ar[r,"\mathrm{id}"'] & \TwoCat X\otimes^{\mathrm{op}(\epsilon)}\TwoCat E. \end{tikzcd}$$
    Thus the comparison for the terminal parameter category is canonically the identity.

    For products of parameter categories, suppressing canonical associativity equivalences, we have the commutative diagram
    $$\begin{tikzcd}[column sep=large] \TwoCat X\otimes^{\mathrm{op}(\epsilon)}((K\times L)\times\TwoCat E) \ar[r,"\lambda_{K,L\times\TwoCat E,\TwoCat X}"] \ar[d,"\lambda_{K\times L,\TwoCat E,\TwoCat X}"'] & K\times(\TwoCat X\otimes^{\mathrm{op}(\epsilon)}(L\times\TwoCat E)) \ar[d,"\mathrm{id}_K\times\lambda_{L,\TwoCat E,\TwoCat X}"] \\ (K\times L)\times(\TwoCat X\otimes^{\mathrm{op}(\epsilon)}\TwoCat E) \ar[r,"\simeq"'] & K\times(L\times(\TwoCat X\otimes^{\mathrm{op}(\epsilon)}\TwoCat E)). \end{tikzcd}$$
    Both composites have the same projections to $K$, $L$, and $\TwoCat X\otimes^{\mathrm{op}(\epsilon)}\TwoCat E$.
    The universal property of products gives these identifications coherently for any number of parameter categories.
    The pointwise monoidal comparisons inherit these coherences by their pullback construction, and Yoneda transports them to the comparisons $\sigma_{K,\TwoCat N^\otimes}$.
\end{proof}

\begin{constr} \label{constr:2-functoriality-of-lax-tensor-oplax:mapping-categories-upgrade}
    We construct, for every pair of symmetric monoidal $2$-categories $\TwoCat M^\otimes$ and $\TwoCat N^\otimes$, a functor
    $$g \colon \Map{\TwoCAlgDelta{\TwoCatInfTwo}}{\TwoCat M^\otimes}{\TwoCat N^\otimes} \longrightarrow \Map{\TwoCatInfTwo}{\FunPartEpsDeltaTensor{\dagger}{\TwoCat C^\otimes}{\TwoCat M^\otimes}}{\FunPartEpsDeltaTensor{\dagger}{\TwoCat C^\otimes}{\TwoCat N^\otimes}}.$$
    By \cref{lem:2-sym:category-indexed-cotensors}, for every category $K$ we have a natural equivalence
    \begin{equation*} 
        \Map{\CAlgDelta{\TwoCatInfTwo}}{\TwoCat M^\otimes}{\TwoFun{K}{\TwoCat N^\otimes}^{\ptw}} \simeq \Map{\CatInf}{K}{\Map{\TwoCAlgDelta{\TwoCatInfTwo}}{\TwoCat M^\otimes}{\TwoCat N^\otimes}}.
    \end{equation*}
    We use this equivalence for $K=\Map{\TwoCAlgDelta{\TwoCatInfTwo}}{\TwoCat M^\otimes}{\TwoCat N^\otimes}$.
    Thus, the identity functor of $K$ corresponds to a $\delta$-symmetric monoidal functor
    $$\mathrm{ev}\colon \TwoCat M^\otimes \longrightarrow \TwoFun{K}{\TwoCat N^\otimes}^\ptw.$$
    Applying \cref{constr:2-functoriality-of-lax-tensor-oplax:1-functoriality} and then \cref{lem:2-functoriality-of-lax-tensor-oplax:cotensor-comparison} gives
    \begin{align*}
        \FunPartEpsDeltaTensor{\dagger}{\TwoCat C^\otimes}{\TwoCat M^\otimes} 
        &\xrightarrow{\FunPartEpsDeltaTensor{\dagger}{\TwoCat C^\otimes}{\mathrm{ev}}} \FunPartEpsDeltaTensor{\dagger}{\TwoCat C^\otimes}{\TwoFun{K}{\TwoCat N^\otimes}^\ptw} \\
        &\xrightarrow{\sigma_{K,\TwoCat N^\otimes}} \TwoFun{K}{\FunPartEpsDeltaTensor{\dagger}{\TwoCat C^\otimes}{\TwoCat N^\otimes}}.
    \end{align*}
    Equivalently, by currying twice, this is a functor
    $$g \colon K\longrightarrow \Map{\TwoCatInfTwo}{\FunPartEpsDeltaTensor{\dagger}{\TwoCat C^\otimes}{\TwoCat M^\otimes}}{\FunPartEpsDeltaTensor{\dagger}{\TwoCat C^\otimes}{\TwoCat N^\otimes}},$$
    which has the required signature.
\end{constr}

\begin{lem} \label{lem:2-functoriality-of-lax-tensor-oplax:compatibility-on-cores}
    For every pair of symmetric monoidal $2$-categories $\TwoCat M^\otimes$ and $\TwoCat N^\otimes$, the square
    \begin{center}
        \begin{tikzcd}
            \Map{\CAlgDelta{\TwoCatInfTwo}}{\TwoCat M^\otimes}{\TwoCat N^\otimes} \ar[d, hook] \ar[r, "f"] & \Map{\CatInfTwo}{\FunPartEpsDeltaTensor{\dagger}{\TwoCat C^\otimes}{\TwoCat M^\otimes}}{\FunPartEpsDeltaTensor{\dagger}{\TwoCat C^\otimes}{\TwoCat N^\otimes}} \ar[d, hook] \\
            \Map{\TwoCAlgDelta{\TwoCatInfTwo}}{\TwoCat M^\otimes}{\TwoCat N^\otimes} \ar[r, "g"'] & \Map{\TwoCatInfTwo}{\FunPartEpsDeltaTensor{\dagger}{\TwoCat C^\otimes}{\TwoCat M^\otimes}}{\FunPartEpsDeltaTensor{\dagger}{\TwoCat C^\otimes}{\TwoCat N^\otimes}}
        \end{tikzcd}
    \end{center}
    commutes.
    Here the vertical maps are the inclusions of cores, the top map is induced by \cref{constr:2-functoriality-of-lax-tensor-oplax:1-functoriality}, and the bottom map is the one constructed in \cref{constr:2-functoriality-of-lax-tensor-oplax:mapping-categories-upgrade}.
\end{lem}
\begin{proof}
    Since the top left corner is an anima, using the adjunction $\mathrm{incl} \colon \AnL \rightleftarrows \CatInf \noloc \iota_0$ it suffices to show that the square 
    \begin{center}
        \begin{tikzcd}
            \Map{\CAlgDelta{\TwoCatInfTwo}}{\TwoCat M^\otimes}{\TwoCat N^\otimes} \ar[d, hook] \ar[r, "f"] & \Map{\CatInfTwo}{\FunPartEpsDeltaTensor{\dagger}{\TwoCat C^\otimes}{\TwoCat M^\otimes}}{\FunPartEpsDeltaTensor{\dagger}{\TwoCat C^\otimes}{\TwoCat N^\otimes}} \ar[d, hook] \\
            \iota_0 \Map{\TwoCAlgDelta{\TwoCatInfTwo}}{\TwoCat M^\otimes}{\TwoCat N^\otimes} \ar[r, "{\iota_0 g}"'] & \iota_0 \Map{\TwoCatInfTwo}{\FunPartEpsDeltaTensor{\dagger}{\TwoCat C^\otimes}{\TwoCat M^\otimes}}{\FunPartEpsDeltaTensor{\dagger}{\TwoCat C^\otimes}{\TwoCat N^\otimes}}
        \end{tikzcd}
    \end{center}
    commutes.
    Note that 
    $$\iota_0 \Map{\TwoCAlgDelta{\TwoCatInfTwo}}{\TwoCat M^\otimes}{\TwoCat N^\otimes} = \Map{\iota_1 \TwoCAlgDelta{\TwoCatInfTwo}}{\TwoCat M^\otimes}{\TwoCat N^\otimes} = \Map{\CAlgDelta{\TwoCatInfTwo}}{\TwoCat M^\otimes}{\TwoCat N^\otimes},$$
    and similarly 
    \begin{align*}
        &\iota_0 \Map{\TwoCatInfTwo}{\FunPartEpsDeltaTensor{\dagger}{\TwoCat C^\otimes}{\TwoCat M^\otimes}}{\FunPartEpsDeltaTensor{\dagger}{\TwoCat C^\otimes}{\TwoCat N^\otimes}} \\
        &= \Map{\iota_1\TwoCatInfTwo}{\FunPartEpsDeltaTensor{\dagger}{\TwoCat C^\otimes}{\TwoCat M^\otimes}}{\FunPartEpsDeltaTensor{\dagger}{\TwoCat C^\otimes}{\TwoCat N^\otimes}} \\
        &= \Map{\CatInfTwo}{\FunPartEpsDeltaTensor{\dagger}{\TwoCat C^\otimes}{\TwoCat M^\otimes}}{\FunPartEpsDeltaTensor{\dagger}{\TwoCat C^\otimes}{\TwoCat N^\otimes}}.
    \end{align*}
    Hence, the vertical maps are the identity.
    Therefore, it suffices to show that the horizontal maps $f$ and $\iota_0 g$ are homotopic.
    By Yoneda, we have to provide for every anima $S \in \AnL$ a homotopy $$\Map{\AnL}{S}{f} \simeq \Map{\AnL}{S}{\iota_0 g},$$ natural in $S$.
    Moreover, again by adjunction, we have $\Map{\AnL}{S}{\iota_0 g} \cong \Map{\CatInf}{S}{g}$.

    For every anima $S$, write $(\TwoCat N^\otimes)^S=\TwoFun{S}{\TwoCat N^\otimes}^{\ptw}$.
    By \cref{lem:2-sym:category-indexed-cotensors} and the adjunction $\mathrm{incl} \dashv \iota_0$, this is the cotensor of $\TwoCat N^\otimes$ by $S$ in $\CAlgDelta{\TwoCatInfTwo}$.
    Similarly, the cotensor of $\FunPartEpsDeltaTensor{\dagger}{\TwoCat C^\otimes}{\TwoCat N^\otimes}$ by $S$ in $\CatInfTwo$ is $\TwoFun{S}{\FunPartEpsDeltaTensor{\dagger}{\TwoCat C^\otimes}{\TwoCat N^\otimes}}$.

    We first show that $\sigma_{S,\TwoCat N^\otimes}$ agrees with the canonical cotensor comparison
    $$\kappa_{S,\TwoCat N^\otimes}\colon \FunPartEpsDeltaTensor{\dagger}{\TwoCat C^\otimes}{(\TwoCat N^\otimes)^S}\longrightarrow\TwoFun{S}{\FunPartEpsDeltaTensor{\dagger}{\TwoCat C^\otimes}{\TwoCat N^\otimes}}.$$
    By definition, $\kappa_{S,\TwoCat N^\otimes}$ corresponds under the cotensor adjunction to the composite
    \begin{equation} \label{eq:weird-composite}
        S\xrightarrow{\mathrm{ev}}\Map{\CAlgDelta{\TwoCatInfTwo}}{(\TwoCat N^\otimes)^S}{\TwoCat N^\otimes}\xrightarrow{\FunPartEpsDeltaTensor{\dagger}{\TwoCat C^\otimes}{-}}\Map{\CatInfTwo}{\FunPartEpsDeltaTensor{\dagger}{\TwoCat C^\otimes}{(\TwoCat N^\otimes)^S}}{\FunPartEpsDeltaTensor{\dagger}{\TwoCat C^\otimes}{\TwoCat N^\otimes}}.
    \end{equation}
    Here, the first map is the evaluation family of the cotensor.

    For every $s\colon *\to S$, naturality of $\sigma_{-,\TwoCat N^\otimes}$ gives
    $$\mathrm{ev}_s\circ\sigma_{S,\TwoCat N^\otimes}\cong\sigma_{*,\TwoCat N^\otimes}\circ \FunPartEpsDeltaTensor{\dagger}{\TwoCat C^\otimes}{\mathrm{ev}_s}\cong \FunPartEpsDeltaTensor{\dagger}{\TwoCat C^\otimes}{\mathrm{ev}_s},$$
    where the second equivalence follows from \cref{lem:2-functoriality-of-lax-tensor-oplax:cotensor-comparison}.
    This equivalence is natural in $s\in\Map{}{*}{S}\cong S$: it is obtained from the naturality of $\sigma$ on this entire mapping anima.
    Thus the double currying of $\sigma_{S,\TwoCat N^\otimes}$ is naturally equivalent to \eqref{eq:weird-composite}.
    Hence, $\sigma_{S,\TwoCat N^\otimes}\cong\kappa_{S,\TwoCat N^\otimes}$, naturally in $S$.

    Now the cotensor adjunctions identify the source and target of $\Map{\AnL}{S}{f}$ with
    $$\Map{\AnL}{S}{\Map{\CAlgDelta{\TwoCatInfTwo}}{\TwoCat M^\otimes}{\TwoCat N^\otimes}}\cong\Map{\CAlgDelta{\TwoCatInfTwo}}{\TwoCat M^\otimes}{(\TwoCat N^\otimes)^S}$$
    and
    \begin{align*}
        &\Map{\AnL}{S}{\Map{\CatInfTwo}{\FunPartEpsDeltaTensor{\dagger}{\TwoCat C^\otimes}{\TwoCat M^\otimes}}{\FunPartEpsDeltaTensor{\dagger}{\TwoCat C^\otimes}{\TwoCat N^\otimes}}}\\
        &\cong\Map{\CatInfTwo}{\FunPartEpsDeltaTensor{\dagger}{\TwoCat C^\otimes}{\TwoCat M^\otimes}}{\TwoFun{S}{\FunPartEpsDeltaTensor{\dagger}{\TwoCat C^\otimes}{\TwoCat N^\otimes}}},
    \end{align*}
    respectively.
    Under these identifications, $\Map{\AnL}{S}{f}$ is the map
    $$u\longmapsto\kappa_{S,\TwoCat N^\otimes}\circ \FunPartEpsDeltaTensor{\dagger}{\TwoCat C^\otimes}{u}.$$
    On the other hand, put $K=\Map{\TwoCAlgDelta{\TwoCatInfTwo}}{\TwoCat M^\otimes}{\TwoCat N^\otimes}$.
    A functor $a\colon S\to K$ corresponds under \cref{lem:2-sym:category-indexed-cotensors} to a morphism $u_a\colon\TwoCat M^\otimes\to(\TwoCat N^\otimes)^S$.
    Explicitly, $u_a$ is given by the composition 
    $$\TwoCat M^\otimes \xrightarrow{\mathrm{ev}} (\TwoCat N^\otimes)^K \xrightarrow{(\TwoCat N^\otimes)^a} (\TwoCat N^\otimes)^S,$$
    where the first map is the evaluation map of the cotensor adjunction.
    Hence, by the construction of $g$ and naturality of $\sigma_{-,\TwoCat N^\otimes}$ under this restriction, $g\circ a$ corresponds to
    $$\FunPartEpsDeltaTensor{\dagger}{\TwoCat C^\otimes}{\TwoCat M^\otimes}\xrightarrow{\FunPartEpsDeltaTensor{\dagger}{\TwoCat C^\otimes}{u_a}}\FunPartEpsDeltaTensor{\dagger}{\TwoCat C^\otimes}{(\TwoCat N^\otimes)^S}\xrightarrow{\sigma_{S,\TwoCat N^\otimes}}\TwoFun{S}{\FunPartEpsDeltaTensor{\dagger}{\TwoCat C^\otimes}{\TwoCat N^\otimes}}.$$
    Consequently, under the same identifications, $\Map{\CatInf}{S}{g}$ is the map
    $$u\longmapsto\sigma_{S,\TwoCat N^\otimes}\circ \FunPartEpsDeltaTensor{\dagger}{\TwoCat C^\otimes}{u}.$$
    The natural equivalence $\sigma_{S,\TwoCat N^\otimes}\cong\kappa_{S,\TwoCat N^\otimes}$ therefore gives
    $$\Map{\AnL}{S}{f}\cong\Map{\CatInf}{S}{g}\cong\Map{\AnL}{S}{\iota_0 g},$$
    naturally in $S$, which is what we needed to prove.
\end{proof}

\begin{prop} \label{prop:2-functoriality-of-lax-tensor-oplax:composition}
    Let $\TwoCat M^\otimes$, $\TwoCat N^\otimes$ and $\TwoCat P^\otimes$ be symmetric monoidal $2$-categories.
    There is a natural equivalence
    $$\FunPartEpsDeltaTensor{\dagger}{\TwoCat C^\otimes}{-} \circ \FunPartEpsDeltaTensor{\dagger}{\TwoCat C^\otimes}{-} \simeq \FunPartEpsDeltaTensor{\dagger}{\TwoCat C^\otimes}{- \circ -},$$
    of functors 
    $$\Map{\TwoCAlgDelta{\TwoCatInfTwo}}{\TwoCat M^\otimes}{\TwoCat N^\otimes} \times \Map{\TwoCAlgDelta{\TwoCatInfTwo}}{\TwoCat N^\otimes}{\TwoCat P^\otimes} \to \Map{\TwoCatInfTwo}{ \FunPartEpsDeltaTensor{\dagger}{\TwoCat C^\otimes}{\TwoCat M^\otimes}}{\FunPartEpsDeltaTensor{\dagger}{\TwoCat C^\otimes}{\TwoCat P^\otimes}}.$$
    Its components agree with the composition equivalences of \cref{constr:2-functoriality-of-lax-tensor-oplax:1-functoriality}.
    Together with the identity comparisons, these equivalences satisfy the associativity and unit coherences and are compatible with whiskering.
\end{prop}
\begin{proof}
    Consider categories $K$ and $L$, and families
    $$a\colon K\longrightarrow \Map{\TwoCAlgDelta{\TwoCatInfTwo}}{\TwoCat M^\otimes}{\TwoCat N^\otimes} \qquad \text{and} \qquad b\colon L\longrightarrow \Map{\TwoCAlgDelta{\TwoCatInfTwo}}{\TwoCat N^\otimes}{\TwoCat P^\otimes}.$$
    By \cref{lem:2-sym:category-indexed-cotensors}, these correspond to $\delta$-symmetric monoidal functors
    $$\widetilde a\colon \TwoCat M^\otimes\longrightarrow\TwoFun{K}{\TwoCat N^\otimes}^\ptw \qquad \text{and} \qquad \widetilde b\colon \TwoCat N^\otimes\longrightarrow \TwoFun{L}{\TwoCat P^\otimes}^\ptw.$$
    The composite family, parametrized by $K\times L$, corresponds to
    $$\TwoCat M^\otimes \xrightarrow{\widetilde a} \TwoFun{K}{\TwoCat N^\otimes}^\ptw \xrightarrow{\TwoFun{K}{\widetilde b}^\ptw} \TwoFun{K}{\TwoFun{L}{\TwoCat P^\otimes}^\ptw}^\ptw \simeq \TwoFun{K\times L}{\TwoCat P^\otimes}^\ptw.$$
    Here, the last equivalence is \Cref{lem:2-sym:pointwise-comparison}.
    Apply $\FunPartEpsDeltaTensor{\dagger}{\TwoCat C^\otimes}{-}$ to this composite.
    We obtain the following diagram:
    \begin{center}
        \begin{tikzcd}[column sep=huge]
            \FunPartEpsDeltaTensor{\dagger}{\TwoCat C^\otimes}{\TwoCat M^\otimes} \ar[d, "\tilde{a}"'] \\

            \FunPartEpsDeltaTensor{\dagger}{\TwoCat C^\otimes}{\TwoFun{K}{\TwoCat N^\otimes}^\ptw} \ar[d, "\tilde{b}"'] \ar[r, "\sigma_{K,\TwoCat N^\otimes}"] 
                &\TwoFun{K}{\FunPartEpsDeltaTensor{\dagger}{\TwoCat C^\otimes}{\TwoCat N^\otimes}} \ar[d, "\tilde{b}"]\\

            \FunPartEpsDeltaTensor{\dagger}{\TwoCat C^\otimes}{\TwoFun{K}{\TwoFun{L}{\TwoCat P^\otimes}^\ptw}^\ptw} \ar[dd, "\simeq"'] \ar[r, "{\sigma_{K, \TwoFun{L}{\TwoCat P^\otimes}^\ptw}}"]
                &\TwoFun{K}{\FunPartEpsDeltaTensor{\dagger}{\TwoCat C^\otimes}{\TwoFun{L}{\TwoCat P^\otimes}^\ptw}} \ar[d, "\sigma_{L,\TwoCat P^\otimes}"]\\
            
                &\TwoFun{K}{\TwoFun{L}{\FunPartEpsDeltaTensor{\dagger}{\TwoCat C^\otimes}{\TwoCat P^\otimes}}} \ar[d, "\simeq"]\\

            \FunPartEpsDeltaTensor{\dagger}{\TwoCat C^\otimes}{\TwoFun{K \times L}{\TwoCat P^\otimes}^\ptw} \ar[r, "\sigma_{K\times L,\TwoCat P^\otimes}"'] 
                &\TwoFun{K \times L}{\FunPartEpsDeltaTensor{\dagger}{\TwoCat C^\otimes}{\TwoCat P^\otimes}}
        \end{tikzcd}
    \end{center}
    Here, the top square commutes by naturality of $\sigma_{K,-}$, and the bottom square commutes 
    by \cref{lem:2-functoriality-of-lax-tensor-oplax:cotensor-comparison}.
    Currying again, this identifies the resulting $K\times L$-family with the pointwise composite of the families obtained separately from $a$ and $b$.
    
    Taking $a$ and $b$ to be the universal families (i.e.\ taking $K = \Map{\TwoCAlgDelta{\TwoCatInfTwo}}{\TwoCat M^\otimes}{\TwoCat N^\otimes}$ and $a$ the identity,
    and $L = \Map{\TwoCAlgDelta{\TwoCatInfTwo}}{\TwoCat N^\otimes}{\TwoCat P^\otimes}$ and $b$ the identity) gives the required map.
    Restricting to $K$ and $L$ being anima identifies its components with the composition equivalences of the original functor.

    Taking one family to be constant gives compatibility with left or right whiskering.
    Taking a family consisting of an identity gives the identity comparison.
    Finally, the same calculation with three parameter categories proves compatibility with associativity, using the coherent product compatibility of \cref{lem:2-functoriality-of-lax-tensor-oplax:cotensor-comparison}.
    The unit coherences follow from $\sigma_{*,\TwoCat N^\otimes}\simeq\mathrm{id}$.
\end{proof}

\begin{thm} \label{lem:2-functoriality-of-lax-tensor-oplax:main-thm}
    There exists a $2$-functor
    $$\FunPartEpsDeltaTensor{\dagger}{\TwoCat C^\otimes}{-}\colon \TwoCAlgDelta{\TwoCatInfTwo} \longrightarrow \underline h\,\TwoCatInfTwo$$
    into the homotopy $(2,2)$-category.
    Its underlying functor agrees with the composite of the functor from \cref{constr:2-functoriality-of-lax-tensor-oplax:1-functoriality} and the canonical projection to the underlying category of $\underline h\,\TwoCatInfTwo$.
\end{thm}
\begin{proof}
    We first construct a pseudofunctor of homotopy $2$-categories
    $$\underline h\,\TwoCAlgDelta{\TwoCatInfTwo} \longrightarrow \underline h\,\TwoCatInfTwo.$$
    On objects, it sends $\TwoCat N^\otimes$ to
    $$\FunPartEpsDeltaTensor{\dagger}{\TwoCat C^\otimes}{\TwoCat N^\otimes}.$$
    On each hom-category, use the homotopy category of the functor from \cref{constr:2-functoriality-of-lax-tensor-oplax:mapping-categories-upgrade}.
    This preserves identity $2$-morphisms and vertical composition, since that construction is a functor of categories.

    The composition and identity comparisons of \cref{prop:2-functoriality-of-lax-tensor-oplax:composition} supply the invertible composition and unit constraints.
    Their naturality and coherence give the pseudofunctor axioms.
    Composing with the canonical projection
    $$\TwoCAlgDelta{\TwoCatInfTwo} \longrightarrow \underline h\,\TwoCAlgDelta{\TwoCatInfTwo}$$
    gives the asserted $2$-functor.
    Compatibility with the original functor follows from \cref{lem:2-functoriality-of-lax-tensor-oplax:compatibility-on-cores} and the identification of the composition and unit comparisons in \cref{prop:2-functoriality-of-lax-tensor-oplax:composition}.
\end{proof}

\begin{rmk}
    With more work one can show that there actually exists a $2$-functor 
    $$\FunPartEpsDeltaTensor{\dagger}{\TwoCat C^\otimes}{-} \colon \TwoCAlgDelta{\TwoCatInfTwo} \to \TwoCatInfTwo.$$
    Since we only need the functor in question to preserve adjoints, the above formulation suffices and we will not 
    concern ourselves with the construction of the full coherence datum.
\end{rmk}

\begin{cor} \label{lem:2-functoriality-of-lax-tensor-oplax:pres-of-adjoints}
    Let $L\colon \TwoCat M^\otimes \rightleftarrows \TwoCat N^\otimes \noloc R$
    be an adjunction in $\TwoCAlgDelta{\TwoCatInfTwo}$. Then
    $$\FunPartEpsDeltaTensor{\dagger}{\TwoCat C^\otimes}{L} \dashv \FunPartEpsDeltaTensor{\dagger}{\TwoCat C^\otimes}{R}.$$
\end{cor}
\begin{proof}
    Applying the $2$-functor constructed in \cref{lem:2-functoriality-of-lax-tensor-oplax:main-thm} gives an adjunction
    $$\FunPartEpsDeltaTensor{\dagger}{\TwoCat C^\otimes}{L} \dashv \FunPartEpsDeltaTensor{\dagger}{\TwoCat C^\otimes}{R}$$
    in $\underline h\,\TwoCatInfTwo$, since $2$-functors preserve adjunctions.
    A pair of $1$-morphisms in a $2$-category is adjoint if and only if its image in the homotopy $(2,2)$-category is adjoint; see \cite[Definition 4.1.1 and the following discussion]{AGH24}.
    Hence, the same pair is adjoint in $\TwoCatInfTwo$.
\end{proof}

\begin{notation} \label{def:2sym:lax-functor-2-category}
    We have the following variation of the above: If $\TwoCat C^\otimes$ is a symmetric monoidal $2$-category, 
    and $\epsilon, \delta \in \oplaxvarset$, we can equip $\TwoCat C^\otimes$ with the flat marking $\flat$.
    We then write
    $$\defnot{\FunEpsDeltaTensor{\TwoCat C^\otimes}{-}} \coloneqq \FunPartEpsDeltaTensor{\flat}{\TwoCat C^\otimes}{-}.$$
    In particular, we obtain the category $\defnot{\FunLaxTensor{\TwoCat C^\otimes}{-}}$ (for $\delta = \mathrm{lax}$ and $\epsilon = \emptyset$).
\end{notation}

\begin{rmk}
    If $\dagger = \flat$, $\delta = \emptyset$ and $\epsilon \in \{\,\mathrm{lax},\,\mathrm{oplax}\,\}$, we recover the categories defined in \cite[Definition 2.17]{Carmeli2015charactersandtransfers} and further studied in \cite[\S2.2]{Neuhauser2026rigidalgebras}.
    Similarly, if $\dagger = \flat$, $\delta = \mathrm{lax}$ and $\epsilon \in \{\,\mathrm{lax},\,\mathrm{oplax}\,\}$, we recover the categories discussed in \cite[Definition 6.10]{cnossen2025universalityspan2categoriesconstruction}.
\end{rmk}

\begin{lem} \label{lem:oplax-lax-tensor:marked-currying}
    Let $(\TwoCat C^\otimes,\dagger)$ be a marked symmetric monoidal
    $2$-category, let $\TwoCat D^\otimes$ be a symmetric monoidal
    $2$-category, and let $\TwoCat I$ be a $2$-category.
    Let $\TwoCat A$ be the full sub-$2$-category of
    $$\FunPartOplaxLaxTensor{\dagger}{\TwoCat C^\otimes}{\FunLax{\TwoCat I}{\TwoCat D^\otimes}^{\ptw}}$$
    spanned by those objects whose underlying functor sends
    marked morphisms of $\TwoCat C$ to strong natural transformations.
    Then there is a canonical equivalence
    $$\TwoCat A \simeq \FunLax{\TwoCat I}{\FunPartOplaxLaxTensor{\dagger}{\TwoCat C^\otimes}{\TwoCat D^\otimes}}.$$
\end{lem}
\begin{proof}
    For every test $2$-category $\TwoCat K$, we have by \cref{lem:2-functoriality-of-lax-tensor-oplax:1-functoriality-univ-prop} 
    a natural equivalence
    \begin{align*}
        &\Map{\CatInfTwo}{\TwoCat K}{\FunPartOplaxLaxTensor{\dagger}{\TwoCat C^\otimes}{\FunLax{\TwoCat I}{\TwoCat D^\otimes}^{\ptw}}} \\
        &\cong 
        \MapMarked{\CAlgLax{\TwoCatInfTwo}}{\dagger}{\TwoCat C^\otimes}{\FunLax{\TwoCat K}{\FunLax{\TwoCat I}{\TwoCat D^\otimes}^{\ptw}}^{\ptw}},
    \end{align*}
    where the target carries the strong marking, i.e.\ the marking given by $\TwoFun{\TwoCat K}{\FunLax{\TwoCat I}{\TwoCat D^\otimes}^{\ptw}}^{\ptw}$.
    From \Cref{lem:pointwise:internal-currying-pointwise-structure} we obtain 
    $$\FunLax{\TwoCat K}{\FunLax{\TwoCat I}{\TwoCat D^\otimes}^{\ptw}}^{\ptw} \cong \FunLax{\TwoCat K \otimes^\gr \TwoCat I}{\TwoCat D^\otimes}^{\ptw},$$
    natural in $\TwoCat D^\otimes$, $\TwoCat K$ and $\TwoCat I$.
    
    Thus, we obtain an equivalence 
    \begin{equation} \label{eq:oplax-lax-tensor:strong-not-so-strong}
        \Map{\CatInfTwo}{\TwoCat K}{\FunPartOplaxLaxTensor{\dagger}{\TwoCat C^\otimes}{\FunLax{\TwoCat I}{\TwoCat D^\otimes}^{\ptw}}} \cong 
        \MapMarked{\CAlgLax{\TwoCatInfTwo}}{\dagger}{\TwoCat C^\otimes}{\FunLax{\TwoCat K \otimes^\gr \TwoCat I}{\TwoCat D^\otimes}^{\ptw}},
    \end{equation}
    where now the target carries the \emph{transported marking via the above equivalence coming from the strong marking},
    we denote it by $\natural$.
    Note that the strong marking on $\FunLax{\TwoCat K \otimes^\gr \TwoCat I}{\TwoCat D^\otimes}^{\ptw}$ is contained in $\natural$.
    We claim that a functor $F \colon \TwoCat K \to \FunPartOplaxLaxTensor{\dagger}{\TwoCat C^\otimes}{\FunLax{\TwoCat I}{\TwoCat D^\otimes}^{\ptw}}$
    lies inside of $\TwoCat A$ if and only if under \eqref{eq:oplax-lax-tensor:strong-not-so-strong} it
    corresponds to a functor that sends a marked morphism in $\TwoCat C$ to the strong marking.

    Write $H \colon \TwoCat C^\otimes \to \FunLax{\TwoCat K}{\FunLax{\TwoCat I}{\TwoCat D^\otimes}^{\ptw}}^{\ptw}$ and $G \colon \TwoCat C^\otimes \to \FunLax{\TwoCat K \otimes^\gr \TwoCat I}{\TwoCat D^\otimes}^{\ptw}$ for the lax symmetric
    monoidal functors corresponding to $F$ before and after
    Gray currying, respectively.
    For a marked morphism $f$ of $\TwoCat C$, the marked
    universal property makes $H(f)$ strong.
    By Gray adjunction and
    \cite[Corollary 2.7.12]{AGH24}, it is therefore represented by
    a functor
    $$\beta_f \colon (\Delta^1 \times \TwoCat K) \otimes^\gr \TwoCat I \to \TwoCat D.$$
    Precomposing with the canonical comparison
    $\Delta^1 \otimes^\gr \TwoCat K \to \Delta^1 \times \TwoCat K$
    in the first two factors represents $G(f)$.

    By the dual form of \cite[Corollary 2.8.3]{AGH24} (using the second pushout of \cite[Proposition 2.8.1]{AGH24}),
    $\beta_f$ factors through
    $$p \colon (\Delta^1 \times \TwoCat K) \otimes^\gr \TwoCat I \to \Delta^1 \times (\TwoCat K \otimes^\gr \TwoCat I)$$
    if and only if, for every object $k \in \TwoCat K$,
    its restriction to $\Delta^1 \otimes^\gr \{k\} \otimes^\gr \TwoCat I$
    factors through $\Delta^1 \times \TwoCat I$.
    By \cite[Corollary 2.7.12]{AGH24}, this says precisely that $F(k)(f)$ is strong.

    Thus, if $F$ factors through $\TwoCat A$, then
    $\beta_f$ factors through $p$, so $G(f)$ is strong.
    Conversely, if $G(f)$ is strong, its restriction to
    $\{k\} \otimes^\gr \TwoCat I$ is strong, and this
    restriction is $F(k)(f)$.
    Since $\TwoCat A$ is full, this proves the claim.

    Finally, any strong transformation on
    $\TwoCat K \otimes^\gr \TwoCat I$ is represented by a functor
    $\Delta^1 \times (\TwoCat K \otimes^\gr \TwoCat I) \to \TwoCat D$.
    Precomposition with $p$ exhibits its uncurried transformation as strong in the $\TwoCat K$-direction.
    Hence the strong marking is contained in $\natural$.

    Hence, we obtain
    \begin{alignat*}{2}
        \Map{\CatInfTwo}{\TwoCat K}{\TwoCat A}
        &\cong \MapMarked{\CAlgLax{\TwoCatInfTwo}}{\dagger}{\TwoCat C^\otimes}{\FunLax{\TwoCat K \otimes^\gr \TwoCat I}{\TwoCat D^\otimes}^{\ptw}} &&\tag{1}\\
        &\cong \Map{\CatInfTwo}{\TwoCat K \otimes^\gr \TwoCat I}{\FunPartOplaxLaxTensor{\dagger}{\TwoCat C^\otimes}{\TwoCat D^\otimes}} &&\tag{2}\\
        &\cong \Map{\CatInfTwo}{\TwoCat K}{\FunLax{\TwoCat I}{\FunPartOplaxLaxTensor{\dagger}{\TwoCat C^\otimes}{\TwoCat D^\otimes}}}. &&\tag{3}
    \end{alignat*}
    Here the marked target carries the strong marking. Equivalence (1) was explained above.
    Equivalence (2) is just \cref{lem:2-functoriality-of-lax-tensor-oplax:1-functoriality-univ-prop},
    and equivalence (3) holds by definition.
    The lemma follows by Yoneda.
\end{proof}

These preparations allow us to prove the main theorem of this section:
\begin{thm} \label{lem:tensor-product:lax-sym-mon-lax-lim-existence}
    Let $\TwoCat{I}$ be a $2$-category and let $\TwoCat D^\otimes$ be a symmetric monoidal $2$-category that admits all lax limits of shape $\TwoCat{I}$.
    Then the $2$-category $\FunPartOplaxLaxTensor{\dagger}{\TwoCat C^\otimes}{\TwoCat D^\otimes}$ admits all lax limits of shape $\TwoCat{I}$,
    and the forgetful functor of \Cref{lem:2-functoriality-of-lax-tensor-oplax:forgetful-map}
    $$V \colon \FunPartOplaxLaxTensor{\dagger}{\TwoCat C^\otimes}{\TwoCat D^\otimes} \to \FunPartOpLax{\dagger}{\TwoCat C}{\TwoCat D}$$
    preserves these lax limits.
\end{thm}
\begin{proof}
    The proof proceeds similar to the one of \cite[Proposition 5.4.2]{abellán2026freefibrationslaxcolimits}.
    As $\TwoCat D$ admits all lax limits of shape $\TwoCat{I}$, there is an adjunction 
    $$\mathrm{const} \colon \TwoCat D \leftrightarrows \FunLax{\TwoCat I}{\TwoCat D} \noloc \laxlim_{\TwoCat I}.$$ 
    Moreover, $\mathrm{const}$ is symmetric monoidal and $\laxlim_{\TwoCat I}$ is equipped with the corresponding lax symmetric monoidal structure;
    and they form an adjunction in $\TwoCAlgLax{\TwoCatInfTwo}$: This is explained in \Cref{lem:lax-lim:partial-lax-lim-lax-sym-mon}.
    The functor $\FunPartOplaxLaxTensor{\dagger}{\TwoCat C^\otimes}{-}$ preserves this adjunction by \cref{lem:2-functoriality-of-lax-tensor-oplax:pres-of-adjoints}
    and thus yields an adjunction 
    \begin{equation*}
        \mathrm{const} \circ - \colon \FunPartOplaxLaxTensor{\dagger}{\TwoCat C^\otimes}{\TwoCat D^\otimes} \leftrightarrows \FunPartOplaxLaxTensor{\dagger}{\TwoCat C^\otimes}{\FunLax{\TwoCat I}{\TwoCat D^\otimes}^\ptw} \noloc \laxlim_{\TwoCat I} \circ -.
    \end{equation*}
    Let $\TwoCat A$ denote the full sub-$2$-category of the right hand side spanned by those functors which map marked morphisms to strong natural transformations. 
    The functor $\mathrm{const} \circ -$ clearly factors over $\TwoCat A$ and hence, the adjunction restricts to $\TwoCat A$. 
    But now, $\TwoCat A$ is equivalent to the category $\FunLax{\TwoCat I}{\FunPartOplaxLaxTensor{\dagger}{\TwoCat C^\otimes}{\TwoCat D^\otimes}}$ by \cref{lem:oplax-lax-tensor:marked-currying}. 
    In particular, we get an adjunction
    \begin{equation*}
        \FunPartOplaxLaxTensor{\dagger}{\TwoCat C^\otimes}{\TwoCat D^\otimes} \leftrightarrows \FunLax{\TwoCat I}{\FunPartOplaxLaxTensor{\dagger}{\TwoCat C^\otimes}{\TwoCat D^\otimes}},
    \end{equation*}
    where the left adjoint is the constant diagram functor. 
    In particular, all lax limits of shape $\TwoCat I$ exist in $\FunPartOplaxLaxTensor{\dagger}{\TwoCat C^\otimes}{\TwoCat D^\otimes}$. 
    As all the steps in this proof are compatible with the forgetful functor, by comparing with the proof of \cite[Proposition 5.4.2]{abellán2026freefibrationslaxcolimits} we see that these lax limits are actually computed after forgetting the symmetric monoidal structure.  
\end{proof}

\begin{lem} \label{lem:oplax-lax-tensor:coop}
    Let $\TwoCat C^\otimes$, $\TwoCat M^\otimes$ be two symmetric monoidal $2$-categories,
    and let $\epsilon \in \oplaxvarset$.
    There is an equivalence of $2$-categories
    \begin{equation*}
        \coop{(\FunEpsLaxTensor{\TwoCat{C}^\otimes}{\TwoCat{M}^\otimes})} \cong \FunEpsOplaxTensor{\coop{(\TwoCat{C}^\otimes)}}{\coop{(\TwoCat{M}^\otimes)}}.
    \end{equation*}
    This equivalence is natural in $\TwoCat C^\otimes$ as an object of $\CAlgLax{\TwoCatInfTwo}$,
    i.e.\ it gives an equivalence of functors 
    $$\coop{(\FunEpsLaxTensor{-}{\TwoCat{M}^\otimes})} \cong \FunEpsOplaxTensor{\coop{(-)}}{\coop{(\TwoCat{M}^\otimes)}} \colon \op{\CAlgLax{\TwoCatInfTwo}} \to \CatInfTwo.$$
    In particular, for a lax symmetric monoidal functor $F \colon \TwoCat{C}^\otimes \to \TwoCat{D}^\otimes$ there is a commutative diagram 
    \begin{center}
        \begin{tikzcd}
            \coop{(\FunEpsLaxTensor{\TwoCat{D}^\otimes}{\TwoCat{M}^\otimes})} \ar[r, "\simeq"] \ar[d, "-\circ F"'] & \FunEpsOplaxTensor{\coop{(\TwoCat{D}^\otimes)}}{\coop{(\TwoCat{M}^\otimes)}} \ar[d, "-\circ \coop{F}"] \\
            \coop{(\FunEpsLaxTensor{\TwoCat{C}^\otimes}{\TwoCat{M}^\otimes})} \ar[r, "\simeq"] & \FunEpsOplaxTensor{\coop{(\TwoCat{C}^\otimes)}}{\coop{(\TwoCat{M}^\otimes)}} \rlap{.}
        \end{tikzcd}
    \end{center}
\end{lem}
\begin{proof}    
    Let $\TwoCat T$ be a $2$-category.
    We have the following chain of natural equivalences:
    \begin{alignat*}{2}
        &\Map{\CatInfTwo}{\TwoCat T}{\coop{(\FunEpsLaxTensor{\TwoCat{C}^\otimes}{\TwoCat{M}^\otimes})}} \\
        \cong &\Map{\CatInfTwo}{\coop{\TwoCat T}}{\FunEpsLaxTensor{\TwoCat{C}^\otimes}{\TwoCat{M}^\otimes}} &&\tag{1}\\
        \cong &\Map{\CAlgLax{\TwoCatInfTwo}}{\TwoCat C^\otimes}{\FunNEps{\coop{\TwoCat T}}{\TwoCat M^\otimes}^{\ptw}} &&\tag{2}\\
        \cong &\Map{\CAlgLax{\TwoCatInfTwo}}{\TwoCat C^\otimes}{\coop{(\FunNEps{\TwoCat T}{\coop{(\TwoCat M^\otimes)}}^{\ptw})}} &&\tag{3}\\
        \cong &\Map{\CAlgOpLax{\TwoCatInfTwo}}{\coop{(\TwoCat{C}^\otimes)}}{\FunNEps{\TwoCat T}{\coop{(\TwoCat{M}^\otimes)}}^{\ptw}} &&\tag{4}\\
        \cong &\Map{\CatInfTwo}{\TwoCat T}{\FunEpsOplaxTensor{\coop{(\TwoCat{C}^\otimes)}}{\coop{(\TwoCat{M}^\otimes)}}}. &&\tag{5}
    \end{alignat*}
    Here, (1) holds because $\coop{(-)} \colon \CatInfTwo \to \CatInfTwo$ is an autoequivalence.
    Equivalences (2) and (5) are both instances of \cref{lem:2-functoriality-of-lax-tensor-oplax:1-functoriality-univ-prop}.
    Equivalence (3) is provided by \cref{lem:2-sym:pointwise-coop-interchange-law}.
    Equivalence (4) uses that $\coop{(-)}$ upgrades to an equivalence $\CAlgLax{\TwoCatInfTwo} \to \CAlgOpLax{\TwoCatInfTwo}$ 
    (with inverse also given by $\coop{(-)}$), cf.\ \cref{lem:2-sym:oplax-lax-CAlg-coop-equivalence}.
    Hence, the required equivalence exists by the Yoneda lemma.

    Since all of the above equivalences were natural in $\TwoCat C^\otimes$, 
    it upgrades to a natural equivalence of functors.
\end{proof}

%% file: symmetric-monoidal/calg.tex
\subsection{Commutative algebras} \label{sec:sym-calg}
In the last subsection, we discuss the $2$-category of commutative algebras in a symmetric monoidal $2$-category,
as well as its basic functoriality.

\begin{defn}
    We define the \defemph{commutative algebra functor} 
    $$\defnot{\TwoCAlg{-}} \colon \CAlgLax{\TwoCatInfTwo} \to \CatInfTwo$$
    as the functor given by $\FunLaxTensor{\Delta^{0,\times}}{-}$ \footnote{Here, $\Delta^0$ carries the unique symmetric monoidal structure, which is also cartesian monoidal.} of \cref{def:2sym:lax-functor-2-category}.
\end{defn}

\begin{lem} \label{lem:CAlg:anima-CAlg-Mapping-anima}
    Let $\TwoCat C^\otimes, \TwoCat D^\otimes \in \CAlgLax{\TwoCatInfTwo}$.
    We have 
    $$\iota_0 \FunLaxTensor{\TwoCat C^\otimes}{\TwoCat D^\otimes} \cong \Map{\CAlgLax{\TwoCatInfTwo}}{\TwoCat C^\otimes}{\TwoCat D^\otimes},$$
    natural in $\TwoCat C^\otimes$ and $\TwoCat D^\otimes$.
\end{lem}
\begin{proof}
    We have the following chain of natural equivalences:
    \begin{alignat*}{2}
        \iota_0 \FunLaxTensor{\TwoCat C^\otimes}{\TwoCat D^\otimes} 
        &\cong \Map{\CatInfTwo}{\Delta^0}{\FunLaxTensor{\TwoCat C^\otimes}{\TwoCat D^\otimes}} \\
        &\cong \Map{\CAlgLax{\TwoCatInfTwo}}{\TwoCat C^\otimes}{\TwoFun{\Delta^0}{\TwoCat D^\otimes}^\ptw} &&\tag{$\star$} \\
        &\cong \Map{\CAlgLax{\TwoCatInfTwo}}{\TwoCat C^\otimes}{\TwoCat D^\otimes}\rlap{.}
    \end{alignat*}
    Here, $(\star)$ holds by \cref{lem:2-functoriality-of-lax-tensor-oplax:1-functoriality-univ-prop}, 
    applied with $\epsilon = \emptyset$, $\delta = \mathrm{lax}$ and $\dagger = \flat$.
\end{proof}

\begin{defn}
    Let $\TwoCat{C}$ be a $2$-category admitting finite products. 
    We say that a functor $F \colon \Fin_* \to \TwoCat{C}$ \defemphindex{satisfies the Segal condition}{Segal condition} if for each $n \geq 0$ the maps $\rho_i \colon \langle n \rangle \to \langle 1 \rangle$ mapping $i$ to $1$ and everything else to $*$ 
    induce an equivalence $F(\langle n \rangle) \cong \prod_{i=1}^{n} F(\langle 1 \rangle)$.
    We denote by $\FunSegal{\Fin_*}{\TwoCat{C}} \subseteq \TwoFun{\Fin_*}{\TwoCat{C}}$ the full sub-$2$-category spanned by those functors satisfying the Segal condition.
\end{defn}

\begin{prop} \label{lem:CAlg:anima-CAlg-Fun-Segal}
    Let $\TwoCat C$ be a $2$-category with finite products, and equip it with the cartesian symmetric monoidal structure.
    We have
    $$\iota_0\TwoCAlg{\TwoCat C^\times} \cong \iota_0\FunSegal{\Fin_*}{\TwoCat C}.$$
\end{prop}
\begin{proof}
    By \cref{lem:CAlg:anima-CAlg-Mapping-anima}, we have
    $$\iota_0\TwoCAlg{\TwoCat C^\times} \cong \Map{\CAlgLax{\TwoCatInfTwo}}{\Delta^{0,\times}}{\TwoCat C^\times}.$$
    The total $2$-category presenting $\Delta^{0,\times}$ is $\Fin_*$.
    Thus, the right-hand side is the anima of sections of
    $$\TwoCat C^\times \longrightarrow \Fin_*$$
    preserving $2$-cocartesian $1$-morphisms over inert maps.
    Since $\Fin_*$ is a $1$-category, these sections factor through
    $$\iota_1(\TwoCat C^\times) \cong (\iota_1\TwoCat C)^\times.$$
    Under this identification, preservation of $2$-cocartesian
    $1$-morphisms over inert maps agrees with preservation of
    cocartesian morphisms over inert maps: in both cases, an arrow
    over an inert map is cocartesian precisely when its component
    maps are equivalences. Hence,
    $$\Map{\CAlgLax{\TwoCatInfTwo}}{\Delta^{0,\times}}{\TwoCat C^\times} \cong \iota_0\CAlg{(\iota_1\TwoCat C)^\times}.$$
    By \cite[Proposition 2.4.2.5]{higheralgebra}, we have
    $$\CAlg{(\iota_1\TwoCat C)^\times} \cong \FunSegal{\Fin_*}{\iota_1\TwoCat C}.$$
    Finally, the adjunction $\operatorname{incl} \dashv \iota_1$
    identifies functors from $\Fin_*$ to $\TwoCat C$ with functors from $\Fin_*$
    to $\iota_1\TwoCat C$ at the level of
    mapping anima. Since $\iota_1$ preserves finite products,
    and since a morphism in a $2$-category $\TwoCat C$ is an equivalence if and only if 
    it is an equivalence in $\iota_1 \TwoCat C$, the Segal conditions agree. Thus,
    $$\iota_0\FunSegal{\Fin_*}{\iota_1\TwoCat C} \cong \iota_0\FunSegal{\Fin_*}{\TwoCat C}.$$
    Combining these equivalences proves the claim.
\end{proof}

\begin{notation}
    We write $\defnot{\CAlgOpLax{\TwoCatInf}} \subseteq \CAlgOpLax{\TwoCatInfTwo}$ for the full sub-category 
    spanned by those symmetric monoidal $2$-categories such that the underlying category is a $1$-category.
\end{notation}

\begin{lem} \label{lem:tensor-product:CAlg-of-Arr-lax-is-oplax-sym-mon-functors}
    The anima $\iota_0\TwoCAlg{\FunLax{\Delta^1}{\TwoCatInf^\times}^\ptw}$ is equivalent to the anima of 
    oplax symmetric monoidal functors $\Map{\CatInfTwo}{\Delta^1}{\CAlgOpLax{\TwoCatInf}}$. 
\end{lem}
\begin{proof}
    The pointwise structure $\FunLax{\Delta^1}{\TwoCatInf^\times}^\ptw$ is cartesian by \cref{lem:pointwise:cartesian-pointwise}.
    Hence, by \cref{lem:CAlg:anima-CAlg-Fun-Segal},
    $$\iota_0\TwoCAlg{\FunLax{\Delta^1}{\TwoCatInf^\times}^\ptw} \cong \iota_0\FunSegal{\Fin_*}{\FunLax{\Delta^1}{\TwoCatInf}}.$$ 
    Using that 
    \begin{equation*}    
        \Map{\CatInfTwo}{\Fin_*}{\FunLax{\Delta^1}{\TwoCatInf}} \cong \Map{\CatInfTwo}{\Delta^1}{\FunOpLax{\Fin_*}{\TwoCatInf}}
    \end{equation*}
    by definition, 
    $\iota_0\FunSegal{\Fin_*}{\FunLax{\Delta^1}{\TwoCatInf}} \subseteq \Map{\CatInfTwo}{\Delta^1}{\FunOpLax{\Fin_*}{\TwoCatInf}}$ can be seen as a subanima. 
    We claim that its image is given by those oplax natural transformations $\alpha \colon F \to G$, where $F$ and $G$ satisfy the Segal condition and such that 
    $\alpha$ is strong over inert morphisms. 

    Let $H \colon \Fin_* \to \FunLax{\Delta^1}{\TwoCatInf}$ be a functor and let $\alpha \colon F \to G$ be the corresponding oplax natural transformation under the above equivalence.
    Assume first that $H$ satisfies the Segal condition, i.e.\ the maps $\rho^n_i \colon \langle n \rangle \to \langle 1 \rangle$ mapping $i$ to $1$ and everything else to $*$ induce an equivalence $H(\langle n \rangle) \xrightarrow{\simeq} \prod_{i=1}^{n} H(\langle 1\rangle)$.
    Then also $F= \ev{0} H$ and $G=\ev{1}H$ satisfy the Segal condition because $\ev{0}, \ev{1} \colon \FunLax{\Delta^1}{\TwoCatInf} \to \TwoCatInf$ are product-preserving by \cite[Proposition 5.4.2]{abellán2026freefibrationslaxcolimits} 
    (note that products are partially oplax limits over diagrams with no nontrivial arrows).
    The functor $H$ satisfies the Segal condition if and only if for all objects $\langle n \rangle \in \Fin_*$, the map $H(\langle n \rangle) \to \prod_{i=1}^{n} H(\langle 1\rangle)$ induced by the $\rho_i^n$ is an equivalence in $\FunLax{\Delta^1}{\TwoCatInf}$. 
    We can write this map as a diagram of the form 
    \begin{center}
        \begin{tikzcd}[column sep=large]
            F(\langle n \rangle) \ar[r, "\alpha_{\langle n \rangle}"] \ar[d] & G(\langle n \rangle) \ar[d] \\
            \prod_{i=1}^{n} F(\langle 1 \rangle) \ar[r, "\prod_{i=1}^{n}\alpha_{\langle 1 \rangle}"'] \laxarrowur{\eta_n} & \prod_{i=1}^{n} G(\langle 1 \rangle) \rlap{.}
        \end{tikzcd}
    \end{center}
    As $F$ and $G$ satisfy the Segal condition, the two vertical maps in the diagram are equivalences. 
    Consequently, $H$ satisfies the Segal condition if and only if $F$ and $G$ do so and the maps $\eta_n$ are equivalences.
    The map $\eta_n$ is defined via the diagrams
    \begin{center}
        \begin{tikzcd}
            F(\langle n \rangle) \ar[r, "\alpha_{\langle n \rangle}"] \ar[d, "F(\rho_i^n)"'] & G(\langle n \rangle) \ar[d, "G(\rho_i^n)"] \\
            F(\langle 1 \rangle) \ar[r, "\alpha_{\langle 1 \rangle}"'] \laxarrowur{\alpha_{\rho_i^n}} & G(\langle 1 \rangle)
        \end{tikzcd}
    \end{center}
    coming from the oplax naturality of $\alpha$ applied to $\rho_i^n$ for all $i \in \{1, \dots, n\}$. 
    Hence, $\eta_n$ is an equivalence if and only if $\alpha$ is strong over the maps $\rho_i^n$. 

    We want to show that this is in fact equivalent to $\alpha$ being strong over all inert morphisms.
    So let $f \colon \langle m \rangle \to \langle n \rangle$ be an arbitrary inert morphism in $\Fin_*$. 
    We have to check that the map $$\alpha_{f} \colon \alpha_{\langle n \rangle} F(f) \to G(f) \alpha_{\langle m \rangle}$$ is an equivalence.
    As $G$ satisfies the Segal condition, the maps $\rho_i^n$ yield an equivalence $$G(\langle n \rangle) \cong \prod_{i=1}^n G(\langle 1 \rangle).$$
    Thus, it is enough to show that $G(\rho_i^n) \alpha_{f}$ is an equivalence for each $i$. 
    Consider the commutative diagram
    \begin{center}
        \begin{tikzcd}
            \alpha_{\langle 1 \rangle} F(\rho_i^n f) \ar[r, "\alpha_{\rho_i^n f}"] \ar[d, equal] & G(\rho_i^n f) \alpha_{\langle m \rangle} \ar[dd, equal] \\
             \alpha_{\langle 1 \rangle} F(\rho_i^n) F(f) \ar[d, "{\alpha_{\rho_i^n}F(f)}"'] & \\
             G(\rho_i^n) \alpha_{\langle n \rangle} F(f) \ar[r, "G(\rho_i^n) \alpha_f"'] & G(\rho_i^n) G(f) \alpha_{\langle m \rangle} \rlap{.}
        \end{tikzcd}
    \end{center}
    As $\alpha$ is strong over the maps $\rho_i^n$, the maps $\alpha_{\rho_i^n}$ and $\alpha_{\rho_i^n f}$ are equivalences (note that $\rho_i^n f = \rho_{j}^m$ for the unique $j \in \{1, \dots , m\}$ with $f(j)=i$). 
    In particular, the diagram shows that also $G(\rho_i^n) \alpha_{f}$ is an equivalence.
    
    Thus, $\eta_n$ is an equivalence for all $n$ if and only if $\alpha$ is strong over all inert morphisms.
    This finishes the proof of the claim about the image of 
    $$\iota_0\FunSegal{\Fin_*}{\FunLax{\Delta^1}{\TwoCatInf}} \subseteq \Map{\CatInfTwo}{\Delta^1}{\FunOpLax{\Fin_*}{\TwoCatInf}}.$$
    To complete the proof of the lemma, we apply \cite[Corollary 3.5.10]{AGH24} and see that $\alpha$ as above corresponds to a functor between the cartesian unstraightenings 
    which preserves $2$-cartesian $1$-morphisms lying over opposites of inert morphisms. 
    This is exactly an oplax symmetric monoidal functor by definition.
\end{proof}

%% file: 3FF/def-3FF.tex
We start this section by defining in \cref{sec:3FF-def} the main actor of this paper: a $3$-functor formalism,
and its variant, a $6\frac{1}{2}$-functor formalism. We explain how one recovers a symmetric monoidal structure 
from such an object, and discuss some immediate consequences.
\cref{sec:3FF-constr} is the counterpart to \cref{sec:constr-2FF}: Here, we explain how to construct a $3$-functor formalism,
essentially dualizing the construction from \cite{oldCLL}.
Finally, in \cref{sec:3ff-Cartier}, we explain how to extend the results of \cref{sec:Cart-2FF} to the setting of $3$-functor formalisms,
using the existence result about lax limits from \cref{sec:sym-oplax-lax}.

\subsection{Definition of \texorpdfstring{$3$}{3}-functor formalisms and consequences} \label{sec:3FF-def}

\begin{defn} \label{def:3FF:cartesian-Nagata}
    A \defemphindex{symmetric monoidal Nagata setup}{Nagata setup!symmetric monoidal} $(\Cat C^\otimes, I, P)$ is 
    a symmetric monoidal category $\Cat C^\otimes$, together with a Nagata setup 
    $(\Cat C, \all, I, P)$ on the underlying category $\Cat C$ (so in particular, now $E = \all$),
    such that moreover 
    \begin{enumerate}[label=(\arabic*)]
        \item $I \otimes I \subseteq I$,
        \item $P \otimes P \subseteq P$, 
        \item the functor $- \otimes - \colon \Cat C \times \Cat C \to \Cat C$ preserves pullbacks along
            morphisms in $I \times I$ and along morphisms in $P \times P$.
    \end{enumerate}
\end{defn}

Let us explain why we included $E = \all$ in the definition of a symmetric monoidal Nagata setup.
\begin{rmk} \label{rmk:3FF:symmetric-Nagata-E-eq-all}
    Let $(\Cat C, E, I, P)$ be a Nagata setup, where $\Cat C$ admits finite products and is equipped with the cartesian symmetric monoidal structure.
    We will describe later (\Cref{lem:tensor-product:DX-is-sym-mon}) how a lax symmetric monoidal functor $\Cat D^\otimes \colon \SpanStd^\otimes \to \TwoCatInf^\times$ induces a symmetric monoidal structure on $\Cat D^\otimes (X)$ for each object $X \in \Cat C$. 
    This requires us to equip each object $X \in \SpanStd$ with a canonical commutative algebra structure. 
    In particular, to define the unit we have to assume that the maps $X \to *$ belong to $E$ for each $X$. 
    But as $E$ is left cancellable by \cite[Lemma 3.21]{cnossen2025universalityspan2categoriesconstruction} this already implies $E = \all$.
\end{rmk}

As a special case, we will often use the following definition:
\begin{defn}
    A \defemphindex{cartesian Nagata setup}{Nagata setup!cartesian} $(\Cat C^\times, I, P)$ 
    is a symmetric monoidal Nagata setup where $\Cat C$ is equipped with the cartesian monoidal structure.
    In particular, we assume that $\Cat C$ admits all finite products.
\end{defn}

Cartesian Nagata setups exist in abundance:
\begin{lem} \label{lem:3FF:Nagata-is-cartesian}
    Let $(\Cat C, \all, I, P)$ be a Nagata setup, such that $\Cat C$ admits all finite products.
    Then $(\Cat C^\times, I, P)$ is a cartesian Nagata setup.
\end{lem}
\begin{proof}
    We have to check the properties of \cref{def:3FF:cartesian-Nagata}.
    \begin{enumerate}[label=(\arabic*)]
        \item Let $j \colon U \to X$ and $j' \colon U' \to X'$ be morphisms in $I$. We claim that also $j \times j' \in I$.
            Indeed, $j \times j'$ is the composition of $\id{} \times j'$ and $j \times \id{}$, which are in turn pullbacks of $j$ and $j'$.
            Hence, the claim follows since $I$ is a subcategory stable under pullbacks.
        \item The case for $P$ has an analogous proof.
        \item That $- \times -$ commutes with pullbacks along $I \times I$ and $P \times P$ is clear since products commute with all limits. \qedhere
    \end{enumerate}
\end{proof}

\begin{constr} \label{constr:3FF:Span-symmetric-monoidal}
    Let $(\Cat C^\otimes, I, P)$ be a symmetric monoidal Nagata setup.
    It is explained in \cite[beginning of section 5]{cnossen2025universalityspan2categoriesconstruction} how $\Span{\Cat C}{\all}{\all}{P}{I}$ 
    can be equipped with an induced symmetric monoidal structure $\defnot{\Span{\Cat C}{\all}{\all}{P}{I}^\otimes} \colon \Fin_* \to \TwoCatInfTwo$ for which the canonical functor 
    $\op{\Cat C} \to \Span{\Cat C}{\all}{\all}{P}{I}$ is symmetric monoidal. 
    Postcomposing with the product-preserving functor $\coop{(-)}$ and using the definition of $\SpanStd$ (for $E = \all$) we also get an induced symmetric monoidal structure on $\Span{\Cat C}{\all}{\all}{I}{P}$. 
    Moreover, the canonical functor $h \colon \Cat C \to \Span{\Cat C}{\all}{\all}{I}{P}$ is symmetric monoidal.
\end{constr}
    
\begin{defn}
    Let $(\Cat C^\otimes, I, P)$ be a symmetric monoidal Nagata setup and let $\TwoCat{M}^\otimes$ be a symmetric monoidal $2$-category. 
    A \defemphindex{$3$-functor formalism (with values in $\TwoCat{M}^\otimes$)}{$3$-functor formalism} is a lax symmetric monoidal functor $\Cat D^\otimes \colon \Span{\Cat C}{\all}{\all}{I}{P}^\otimes \to \TwoCat{M}^\otimes$.
\end{defn}

\begin{rmk}
    Any $3$-functor formalism has an underlying $2$-functor formalism by forgetting the lax symmetric monoidal structure
    (i.e.\ by applying the forgetful functor $\CAlgLax{\TwoCatInfTwo} \to \TwoCatInfTwo$ of \cref{rmk:2sym:forgetf-functor}).
\end{rmk}

\begin{lem}[{\cite[Corollary 2.4.3.10]{higheralgebra}}] \label{lem:tensor-product:algebras-cocartesian-monoidal-structure}
    Let $\Cat C^\amalg$ be a category with finite coproducts, equipped with the cocartesian symmetric monoidal structure.
    Then the forgetful functor $\CAlg{\Cat C^\amalg} \xrightarrow{\simeq} \Cat C$ is an equivalence of categories. 
    Indeed, every object $X \in \Cat C$ carries a canonical commutative algebra structure given by the fold map $X \amalg X \to X$.
\end{lem}

\begin{cor} \label{lem:tensor-product:algebras-in-spans}
    Let $(\Cat C^\times, I, P)$ be a cartesian Nagata setup, so that we may equip $\op{\Cat C}$ with the cocartesian monoidal structure.
    Then the functor $\op{\Cat C} \to \Span{\Cat C}{\all}{\all}{I}{P}$ factors over $\CAlg{\Span{\Cat C}{\all}{\all}{I}{P}^\otimes}$, 
    equipping each object $X \in \Span{\Cat C}{\all}{\all}{I}{P}$ with the canonical commutative algebra structure given by the span $X \times X \xleftarrow{\Delta_X} X \xrightarrow{\id{X}} X$, 
    where $\Delta_X$ denotes the diagonal.
\end{cor}
\begin{proof}
    One can argue as in \cite[beginning of section 5]{cnossen2025universalityspan2categoriesconstruction} that the functor $\op{\Cat C} \to \Span{\Cat C}{\all}{\all}{I}{P}$ is symmetric monoidal. 
    In particular, it induces a functor $\CAlg{\Cat C^{\mathrm{op}, \amalg}} \to \CAlg{\Span{\Cat C}{\all}{\all}{I}{P}^\otimes}$. 
    But $\CAlg{\Cat C^{\mathrm{op}, \amalg}} \cong \op{\Cat C}$ by \Cref{lem:tensor-product:algebras-cocartesian-monoidal-structure}.

    To get the explicit description of the algebra structure, note that a fold map in $\op{\Cat C}$ 
    is given by the diagonal in $\Cat C$.
\end{proof}

\begin{lem} \label{lem:tensor-product:DX-is-sym-mon}
    Let $(\Cat C^\times, I, P)$ be a cartesian Nagata setup.
    Let $\Cat D^\otimes \colon \Span{\Cat C}{\all}{\all}{I}{P}^\otimes \to \TwoCatInf^\times$ be a $3$-functor formalism with values in $\TwoCatInf^\times$ and let $X \in \Cat C$ be an object. 
    Then the category $\Cat D^\otimes (X)$ carries an induced symmetric monoidal structure whose multiplication is given by the composition 
    \begin{equation*}
        \defnot{- \shriektensor -} \colon \Cat D^\otimes (X) \times \Cat D^\otimes (X) \xrightarrow{\defnot{-\boxtimes-}} \Cat D^\otimes (X \times X) \xrightarrow{\Delta_X^!} \Cat D^\otimes (X).
    \end{equation*}
    Here, the first arrow, denoted $\defnot{-\boxtimes -}$, is given by the lax symmetric monoidal structure of $\Cat D^\otimes$ and $\Delta_X \colon X \to X \times X$ denotes the diagonal.
    We call this symmetric monoidal structure the \defemph{shriek tensor product}.
\end{lem}
\begin{proof}
    As $\Cat D^\otimes$ is a lax symmetric monoidal functor, it induces a functor 
    \begin{equation*}    
        \TwoCAlg{\Cat D^\otimes} \colon \TwoCAlg{\Span{\Cat C}{\all}{\all}{I}{P}^\otimes} \to \TwoCAlg{\TwoCatInf^\times} 
    \end{equation*}
    by applying the functor $\TwoCAlg{-} \colon \CAlgLax{\TwoCatInfTwo} \to \CatInfTwo$. 
    Each object $X \in \Span{\Cat C}{\all}{\all}{I}{P}$ carries by \Cref{lem:tensor-product:algebras-in-spans} a canonical commutative algebra structure whose multiplication is given by $$\mu_X \colon X \times X \xleftarrow{\Delta_X} X \xrightarrow{\id{X}} X.$$ 
    In particular, $\Cat D^\otimes (X) \in \CAlg{\TwoCatInf^\times}$ is a symmetric monoidal category whose tensor functor is given by 
    \begin{equation*}
        \Cat D^\otimes (X) \times \Cat D^\otimes (X) \to \Cat D^\otimes (X \times X) \xrightarrow{\Cat D^\otimes(\mu_X) = \Delta_X^!} \Cat D^\otimes (X). \qedhere
    \end{equation*}
\end{proof}

\begin{defn} \label{def:3ff:defn}
    Let $(\Cat C^\times, I, P)$ be a cartesian Nagata setup.
    A \defemphindex{$6\frac{1}{2}$-functor formalism with respect to some class of morphisms $T \subseteq \Cat C$}{$6\frac{1}{2}$-functor formalism} 
    is a $3$-functor formalism $\Cat D^\otimes \colon \Span{\Cat C}{\all}{\all}{I}{P}^\otimes \to \TwoCatInf^\times$ with values in $\TwoCatInf$
    such that the underlying $2$-functor formalism is a $4\frac{1}{2}$-formalism with respect to $T$ in the sense of \Cref{def:2ff:defn}, and moreover
    \begin{enumerate}[label=(\arabic*)]
        \item[\textbf{(FF5)}] \label{axiom:FF5} for every $X \in \Cat C$ the symmetric monoidal category $\Cat D^\otimes(X)$ is presentably symmetric monoidal\footnote{In other words, the tensor product preserves colimits in each variable. In particular, by the adjoint functor theorem, the symmetric monoidal structure is closed.}.
    \end{enumerate}
\end{defn}

We now collect some consequences of the existence of a $3$-functor formalism.

\begin{lem} \label{lem:tensor-product:upper-shriek-symmetric-monoidal}
    Let $(\Cat C^\times, I, P)$ be a cartesian Nagata setup.
    Let $\Cat D^\otimes \colon \Span{\Cat C}{\all}{\all}{I}{P}^\otimes \to \TwoCatInf^\times$ be a $3$-functor formalism. Let $f \colon X \to Y$ be a morphism in $\Cat C$. 
    Then the functor $f^! \colon \Cat D(Y) \to \Cat D(X)$ is symmetric monoidal with respect to the symmetric monoidal structures described in \Cref{lem:tensor-product:DX-is-sym-mon}. 
\end{lem}
\begin{proof}
    By the proof of \Cref{lem:tensor-product:DX-is-sym-mon}, it is enough to show that the span $Y \xleftarrow{f} X \xrightarrow{\id{X}} X$ is actually 
    a $1$-morphism in $\CAlg{\Span{\Cat C}{\all}{\all}{I}{P}^\otimes}$. 
    But this follows directly from \Cref{lem:tensor-product:algebras-in-spans}.
\end{proof}

\begin{lem} \label{lem:consequences:projection-formula}
    Let $(\Cat C^\times, I, P)$ be a cartesian Nagata setup.
    Let $\Cat D^\otimes \colon \Span{\Cat C}{\all}{\all}{I}{P}^\otimes \to \TwoCatInf^\times$ be a $3$-functor formalism. Let $f \colon X \to Y$ be a morphism in $\Cat C$. 
    Then the functor $f_* \colon \Cat D(X) \to \Cat D(Y)$ satisfies the projection formula: For two objects $M \in \Cat D(X)$, $N \in \Cat D(Y)$, there is an equivalence 
    $$f_*(M \shriektensor f^! N) \cong f_*(M) \shriektensor N.$$
\end{lem}
\begin{proof}
    Denote by $\Gamma_f \colon X \to X \times Y$ the graph of $f$, i.e.\ $\Gamma_f = (\id{X} \times f) \circ \Delta_X$, where $\Delta_X \colon X \to X \times X$ denotes the diagonal. 
    Consider the span $X \times Y \xleftarrow{\Gamma_f} X \xrightarrow{f} Y$. 
    Applying $\Cat D$ to it yields $f_* \Gamma_f^!$. 
    If we evaluate this at $M \boxtimes N \in \Cat D(X \times Y)$, we get 
    $$f_* \Gamma_f^!(M \boxtimes N) \cong f_* \Delta_X^! (\id{X} \times f)^!(M \boxtimes N) \cong f_*(M \shriektensor f^! N),$$
    where we used \Cref{lem:tensor-product:DX-is-sym-mon} in the second equivalence.
    On the other hand, the above span can equivalently be written as composition of the spans $X \times Y \xleftarrow{\id{X \times Y}} X \times Y \xrightarrow{f \times \id{Y}} Y \times Y$ 
    and $Y \times Y \xleftarrow{\Delta_Y} Y \xrightarrow{\id{Y}} Y$.
    Applying $\Cat D$ to this composition yields the functor $\Delta_Y^! (f \times \id{Y})_*$ and evaluating it at $M \boxtimes N$ gives 
    $$\Delta_Y^! (f \times \id{Y})_* (M \boxtimes N) \cong f_*(M) \shriektensor N,$$
    where we used \Cref{lem:tensor-product:DX-is-sym-mon} again.
\end{proof}

%% file: 3FF/construction-of-3FF.tex
\subsection{Construction of \texorpdfstring{$3$}{3}-functor formalisms} \label{sec:3FF-constr}

\begin{setting}
    In this section, fix a symmetric monoidal Nagata setup $(\Cat C^\otimes, I, P)$,
    and let $\TwoCat M^\otimes$ be a symmetric monoidal $2$-category.
\end{setting}

We will need the following definitions, which are all variations of \cite[Definition 5.13]{cnossen2025universalityspan2categoriesconstruction} and related notions,
and are the multiplicative versions of \Cref{def:cff:suitable-decomposition-all-BC}.
\begin{defn} \label{def:constr-3ff:all-BC}
    Let $\epsilon \in \{\,\mathrm{oplax},\, \emptyset\,\}$ and $\delta \in \{\,\mathrm{lax},\, \mathrm{oplax}\,\}$.
    \begin{enumerate}[label=(\arabic*)]
        \item Let $\Cat D_*^\otimes \colon \Cat C^\otimes \to \TwoCat{M}^\otimes$ be a $\delta$-symmetric monoidal functor. 
            We say that $\Cat D_*^\otimes$ is a \defemphindex{$\delta$-symmetric monoidal functor satisfying all Beck--Chevalley conditions with respect to $(I,P)$}{Beck--Chevalley condition!satisfies all} if the following conditions hold:
            \begin{enumerate}[label=(\alph*)]
                \item The underlying functor $\Cat D_* \colon \Cat C \to \TwoCat{M}$ satisfies all Beck--Chevalley conditions with respect to $(I,P)$.
                \item The natural transformation $\Cat D_*(-) \otimes \Cat D_*(-) \to \Cat D_*(- \otimes -)$, resp.\ $\Cat D_*(-\otimes -) \to \Cat D_*(-) \otimes \Cat D_*(-)$, 
                given by the $\delta$-symmetric monoidal structure satisfies all Beck--Chevalley conditions with respect to $(I \times I, P \times P)$.
            \end{enumerate}
        \item Let $F \colon \op{(\Cat C^\otimes)} \to \TwoCat{M}^\otimes$ be a $\delta$-symmetric monoidal functor.
            We say that $F$ is a \defemphindex{$\delta$-symmetric monoidal $(I,P)$-biadjointable functor}{biadjointable} if the following conditions hold:
            \begin{enumerate}[label=(\alph*)]
                \item The underlying functor $F \colon \op{\Cat C} \to \TwoCat{M}$ is $(I,P)$-biadjointable.
                \item The natural transformation $F(-) \otimes F(-) \to F(- \otimes -)$, respectively $F(- \otimes -) \to F(-) \otimes F(-)$, 
                given by the $\delta$-symmetric monoidal structure is $(I \times I, P \times P)$-biadjointable.
            \end{enumerate}
        \item We say that a symmetric monoidal $\epsilon$-natural transformation between $\delta$-symmetric monoidal functors satisfying all Beck--Chevalley conditions with respect to $(I,P)$ \defemphindex{satisfies all Beck--Chevalley conditions with respect to $(I,P)$}{Beck--Chevalley condition!satisfies all} if
            its underlying $\epsilon$-natural transformation satisfies all Beck--Chevalley conditions with respect to $(I,P)$.
        \item We say that a symmetric monoidal $\epsilon$-natural transformation between $\delta$-symmetric monoidal $(I,P)$-biadjointable functors is \defemphindex{$(I,P)$-biadjointable}{biadjointable} if its underlying $\epsilon$-natural transformation is $(I,P)$-biadjointable.
        \item We write $\defnot{\FunEpsDeltaTensorIPBC{\Cat C^\otimes}{\TwoCat{M}^\otimes}} \subseteq \FunEpsDeltaTensor{\Cat C^\otimes}{\TwoCat{M}^\otimes}$ for the $2$-full sub-$2$-category given by 
                $\delta$-symmetric monoidal functors satisfying all Beck--Chevalley conditions with respect to $(I,P)$, and by symmetric monoidal $\epsilon$-natural transformations satisfying all Beck--Chevalley conditions with respect to $(I,P)$.
        \item We write $\defnot{\FunEpsDeltaTensorBadj{\op{(\Cat C^\otimes)}}{\TwoCat{M}^\otimes}} \subseteq \FunEpsDeltaTensor{\op{(\Cat C^\otimes)}}{\TwoCat{M}^\otimes}$ for the $2$-full sub-$2$-category given 
                by $\delta$-symmetric monoidal $(I,P)$-biadjointable functors, and by $(I,P)$-biadjointable $\epsilon$-natural transformations between them.
    \end{enumerate}
\end{defn}

We can upgrade \Cref{lem:constr-2cat:allBC-is-IPbadj} to the symmetric monoidal setting. 
\begin{lem} \label{lem:tensor-product:lax-sym-mon-allBC-is-IPbadj}
    Let $F \colon \Cat C^\otimes \to \TwoCat{M}^\otimes$ be a lax symmetric monoidal functor.
    Then $F$ is a lax symmetric monoidal functor satisfying all Beck--Chevalley conditions with respect to $(I,P)$ if and only if 
    $\coop{F} \colon \coop{(\Cat C^\otimes)} \to \coop{(\TwoCat{M}^\otimes)}$ is an oplax symmetric monoidal $(I,P)$-biadjointable functor. 
    More generally, the functor $\coop{(-)} \colon \CatInfTwo \to \CatInfTwo$ induces an equivalence of $2$-categories 
    \begin{equation*}
        \FunLaxTensorIPBC{\Cat C^\otimes}{\TwoCat{M}^\otimes} \simeq \coop{(\FunOplaxTensorBadj{\coop{(\Cat C^\otimes)}}{\coop{(\TwoCat{M}^\otimes)}})}.
    \end{equation*}
\end{lem}
\begin{proof}
    Since $\coop{(-)}$ induces an equivalence $$\FunLaxTensor{\Cat C^\otimes}{\TwoCat{M}^\otimes} \simeq \coop{(\FunOplaxTensor{\coop{(\Cat C^\otimes)}}{\coop{(\TwoCat{M}^\otimes)}})}$$ by \cref{lem:oplax-lax-tensor:coop},
    it suffices to show that the imposed conditions on objects and morphisms correspond to each other under this equivalence. 
    This follows directly from \Cref{lem:constr-2cat:allBC-is-IPbadj}: The underlying functor of $F$ satisfies all Beck--Chevalley conditions with respect to $(I,P)$ 
    if and only if the underlying functor of $\coop{F}$ is $(I,P)$-biadjointable, 
    and the natural transformation $$F(-) \otimes F(-) \to F(- \otimes -)$$ satisfies all Beck--Chevalley conditions with respect to $(I \times I, P \times P)$ if and only if 
    $$\coop{(F(-) \otimes F(-) \to F(- \otimes -))} =(\coop{F}(- \otimes -) \to \coop{F}(-) \otimes \coop{F}(-))$$ is $(I \times I, P \times P)$-biadjointable.
    By the same argument we see that also the $1$-morphisms on both sides coincide upon applying $\coop{(-)}$.
\end{proof}

We explained in \cref{constr:3FF:Span-symmetric-monoidal} that the canonical functor $h \colon \Cat C^\otimes \to \Span{\Cat C}{\all}{\all}{I}{P}^\otimes$ is symmetric monoidal.
In particular, precomposition with $h$ induces functors
\[\FunLaxTensor{\Span{\Cat C}{\all}{\all}{I}{P}^\otimes}{\TwoCat{M}^\otimes} \to \FunLaxTensor{\Cat C^\otimes}{\TwoCat{M}^\otimes}\]
and 
\[\FunOplaxLaxTensor{\Span{\Cat C}{\all}{\all}{I}{P}^\otimes}{\TwoCat{M}^\otimes} \to \FunOplaxLaxTensor{\Cat C^\otimes}{\TwoCat{M}^\otimes}.\]
Our goal in this section is to prove the following version of \cite[Theorem 6.13]{cnossen2025universalityspan2categoriesconstruction}.

\begin{thm} \label{lem:tensor-product:lax-sym-mon-UP-span}
    The above functor induces an equivalence of $2$-categories  
    \begin{equation*}
        \FunOplaxLaxTensor{\Span{\Cat C}{\all}{\all}{I}{P}^\otimes}{\TwoCat{M}^\otimes} \xrightarrow{\simeq} \FunOplaxIPBCLaxTensor{\Cat C^\otimes}{\TwoCat{M}^\otimes}.
    \end{equation*}
\end{thm}

\begin{rmk}
    If one tries to prove the above result in the same way as \Cref{lem:constr-2cat:ext-and-unique}, then the problem arises that the application of $\coop{(-)}$ in the proof 
    would turn the lax symmetric monoidal structure into an oplax symmetric monoidal one. 
    But \cite[Theorem 6.13]{cnossen2025universalityspan2categoriesconstruction} is stated only for lax symmetric monoidal functors and the arguments in Section 5 of \loccit rely on this. 
    Bastiaan Cnossen made us aware that in an older version of \cite{cnossen2025universalityspan2categoriesconstruction} the authors used a different approach to prove the equivalence 
    \begin{equation*}
        \FunLaxTensor{\Span{\Cat C}{\all}{\all}{P}{I}^\otimes}{\TwoCat{M}^\otimes} \xrightarrow{\simeq} \FunLaxTensorBadj{\op{(\Cat C^\otimes)}}{\TwoCat{M}^\otimes}
    \end{equation*}
    and that this approach might dualize to give the same statement for oplax symmetric monoidal functors.
    Once we established this equivalence we can use it to prove the versions of \cite[Theorem 6.13]{cnossen2025universalityspan2categoriesconstruction} that we need.
    
    %That's why we give a different argument here. 
    %However, once we have proven an analogue of \cite[Theorem 5.14]{cnossen2025universalityspan2categoriesconstruction} for the underlying groupoids, 
    %the proof of \Cref{lem:tensor-product:lax-sym-mon-UP-span} proceeds as the proof of \cite[Theorem 6.13]{cnossen2025universalityspan2categoriesconstruction}.
\end{rmk}

In the following we adopt the notation from \cite[Section 4.3]{oldCLL} (\cite{oldCLL} refers to version $1$ of \cite{cnossen2025universalityspan2categoriesconstruction}). 
In particular, we use the $3$-adjunction\footnote{The next proposition and theorem are the only places in the paper where we make (light) use of the theory of $(\infty,3)$-categories.} $(\TwoCatInfTwo)^{\mathrm{cc}}_{/ \SpanFin} \rightleftarrows (\TwoCatInfTwo)^{\op{\Fin}\mathrm{-cc}}_{/ \SpanFin} \noloc \mathrm{Rig}$ from the proof of \cite[Lemma 4.29]{oldCLL}, 
where the left adjoint is the inclusion and restricts to an inclusion $\CAlg{\TwoCatInfTwo} \hookrightarrow \CAlgLax{\TwoCatInfTwo}$.

For the next two results we will use the definition of symmetric monoidal ($2$-)categories and lax symmetric monoidal functors between them via $\SpanFin$ instead of $\Fin_*$ because 
the arguments of \cite{oldCLL} carry over better in this way.
In this approach, a symmetric monoidal $2$-category is given by a finite-product-preserving functor $\SpanFin \to \CatInfTwo$ and a lax symmetric monoidal functor is 
a functor between the $2$-cocartesian unstraightenings preserving $2$-cocartesian $1$-morphisms lying over backwards spans.
Note that by \cite[Proposition C.1]{bachmann2020normsmotivichomotopytheory} and \cite[Corollary 5.1.15]{Barkan_2025} the $\SpanFin$-approach and the $\Fin_*$-approach yield the same definitions of symmetric monoidal ($2$-)categories and lax symmetric monoidal functors.

\begin{prop} \label{lem:tensor-product:lax-sym-mon-allBC-via-Rig}
    A functor $F \colon \Cat C^\otimes \to \TwoCat{M}^\otimes$ is a lax symmetric monoidal functor satisfying all Beck--Chevalley conditions with respect to $(I,P)$ if and only if 
    the associated functor $\widetilde{F} \colon \Cat C^\otimes \to \mathrm{Rig}(\TwoCat{M}^\otimes)$ levelwise satisfies all Beck--Chevalley conditions with respect to $(I,P)$.

    Moreover, if $F, G \colon \Cat C^\otimes \to \TwoCat{M}^\otimes$ are lax symmetric monoidal functors satisfying all Beck--Chevalley conditions with respect to $(I,P)$,
    and $t \colon F \to G$ is a symmetric monoidal natural transformation, 
    then the associated natural transformation $\widetilde{t} \colon \widetilde{F} \to \widetilde{G}$ satisfies all Beck--Chevalley conditions with respect to $(I,P)$ 
    if and only if the underlying natural transformation $t_{[1]} \colon F([1]) \to G([1])$ does so.
\end{prop}
\begin{proof}
    We closely follow the proof of \cite[Proposition 4.31]{oldCLL}.
    By definition, the functor $\widetilde{F}$ levelwise satisfies all Beck--Chevalley conditions with respect to $(I,P)$ if and only if for each $n \geq 0$, 
    the functor $\widetilde{F}([n]) \colon \Cat C^\otimes([n]) \to \mathrm{Rig}(\TwoCat{M}^\otimes)([n])$ satisfies all Beck--Chevalley conditions with respect to $(I^n, P^n)$. 
    This is equivalent to $\coop{\widetilde{F}([n])}$ being $(I^n, P^n)$-biadjointable for each $n \geq 0$ by \Cref{lem:constr-2cat:allBC-is-IPbadj}.
    Using the $3$-adjunction $\mathbb{B} \colon \mathrm{BAdjData}_2 \rightleftarrows \TwoCatInfTwo \noloc (-)_{\mathrm{badj}}$ from \cite[Proposition 2.15]{oldCLL}, 
    this is equivalent to $\coop{\widetilde{F}([n])}$ extending over the functor $h \colon \coop{\Cat C^\otimes([n])} \to \mathbb{B}^\otimes([n])$, where 
    $$\mathbb{B}^\otimes \coloneqq \mathbb{B} \circ \coop{(\Cat C^\otimes)} \colon \SpanFin \to \TwoCatInfTwo.\footnote{Here, we use $\coop{I}$ and $\coop{P}$ to upgrade $\coop{(\Cat C^\otimes)} \colon \SpanFin \to \TwoCatInf$ to a functor with target $\mathrm{BAdjData}_2$ via the construction from \cite[Example 2.12]{oldCLL}.}$$
    Applying $\coop{(-)}$, this is equivalent to $\widetilde{F}([n])$ extending over $\coop{h} \colon \Cat C^\otimes([n]) \to \coop{\mathbb{B}^\otimes([n])}$. 
    Proceeding exactly as in the proof of \cite[Proposition 4.31]{oldCLL}, we see that this is equivalent to the statement that the lax symmetric monoidal functor 
    $\psi_n \colon \Hom{\SpanFin}{[n]}{-} \to \Fun{\Cat C^\otimes([n])}{\TwoCat{M}^\otimes(-)}$ factors over the $2$-full subcategory 
    \begin{equation*}
        \TwoFun{\coop{\mathbb{B}^\otimes([n])}}{\TwoCat{M}^\otimes(-)} \cong \coop{\tightunderline{\mathrm{Hom}}_{\mathrm{BAdjData}_2}(\coop{\Cat C^\otimes([n])},(\coop{\TwoCat{M}^\otimes(-)})_{\mathrm{badj}})},
    \end{equation*}
    where the right-hand side denotes the mapping $2$-category in the $3$-category $\mathrm{BAdjData}_2$.

    On $[k] \in \SpanFin$, the latter category is given by $\coop{\TwoFunBadj{\coop{\Cat C^\otimes([n])}}{\coop{\TwoCat{M}^\otimes([k])}}{I^n}{P^n}}$ (cf.\ \cite[Remark 2.14]{oldCLL}), which is in turn equivalent to 
    $\TwoFunVarBC{\Cat C^\otimes([n])}{\TwoCat{M}^\otimes([k])}{I^n}{P^n}$ by \Cref{lem:constr-2cat:allBC-is-IPbadj}.
    Following \cite[Remark 4.32]{oldCLL}, $\psi_n$ factors over $\TwoFunVarBC{\Cat C^\otimes([n])}{\TwoCat{M}^\otimes(-)}{I^n}{P^n}$ if and only if it does so on the object $[1] \in \SpanFin$ and on the forward maps $[0] \to [1]$ and $[2] \to [1]$. 

    Now we can analyze this condition similarly to the proof of \cite[Proposition 4.31]{oldCLL} (replacing $(I,P)$-biadjointability by the property of satisfying all Beck--Chevalley conditions with respect to $(I,P)$).
    This lets us conclude that the factorization of $\psi_n([1])$ for all $n$ is equivalent to the underlying functor of $F$ satisfying all Beck--Chevalley conditions with respect to $(I,P)$, 
    the factorization of $\psi_n([0] \to [1])$ for all $n$ imposes no condition on $F$, and
    the factorization of $\psi_n([2] \to [1])$ for all $n$ is equivalent to the natural transformation $F(-) \otimes F(-) \to F(- \otimes -)$ satisfying all Beck--Chevalley conditions with respect to $(I \times I,P \times P)$.
    
    As in the proof of \cite[Proposition 4.31]{oldCLL}, the argument for natural transformations is similar.
\end{proof}

\begin{thm} \label{lem:tensor-product:oplax-sym-mon-CLL}
    The restriction functor
    \begin{equation*}
        \FunOplaxTensor{\Span{\Cat C}{\all}{\all}{P}{I}^\otimes}{\TwoCat{M}^\otimes} \to \FunOplaxTensorBadj{\op{(\Cat C^\otimes)}}{\TwoCat{M}^\otimes}
    \end{equation*}
    is an equivalence of $2$-categories.
\end{thm}
\begin{proof}
    Consider the following commutative diagram where the horizontal functors are given by restriction: 
    \begin{center}
        \begin{tikzcd}
            \FunOplaxTensor{\Span{\Cat C}{\all}{\all}{P}{I}^\otimes}{\TwoCat{M}^\otimes} \ar[r] \ar[d, "\simeq"'] & \FunOplaxTensor{\op{(\Cat C^\otimes)}}{\TwoCat{M}^\otimes} \ar[d, "\simeq"] \\
            \coop{(\FunLaxTensor{\coop{(\Span{\Cat C}{\all}{\all}{P}{I}^\otimes)}}{\coop{(\TwoCat{M}^\otimes)}})} \ar[r] \ar[d, "\simeq"'] & \coop{(\FunLaxTensor{\Cat C^\otimes}{\coop{(\TwoCat{M}^\otimes)}})} \ar[d, "\simeq"] \\
            \coop{\FunTensor{\coop{(\Span{\Cat C}{\all}{\all}{P}{I}^\otimes)}}{\mathrm{Rig}(\coop{(\TwoCat{M}^\otimes)})}} \ar[d, "\simeq"'] \ar[r] & \coop{\FunTensor{\Cat C^\otimes}{\mathrm{Rig}(\coop{(\TwoCat{M}^\otimes)})}} \ar[d, "\simeq"] \\
            \FunTensor{\Span{\Cat C}{\all}{\all}{P}{I}^\otimes}{\coop{\mathrm{Rig}(\coop{(\TwoCat{M}^\otimes)})}} \ar[r] & \FunTensor{\op{(\Cat C^\otimes)}}{\coop{\mathrm{Rig}(\coop{(\TwoCat{M}^\otimes)})}} \rlap{.}
        \end{tikzcd}
    \end{center}
    Note that $\FunTensor{-}{-}$ is an abuse of notation as the targets are not symmetric monoidal $2$-categories in general.
    Instead it denotes the mapping $2$-category in $\TwoFun{\SpanFin}{\TwoCatInfTwo}$, where the latter is regarded as a $3$-category.
    Here, the upper and lower square are induced by applying $\coop{(-)}$ and the square in the middle commutes by naturality of the adjunction equivalence for $\mathrm{Rig}$ (cf.\ \cite[Proof of Theorem 4.33]{oldCLL}). 
    By \cite[Proposition 4.24]{oldCLL} the bottom horizontal map is a locally full inclusion whose image is spanned by the levelwise $(I,P)$-biadjointable functors and natural transformations. 
    Thus, it remains to show that via the equivalences on the right, the condition of being levelwise $(I,P)$-biadjointable corresponds to being oplax symmetric monoidal $(I,P)$-biadjointable. 
    First note that a functor $F \colon \op{(\Cat C^\otimes)} \to \coop{\mathrm{Rig}(\coop{(\TwoCat{M}^\otimes)})}$ is levelwise $(I,P)$-biadjointable if and only if 
    the functor $\coop{F} \colon \Cat C^\otimes \to \mathrm{Rig}(\coop{(\TwoCat{M}^\otimes)})$ levelwise satisfies all Beck--Chevalley conditions with respect to $(I,P)$ by \Cref{lem:constr-2cat:allBC-is-IPbadj}, 
    and similarly for natural transformations. 
    Moreover, a functor $G \colon \op{(\Cat C^\otimes)} \to \TwoCat{M}^\otimes$ is oplax symmetric monoidal $(I,P)$-biadjointable if and only if 
    the functor $\coop{G} \colon \Cat C^\otimes \to \coop{(\TwoCat{M}^\otimes)}$ is a lax symmetric monoidal functor satisfying all Beck--Chevalley conditions with respect to $(I,P)$ by \Cref{lem:tensor-product:lax-sym-mon-allBC-is-IPbadj}, 
    and similarly for natural transformations. 
    Hence, the claim follows from \Cref{lem:tensor-product:lax-sym-mon-allBC-via-Rig}.
\end{proof}

\begin{prop} \label{lem:tensor-product:UP-span-oplax-oplax-sym-mon}
    The restriction functor 
    \begin{equation*}
        \FunOplaxOplaxTensor{\Span{\Cat C}{\all}{\all}{P}{I}^\otimes}{\TwoCat{M}^\otimes} \to \FunOplaxBadjOplaxTensor{\op{(\Cat C^\otimes)}}{\TwoCat{M}^\otimes}
    \end{equation*}
    is an equivalence of $2$-categories.
\end{prop}
\begin{proof}
    The proof is similar to the proof of \cite[Theorem 6.13]{cnossen2025universalityspan2categoriesconstruction}. 
    By Yoneda, it is enough to show that for each $2$-category $\TwoCat{N}$ the map
    \begin{equation*}
        \Map{\CatInfTwo}{\TwoCat{N}}{\FunOplaxOplaxTensor{\Span{\Cat C}{\all}{\all}{P}{I}^\otimes}{\TwoCat{M}^\otimes}} \to \Map{\CatInfTwo}{\TwoCat{N}}{\FunOplaxOplaxTensor{\op{(\Cat C^\otimes)}}{\TwoCat{M}^\otimes}}
    \end{equation*}
    induced by $h \colon \op{(\Cat C^\otimes)} \to \Span{\Cat C}{\all}{\all}{P}{I}^\otimes$ is an inclusion of path components with essential image those functors $F \colon \TwoCat{N} \to \FunOplaxOplaxTensor{\op{(\Cat C^\otimes)}}{\TwoCat{M}^\otimes}$ 
    such that for each object $X \in \TwoCat{N}$, $F(X)$ is an oplax symmetric monoidal $(I,P)$-biadjointable functor, and such that
    for each $1$-morphism $f$ of $\TwoCat{N}$, the oplax natural transformation $F(f)$ is $(I,P)$-biadjointable.
    By \Cref{lem:2-functoriality-of-lax-tensor-oplax:1-functoriality-univ-prop} (where we equip $\TwoCat C^\otimes$ with the flat marking $\flat$),
    the above map is equivalent to the map 
    \begin{center}
        \begin{tikzcd}
            \Map{\CAlgOpLax{\TwoCatInfTwo}}{\Span{\Cat C}{\all}{\all}{P}{I}^\otimes}{\FunLax{\TwoCat{N}}{\TwoCat{M}^\otimes}^\ptw} \ar[d, "{- \circ h}"] \\
            \Map{\CAlgOpLax{\TwoCatInfTwo}}{\op{(\Cat C^\otimes)}}{\FunLax{\TwoCat{N}}{\TwoCat{M}^\otimes}^\ptw}
        \end{tikzcd}
    \end{center}
    induced by $h$.
    By \Cref{lem:tensor-product:oplax-sym-mon-CLL}, the latter map is an inclusion of path components with essential image given by the oplax symmetric monoidal $(I,P)$-biadjointable functors $F \colon \op{(\Cat C^\otimes)} \to \FunLax{\TwoCat{N}}{\TwoCat{M}^\otimes}^\ptw$. 
    This means $F$ satisfies: 
    \begin{itemize}
        \item The underlying functor of $F$ is $(I,P)$-biadjointable, and 
        \item the natural transformation $F(- \otimes -) \to F(-) \otimes F(-)$ is $(I \times I, P \times P)$-biadjointable.
    \end{itemize}
    The only thing left to show is that the corresponding functor $\widetilde{F} \colon \TwoCat{N} \to \FunOplaxOplaxTensor{\op{(\Cat C^\otimes)}}{\TwoCat{M}^\otimes}$ satisfies the following: 
    \begin{itemize}
        \item For each object $X \in \TwoCat{N}$, $\widetilde{F}(X)$ is an oplax symmetric monoidal $(I,P)$-biadjointable functor, and 
        \item for each $1$-morphism $f$ of $\TwoCat{N}$, the oplax natural transformation $\widetilde{F}(f)$ is $(I,P)$-biadjointable.
    \end{itemize}
    The rest of the proof proceeds exactly as in \cite[Proof of Theorem 6.13]{cnossen2025universalityspan2categoriesconstruction}.
\end{proof}

\begin{proof}[Proof of \cref{lem:tensor-product:lax-sym-mon-UP-span}] \proofof{lem:tensor-product:lax-sym-mon-UP-span}
    Applying $\coop{(-)}$ yields a commutative diagram 
    \begin{center}
        \begin{tikzcd}
            \FunOplaxLaxTensor{\Span{\Cat C}{\all}{\all}{I}{P}^\otimes}{\TwoCat{M}^\otimes} \ar[r] \ar[d, "\simeq"'] & \FunOplaxIPBCLaxTensor{\Cat C^\otimes}{\TwoCat{M}^\otimes} \ar[d, "\simeq"] \\
            \coop{(\FunOplaxOplaxTensor{\Span{\Cat C}{\all}{\all}{P}{I}^\otimes}{\coop{(\TwoCat{M}^\otimes)}})} \ar[r] & \coop{(\FunOplaxBadjOplaxTensor{\op{(\Cat C^\otimes)}}{\coop{(\TwoCat{M}^\otimes)}})} \rlap{,}
        \end{tikzcd}
    \end{center}
    where the horizontal arrows are given by restriction.
    For the vertical equivalences we used \cref{lem:constr-2cat:allBC-is-IPbadj}, \Cref{lem:oplax-lax-tensor:coop} and the definition of $\SpanStd$ (for $E = \all$). Commutativity also follows from \Cref{lem:oplax-lax-tensor:coop}.  
    As the bottom arrow is an equivalence by \Cref{lem:tensor-product:UP-span-oplax-oplax-sym-mon}, it follows that also the top arrow is an equivalence.
\end{proof}

We end this section with the following lemma comparing \cref{lem:constr-2cat:ext-and-unique} with \cref{lem:tensor-product:lax-sym-mon-UP-span}.
\begin{lem} \label{lem:tensor-product:comparison-cosntruction-2ff-3ff}
    The square 
    \begin{center}
        \begin{tikzcd}
            \FunOplaxLaxTensor{\Span{\Cat C}{\all}{\all}{I}{P}^\otimes}{\TwoCat{M}^\otimes} \ar[d] \ar[r, "\simeq"] &\FunOplaxIPBCLaxTensor{\Cat C^\otimes}{\TwoCat{M}^\otimes} \ar[d] \\
            \FunOpLax{\Span{\Cat C}{E}{\all}{I}{P}}{\TwoCat{M}} \ar[r, "\simeq"]&\FunOpLaxIPBC{\Cat C}{\TwoCat{M}}
        \end{tikzcd}
    \end{center}
    commutes.
    Here, the horizontal equivalences are given by restriction via \cref{lem:constr-2cat:ext-and-unique,lem:tensor-product:lax-sym-mon-UP-span},
    and the vertical functors are induced by the forgetful functors from \cref{lem:2-functoriality-of-lax-tensor-oplax:forgetful-map}.
    In particular, if $\Cat D_*^\otimes \colon \Cat C^\otimes \to \TwoCat M^\otimes$ is a lax symmetric monoidal functor satisfying all Beck--Chevalley conditions,
    then the underlying functor of its extension to a $3$-functor formalism $\Cat D^\otimes$ is the extension to a $2$-functor formalism $\Cat D$ 
    of the underlying functor $\Cat D_*$.
\end{lem}
\begin{proof}
    By the naturality of the forgetful functor from \cref{lem:2-functoriality-of-lax-tensor-oplax:forgetful-map}, we obtain a commutative square 
    \begin{center}
        \begin{tikzcd}
            \FunOplaxLaxTensor{\Span{\Cat C}{\all}{\all}{I}{P}^\otimes}{\TwoCat{M}^\otimes} \ar[d] \ar[r] &\FunOplaxLaxTensor{\Cat C^\otimes}{\TwoCat{M}^\otimes} \ar[d] \\
            \FunOpLax{\Span{\Cat C}{E}{\all}{I}{P}}{\TwoCat{M}} \ar[r]&\FunOpLax{\Cat C}{\TwoCat{M}}.
        \end{tikzcd}
    \end{center}
    Since by \cref{lem:constr-2cat:ext-and-unique,lem:tensor-product:lax-sym-mon-UP-span} both horizontal functors are locally full inclusions with essential images given by the right corners of the square of the statement,
    the claim follows.
\end{proof}

%% file: 3FF/frobenius-structures.tex
\subsection{Frobenius structures and \texorpdfstring{$3$}{3}-functor formalisms of Cartier modules} \label{sec:3ff-Cartier}
To upgrade the $2$-functor formalism of Cartier modules to 
a $3$-functor formalism, we need the following multiplicative version of a Frobenius structure:

\begin{defn}
    Let $\Cat C$ be a category with finite products, so that we can equip $\Cat C$ with the cartesian symmetric monoidal structure.
    Then the identity functor $\id{\Cat C}$ has the structure of a symmetric monoidal functor since it preserves products.
    A \defemphindex{symmetric monoidal Frobenius structure}{Frobenius structure!symmetric monoidal} on $\Cat C$ is a symmetric monoidal natural transformation 
    \begin{equation*}
        F \colon \id{\Cat C} \to \id{\Cat C}.
    \end{equation*}
\end{defn}

The data of a Frobenius structure and a symmetric monoidal Frobenius structure are the same:
\begin{lem}[{\cite[Corollary 2.4.1.8]{higheralgebra}}] \label{lem:tensor-product:Frob-structure-is-sym-mon}
    Let $\Cat C$ be a category with finite products.
    The forgetful functor $\mathrm{Fun}^{\otimes}(\Cat C^\times, \Cat C^\times) \to \mathrm{Fun}(\Cat C, \Cat C)$
    is fully faithful, with essential image spanned by those functors that preserve finite products.
    In particular, any Frobenius structure can be uniquely lifted to symmetric monoidal Frobenius structure.
\end{lem}

In view of the above lemma, if $\Cat C$ is a category with finite products,
we will not distinguish between the notions of a Frobenius structure 
and a symmetric monoidal Frobenius structure on $\Cat C$.

\begin{setting}
    For the rest of this section, fix a category $\Cat C$ with finite products,
    and a cartesian Nagata setup $(\Cat C^\times, I, P)$.
    Let $F \colon \id{\Cat C} \to \id{\Cat C}$ be a (symmetric monoidal)\uses{lem:tensor-product:Frob-structure-is-sym-mon} Frobenius structure on $\Cat C$.
    Moreover, let $\Cat D^\otimes$ be a $3$-functor formalism on $\Cat C$,
    i.e.\ a lax symmetric monoidal functor $$\Cat D^\otimes \colon \Span{\Cat C}{\all}{\all}{I}{P}^\otimes \to \TwoCatInf^\times.$$
    Assume from \Cref{lem:tensor-product:extension-of-F_*} on that for every $j \in I$, the square $\Cat D_*(\square^F_j)$ satisfies the left Beck--Chevalley condition.
\end{setting}

\begin{lem} \label{lem:tensor-product:F_*-is-sym-mon}
    The natural transformation $F_* \colon \Cat D_* \to \Cat D_*$ from \Cref{def:constr-cart:def-of-F_*} is symmetric monoidal.
\end{lem}
\begin{proof}
    This holds because $F_* \colon \Cat D_* \to \Cat D_*$ is defined by whiskering $F \colon \id{\Cat C} \to \id{\Cat C}$ with $\Cat D_*$ inside the $2$-category $\TwoCAlgLax{\TwoCatInfTwo}$.
\end{proof}

\begin{lem} \label{lem:tensor-product:extension-of-F_*}
    Then the natural transformation $F_* \colon \Cat D_* \to \Cat D_*$ from \Cref{def:constr-cart:def-of-F_*} extends to a morphism 
    $F_* \colon \Cat D \to \Cat D$ in $\FunPartOplaxLaxTensor{(I,\all)}{\Span{\Cat C}{\all}{\all}{I}{P}^\otimes}{\TwoCatInf^\times}$, i.e.\ $F_*$ is 
    a partially oplax natural transformation between lax symmetric monoidal functors which is a 
    strong natural transformation between lax symmetric monoidal functors on the forward morphisms and on the backward morphisms lying in $I$.
\end{lem}
\begin{proof}
    By \Cref{lem:tensor-product:lax-sym-mon-UP-span}, there is an equivalence of $2$-categories 
    \begin{equation*}
        \FunOplaxLaxTensor{\Span{\Cat C}{\all}{\all}{I}{P}^\otimes}{\TwoCatInf^\times} \xrightarrow{\simeq} \FunOplaxIPBCLaxTensor{\Cat C^\times}{\TwoCatInf^\times}.
    \end{equation*}
    Thus, we have to check that $F_* \colon \Cat D_* \to \Cat D_*$ is a $1$-morphism in the category on the right hand side. 
    By \Cref{lem:tensor-product:F_*-is-sym-mon}, $F_* \colon \Cat D_* \to \Cat D_*$ is a symmetric monoidal natural transformation. 
    Hence, it can be regarded as a morphism in $\FunOplaxLaxTensor{\Cat C^\times}{\TwoCatInf^\times}$ by \Cref{lem:2-functoriality-of-lax-tensor-oplax:explicit-morphisms}. 
    Moreover, we have already shown in the proof of \Cref{lem:BN:F-star-partially-oplax-nat-trans} that the underlying morphism of $F_*$ satisfies all Beck--Chevalley conditions with respect to $(I,P)$. 
    In particular, $F_* \colon \Cat D_* \to \Cat D_*$ can be lifted to an oplax natural transformation $F_* \colon \Cat D \to \Cat D$ between lax symmetric monoidal functors via the equivalence above. 
    By \Cref{lem:BN:F-star-partially-oplax-nat-trans}, the underlying oplax transformation is strong on the forward morphisms and on the backward morphisms lying in $I$. 
\end{proof}

\begin{constr} \label{def:tensor-product:Cart-D-tensor}
    Consider the partially oplax natural transformation $F_* \colon \Cat D^\otimes \to \Cat D^\otimes$ between lax symmetric monoidal functors from \Cref{lem:tensor-product:extension-of-F_*}. 
    Via \Cref{lem:BN:univ-property-2-cat}, this data defines a functor
    $$\Psi \colon \BN \to \FunPartOplaxLaxTensor{(I,\all)}{\Span{\Cat C}{\all}{\all}{I}{P}^\otimes}{\TwoCatInf^\times}.$$
    The target category admits all lax limits by \cref{lem:tensor-product:lax-sym-mon-lax-lim-existence,lem:lax-lim:partial-lax-lim-existence}.
    Hence, we can define the $3$-functor formalism 
    $$\defnot{\Cart(\Cat D^\otimes, F_*)} \coloneqq \laxlim(\Psi) \in \FunPartOplaxLaxTensor{(I,\all)}{\Span{\Cat C}{\all}{\all}{I}{P}^\otimes}{\TwoCatInf^\times}$$
    of \defemphindex{Cartier modules (with respect to $F$)}{Cartier modules}. 
\end{constr}
 
This $3$-functor formalism is an enhancement of the $2$-functor formalism constructed in \cref{def:BN:Cart-Frob-pCFF}.

\begin{lem} \label{lem:Cartier-3FF:underlying-2FF}
    Write
    \begin{equation*}
        V \colon \FunPartOplaxLaxTensor{(I,\all)}{\Span{\Cat C}{\all}{\all}{I}{P}^\otimes}{\TwoCatInf^\times} \to \FunPartOpLax{(I,\all)}{\Span{\Cat C}{\all}{\all}{I}{P}}{\TwoCatInf}    
    \end{equation*}
    for the functor forgetting the lax symmetric monoidal structures from \cref{lem:2-functoriality-of-lax-tensor-oplax:forgetful-map}, then 
    $$V(\Cart(\Cat D^\otimes, F_*)) \cong \Cart(\Cat D, F_*).$$
    In particular, all the results about $\Cart(\Cat D, F_*)$ we proved before still hold for (the underlying $2$-functor formalism of) 
    $\Cart(\Cat D^\otimes, F_*)$.
\end{lem}
\begin{proof}
    Consider the functor $\Psi \colon \BN \to \FunPartOplaxLaxTensor{(I,\all)}{\Span{\Cat C}{\all}{\all}{I}{P}^\otimes}{\TwoCatInf^\times}$ from \Cref{def:tensor-product:Cart-D-tensor}.
    By \cref{lem:tensor-product:extension-of-F_*}, $V \circ \Psi$ is equivalent to the functor induced 
    by the partially oplax natural transformation $F_* \colon \Cat D \to \Cat D$ from \Cref{lem:BN:F-star-partially-oplax-nat-trans}.
    In particular, $\laxlim(V \circ \Psi) \cong \Cart(\Cat D, F_*)$.
    Moreover, $\Cart(\Cat D^\otimes, F_*) = \laxlim \Psi$ by definition.
    Hence, the result holds because $V$ commutes with lax limits by \cref{lem:tensor-product:lax-sym-mon-lax-lim-existence}.
\end{proof}

We next give an explicit description of the symmetric monoidal structure on $\Cart(\Cat D(X), F_{*,X})$ that is induced by the lax symmetric monoidal structure of $\Cart(\Cat D^\otimes, F_*)$ 
as explained in \Cref{lem:tensor-product:DX-is-sym-mon}. 

\begin{defn} \label{def:tensor-product:monoidal-structure}
    The lax symmetric monoidal structure of $\Cart(\Cat D^\otimes, F_*)$ induces a symmetric monoidal structure on $\Cart(\Cat D(X), F_{*,X})$ for each object $X \in \Cat C$ via \Cref{lem:tensor-product:DX-is-sym-mon}. 
    We denote this symmetric monoidal structure by $\defnot{\shriektensor}$ and call it the \defemph{shriek tensor product}. 
\end{defn}

\begin{lem} \label{lem:tensor-product:F_*X-is-oplax-sym-mon}
    Consider the partially oplax natural transformation between lax symmetric monoidal functors $F_* \colon \Cat D \to \Cat D$ from \Cref{lem:tensor-product:extension-of-F_*}. 
    Let $X \in \Cat C$ be an object. 
    Then the functor $F_{*,X} \colon \Cat D(X) \to \Cat D(X)$ carries an oplax symmetric monoidal structure, 
    where $\Cat D(X)$ is equipped with the symmetric monoidal structure from \Cref{lem:tensor-product:DX-is-sym-mon}.
\end{lem}
\begin{proof}
    The partially oplax natural transformation between lax symmetric monoidal functors $F_* \colon \Cat D \to \Cat D$ can be seen as an object of 
    $\TwoFun{\Delta^1}{\FunOplaxLaxTensor{\Span{\Cat C}{\all}{\all}{I}{P}^\otimes}{\TwoCatInf^\times}}$. 
    By \cite[Definition 6.10]{cnossen2025universalityspan2categoriesconstruction}, this corresponds to a lax symmetric monoidal functor 
    $\Span{\Cat C}{\all}{\all}{I}{P}^\otimes \to \FunLax{\Delta^1}{\TwoCatInf^\times}$ sending an object $X \in \Span{\Cat C}{\all}{\all}{I}{P}$ to $F_{*,X} \colon \Cat D(X) \to \Cat D(X)$. 
    In particular, it induces a functor $$\CAlg{\Span{\Cat C}{\all}{\all}{I}{P}^\otimes} \to \CAlg{\FunLax{\Delta^1}{\TwoCatInf^\times}}.$$ 
    But as each object $X \in \Span{\Cat C}{\all}{\all}{I}{P}$ carries a canonical commutative algebra structure by \Cref{lem:tensor-product:algebras-in-spans}, 
    we see that $F_{*,X} \in \CAlg{\FunLax{\Delta^1}{\TwoCatInf^\times}}$. 
    Hence, $F_{*,X} \colon \Cat D(X) \to \Cat D(X)$ is an oplax symmetric monoidal functor by \Cref{lem:tensor-product:CAlg-of-Arr-lax-is-oplax-sym-mon-functors}.
\end{proof}

\begin{lem} \label{lem:tensor-product:U-is-sym-mon}
    The partially oplax natural transformation $U \colon \Cart(\Cat D^\otimes, F_*) \to \Cat D$ between lax symmetric monoidal functors is 
    a natural transformation between lax symmetric monoidal functors.
    In particular, for each object $X \in \Cat C$, the functor $U_X \colon \Cart(\Cat D(X), F_{*,X}) \to \Cat D(X)$ is symmetric monoidal.
\end{lem}
\begin{proof}
    We first give an alternative description of $U \colon \Cart(\Cat D^\otimes, F_*) \to \Cat D$.
    The diagram $$\Psi \colon \BN \to \FunPartOplaxLaxTensor{(I,\all)}{\Span{\Cat C}{\all}{\all}{I}{P}^\otimes}{\TwoCatInf^\times}$$
    defining $\Cart(\Cat D^\otimes, F_*)$ 
    corresponds via \Cref{lem:2-functoriality-of-lax-tensor-oplax:1-functoriality-univ-prop} to 
    an object $$\widetilde{\Psi} \in \MapMarked{\CAlgLax{\TwoCatInfTwo}}{(I,\all)}{\Span{\Cat C}{\all}{\all}{I}{P}^\otimes}{\FunLax{\BN}{\TwoCatInf^\times}^\ptw}.$$ 
    By construction and the proof of \Cref{lem:tensor-product:lax-sym-mon-lax-lim-existence}, we have $\ev{*} \circ \widetilde{\Psi} \cong \Cat D$ and $\laxlim_{\BN} \circ \widetilde{\Psi} \cong \Cart(\Cat D^\otimes, F_*)$.
    Consider the counit $c$ of the adjunction $\diag \colon \TwoCatInf \rightleftarrows \FunLax{\BN}{\TwoCatInf} \noloc \laxlim_{\BN}$. 
    Whiskering it with the functor $\ev{*} \colon \FunLax{\BN}{\TwoCatInf} \to \TwoCatInf$ gives a natural transformation $\ev{*}c \colon \laxlim_{\BN} \to \ev{*}$ 
    of functors $\FunLax{\BN}{\TwoCatInf} \to \TwoCatInf$.
    We recover $U \colon \Cart(\Cat D^\otimes, F_*) \to \Cat D$ by whiskering the natural transformation $\ev{*}c$ with the functor $\widetilde{\Psi}$. 
    Thus, it is enough to show that $\ev{*}c$ is a symmetric monoidal natural transformation. 
    
    For this, first note that the functor $\ev{*} \colon \FunLax{\BN}{\TwoCatInf^\times}^\ptw \to \TwoCatInf^\times$ is symmetric monoidal by the definition of the pointwise symmetric monoidal structure. 
    As the pointwise symmetric monoidal structure on $\FunLax{\BN}{\TwoCatInf^\times}$ is the cartesian symmetric monoidal structure by \cref{lem:pointwise:cartesian-pointwise} and as lax limits commute with finite products, 
    we get that the functor $\laxlim_{\BN} \colon \FunLax{\BN}{\TwoCatInf^\times}^\ptw \to \TwoCatInf^\times$ also is symmetric monoidal. 
    Now we obtain for all $\Phi_1, \Phi_2 \in \FunLax{\BN}{\TwoCatInf}$ a commutative diagram 
    \begin{center}
        \begin{tikzcd}
            \laxlim_{\BN} (\Phi_1 \times \Phi_2) \ar[r, "\simeq"] \ar[d, "{(\ev{*}c)_{\Phi_1 \times \Phi_2}}"'] & \laxlim_{\BN}(\Phi_1) \times \laxlim_{\BN}(\Phi_2) \ar[d, "{(\ev{*}c)_{\Phi_1} \times (\ev{*}c)_{\Phi_2}}"] \\
            \ev{*}(\Phi_1 \times \Phi_2) \ar[r, "\simeq"] & \ev{*}(\Phi_1) \times \ev{*}(\Phi_2) \rlap{.}
        \end{tikzcd}
    \end{center}
    Indeed, this can be checked after projecting to the factors $\ev{*}(\Phi_1)$ and $\ev{*}(\Phi_2)$, respectively, where it amounts to the naturality of $\ev{*}c$ with respect to the projections of $\Phi_1 \times \Phi_2$.
    In particular, these diagrams show that $\ev{*}c$ is a symmetric monoidal natural transformation.
\end{proof}

\begin{notation}
    Let $X \in \Cat C$. We define the \defemphindex{dualizing object of $X$}{dualizing object}, in notation $\defnot{\omega_X} \in \Cat D(X)$ 
    as the tensor unit for the shriek tensor product. 
\end{notation}

\begin{rmk} \label{rmk:tensor-product:shirek-dualizing-object}
    Let $X \in \Cat C$.
    Since $F_X^! \colon \Cat D(X) \to \Cat D(X)$ is symmetric monoidal by \cref{lem:tensor-product:upper-shriek-symmetric-monoidal},
    we see that $\omega_X \cong F_X^! \omega_X$.
\end{rmk}

\begin{prop} \label{lem:tensor-product:shriektensor-on-Cart}
    We have the following description of the shriek tensor product:
    \begin{enumerate}[label=(\arabic*)]
        \item Let $X \in \Cat C$. Assume that $F_X \in P$. Then the tensor unit $\mathbf{1}_{X} \in \Cart(\Cat D(X), F_{*,X})$ is given by 
            $$U_X(\mathbf{1}_{X}) = \omega_X$$
            and
            $$\kappa_X(\mathbf{1}_{X}) \colon F_{*,X} \omega_X \cong F_{*,X} F_X^! \omega_X \xrightarrow{\mathrm{counit}} \omega_X.$$
            Here, the first equivalence is given by \cref{rmk:tensor-product:shirek-dualizing-object}, and the second morphism 
            is the counit of the adjunction $F_{*,X} \dashv F_X^!$ of \cref{lem:consequences:adjoints-of-upper-shriek}, where we use that $F_X \in P$. 
        \item For $X \in \Cat C$ and two objects $M, N \in \Cart(\Cat D(X), F_{*,X})$, their tensor product $M \shriektensor N$ is given by 
            \begin{equation*}
                U_X(M \shriektensor N) \cong U_X(M) \shriektensor U_X(N)
            \end{equation*}
            and
            \begin{equation*}
                \kappa_X(M \shriektensor N) \colon F_{*,X} (U_XM \shriektensor U_XN) \to F_{*,X}U_XM \shriektensor F_{*,X}U_XN \xrightarrow{\kappa_X(M) \shriektensor \kappa_X(N)} U_XM \shriektensor U_XN, 
            \end{equation*}
            where the first arrow is given by the oplax symmetric monoidal structure of $F_{*,X}$ from \Cref{lem:tensor-product:F_*X-is-oplax-sym-mon}.
        \item For every morphism $f \colon X \to Y$ in $\Cat C$, the functor $f^! \colon \Cart(\Cat D(Y), F_{*,Y}) \to \Cart(\Cat D(X), F_{*,X})$ is symmetric monoidal. 
    \end{enumerate}
\end{prop}
\begin{proof}
    We first show (2).
    By \Cref{lem:tensor-product:U-is-sym-mon}, the functor $U_X \colon \Cart(\Cat D(X), F_{*,X}) \to \Cat D(X)$ is symmetric monoidal. 
    In particular, $U_X(M \shriektensor N) \cong U_XM \shriektensor U_XN$.

    We now prove the claim about $\kappa$.
    By definition, $\kappa \colon F_* U \to U$ is a $2$-morphism in the $2$-category $\FunPartOplaxLaxTensor{(I,\all)}{\Span{\Cat C}{\all}{\all}{I}{P}^\otimes}{\TwoCatInf^\times}$. 
    We denote the walking $2$-morphism by $C_2$. 
    By \Cref{lem:2-functoriality-of-lax-tensor-oplax:1-functoriality-univ-prop}, $\kappa$ is equivalently given by an object $\widetilde{\kappa}$ of $\MapMarked{\CAlgLax{\TwoCatInfTwo}}{(I,\all)}{\Span{\Cat C}{\all}{\all}{I}{P}^\otimes}{\FunLax{C_2}{\TwoCatInf^\times}^\ptw}$. 
    The functor $\widetilde{\kappa}$ maps an object $X \in \Span{\Cat C}{\all}{\all}{I}{P}$ to the natural transformation $\kappa_X \colon F_{*,X} U_X \to U_X$. 
    The definition of $\MapMarked{\CAlgLax{\TwoCatInfTwo}}{(I,\all)}{\Span{\Cat C}{\all}{\all}{I}{P}^\otimes}{\FunLax{C_2}{\TwoCatInf^\times}^\ptw}$ provides a commutative diagram 
    \begin{center}
        \begin{tikzcd}
            F_{*,X} U_X \Delta_X^! \boxtimes \ar[r] \ar[d, "\kappa_X"'] & \Delta_X^! F_{*,X\times X} U_{X \times X} \boxtimes \ar[r] \ar[d, "\kappa_{X\times X}"] & \Delta_X^! \boxtimes (F_{*,X} U_X \times F_{*,X} U_X) \ar[d, "\kappa_X \times \kappa_X"] \\
            U_X \Delta_X^! \boxtimes \ar[r, "\simeq"] & \Delta_X^! U_{X \times X} \boxtimes \ar[r, "\simeq"] & \Delta_X^! \boxtimes (U_X \times U_X) \rlap{,}
        \end{tikzcd}
    \end{center}
    where $\Delta_X \colon X \to X \times X$ denotes the diagonal and $\boxtimes$ denotes the functor $\Cat D(X) \times \Cat D(X) \to \Cat D(X \times X)$ given by the lax symmetric monoidal structure of $\Cat D$ (and similarly for $\Cart(\Cat D, F_*)$).  
    Indeed, functoriality of $\widetilde{\kappa}$ for the backwards span $X \times X \xleftarrow{\Delta_X} X \xrightarrow{\id{X}} X$ determines the left hand square, 
    in which the horizontal maps are given by (oplax) naturality of $F_*$ and $U$, respectively. 
    The lax symmetric monoidal structure of $\widetilde{\kappa}$ unwinds to the commutativity of the right hand square, 
    in which the horizontal maps are given by (oplax) symmetric monoidality of $F_*$ and $U$ for $\boxtimes$. 
    If we evaluate the composition of the left vertical map with the bottom horizontal maps at $(M, N) \in \Cart(\Cat D(X), F_{*,X}) \times \Cart(\Cat D(X), F_{*,X})$ we get exactly $\kappa_X(M \shriektensor N)$ 
    by \Cref{lem:tensor-product:DX-is-sym-mon}. 
    Evaluating the composition of the top horizontal maps with the right vertical map at $(M,N)$ yields the map 
    $$F_{*,X} U_X (M \shriektensor N) \to F_{*,X} U_X(M) \shriektensor F_{*,X}U_X(N) \xrightarrow{\kappa_X(M) \shriektensor \kappa_X(N)} U_X(M) \shriektensor U_X(N),$$ 
    where the first arrow is given by the (oplax) symmetric monoidal structure of $F_{*,X}$ and $U_X$, respectively.
    Hence, we conclude the proof of (2) by the commutativity of the above diagram.

    We now show (1). By \Cref{lem:tensor-product:U-is-sym-mon}, the functor $U_X \colon \Cart(\Cat D(X), F_{*,X}) \to \Cat D(X)$ is symmetric monoidal.
    In particular, $U_X(\mathbf{1}_{X}) = \omega_X$. 
    Let $$e_X \colon F_{*,X}\omega_X \longrightarrow \omega_X$$
    denote the unit map of the oplax symmetric monoidal structure from
    \Cref{lem:tensor-product:F_*X-is-oplax-sym-mon}. The nullary monoidal compatibility of $\kappa_X$ gives 
    $$\kappa_X(\mathbf{1}_{X}) \cong e_X,$$
    under the identification $U_X(\mathbf{1}_{X}) \cong \omega_X$.
    This is the nullary analogue of the compatibility used in the proof of (2) above.
    It remains to identify $e_X$. Write $a_X \colon X \to \top$ for the terminal morphism.
    Since $F_\top \colon \top \to \top$ is canonically the identity, the nullary monoidal compatibility of the Frobenius
    transformation at $\top$ is the identity. Consequently, by the construction of the induced
    oplax monoidal structure, $e_X$ is obtained by evaluating the oplax naturality map
    $$F_{*,X}a_X^! \longrightarrow a_X^!$$
    at $\omega_\top$, using $\omega_X \cong a_X^!\omega_\top$.
    Since $F_X \in P$, we have $F_{*,X} \dashv F_X^!$ by \Cref{lem:consequences:adjoints-of-upper-shriek}.
    The preceding oplax naturality map is the mate of the functoriality equivalence
    $$a_X^! \cong F_X^!a_X^!,$$
    induced by $a_XF_X=a_X$.
    Indeed, the corresponding cell between spans is the counit cell for $F_{*,X}\dashv F_X^!$, whiskered with the backward span associated to $a_X$.
    Thus this map is the composite
    $$F_{*,X}a_X^! \xrightarrow{\simeq}  F_{*,X}F_X^!a_X^! \xrightarrow{c\,a_X^!} a_X^!.$$
    Evaluating at $\omega_\top$ therefore gives
    $$\kappa_X(\mathbf{1}_{X}) \cong \left( F_{*,X}\omega_X \xrightarrow{\simeq} F_{*,X}F_X^!\omega_X \xrightarrow{c_{\omega_X}} \omega_X \right),$$
    as claimed.

    (3) follows directly from \Cref{lem:tensor-product:upper-shriek-symmetric-monoidal}.
\end{proof}

\begin{prop} \label{lem:tensor-product:CartD-4FF}
    Consider the $3$-functor formalism $\Cart(\Cat D, F_*)$ from \Cref{def:tensor-product:Cart-D-tensor}. 
    If $\Cat D^\otimes$ is a $6\frac{1}{2}$-functor formalism with respect to a class of morphisms $T \subseteq \Cat C$,
    then $\Cart(\Cat D^\otimes, F_*)$ is a $6\frac{1}{2}$-functor formalism with respect to $T$.
\end{prop}
\begin{proof}
    It follows directly from \Cref{lem:Cartier-3FF:underlying-2FF} and \Cref{lem:constr-cart:main-thm} that the underlying $2$-functor formalism of $\Cart(\Cat D^\otimes, F_*)$ is a $4\frac{1}{2}$-functor formalism. 
    Thus, it remains to show that for every object $X \in \Cat C$ in the symmetric monoidal category $\Cart(\Cat D^\otimes, F_*)(X) \cong \Cart(\Cat D(X), F_{*,X})$ 
    the shriek tensor product is compatible with colimits in each variable.
    As the forgetful functor $U_X \colon \Cart(\Cat D(X), F_{*,X}) \to \Cat D(X)$ is conservative and colimit-preserving by \cite[Corollary 2.7(b), (f)]{us24cart} (using also that $F_{*, X} = \Cat D_*(F_X)$ preserves colimits 
    as it admits a right adjoint by axiom (FF2) of \Cref{def:2ff:defn}), this can be checked after applying $U_X$. 
    Now the claim follows as $U_X$ is symmetric monoidal by \Cref{lem:tensor-product:U-is-sym-mon} and $\Cat D(X)$ is presentably symmetric monoidal by assumption.
\end{proof}

For easier reference we summarize our results of this section.

\begin{thm} \label{lem:tensor-product:main-thm}
    Let $(\Cat C^\times, I, P)$ be a cartesian Nagata setup. Let $F \colon \id{\Cat C} \to \id{\Cat C}$ be a (symmetric monoidal) Frobenius structure on $\Cat C$.
    Let $\Cat D^\otimes \colon \Span{\Cat C}{\all}{\all}{I}{P}^\otimes \to \TwoCatInf^\times$ be a $3$-functor formalism.
    Assume that for every $j \in I$ the square $\Cat D_*(\square_j^F)$ satisfies the left Beck--Chevalley condition.
    Then there is a $3$-functor formalism $\Cart(\Cat D^\otimes, F_*)$ satisfying the following: 
    \begin{enumerate}[label=(\arabic*)]
        \item The underlying $2$-functor formalism of $\Cart(\Cat D^\otimes, F_*)$ is the $2$-functor formalism $\Cart(\Cat D, F_*)$ from \Cref{lem:constr-cart:main-thm}. Consequently, 
        \begin{enumerate}[label=(\alph*)]
            \item If $X \in \Cat C$ is an object, $\Cart(\Cat D, F_*)(X) \cong \Cart(\Cat D(X), F_{*,X})$ is the category of Cartier modules in the category $\Cat D(X)$ with respect to the functor $F_{*,X}$.
            \item There is a natural transformation $U \colon \Cart(\Cat D, F_*) \to \Cat D$ such that for each object $X \in \Cat C$ the functor $U_X$ is equivalent to the forgetful functor $\Cart(\Cat D(X), F_{*,X}) \to \Cat D(X)$. 
            In particular, for each morphism $f \colon X \to Y$ in $\SpanStd$ there is a commutative diagram 
            \begin{center}
                \begin{tikzcd}[column sep=huge]
                    \Cart(\Cat D(X), F_{*,X}) \ar[r, "{\Cart(\Cat D, F_*)(f)}"] \ar[d, "U_X"'] & \Cart(\Cat D(Y), F_{*,Y}) \ar[d, "U_Y"] \\
                    \Cat D(X) \ar[r, "{\Cat D(f)}"'] & \Cat D(Y) \rlap{.}
                \end{tikzcd}
            \end{center}
            \item  If $f \colon X \to Y$ is a morphism in $\SpanStd$, there is a commutative diagram 
            \begin{center}
                \begin{tikzcd}
                    F_{*,Y}U \Cart(\Cat D,F_*)(f) \ar[rr, "{\kappa \Cart(\Cat D,F_*)(f)}"] \ar[d, "\simeq"'] && U \Cart(\Cat D,F_*)(f) \ar[d, "\simeq"] \\
                    F_{*,Y} \Cat D(f)U \ar[r] & \Cat D(f)F_{*,X}U \ar[r, "{\Cat D(f)\kappa}"'] & \Cat D(f)U \rlap{,} 
                \end{tikzcd}
            \end{center}
            where the vertical equivalences are given by the previous property and the bottom left arrow is given by (partially) oplax naturality of $F_*$.
        \end{enumerate}
        \item Let $X \in \Cat C$. Assume that $F_X \in P$. Then the tensor unit $\mathbf{1}_X \in \Cart(\Cat D(X), F_{*,X})$ is given by 
            $$U(\mathbf{1}_{X}) = \omega_X$$
            and
            $$\kappa(\mathbf{1}_{X}) \colon F_{*,X} \omega_X \cong F_{*,X} F_X^! \omega_X \xrightarrow{\mathrm{counit}} \omega_X.$$
            Here, the first equivalence is given by \cref{rmk:tensor-product:shirek-dualizing-object}, and the second morphism 
            is the counit of the adjunction $F_{*,X} \dashv F_X^!$ of \cref{lem:consequences:adjoints-of-upper-shriek}, where we use that $F_X \in P$. 
        \item For $X \in \Cat C$ and two objects $M, N \in \Cart(\Cat D(X), F_{*,X})$, their tensor product $M \shriektensor N$ is given by 
        \begin{equation*}
            U(M \shriektensor N) \cong U(M) \shriektensor U(N)
        \end{equation*}
        and
        \begin{equation*}
            \kappa(M \shriektensor N) \colon F_{*,X} (UM \shriektensor UN) \to F_{*,X}UM \shriektensor F_{*,X}UN \xrightarrow{\kappa(M) \shriektensor \kappa(N)} UM \shriektensor UN, 
        \end{equation*}
        where the first arrow is given by the oplax symmetric monoidal structure of $F_{*,X}$ from \Cref{lem:tensor-product:F_*X-is-oplax-sym-mon}.
    \end{enumerate}
    Moreover, if $\Cat D$ is a $6\frac{1}{2}$-functor formalism with respect to a class of morphisms $T \subseteq \Cat C$,
    then $\Cart(\Cat D, F_*)$ also is a $6\frac{1}{2}$-functor formalism with respect to $T$ whose underlying $4\frac{1}{2}$-functor formalism is the one from \Cref{lem:constr-cart:main-thm}. 
\end{thm}
\begin{proof}
    By \Cref{lem:tensor-product:extension-of-F_*}, the natural transformation $F_* \colon \Cat D_* \to \Cat D_*$ extends to a partially oplax natural transformation $F_* \colon \Cat D \to \Cat D$ between lax symmetric monoidal functors. 
    Now we define $\Cart(\Cat D^\otimes, F_*)$ as in \Cref{def:tensor-product:Cart-D-tensor}. 
    It has the described properties by \Cref{lem:Cartier-3FF:underlying-2FF} and \Cref{lem:tensor-product:shriektensor-on-Cart}. 
    The claim about $6\frac{1}{2}$-functor formalisms was shown in \Cref{lem:tensor-product:CartD-4FF}.
\end{proof}

%% file: derived-schemes/animated-Fp-algebras.tex
We start this section by reviewing the theory of animated $\Fp$-algebras and derived $\Fp$-schemes 
in \Cref{sec:derived-schemes-animated-rings,sec:derived-schemes-derived-schemes}.
Then, in \cref{sec:derived-schemes-ind-coh} we review the theory of ind-coherent sheaves, as 
outlined in \cite{gaitsgory2012indcoherentsheaves,gaitsgory2019study}.
In \cref{sec:derived-schemes-2ff}, using the results of \Cref{sec:2ff}, we construct the 
$4\frac{1}{2}$-functor formalisms of ind-coherent sheaves and Cartier modules.
Then, in \cref{sec:derived-schemes-3ff}, using the results of \Cref{sec:3ff},
we lift these constructions to obtain the $6\frac{1}{2}$-functor formalisms of ind-coherent sheaves and Cartier modules.
Finally, we collect some consequences of the existence of these functor formalisms in \cref{sec:derived-schemes-consequences}.

\subsection{Animated \texorpdfstring{$\Fp$}{Fp}-algebras} \label{sec:derived-schemes-animated-rings}
Let $p$ be a prime.
In this section, we review the theory of animated $\Fp$-algebras.
We write $\defnot{\Ringp}$ for the (1,1)-category of discrete $\Fp$-algebras.

For a small category $\Cat C$ with finite coproducts, recall the \defemph{nonabelian derived category}
$$\defnot{\PSig{\Cat C}} \coloneqq \operatorname{Fun}^{\times}(\op{\Cat C},\An)$$
of finite-product-preserving presheaves on $\Cat C$, cf.
\cite[Definition 5.5.8.8]{highertopoi}.
The Yoneda embedding factors through $\PSig{\Cat C}$ and induces a fully faithful functor
$$\yo \colon \Cat C\hookrightarrow\PSig{\Cat C}$$
which preserves finite coproducts, cf.
\cite[Proposition 5.5.8.10(2)]{highertopoi}.
We will use the following universal property.

\begin{lem}[{\cite[Proposition 5.5.8.15]{highertopoi}}]\label{lem:animated-rings:univ-property}
    Let $\Cat D$ be an $\infty$-category admitting sifted colimits.
    Restriction along the Yoneda embedding induces an equivalence
    $$\defnot{\operatorname{Fun}^{\operatorname{sifted}} \bigl(\PSig{\Cat C},\Cat D\bigr)} \xrightarrow{\ \simeq\ }  \operatorname{Fun}(\Cat C,\Cat D),$$
    where the left-hand side is the full subcategory of functors preserving
    sifted colimits.
\end{lem}

\begin{defn}
    We write $\defnot{\Polyp}\subseteq\Ringp$ for the full subcategory spanned by the
    finitely generated polynomial $\Fp$-algebras
    $$\Fp[x_1,\ldots,x_n]\qquad(n\geq 0).$$
    It has finite coproducts, given by tensor products over $\Fp$.
\end{defn}

\begin{defn}
    We define the category of \defemph{animated $\Fp$-algebras} by
    $$\defnot{\AnRingp}\coloneqq\PSig{\Polyp}.$$
\end{defn}

\begin{prop}[{\cite[Definition 25.1.1.1 and Remark 25.1.1.4]{sag}}] \label{lem:animated-rings:classical-ring-yoneda}
    There is a canonical equivalence
    $$\operatorname{Fun}^{\times}\bigl(\op{\Polyp},\operatorname{Set}\bigr)\xrightarrow{\ \simeq\ }\Ringp.$$
    Consequently, the inclusion
    $\operatorname{Set}\hookrightarrow\An$ induces a fully faithful functor
    $$\yo\colon\Ringp\hookrightarrow\AnRingp,$$
    whose essential image is the full subcategory of $0$-truncated animated
    $\Fp$-algebras. The Yoneda embedding
    $\yo\colon\Polyp\hookrightarrow\AnRingp$ factors through this functor.
\end{prop}

We next define the Frobenius endomorphism of animated $\Fp$-algebras.

\begin{constr} \label{lem:animated-rings:frobenius}
    The universal property of
    \cref{lem:animated-rings:univ-property} identifies restriction along
    $y\colon\Polyp\to\AnRingp$ with an equivalence
    $$\operatorname{Fun}^{\operatorname{sifted}}(\AnRingp,\AnRingp)\xrightarrow{\ \simeq\ }\operatorname{Fun}(\Polyp,\AnRingp).$$
    In particular, natural transformations between sifted-colimit-preserving
    functors are uniquely determined by their restrictions to $\Polyp$.

    The \defemphindex{absolute Frobenius}{Frobenius} maps
    $$\defnot{\operatorname{F}_P}\colon P\longrightarrow P, \qquad a\longmapsto a^p,$$
    are natural in the polynomial $\Fp$-algebra $P$ and hence define a natural
    transformation
    $$\defnot{\operatorname{F}}\colon\id{\Polyp}\longrightarrow\id{\Polyp}.$$
    After applying the Yoneda embedding and extending uniquely under sifted
    colimits, we obtain a natural transformation
    $$\defnot{\operatorname{F}}\colon \id{\AnRingp}\longrightarrow\id{\AnRingp}.$$
    Its restriction to discrete $\Fp$-algebras is the usual absolute
    Frobenius. This defines a Frobenius structure on $\AnRingp$ in the sense
    of \cref{def:constr-cart:frob-structure}.
\end{constr}

\begin{constr}[{\cite[Construction 25.1.2.1 and Proposition 25.1.2.2]{sag}}] \label{constr:animated-ring:htpy-group}
    Sending a polynomial $\Fp$-algebra to its Eilenberg--Mac Lane
    $\mathbb E_\infty$-ring defines a functor
    $$\Polyp\longrightarrow\mathrm{CAlg}_{H\Fp}(\mathrm{Sp}_{\geq 0}).$$
    By \cref{lem:animated-rings:univ-property}, it admits an essentially unique
    sifted-colimit-preserving extension
    $$\defnot{H}\colon\AnRingp\longrightarrow\mathrm{CAlg}_{H\Fp}(\mathrm{Sp}_{\geq 0}),$$
    called the \defemphindex{underlying $\mathbb E_\infty$-ring functor}{underlying $\mathbb E_\infty$-ring}.
    The functor $H$ preserves limits and colimits and is conservative.

    For $A\in\AnRingp$ and $n\geq 0$, we define the \defemphindex{$n$-th homotopy group}{homotopy group} of $A$ as
    $$\defnot{\pi_n(A)}\coloneqq\pi_n\bigl(H(A)\bigr).$$
    Then $\pi_0(A)$ is naturally a discrete $\Fp$-algebra, and
    $$\defnot{\pi_*(A)}\coloneqq\bigoplus_{n\geq 0}\pi_n(A)$$
    is a graded-commutative $\pi_0(A)$-algebra. In particular,
    $\pi_n(A)$ is naturally a $\pi_0(A)$-module for every $n>0$;
    see \cite[\S25.1.3]{sag}.
\end{constr}

\begin{lem}[{\cite[\S25.1.3]{sag}}] \label{lem:animated-ring:whitehead} \uses{constr:animated-ring:htpy-group}
    Let $f\colon A\to B$ be a morphism of animated $\Fp$-algebras.
    Then $f$ is an equivalence if and only if
    $$\pi_n(f)\colon\pi_n(A)\longrightarrow\pi_n(B)$$
    is an isomorphism for every $n\geq 0$.
\end{lem}

%% file: derived-schemes/derived-schemes.tex
\subsection{Derived Schemes} \label{sec:derived-schemes-derived-schemes}
In this section, we review the theory of derived schemes.
Our main reference is Khan's thesis \cite{khan2016motivic}.

\begin{defn}
    Write $\defnot{\AffDSchp} \coloneqq \op{(\AnRingp)}$
    for the opposite of the category of animated $\Fp$-algebras.
    We call this category the \defemphindex{category of affine derived $\Fp$-schemes}{derived scheme!affine}.
\end{defn}

\begin{defn}
    Write $\defnot{\PreStkp} \coloneqq \Fun{\op{(\AffDSchp)}}{\An}$
    for the category of presheaves of anima on affine derived $\Fp$-schemes.
    We call this category the \defemphindex{category of prestacks over $\Fp$}{prestack}.

    Write $\yo \colon \AffDSchp \hookrightarrow \PreStkp$ for the Yoneda embedding.
\end{defn}

\begin{defn}[{\cite[Definitions 4.4.2 and 4.4.3]{khan2016motivic}}]
    Let $\varphi \colon A \to B$ be a morphism of animated $\Fp$-algebras.
    The induced morphism $\Spec{B} \to \Spec{A}$ is called
    \defemph{flat} if $\pi_0(\varphi)$ is flat and, for every $n \ge 0$, the canonical morphism
    \begin{equation*}
        \pi_n(A) \otimes_{\pi_0(A)} \pi_0(B) \longrightarrow \pi_n(B)
    \end{equation*}
    is an isomorphism. It is called an \defemph{open immersion} if it is flat and the induced
    morphism $\Spec{\pi_0(B)} \to \Spec{\pi_0(A)}$ is an open immersion of classical schemes.

    A family $(\Spec{B_\alpha} \to \Spec{A})_\alpha$ is a \defemphindex{Zariski covering family}{Zariski!cover}
    if every morphism in the family is an open immersion and the induced family of classical
    open immersions
    \begin{equation*}
        \bigl(\Spec{\pi_0(B_\alpha)} \longrightarrow \Spec{\pi_0(A)}\bigr)_\alpha
    \end{equation*}
    is jointly surjective. These covering families define a Grothendieck topology, the \defemphindex{Zariski topology}{Zariski!toplogy} on $\AffDSchp$.
\end{defn}

\begin{defn}[{\cite[Definition 4.4.4]{khan2016motivic}}]
    We write
    \begin{equation*}
        \defnot{\StkZar} \coloneqq \Shv{\zar}{\AffDSchp}{\An} \subseteq \PreStkp
    \end{equation*}
    for the $\infty$-topos of sheaves on affine derived $\Fp$-schemes
    with respect to the Zariski topology.
    An object of $\StkZar$ will be called a \defemphindex{derived (Zariski) stack}{Zariski!stack}.
\end{defn}

\begin{defn}[{\cite[Definition 4.5.2]{khan2016motivic}}]
    \label{defn:derived-schemes:open-immersion-stacks}
    Let $j \colon U \to X$ be a morphism in $\StkZar$.
    If $X$ is affine, we call $j$ an \defemph{open immersion} if it is a monomorphism and
    there exists an effective epimorphism $\sqcup_\alpha U_\alpha \to U$, with each $U_\alpha$
    affine, such that every composite $U_\alpha \to X$ is an open immersion of affine derived schemes.
    For arbitrary $X$, we call $j$ an open immersion if, for every affine derived scheme $T$
    and every morphism $T \to X$, the base change
    \begin{equation*}
        U \times_X T \longrightarrow T
    \end{equation*}
    is an open immersion in the preceding sense.
\end{defn}

\begin{defn}[{\cite[§4.5.3]{khan2016motivic}}]
    We say that $X \in \StkZar$ \defemphindex{admits an affine Zariski cover}{Zariski!cover}
    if there exists an effective epimorphism $\sqcup_\alpha U_\alpha \to X$
    in $\StkZar$, where every $U_\alpha$ is an affine derived scheme and every
    $U_\alpha \to X$ is an open immersion.
\end{defn}

\begin{defn}[{\cite[Definition 4.5.4]{khan2016motivic}}]
    Let $X \in \StkZar$.
    We say that $X$ is a \defemphindex{derived $\Fp$-scheme}{derived scheme} if
    $X$ admits an affine Zariski cover.
    Write $\defnot{\DSchp} \subseteq \StkZar$ for the full subcategory spanned
    by derived $\Fp$-schemes.
\end{defn}

\begin{constr}[{\cite[§4.5.6]{khan2016motivic}}] \label{constr:derived-schemes:classical}
    Consider the inclusion $\AffSchp \hookrightarrow \AffDSchp$ of classical affine
    $\Fp$-schemes into affine derived $\Fp$-schemes, given as the opposite of the functor from \cref{lem:animated-rings:classical-ring-yoneda}. Restriction along this inclusion sends
    a derived scheme to a classical scheme and defines a functor
    \begin{equation*}
        \textcolor{purple}{\classical{(-)}} \colon \DSchp \longrightarrow \Schp.
    \end{equation*}
    It is right adjoint to the canonical fully faithful inclusion of classical schemes:
    \begin{equation*}
        \defnot{\mathrm{incl}} \colon \Schp \rightleftarrows \DSchp \noloc \classical{(-)}.
    \end{equation*}
    On affine schemes, this adjunction is the canonical equivalence
    \begin{align*}
        \Map{\AffDSchp}{\Spec{R}}{\Spec{A}}
        &= \Map{\AnRingp}{A}{R} \\
        &\cong \Map{\Ringp}{\pi_0(A)}{R} \\
        &\cong \Map{\AffSchp}{\Spec{R}}{\Spec{\pi_0(A)}},
    \end{align*}
    for $R$ classical; the global adjunction follows by Zariski descent. We call
    $\classical{X}$ the \defemph{underlying classical scheme} of $X$. In particular,
    \begin{equation*}
        \classical{\Spec{A}} \cong \Spec{\pi_0(A)}.
    \end{equation*}
    See also \cite[§2.2]{toen2014derived}.
\end{constr}

\begin{defn}[{\cite[§2.4]{kelly2026pro}}]
    Let $A \to B$ be a morphism of animated rings. We say that $B$ is
    \defemphindex{almost of finite presentation over $A$}{almost of finite presentation} if, for every $n \ge 0$,
    the truncation $\tau_{\le n}B$ is compact in the $\infty$-category of
    $n$-truncated animated $A$-algebras.

    A morphism $f \colon X \to Y$ of derived schemes is called
    \defemphindex{locally almost of finite presentation}{almost of finite presentation!locally}, abbreviated \defemphnoindex{lafp}, if,
    for every pair of affine open subschemes
    \begin{equation*}
        U = \Spec{A} \subseteq Y,
        \qquad
        V = \Spec{B} \subseteq X
    \end{equation*}
    with $f(V) \subseteq U$, the induced morphism $A \to B$ is almost of finite presentation.
\end{defn}

We will be interested in the following properties of morphisms of derived $\Fp$-schemes:
\begin{defn} \label{defn:derived-schemes:properties-of-morphisms} \uses{constr:derived-schemes:classical}
    Let $f \colon X \to Y$ be a morphism of derived schemes.
    We say that $f$ is
    \begin{enumerate}[label=(\alph*)]
        \item \defemph{separated} if $\classical{f}$ is separated, cf.\ \cite[§6.1.1]{khan2016motivic},
        %\item \defemph{of finite type} if $\classical{f}$ is of finite type, cf.\ \cite[§6.1.1]{khan2016motivic},
        \item \defemphindex{locally almost of finite presentation}{almost of finite presentation!locally} if it is lafp in the sense above,
        \item \defemph{proper} if $\classical{f}$ is proper, cf.\ \cite[§6.1.1]{khan2016motivic},
        \item \defemph{flat} if it is Zariski-locally on source and target induced by a flat
        morphism of animated rings in the sense above,
        cf.\ \cite[Definition 4.4.2(ii) and Definition 5.3.3]{khan2016motivic},
        \item an \defemph{open immersion} if it is an open immersion in the sense of
        \cref{defn:derived-schemes:open-immersion-stacks}; equivalently, if it is flat and
        $\classical{f}$ is an open immersion,
        cf.\ \cite[Definition 4.4.2(iii) and Definition 4.5.2]{khan2016motivic},
        \item \defemphindex{étale}{etale@étale} if it is flat and $\classical{f}$ is étale,
        cf.\ \cite[Definitions 5.3.2 and 5.3.3]{khan2016motivic}, 
        \item a \defemph{classical isomorphism} if $\classical{f}$ is an isomorphism.
    \end{enumerate}
\end{defn}

\begin{lem} \label{lem:derived-schemes:properties}
    All the properties of morphisms listed in
    \cref{defn:derived-schemes:properties-of-morphisms} are stable under composition and base change.
\end{lem}
\begin{proof}
    The functor $\classical{(-)}$ is a right adjoint and therefore preserves pullbacks; it also
    preserves compositions. The assertions for separated, proper and classical-isomorphism morphisms therefore follow from the corresponding classical assertions.

    For flat morphisms, the claim follows immediately from the definition; the formulas are preserved under tensor-product
    base change and under composition. The assertions for open immersions and étale morphisms then
    follow from their characterizations as flat morphisms satisfying the corresponding classical
    condition. Finally, morphisms locally almost of finite presentation are stable under composition and
    base change by \cite[§2.4]{kelly2026pro}.
\end{proof}

Flat morphisms have the following useful property:
\begin{lem} \label{lem:derived-scheme:cartesian-classical-flat}
    Consider a commutative square $\square$ of derived schemes as follows:
    \begin{center}
        \begin{tikzcd}
            W \commsymb \ar[r, "f'"] \ar[d, "g'"] & Z \ar[d, "g"] \\
            X \ar[r, "f"] & Y\rlap{.}
        \end{tikzcd}
    \end{center}
    Suppose that $f$ is flat. Then $\square$ is cartesian if and only if $f'$ is flat and
    $\classical{\square}$ is cartesian as a diagram of classical schemes.
\end{lem}
\begin{proof}
    If $\square$ is cartesian, then $f'$ is flat by base change and
    $\classical{\square}$ is cartesian because $\classical{(-)}$ preserves pullbacks.

    Conversely, suppose that $f'$ is flat and $\classical{\square}$ is cartesian. Write
    $W' \coloneqq X \times_Y Z$ and let $\varphi \colon W \to W'$ be the canonical morphism.
    Then $\classical{\varphi}$ is an isomorphism. The projection $W' \to Z$ is flat because it is
    the base change of the flat morphism $f$.

    Equivalences in an $\infty$-topos can be checked after pullback along an effective epimorphism
    \cite[Lemma 6.2.3.16]{highertopoi}. We may therefore pull back successively along affine
    Zariski covers of $Z$ and of $W'$ and assume that both $Z$ and $W'$ are affine. After the second
    pullback, the classical truncation of the resulting source is isomorphic to the affine scheme
    $\classical{W'}$; hence the source is affine as well. Here we use that a derived scheme is affine
    if and only if its classical truncation is affine; see \cite[§2.2]{toen2014derived}.

    We are thus reduced to a diagram of animated rings
    \begin{center}
        \begin{tikzcd}
            B && B' \ar[ll, "b"] \\
            & A \ar[ul, "c"] \ar[ur, "a"']\rlap{,}
        \end{tikzcd}
    \end{center}
    where $a$ and $c$ are flat, $c \simeq b a$, and $\pi_0(b)$ is an isomorphism.
    By definition, for every $n \ge 0$ we have
    that $\pi_n(b)$ is given by the composition
    \begin{equation*}
        \pi_n(B')
        \cong \pi_n(A) \otimes_{\pi_0(A)} \pi_0(B')
        \xrightarrow{\ \id{} \otimes \pi_0(b)\ }
        \pi_n(A) \otimes_{\pi_0(A)} \pi_0(B)
        \cong \pi_n(B).
    \end{equation*}
    Thus $b$ induces an isomorphism on every homotopy group and is therefore an equivalence by
    \cref{lem:animated-ring:whitehead}. Hence $\varphi$ is an equivalence, as required.
\end{proof}

Recall the Frobenius structure on animated $\Fp$-algebras from
\cref{lem:animated-rings:frobenius}. We can extend this to a Frobenius structure on prestacks,
and in particular obtain a Frobenius structure on derived $\Fp$-schemes.

\begin{lem} \label{lem:derived-schemes:frobenius-structure}
    There is a Frobenius structure
    $\defnot{F} \colon \id{\PreStkp} \to \id{\PreStkp}$
    such that, for every affine derived $\Fp$-scheme $X$,
    the morphism $F_{\yo X} \colon \yo X \to \yo X$ is induced by the Frobenius endomorphism
    of the corresponding animated $\Fp$-algebra. Moreover, $F$ restricts to a Frobenius structure
    \begin{equation*}
        \defnot{F} \colon \id{\DSchp} \longrightarrow \id{\DSchp}.
    \end{equation*}
\end{lem}
\begin{proof}
    Let $F^{\mathrm{aff}} \colon \id{\AffDSchp} \to \id{\AffDSchp}$ be the natural transformation
    induced by Frobenius on animated $\Fp$-algebras. Applying the Yoneda embedding gives a natural
    transformation
    \begin{equation*}
        \yo F^{\mathrm{aff}} \colon \yo \longrightarrow \yo.
    \end{equation*}
    Since $\PreStkp$ is the free cocompletion of $\AffDSchp$
    \cite[Theorem 5.1.5.6]{highertopoi}, and the identity functor of $\PreStkp$
    preserves colimits, left Kan extension along $\yo$ produces a natural transformation
    $$F \colon \id{\PreStkp} \to \id{\PreStkp}$$ with the required values on representables.
    Since $\DSchp \subseteq \PreStkp$ is a full subcategory, the component $F_X \colon X \to X$
    is a morphism of derived schemes whenever $X$ is a derived scheme.
\end{proof}

\begin{defn}
    Let $X \in \DSchp$ be a derived $\Fp$-scheme.
    We call the endomorphism $\defnot{F_X} \colon X \to X$
    from \cref{lem:derived-schemes:frobenius-structure}
    the \defemphindex{absolute Frobenius endomorphism}{Frobenius}.
\end{defn}

This definition is compatible with the classical definition of the Frobenius endomorphism:
\begin{lem} \label{lem:derived-schemes:Frob-structure-classical}
    Let $X \in \DSchp$ be a derived $\Fp$-scheme.
    Then $\classical{(F_X)} \colon \classical{X} \to \classical{X}$
    is the absolute Frobenius endomorphism of $\classical{X}$, i.e.\ the morphism that is
    the identity on the underlying topological space and is given by the $p$-th power map on
    the structure sheaf.
\end{lem}
\begin{proof}
    The functor $\classical{(-)}$ is restriction from affine derived $\Fp$-schemes to classical
    affine $\Fp$-schemes. Under this restriction, the Frobenius natural transformation on animated
    $\Fp$-algebras becomes the ordinary Frobenius natural transformation on classical
    $\Fp$-algebras, since
    \begin{equation*}
        \pi_0(F_A) = F_{\pi_0(A)}
    \end{equation*}
    for every animated $\Fp$-algebra $A$. Consequently, the restriction of $F_X$ to the classical
    site is the morphism obtained by gluing the usual Frobenius morphisms on classical affine opens.
    This is precisely the absolute Frobenius of $\classical{X}$.
\end{proof}

\begin{defn}[{\cite[§0.6.8]{gaitsgory2012indcoherentsheaves}}]
    An animated $\Fp$-algebra $A \in \AnRingp$ is called \defemph{noetherian} if $\pi_0(A)$
    is a noetherian ring and $\pi_i(A)$ is a finitely generated $\pi_0(A)$-module for every
    $i \ge 0$.

    A derived $\Fp$-scheme $X \in \DSchp$ is called \defemph{noetherian} if it admits a finite
    Zariski cover by spectra of noetherian animated $\Fp$-algebras.

    We denote by $\textcolor{purple}{\DSchpnoeth} \subseteq \DSchp$ the full subcategory spanned by the noetherian
    derived $\Fp$-schemes.
\end{defn}

\begin{lem} \label{lem:derived-schemes:noetherian-fiber-product-aft}
    Let $X$, $Y$, and $Z$ be noetherian derived $\Fp$-schemes, and let
    $f \colon X \to Y$ and $g \colon Z \to Y$ be morphisms. If $f$ is locally almost of finite
    presentation, then the fiber product $X \times_Y Z$ is noetherian.
\end{lem}
\begin{proof}
    Choose a finite affine open cover $Y = \bigcup_i Y_i$, with $Y_i = \Spec{A_i}$ and each
    $A_i$ noetherian. The open derived subschemes $f^{-1}(Y_i) \subseteq X$ and
    $g^{-1}(Y_i) \subseteq Z$ are noetherian and therefore quasi-compact. Hence we may choose
    finite affine open covers
    \begin{equation*}
        f^{-1}(Y_i) = \bigcup_j U_{ij},
        \qquad
        g^{-1}(Y_i) = \bigcup_k V_{ik},
    \end{equation*}
    where $U_{ij} = \Spec{B_{ij}}$ and $V_{ik} = \Spec{C_{ik}}$ are noetherian affine derived schemes.
    The affine derived schemes
    \begin{equation*}
        U_{ij} \times_{Y_i} V_{ik}
        \cong
        \Spec{B_{ij} \otimes_{A_i} C_{ik}}
    \end{equation*}
    form a finite Zariski cover of $X \times_Y Z$.

    Since $f$ is locally almost of finite presentation, every morphism $A_i \to B_{ij}$ is almost
    of finite presentation. This property is stable under base change, so
    \begin{equation*}
        C_{ik} \longrightarrow B_{ij} \otimes_{A_i} C_{ik}
    \end{equation*}
    is almost of finite presentation. Since $C_{ik}$ is noetherian, the noetherian
    characterization of almost finitely presented algebras in \cite[§2.4]{kelly2026pro} implies that
    $B_{ij} \otimes_{A_i} C_{ik}$ is noetherian. Thus the displayed finite affine cover is a cover
    by noetherian affine derived schemes, and $X \times_Y Z$ is noetherian.
\end{proof}

\begin{lem} \label{lem:noetherian:lafp-of-finite-type}
    Let $f \colon X \to Y$ be a morphism of derived noetherian schemes.
    Then $f$ is lafp if and only if $\classical{f}$ is of finite type.
    
    In particular, if $f$ is an open immersion, étale or proper, then $f$ is lafp.
\end{lem}
\begin{proof}
    Every morphism between noetherian schemes is quasi-compact.
    The equivalence therefore follows by applying the noetherian criterion of \cite[\S2.4]{kelly2026pro} affine-locally.
    The final assertion follows from the corresponding classical properties of open immersions, étale morphisms and proper morphisms.
\end{proof}

\begin{notation}
    Let $f \colon X \to Y$ be a morphism of derived noetherian schemes.
    In view of \cref{lem:noetherian:lafp-of-finite-type}, we say that 
    $f$ is \defemph{almost of finite type}, or \defemphnoindex{aft} if $f$ is locally almost of finite presentation,
    or, equivalently, $\classical{f}$ is of finite type.
\end{notation}

%% file: derived-schemes/cff-derived-schemes.tex
\subsection{Ind-coherent sheaves} \label{sec:derived-schemes-ind-coh}

In this section we define and collect some results about the category $\IndCoh(X)$ of ind-coherent sheaves 
on a noetherian derived scheme $X$. The main references are \cite{gaitsgory2019study,gaitsgory2012indcoherentsheaves}.

\begin{rmk}
    In \cite{gaitsgory2012indcoherentsheaves} and \cite{gaitsgory2019study}, they work over a field of characteristic $0$.
    All of their results used in this article work with the same proof in characteristic $p$.
\end{rmk}

\begin{defn}[{\cite[Chapter 3, \S\S1.1 and 1.5]{gaitsgory2019study}}]
    On derived affine schemes, \defemphindex{quasi-coherent pullback}{pullback!quasi-coherent} is the functor
    \begin{equation*}
        \defnot{\mathrm{QCoh}^*_{\AffDSchp}} \colon \op{({\AffDSchp})} \to \CatInf
    \end{equation*}
    which sends $\Spec{A}$ to the category $\mathrm{Mod}_A$ of $A$-modules (defined as the stable category of $HA$-module spectra $\mathrm{Mod}_{HA}(\mathrm{Sp})$) and a morphism
    $\Spec{B} \to \Spec{A}$ to the extension-of-scalars functor
    $B \otimes_A - \colon \mathrm{Mod}_A \to \mathrm{Mod}_B$.
    We define $\defnot{\QCoh^*}$ on derived prestacks, and in particular on derived schemes, as the right Kan extension of
    $\mathrm{QCoh}^*_{\AffDSchp}$ along the inclusion of derived affine schemes. Thus
    \begin{equation*}
        \defnot{\QCoh(X)} \cong \lim_{(\Spec{A} \to X)\in ((\AffDSchp)_{/X})^{\mathrm{op}}} \mathrm{Mod}_A.
    \end{equation*}

    The \defemph{standard t-structure} on $\QCoh(X)$ is characterized by the property that
    $M \in \QCoh(X)_{\ge 0}$ if and only if its pullback to every derived affine open
    $\Spec A \hookrightarrow X$ belongs to $\mathrm{Mod}_{A,\ge 0}$. On $\Spec A$ this is the standard
    t-structure on $A$-modules, whose heart is canonically equivalent to the category of
    $\pi_0(A)$-modules.

    Moreover, for any $f$ the functor $\mathrm{QCoh}^*(f)$ admits a right adjoint, which we denote by $\defnot{f_*}$.
\end{defn}

\begin{defn}[{\cite[Definition 1.1.2]{gaitsgory2012indcoherentsheaves}}]
    Let $X$ be a derived noetherian scheme.
    We define the category of \defemphindex{coherent sheaves}{sheaf!coherent} $\defnot{\mathrm{Coh}(X)} \subseteq \QCoh(X)$ as the full subcategory consisting of those quasi-coherent modules 
    $M$ that are bounded (in the standard t-structure), and such that $\pi_n(M) \in \QCoh(X)^\heartsuit \cong \QCoh(\classical{X})^\heartsuit$ is coherent 
    (in the sense of \cite[\href{https://stacks.math.columbia.edu/tag/01XY}{Section 01XY}]{stacks-project}).

    It is clear that the t-structure on $\QCoh(X)$ restricts to a t-structure on $\mathrm{Coh}(X)$.
\end{defn}

\begin{defn}[{\cite[Definition 1.1.3]{gaitsgory2012indcoherentsheaves}}]
    Let $X$ be a derived noetherian scheme.
    We define the category of \defemphindex{ind-coherent sheaves}{sheaf!ind-coherent} as $\defnot{\IndCoh(X)} \coloneqq \mathrm{Ind}(\mathrm{Coh}(X))$,
    the ind-completion of the category of coherent sheaves.

    In particular, there is a canonical comparison functor $\defnot{\Psi} \colon \IndCoh(X) \to \QCoh(X)$,
    given by ind-extending the inclusion $\mathrm{Coh}(X) \to \QCoh(X)$.
\end{defn}

\begin{rmk} \label{rmk:indcoh:classical-regular-qcoh}
    Let $X$ be a derived noetherian scheme. The canonical functor $\Psi \colon \IndCoh(X) \to \QCoh(X)$
    is an equivalence if and only if $X$ is classical and regular, see \cite[Lemma 1.1.6 and Proposition 1.6.4]{gaitsgory2012indcoherentsheaves}.
\end{rmk}

\begin{defn}
    Let $X$ be a derived noetherian scheme.
    The category $\IndCoh(X)$ carries a unique t-structure which is compatible with filtered colimits and right-complete,
    such that the inclusion $\mathrm{Coh}(X) \subseteq \IndCoh(X)$ is t-exact, 
    see e.g.\ \cite[Proposition 2.13]{antieau2016k}.
    We call this t-structure the \defemph{standard t-structure}.
\end{defn}

\begin{notation}
    If $\Cat D$ is a stable category with a t-structure $(\Cat D_{\ge 0}, \Cat D_{\le 0})$,
    we write $\Cat D^+ \coloneqq \bigcup_{n} \Cat D_{\le n}$ for the \defemphnoindex{bounded above} part of $\Cat D$.
\end{notation}

\begin{lem}[{\cite[Lemma 1.2.2 and Proposition 1.2.4]{gaitsgory2012indcoherentsheaves}}] \label{lem:ind-coh:bounded-above}
    Let $X$ be a derived noetherian scheme. The canonical functor $\Psi \colon \IndCoh(X) \to \QCoh(X)$ is t-exact,
    and restricts to an equivalence $$\IndCoh(X)^+ \xrightarrow{\simeq} \QCoh(X)^+.$$
\end{lem}

\begin{prop}[{\cite[Proposition 3.1.1]{gaitsgory2012indcoherentsheaves}}] \label{lem:ind-coh:incoherent-functoriality}
    Let $f \colon X \to Y$ be a morphism of derived noetherian schemes.
    Then there exists a colimit-preserving functor $\defnot{f_*} \colon \IndCoh(X) \to \IndCoh(Y)$
    such that the following diagram commutes:
    \begin{center}
        \begin{tikzcd}
            \IndCoh(X) \ar[r, "f_*"] \ar[d, "\Psi"] &\IndCoh(Y) \ar[d, "\Psi"] \\
            \QCoh(X) \ar[r, "f_*"] &\QCoh(Y) \rlap{.}
        \end{tikzcd}
    \end{center}
\end{prop}

%% file: derived-schemes/cff-cartier-modules-derived-schemes.tex
\subsection{The \texorpdfstring{$4\frac{1}{2}$}{Four-and-a-half}-functor formalism of Cartier modules} \label{sec:derived-schemes-2ff}
In this section we review the $4\frac{1}{2}$-functor formalism of ind-coherent sheaves 
and construct the $4\frac{1}{2}$-functor formalism of Cartier modules on derived noetherian $\Fp$-schemes.

\begin{defn} \label{def:derived-schemes:nagata-subcategories}
    We write 
    \begin{itemize}
        \item $\defnot{E}$ for the class of separated almost of finite type morphisms in $\DSchpnoeth$,
        \item $\defnot{I}$ for the class of open immersions in $\DSchpnoeth$, and
        \item $\defnot{P}$ for the class of proper morphisms in $\DSchpnoeth$.
    \end{itemize}
    They are all wide subcategories of $\DSchpnoeth$ by \cref{lem:derived-schemes:properties}.
    Moreover, $I \subseteq E$: by \cite[\S\S4.4.2 and 4.5.2]{khan2016motivic}, open immersions
    are monomorphisms and are locally given by flat morphisms of finite presentation. In particular,
    they are separated and almost of finite type. 
    Similarly, $P \subseteq E$: Indeed, every proper morphism is separated and classically of finite type (since this can be checked classically where it holds by definition).
\end{defn}

\begin{lem} \label{lem:derived-schemes:suitable-decomposition}
    The pair $(\DSchpnoeth, E)$ is a geometric setup, and
    the pair $(I, P)$ is a suitable decomposition of $E$.
\end{lem}
\begin{proof}
    For the first statement, it suffices to show that if $f \colon X \to Y$ is in $E$, and $g \colon Z \to Y$ is an arbitrary morphism,
    with $X, Y$ and $Z$ noetherian, then $X \times_Y Z$ is again noetherian, and the projection $X \times_Y Z \to Z$ is in $E$.
    The first claim is \cref{lem:derived-schemes:noetherian-fiber-product-aft}, and the second claim was proven in \cref{lem:derived-schemes:properties}.

    We check the axioms of a suitable decomposition from \cref{def:cff:suitable-decomposition}.
    \begin{enumerate}[label=\textbf{(\alph*)}]
        \item Both $I$ and $P$ are stable under base change by \cref{lem:derived-schemes:properties}.
        \item Let $f \colon X \to Y$ lie in $E$.
            %Since $X$ and $Y$ are noetherian and $\classical{f}$ is of finite type \todoref, 
            %the noetherian characterization in \cite[\S2.4]{kelly2026pro} implies that $f$ is almost of finite presentation.
            By \cite[Theorem 8.15]{scholze2025six}, there is a factorization
            $$X \xrightarrow{j} \overline{X} \xrightarrow{p} Y$$
            with $j$ an open immersion and $p$ proper in the sense of
            \cite[Definition 8.9]{scholze2025six}. In particular, $p$ is almost of finite presentation.
            Since $Y$ is noetherian, the same noetherian characterization implies that $\overline{X}$ is noetherian.
            Thus this factorization lies in $\DSchpnoeth$, with $j \in I$ and $p \in P$. \footnote{Note that \cite[Proposition 6.3.4]{khan2016motivic} is not enough: The factorization there is not guaranteed to be noetherian again.
            Similarly, the argument in \cite[\S~6.3.1]{gaitsgory2012indcoherentsheaves} does not prove that the compactification 
            is noetherian. 
            The argument in \cite[Chapter 5, Proposition 2.1.6]{gaitsgory2019study} only applies to derived schemes that are almost of finite type 
            over a field, so it is also not applicable in our situation.}
        \item Let $f \in I \cap P$. We have to show that $f$ is $n$-truncated
            for some $n \ge -2$. In fact, every open immersion is a monomorphism and hence
            $(-1)$-truncated; see \cite[\S\S2.2.3, 4.4.2, and 4.5.2]{khan2016motivic}.
        \item Let $f \colon Y \to X$ and $g \colon Z \to Y$ be composable morphisms.

            Suppose first that $f \in P$. We have to show that $g \in P$ if and only if
            $f \circ g \in P$. If $g$ is proper, then so is $f \circ g$ by
            \cref{lem:derived-schemes:properties}. Conversely, suppose that $f \circ g$ is proper.
            Then $\classical{f} \circ \classical{g}$ is proper and $\classical{f}$ is separated,
            so \cite[\href{https://stacks.math.columbia.edu/tag/01W6}{Lemma 01W6}]{stacks-project}
            implies that $\classical{g}$ is proper. Hence $g$ is proper.

            Suppose now that $f \in I$. We have to show that $g \in I$ if and only if
            $f \circ g \in I$. One direction follows from stability under composition.
            Conversely, suppose that $f \circ g$ is an open immersion. Since $f$ is a monomorphism,
            the commutative square
            \begin{center}
                \begin{tikzcd}
                    Z \ar[r, "g"] \ar[d, "{\id{Z}}"'] & Y \ar[d, "f"] \eqsymb &Z\ar[r, "g"] \ar[d, "\id{Z}"] &Y \ar[d, "\id{Y}"] \ar[r, "\id{Y}"] &Y \ar[d, "f"]\\
                    Z \ar[r, "f \circ g"'] & X                                &Z\ar[r, "g"] &Y \ar[r, "f"] &X
                \end{tikzcd}
            \end{center}
            is cartesian (by the pasting law for pullbacks). Thus $g$ is the base change of the open immersion $f \circ g$ along $f$,
            and is therefore an open immersion by \cref{lem:derived-schemes:properties}.
            \qedhere
    \end{enumerate}
\end{proof}

Our starting point for the functor formalism of ind-coherent sheaves is the following:
\begin{prop}[{\cite[Proposition 3.2.4]{gaitsgory2012indcoherentsheaves}}] \label{lem:derived-schemes:qcoh-covariant}
    There is a functor 
    \begin{equation*}
        \defnot{\IndCoh_*} \colon \DSchpnoeth\to \TwoCatInf,
    \end{equation*}
    such that $\IndCoh_*(X)$
    is the category of (derived) ind-coherent modules on $X$
    for every derived scheme $X$.
    Moreover, $\IndCoh_*(f)$ is given by the pushforward functor $f_* \colon \IndCoh(X)\to \IndCoh(Y)$ 
    of \cref{lem:ind-coh:incoherent-functoriality}.
\end{prop}

This functor satisfies the Beck--Chevalley identities needed below. It also admits pullback functors
for the class of eventually coconnective morphisms, which we recall next.
\begin{defn}[{\cite[Definition 3.5.2]{gaitsgory2012indcoherentsheaves}}]
    Let $f \colon X \to Y$ be a morphism of derived noetherian $\Fp$-schemes.
    We say that $f$ is \defemph{eventually coconnective} if the functor $f^* \colon \QCoh(Y) \to \QCoh(X)$ 
    sends $\mathrm{Coh}(Y)$ to $\QCoh(X)^+$.
\end{defn}

\begin{lem}[{\cite[Section 4.5.1]{gaitsgory2012indcoherentsheaves}}] \label{lem:derived-schemes:flat-eventually-coconnective}
    Let $f \colon X \to Y$ be a flat morphism of derived noetherian $\Fp$-schemes.
    Then $f$ is eventually coconnective.
    In particular, open immersions and \'{e}tale morphisms
    are eventually coconnective.
\end{lem}

\begin{prop} \label{lem:derived-schemes:IndCoh-coherent}
    The functor $\IndCoh_*$ of \cref{lem:derived-schemes:qcoh-covariant} extends to a $4\frac{1}{2}$-functor formalism 
    \begin{equation*}
        \defnot{\IndCoh} \colon \Span{\DSchpnoeth}{E}{\all}{I}{P} \to \TwoCatInf
    \end{equation*}
    with respect to the class of eventually coconnective morphism $f$.
\end{prop}
\begin{proof}
    Note first that $(\DSchpnoeth, E, I, P)$ is a Nagata setup by \cref{lem:derived-schemes:suitable-decomposition}.
    We check that $\IndCoh_*$ satisfies all Beck--Chevalley conditions.
    It then extends to a $2$-functor formalism by \cref{lem:constr-2cat:ext-and-unique}.
    \begin{enumerate}[label=Proof of \textbf{(BC\arabic{enumi})}:]
        \item This is a special case of \cite[Lemma 3.5.8]{gaitsgory2012indcoherentsheaves},
            as open immersions are eventually coconnective by \cref{lem:derived-schemes:flat-eventually-coconnective}.
        \item This is a special case of \cite[Lemma 3.6.9]{gaitsgory2012indcoherentsheaves}, using that 
            eventually coconnective morphisms almost of finite type are of finite Tor-amplitude by
            \cite[Lemma 3.6.3]{gaitsgory2012indcoherentsheaves}.
        \item This is exactly \cite[Corollary 3.3.6 and \S3.3.7]{gaitsgory2012indcoherentsheaves}.
        \item This is exactly \cite[Proposition 3.4.2]{gaitsgory2012indcoherentsheaves}.
        \item This is \cite[Proposition 7.1.6]{gaitsgory2012indcoherentsheaves}: the morphisms in $E$
            are almost of finite type, while open immersions have bounded Tor-amplitude.
    \end{enumerate}
    We now show that this $2$-functor formalism is a $4\frac{1}{2}$-functor formalism with respect to the class of eventually coconnective morphisms in the sense of \cref{def:2ff:defn}:
    \begin{enumerate}[label=Proof of \textbf{(FF\arabic{enumi})}:]
        \item By definition, $\IndCoh(X) \coloneqq \mathrm{Ind}(\mathrm{Coh}(X))$ is stable presentable.
        \item Let $f$ be a morphism in $\DSchpnoeth$. By the adjoint functor theorem,
            it suffices to show that $f_*$ preserves colimits,
            which is shown in \cite[Proposition 3.1.1]{gaitsgory2012indcoherentsheaves}.
        \item Let $f \in E$. Choose a compactification $f = pj$, with $p \in P$ and $j \in I$.
            Then $j^!$ has a right adjoint given by $j_*$, cf.\ \Cref{lem:consequences:adjoints-of-upper-shriek}.
            Hence, it suffices to show that $p^!$ has a right adjoint, or, equivalently by the adjoint functor theorem,
            that it preserves colimits. For this, see \cite[\S3.3.7]{gaitsgory2012indcoherentsheaves}.
        \item Let $f$ be eventually coconnective. Then $f_*$ has a left adjoint by \cite[Lemma 3.5.8]{gaitsgory2012indcoherentsheaves}.
    \end{enumerate}
\end{proof}

Recall the Frobenius structure $F \colon \id{\DSchp} \to \id{\DSchp}$ from \cref{lem:derived-schemes:frobenius-structure},
which restricts to a Frobenius structure $F \colon \id{\DSchpnoeth} \to \id{\DSchpnoeth}$. The following proposition,
a derived lift of a classical observation by Böckle--Blickle \cite{blickle2013cartiercrystals}, is the key geometric compatibility 
between Frobenius and étale pullback that makes our construction work.

\begin{prop} \label{lem:applications-constr-CartD:i^*-commutes-with-F_*}
    Let $f \colon X \to Y$ be a morphism of derived noetherian $\Fp$-schemes, 
    and write $\square^F_f$ for the naturality square
    \begin{center}
        \begin{tikzcd}
            X \ar[r, "F_X"] \ar[d, "f"'] & X \ar[d, "f"] \\
            Y \ar[r, "F_Y"'] & Y \rlap{.}
        \end{tikzcd}
    \end{center}
    If $f$ is étale, then $\square^F_f$ is cartesian, and $\IndCoh_*(\square^F_f)$ satisfies the left Beck--Chevalley condition.
\end{prop}
\begin{proof}
    To see that the square is cartesian,
    note that since $f$ is étale, 
    it is flat and moreover $\classical{f}$ is étale.
    Moreover, $\classical{(F_X)}$ and $\classical{(F_Y)}$
    are just the Frobenius endomorphisms
    on $\classical{X}$ and $\classical{Y}$, cf.\ \cref{lem:derived-schemes:Frob-structure-classical}.
    In particular, we see that $\classical{(\square^F_f)}$ is cartesian 
    by \cite[Lemma 2.2.1]{blickle2013cartiercrystals} (note that 
    in the reference they write $\sigma$ for the absolute Frobenius). 
    Hence, since $f$ is flat, we conclude from \cref{lem:derived-scheme:cartesian-classical-flat} 
    that also $\square^F_f$ is cartesian.

    Note that $f_*$ has a left adjoint $f^*$ 
    since étale morphisms are eventually coconnective by \cref{lem:derived-schemes:flat-eventually-coconnective},
    and $\IndCoh$ is a $4\frac{1}{2}$-functor formalism with respect to eventually coconnective 
    morphisms, cf.\ \cref{lem:derived-schemes:IndCoh-coherent}.
    That $\IndCoh_*(\square^F_f)$ satisfies the left Beck--Chevalley condition 
    follows now from \cite[Lemma 3.6.9]{gaitsgory2012indcoherentsheaves}
    (using that $f$ has finite Tor-amplitude by \cite[Lemma 3.6.3]{gaitsgory2012indcoherentsheaves}),
    since we have already seen above that $\square^F_f$ is cartesian.
\end{proof}

\begin{thm} \label{lem:cartier-modules:main-thm-2ff}
    There exists a $4\frac{1}{2}$-functor formalism 
    \begin{equation*}
        \defnot{\Cart} \colon \Span{\DSchpnoeth}{E}{\all}{I}{P} \to \TwoCatInf
    \end{equation*}
    with respect to the class of eventually coconnective morphisms.
    Moreover the following holds:
    \begin{enumerate}[label=(\arabic*)]
        \item If $X \in \DSchpnoeth$, then $\Cart(X) \cong \Cart(\IndCoh(X), \IndCoh_*(F_X))$ is the category of Cartier-modules 
            in the category of ind-coherent modules on $X$ with respect to the Frobenius endomorphism on $X$, cf.\ \cite[Definition 2.4]{us24cart}.
        \item There is a natural transformation 
            $U \colon \Cart \to \IndCoh$,
            such that for each $X \in \DSchpnoeth$ the functor $U_X$ is equivalent to the forgetful functor
            \begin{equation*}
                \Cart(\IndCoh(X), \IndCoh_*(F_X)) \to \IndCoh(X).
            \end{equation*}
            In particular, for each morphism $(f\colon X \to Y) \in \Span{\DSchpnoeth}{E}{\all}{I}{P}$ there is a commutative diagram
            \begin{center}
                \begin{tikzcd}
                    \Cart(\IndCoh(X), \IndCoh_*(F_X)) \ar[r, "\Cart(f)"] \ar[d, "U_X"'] & \Cart(\IndCoh(Y), \IndCoh_*(F_Y)) \ar[d, "U_Y"] \\
                    \IndCoh(X) \ar[r, "\IndCoh(f)"'] & \IndCoh(Y) \rlap{.}
                \end{tikzcd}
            \end{center}
        \item Let $(f \colon X \to Y) \in \Span{\DSchpnoeth}{E}{\all}{I}{P}$ be a morphism.
            Then the structure morphism $\kappa(\Cart(f)(M))$ for a Cartier module $M \in \Cart(\IndCoh(X), \IndCoh_*(F_X))$ is given by 
            \begin{equation*}
                F_{*,Y} U \Cart(f)(M) \cong F_{*,Y} \IndCoh(f) U(M) \to \IndCoh(f) F_{*,X} U(M) \xrightarrow{\kappa(M)} \IndCoh(f) U(M) \, ,
            \end{equation*}
            where the equivalence holds by the previous part and the map $F_{*,Y} \IndCoh(f) \to \IndCoh(f) F_{*,X}$ is given by the partially oplax naturality of $F_*$ (in particular, it is an equivalence if $f \in (I,\all)$).
    \end{enumerate}
\end{thm}
\begin{proof}
    Recall $\IndCoh$, the $4\frac{1}{2}$-functor formalism with respect to the class of eventually coconnective morphisms of \cref{lem:derived-schemes:IndCoh-coherent}.
    For any open immersion $f$, the square $\IndCoh_*(\square^F_f)$ satisfies the left Beck--Chevalley condition by \Cref{lem:applications-constr-CartD:i^*-commutes-with-F_*}.
    Hence, the theorem is an application of \cref{lem:constr-cart:main-thm}.
\end{proof}

We end this section by comparing Cartier modules in ind-coherent sheaves to Cartier modules in the derived category of abelian $\mathcal O_X$-modules on a classical noetherian $\Fp$-scheme.
\begin{prop} \label{lem:derived-schemes:bdd-above-cartier-modules}
    Let $X$ be a classical noetherian $\Fp$-scheme.
    Then there is a canonical t-exact equivalence
    $$\Cart(\IndCoh(X), \IndCoh_*(F_X))^+ \cong \Cat D^+(\Cart(\QCoh^\heartsuit(X),F_{*,X})).$$
    Here, $\QCoh^\heartsuit(X)$ denotes the abelian category of
    quasi-coherent $\mathcal O_X$-modules. On the left, we equip
    the category of Cartier modules with the induced t-structure
    of \cite[Proposition 3.3]{us24cart}, arising from the
    standard t-structure on $\IndCoh(X)$.
\end{prop}
\begin{proof}
    We have the following chain of equivalences:
    \begin{align*}
        \Cart(\IndCoh(X), \IndCoh_*(F_X))^+
        &\cong \Cart(\IndCoh(X)^+, \IndCoh_*(F_X)^+) \\
        &\cong \Cart(\QCoh(X)^+, \QCoh_*(F_X)^+) \\
        &\cong \Cat D^+(\Cart(\QCoh^\heartsuit(X), F_{*,X})).
    \end{align*}
    Here, the first equivalence holds by \cite[Proposition 3.3]{us24cart}, 
    the second equivalence holds by \cref{lem:ind-coh:bounded-above,lem:ind-coh:incoherent-functoriality},
    and the last equivalence holds by \cite[Theorem 5.1]{us24cart} and \cite[\href{https://stacks.math.columbia.edu/tag/09T4}{Proposition 09T4}]{stacks-project} 
    \footnote{note that in our convention, $\QCoh(X)$ is right Kan extended from affine schemes, and hence equivalent to $\Cat D_{\mathrm{qc}}(X)$,
    instead of $\Cat D(\QCoh^\heartsuit(X))$ in general. The stacks project reference shows that for classical noetherian schemes, this distinction does not matter.}.
\end{proof}

%% file: derived-schemes/3ff-cart.tex
\subsection{The \texorpdfstring{$6\frac{1}{2}$}{six-and-a-half}-functor formalism of Cartier modules} \label{sec:derived-schemes-3ff}
In this section we discuss multiplicative enhancements of the $4\frac{1}{2}$-functor formalisms of \cref{lem:derived-schemes:IndCoh-coherent,lem:cartier-modules:main-thm-2ff}.
The first problem that we encounter is that $\DSchpnoeth$ is \emph{not} closed under products in all derived $\Fp$-schemes:
For example, the ring $\overline{\mathbb{F}}_p \otimes_{\Fp} \overline{\mathbb{F}}_p$ is not noetherian.
Therefore, and for reasons discussed in \cref{rmk:3FF:symmetric-Nagata-E-eq-all}, we have to restrict our geometry to derived schemes separated and almost of finite type over $\Fp$.

\begin{rmk}
    In \cite{gaitsgory2019study} they circumvent this problem by working over a base.
    We cannot do this in our later application, as the Frobenius is not the identity on general $\Fp$-schemes, 
    and thus will in general not define a Frobenius structure 
    on $(\DSchpnoeth)_{/S}$. Thus, we obtain our multiplicative results only for Cartier modules 
    on separated and almost of finite type derived $\Fp$-schemes.
\end{rmk}

\begin{notation}
    We write $\textcolor{purple}{\DSchsaft} \subseteq \DSchpnoeth$ for the full subcategory 
    spanned by derived $\Fp$-schemes that are separated and almost of finite type over $\Fp$.

    Furthermore, we write $\defnot{I}, \defnot{P} \subseteq \DSchsaft$ for the wide subcategories 
    spanned by open immersions and proper morphisms, respectively. 
\end{notation}

\begin{lem} \label{lem:derived-schemes:DSaft-Nagata}
    The triple $((\DSchsaft)^\times, I, P)$ is a cartesian Nagata setup.
\end{lem}
\begin{proof}
    By \cref{lem:derived-schemes:noetherian-fiber-product-aft,lem:derived-schemes:properties}, $\DSchsaft$ has all finite products.
    That $(\DSchsaft, \all, I, P)$ is a Nagata setup is analogous to \cref{lem:derived-schemes:suitable-decomposition}.
    Note that any map in $\DSchsaft$ is itself separated and almost of finite type by right-cancellability and the assumption 
    that the structure morphisms have this property.
    We conclude via \cref{lem:3FF:Nagata-is-cartesian}.
\end{proof}

\begin{prop} \label{lem:derived-schemes:IndCoh-tensor}
    There is a lax symmetric monoidal functor 
    \begin{equation*}
        \defnot{\IndCoh_*^\otimes} \colon (\DSchsaft)^\times \to \TwoCatInf^\times,
    \end{equation*}
    such that the underlying functor is the (restriction of the) functor $\IndCoh_*$ from \cref{lem:derived-schemes:qcoh-covariant}.
\end{prop}
\begin{proof}
    Use \cite[Chapter 4, Proposition 6.3.7]{gaitsgory2019study}, and compose with the lax symmetric monoidal 
    forgetful functor $\Pr^{L,\otimes}_{\Fp} \to \CatInf^\times$\footnote{In the reference,
    they write $\mathrm{DGCat_{cont}}$ for the category of $\Fp$-linear stable presentable categories equipped with the $\Fp$-linear Lurie tensor product,
    whereas we use the notation $\Pr^{L,\otimes}_{\Fp}$.}.
\end{proof}

\begin{prop} \label{lem:derived-schemes:IndCohSymMonAllBC}
    The lax symmetric monoidal functor 
    $$\IndCoh_*^\otimes \colon (\DSchsaft)^\times \to \TwoCatInf^\times$$
    from \cref{lem:derived-schemes:IndCoh-tensor} is a lax symmetric monoidal functor satisfying all Beck--Chevalley conditions 
    in the sense of \cref{def:constr-3ff:all-BC} (1).
\end{prop}
\begin{proof}
    The underlying functor satisfies all Beck--Chevalley conditions by the proof of \cref{lem:derived-schemes:IndCoh-coherent}.
    By \cref{def:constr-3ff:all-BC}, it remains to check the external-product transformation.
    For morphisms $f_i\colon X_i\to Y_i$, consider the naturality square $\square$:
    $$\begin{tikzcd}[column sep=large] \IndCoh(X_1)\times\IndCoh(X_2) \ar[r,"\boxtimes"] \ar[d,"f_{1,*}\times f_{2,*}"'] & \IndCoh(X_1\times X_2) \ar[d,"(f_1\times f_2)_*"] \\ \IndCoh(Y_1)\times\IndCoh(Y_2) \ar[r,"\boxtimes"'] & \IndCoh(Y_1\times Y_2). \end{tikzcd}$$
    We must show that the mates
    $$\mathrm{Ex}^{\mathrm{L}}(\square)\colon (j_1\times j_2)^*\circ\boxtimes \longrightarrow \boxtimes\circ(j_1^*\times j_2^*)$$
    for open immersions $j_i$, and
    $$\mathrm{Ex}^{\mathrm{R}}(\square)\colon \boxtimes\circ(p_1^!\times p_2^!) \longrightarrow (p_1\times p_2)^!\circ\boxtimes$$
    for proper morphisms $p_i$, are equivalences.

    Replace the Cartesian products of categories in $\square$ by their presentable $\Fp$-linear tensor products.
    The resulting horizontal functors are equivalences by \cite[Proposition 4.6.2]{gaitsgory2012indcoherentsheaves}.
    Moreover, the adjunctions $j_i^*\dashv j_{i,*}$ and $p_{i,*}\dashv p_i^!$ consist of colimit-preserving functors by \cite[Lemma 3.5.8, Corollary 3.3.6 and \S3.3.7]{gaitsgory2012indcoherentsheaves}, and hence can be tensored.
    Thus the corresponding mates for this new square are equivalences, since its horizontal functors are equivalences.
    Precomposing with
    $$\IndCoh(Y_1)\times\IndCoh(Y_2)\longrightarrow\IndCoh(Y_1)\otimes_{\Fp}\IndCoh(Y_2)$$
    recovers precisely $\mathrm{Ex}^{\mathrm{L}}(\square)$ and $\mathrm{Ex}^{\mathrm{R}}(\square)$: the canonical bilinear functors are natural with respect to continuous functors and natural transformations, hence compatible with the units and counits defining these mates.
    This proves the claim.
\end{proof}

\begin{prop} \label{lem:derived-schemes:IndCoh-3FF}
    There is a $6\frac{1}{2}$-functor formalism 
    \begin{equation*}
        \defnot{\IndCoh^\otimes} \colon \Span{\DSchsaft}{\all}{\all}{I}{P}^\otimes \to \TwoCatInf^\times
    \end{equation*}
    with respect to the class of eventually coconnective morphism $f$.
    Moreover, the underlying $4\frac{1}{2}$-functor formalism $\IndCoh \colon \Span{\DSchsaft}{\all}{\all}{I}{P} \to \TwoCatInf$
    is the restriction of the $4\frac{1}{2}$-functor formalism from \cref{lem:derived-schemes:IndCoh-coherent} 
    to the $2$-category $\Span{\DSchsaft}{\all}{\all}{I}{P}$.
\end{prop}
\begin{proof}
    First note that $((\DSchsaft)^\times, I, P)$ is a cartesian Nagata setup by \cref{lem:derived-schemes:DSaft-Nagata}.
    We want to apply \cref{lem:tensor-product:lax-sym-mon-UP-span} to the lax symmetric monoidal functor $\IndCoh_*^\otimes \colon (\DSchsaft)^\times \to \TwoCatInf^\times$ 
    of \cref{lem:derived-schemes:IndCoh-tensor}. For this, we have to check that $\IndCoh_*^\otimes$ is a lax symmetric monoidal functor satisfying all Beck--Chevalley conditions.
    This was shown in \cref{lem:derived-schemes:IndCohSymMonAllBC}. Hence, we obtain a $3$-functor formalism 
    $$\IndCoh^\otimes \colon \Span{\DSchsaft}{\all}{\all}{I}{P}^\otimes \to \TwoCatInf^\times.$$
    Its underlying $2$-functor formalism is the $2$-functor formalism from \cref{lem:derived-schemes:IndCoh-coherent} by 
    \cref{lem:tensor-product:comparison-cosntruction-2ff-3ff}

    To see that the constructed $3$-functor formalism is a $6\frac{1}{2}$-functor formalism,
    note first that the underlying $2$-functor formalism is a $4\frac{1}{2}$-functor formalism with respect to 
    eventually coconnective morphisms, again by \cref{lem:derived-schemes:IndCoh-coherent}. 
    Hence, it suffices to prove (FF5) of \cref{def:3ff:defn}, i.e.\ that the tensor product preserves colimits in each variable.
    By \cref{lem:tensor-product:DX-is-sym-mon}, the tensor functor is given by the composition 
    $$\IndCoh(X) \times \IndCoh(X) \to \IndCoh(X \times X) \xrightarrow{\Delta_X^!} \IndCoh(X),$$
    where the first functor is induced by the lax monoidal structure of $\IndCoh$.
    Note that the second functor preserves colimits, since it is a left adjoint by (FF3) of \cref{def:2ff:defn}, which we have already shown.
    It was shown in \cite[Chapter 4, Proposition 6.3.7]{gaitsgory2019study} that the first functor preserves colimits separately in each variable, proving (FF5).
\end{proof}

\begin{notation}
    By \cref{lem:tensor-product:DX-is-sym-mon}, for any $X \in \DSchsaft$ we get a symmetric monoidal
    structure on $\IndCoh(X)$, called the \defemph{shriek tensor product}.
    We denote its tensor functor by $\defnot{\shriektensor}$,
    and write $$\defnot{\omega_X} \in \IndCoh(X)$$ for the tensor unit and call it the \defemphindex{dualizing object of $X$}{dualizing object}.
\end{notation}
\begin{rmk}
    We explain why the tensor unit is called the dualizing object.
    Write $a_X \colon X \to \Spec{\Fp}$ for the terminal morphism.
    By \cref{lem:tensor-product:upper-shriek-symmetric-monoidal},
    we know that $$\omega_X \cong a_X^! \omega_{\Fp} \cong a_X^! \Fp.$$
    For the last equivalence we argue as follows: Under the canonical equivalence
    $\IndCoh(\Spec{\Fp}) \simeq \operatorname{Mod}_{\Fp}$ from \cref{rmk:indcoh:classical-regular-qcoh},
    the external product identifies with the usual tensor product
    by \cite[Chapter 4, Lemma 6.3.2]{gaitsgory2019study}.
    Since the diagonal of $\Spec{\Fp}$ over $\Fp$ is the identity,
    the shriek tensor product also identifies with the usual tensor
    product. In particular, $\omega_{\Fp} \cong \Fp$.
\end{rmk}

\begin{thm} \label{lem:cartier-modules:main-thm-3ff}
    There exists a $6\frac{1}{2}$-functor formalism 
    $$\defnot{\Cart^\otimes} \colon \Span{\DSchsaft}{\all}{\all}{I}{P}^\otimes \to \TwoCatInf^\times$$
    with respect to the class of eventually coconnective morphisms.
    Moreover, the following holds:
    \begin{enumerate}[label=(\arabic*)]
        \item The underlying $4\frac{1}{2}$-functor formalism $\Cart \colon \Span{\DSchsaft}{\all}{\all}{I}{P} \to \TwoCatInf$
            is the restriction of the $4\frac{1}{2}$-functor formalism from \cref{lem:cartier-modules:main-thm-2ff} 
            to the $2$-category $\Span{\DSchsaft}{\all}{\all}{I}{P}$.
        \item Let $X \in \DSchsaft$. Then the tensor unit $\mathbf{1}_{\Cart(X)}$ is given by 
            $$U(\mathbf{1}_{\Cart(X)}) = \omega_X$$
            and
            $$\kappa(\mathbf{1}_{\Cart(X)}) \colon F_{*,X} \omega_X \cong F_{*,X} F_X^! \omega_X \xrightarrow{\mathrm{counit}} \omega_X.$$
            Here, the first equivalence is given by \cref{rmk:tensor-product:shirek-dualizing-object}, and the second morphism 
            is the counit of the adjunction $F_{*,X} \dashv F_X^!$ of \cref{lem:consequences:adjoints-of-upper-shriek}, where we use that $F_X$ is proper. 
        \item For $X \in \DSchsaft$ and Cartier modules $M, N \in \Cart(\IndCoh(X), \IndCoh_*(F_X))$ their shriek tensor product $M \shriektensor N$ is given by 
            \begin{equation*}
                U(M \shriektensor N) \cong U(M) \shriektensor U(N)
            \end{equation*}
            and
            \begin{equation*}
                \kappa(M \shriektensor N) \colon F_{*,X} (UM \shriektensor UN) \to F_{*,X}UM \shriektensor F_{*,X}UN \xrightarrow{\kappa(M) \shriektensor \kappa(N)} UM \shriektensor UN, 
            \end{equation*}
            where the first arrow is given by the oplax symmetric monoidal structure of $F_{*,X}$.
    \end{enumerate}
\end{thm}
\begin{proof}
    Recall $\IndCoh^\otimes$, the $6\frac{1}{2}$-functor formalism with respect to the class of eventually coconnective morphisms of \cref{lem:derived-schemes:IndCoh-3FF}.
    For any open immersion $f$, the square $\IndCoh_*(\square^F_f)$ satisfies the left Beck--Chevalley condition by \Cref{lem:applications-constr-CartD:i^*-commutes-with-F_*}.
    Moreover, for any $X \in \DSchsaft$ the morphism $F_X$ is proper: Indeed, this can be checked on classical schemes, where it holds by \cite[\href{https://stacks.math.columbia.edu/tag/0CCD}{Tag 0CCD}]{stacks-project} as $\classical{X}$ is of finite type over $\Fp$.
    Hence, the theorem is an application of \cref{lem:tensor-product:main-thm}.
\end{proof}

%% file: derived-schemes/consequences.tex
\subsection{Applications} \label{sec:derived-schemes-consequences}

\begin{prop} [Existence of upper shriek functors] \label{lem:derived-schemes-consequences:adjoints-shriek}
    Let $f \colon X \to Y$ be a morphism between (derived) noetherian $\Fp$-schemes,
    which is separated and almost of finite type.
    Then there exists a functor 
    \begin{equation*}
        f^! \colon \Cart(Y) \to \Cart(X).
    \end{equation*}
    If $f$ is proper, then $f^!$ is right adjoint to $f_* \colon \Cart(X) \to \Cart(Y)$.
    If $f$ is an open immersion, then $f^!$ is left adjoint to $f_*$.
\end{prop}
\begin{proof}
    This is clear from the existence of the $2$-functor formalism, cf.\ \cref{lem:cartier-modules:main-thm-2ff},
    and \cref{lem:consequences:adjoints-of-upper-shriek}.
\end{proof}

We can spell out more concretely, what the upper shriek functor does:
\begin{prop}[Upper Shriek and Compactifications] \label{lem:derived-schemes-consequences:compactification}
    Let $f \colon X \to Y$ be a morphism between (derived) noetherian $\Fp$-schemes,
    which is separated and almost of finite type.
    Choose any compactification of $f$,
    i.e.\ a factorization of $f$ into 
    \begin{equation*}
        X \xhookrightarrow{j} \overline{X} \xrightarrow{p} Y,
    \end{equation*}
    such that $j$ is an open immersion, $p$ is proper and $\overline{X}$ is noetherian.
    Then there is a canonical equivalence 
    \begin{equation*}
        f^! \cong j^* p^\times,
    \end{equation*}
    where $j^*$ is a left adjoint to $j_*$,
    and $p^\times$ is a right adjoint to $p_*$.
\end{prop}
\begin{proof}
    This is a special case of \cref{lem:consequences:compactification}, using the existence of the $2$-functor formalism,
    cf.\ \cref{lem:cartier-modules:main-thm-2ff}.
\end{proof}

\begin{prop} [Base change] \label{lem:derived-schemes-consequences:BC}
    Let
    \begin{center}
        \begin{tikzcd}
            Y' \ar[r, "g'"] \ar[d, "f'"'] \cartsymb &X' \ar[d, "f"] \\
            Y \ar[r, "g"'] & X,
        \end{tikzcd}
    \end{center}
    be a cartesian diagram of derived noetherian $\Fp$-schemes, where $f$ (and thus $f'$) is separated and almost of finite type.
    Then there is a canonical equivalence of functors
    \begin{equation*}
        f^! g_* \cong g'_* f'^! \colon \Cart(Y) \to \Cart(X'). 
    \end{equation*}
\end{prop}
\begin{proof}
    This is a special case of \cref{lem:consequences:BC-for-E}, using the existence of the $2$-functor formalism, cf.\ \cref{lem:cartier-modules:main-thm-2ff}.
\end{proof}

\begin{rmk}
    See the proof of \cref{lem:consequences:BC-for-E} and the related \cref{lem:consequences:BC-for-P,lem:consequences:BC-for-I} 
    for the exact construction of the canonical map in \cref{lem:derived-schemes-consequences:BC}.
\end{rmk}

As a special case, we have the following base change result which does \emph{not} involve derived schemes:
\begin{cor} \label{lem:consequences:Tor-unital-basechange}
    Let
    \begin{center}
        \begin{tikzcd}
            &X' \ar[d, "f"] \\
            Y \ar[r, "g"'] & X,
        \end{tikzcd}
    \end{center}
    be a span of classical noetherian $\Fp$-schemes, where $f$ is separated and almost of finite type.
    Assume moreover that $f$ and $g$ are \defemph{Tor independent}, i.e.\ 
    for all compatible affine charts $\Spec{A} \subseteq X$, $\Spec{B} \subseteq Y$ and $\Spec{C} \subseteq X'$,
    we have
    $$\mathrm{Tor}^i_{A}(B, C) = 0$$
    for all $i > 0$.
    Then the derived pullback $X' \times_X Y$ is a classical scheme (in particular, it is the classical fiber product), and there is an equivalence of functors 
    \begin{equation*}
        f^! g_* \cong g'_* f'^! \colon \Cart(Y) \to \Cart(X'),
    \end{equation*}
    where $f' \colon X' \times_X Y \to Y$ and $g' \colon X' \times_X Y \to X'$ are the projections.
\end{cor}
\begin{proof}
    The vanishing of Tor groups is exactly the condition that the derived fiber product is equivalent to the classical fiber product (see e.g.\ \cite[\S~2.5.1, immediately after equation~(2.5.1)]{FengHarrisDerived}).
    Hence, this follows immediately from \cref{lem:derived-schemes-consequences:BC}.
\end{proof}

\begin{prop}[Open base change] \label{lem:derived-schemes-consequences:open-basechange}
    Let $\square$ be a cartesian diagram of noetherian derived $\Fp$-schemes of the form
    \begin{center}
        \begin{tikzcd}
            Y' \ar[r, "g'"] \ar[d, "f'"'] \cartsymb &X' \ar[d, "f"] \\
            Y \ar[r, "g"'] & X,
        \end{tikzcd}
    \end{center}
    and suppose that $f \in I$ and $g \in P$.
    Then $\Cart_*(\square)$ satisfies the double Beck--Chevalley condition.
    In particular, there is a canonical equivalence $$f'^*g^\times \cong g'^\times f^*.$$
\end{prop}
\begin{proof}
    This is a special case of \cref{lem:consequences:contravariant-exchange}, using the existence of the $2$-functor formalism, cf.\ \cref{lem:cartier-modules:main-thm-2ff}.
\end{proof}

\begin{prop} \label{lem:derived-schemes-consequences:symm-mon-structure}
    Let $X$ be a derived noetherian $\Fp$-scheme that is separated and almost of finite type over $\Fp$.
    Then $\Cart(X)$ carries a symmetric monoidal structure such that $U_X \colon \Cart(X) \to \IndCoh(X)$ 
    is symmetric monoidal.
\end{prop}
\begin{proof}
    The existence is a special case of \cref{lem:tensor-product:DX-is-sym-mon}, using the existence of the $3$-functor formalism, cf.\ \cref{lem:cartier-modules:main-thm-3ff}.
    That $U$ is symmetric monoidal was shown in \cref{lem:tensor-product:U-is-sym-mon}.
\end{proof}

\begin{prop} \label{lem:consequences:shriek-symmetric-monoidal}
    Let $f \colon X \to Y$ be a morphism of (derived) noetherian $\Fp$-schemes that are separated and almost of finite type over $\Fp$.
    Then the functor $f^! \colon \Cart(Y) \to \Cart(X)$ is symmetric monoidal with respect to the symmetric monoidal structure of \cref{lem:derived-schemes-consequences:symm-mon-structure}. 
\end{prop}
\begin{proof}
    This is a special case of \cref{lem:tensor-product:upper-shriek-symmetric-monoidal}, using the existence of the $3$-functor formalism, cf.\ \cref{lem:cartier-modules:main-thm-3ff}.
\end{proof}

\begin{lem} \label{lem:consequences:shriek-star-projection-formula}
    Let $f \colon X \to Y$ be a morphism of (derived) noetherian $\Fp$-schemes that are separated and almost of finite type over $\Fp$.
    Then the functor $f_* \colon \Cart(X) \to \Cart(Y)$ satisfies the projection formula: For two Cartier modules $M \in \Cart(X)$, $N \in \Cart(Y)$, there is an equivalence 
    $$f_*(M \shriektensor f^! N) \cong f_*(M) \shriektensor N.$$
\end{lem}
\begin{proof}
    This is a special case of \cref{lem:consequences:projection-formula}, using the existence of the $3$-functor formalism, cf.\ \cref{lem:cartier-modules:main-thm-3ff}.
\end{proof}

%% file: appendix/mapping-objects.tex
\section{Internal mapping objects}
In this short appendix, we discuss left and right mapping objects in $\mathbb{E}_1$-monoidal categories.

Let $\Cat C$ be a category with a monoidal structure $\Cat C^\otimes$.
We write $- \otimes - \colon \Cat C \times \Cat C \to \Cat C$ 
for the tensor functor.

\begin{defn}
    We say that $\Cat C$ is \defemphindex{left closed}{closed!left},
    if there is a functor 
    \begin{equation*}
        \defnot{\iHomL{-}{-}} \colon \op{\Cat C} \times \Cat C \to \Cat C,
    \end{equation*}
    such that for every $X, Y, Z \in \Cat C$, there is an equivalence
    \begin{equation*}
        \Map{\Cat C}{X \otimes Y}{Z} \cong \Map{\Cat C}{Y}{\iHomL{X}{Z}},
    \end{equation*}
    natural in $X$, $Y$ and $Z$.
    Similarly, we say that $\Cat C$ is \defemphindex{right closed}{closed!right},
    if there is a functor 
    \begin{equation*}
        \defnot{\iHomR{-}{-}} \colon \op{\Cat C} \times \Cat C \to \Cat C,
    \end{equation*}
    such that for every $X, Y, Z \in \Cat C$, there is an equivalence
    \begin{equation*}
        \Map{\Cat C}{Y \otimes X}{Z} \cong \Map{\Cat C}{Y}{\iHomR{X}{Z}},
    \end{equation*}
    natural in $X$, $Y$ and $Z$.
\end{defn}

In order to be left (resp.\ right) closed, it is enough that $X \otimes -$ (resp.\ $- \otimes X$)
has a right adjoint, i.e.\ the functoriality of $\iHomL{-}{-}$ (resp.\ of $\iHomR{-}{-}$) in the first variable is automatic.
\begin{lem}[{\cite[Remark 4.2.1.31 and Example 4.2.1.32]{higheralgebra}}] \label{lem:mapping-objects:ihom-existence}
    Suppose that for every $X \in \Cat C$, the functor $X \otimes -$
    has a right adjoint.
    Then $\Cat C$ is left closed.

    Similarly, suppose that for every $X \in \Cat C$, the functor $- \otimes X$
    has a right adjoint.
    Then $\Cat C$ is right closed.
\end{lem}

\begin{lem} \label{lem:mapping-objects:ihom-preserves-limits-both-variables}
    Suppose that $\Cat C$ is left and right closed.
    Then the functors $\iHomL{-}{-}$ and $\iHomR{-}{-}$
    preserve (existing) limits separately in each variable. \footnote{Since both functors are contravariant in the first variable, this means that they send colimits in $\Cat C$ to limits in $\Cat C$ in the first variable.}
\end{lem}
\begin{proof}
    We prove the results only for $\iHomR{-}{-}$,
    the proof for $\iHomL{-}{-}$ is analogous.
    Fix an $X \in \Cat C$.
    First note that since $\iHomR{X}{-}$ is right adjoint to 
    $- \otimes X$, it preserves limits.

    For the other claim, let $\Cat I$ be a category, and $Y_{\bullet} \colon \Cat I \to \Cat C$ be a functor such that its colimit exists. 
    For every $Z \in \Cat C$, we have the following chain of natural equivalences:
    \begin{align*}
        \Map{\Cat C}{Z}{\iHomR{\colim_{i \in \Cat I} Y_i}{X}} 
        &\cong \Map{\Cat C}{Z \otimes (\colim_{i \in \Cat I} Y_i)}{X} \\
        &\cong \Map{\Cat C}{\colim_{i \in \Cat I} (Z \otimes Y_i)}{X} \\
        &\cong \lim_{i \in \op{\Cat I}} \Map{\Cat C}{Z \otimes Y_i}{X} \\
        &\cong \lim_{i \in \op{\Cat I}} \Map{\Cat C}{Z}{\iHomR{Y_i}{X}}.
    \end{align*}
    Here, the first and fourth equivalences hold by adjunction,
    the second equivalence holds because the tensor product commutes with colimits in the second variable (as $\Cat C$ is left closed\footnote{This is the reason we have to assume that $\Cat C$ is left and right closed at the same time}),
    and the third equivalence holds because the mapping space functor
    sends colimits in the first variable to limits of anima.
    This shows that $\iHomR{\colim_{i \in \Cat I} Y_i}{X}$ satisfies the universal property of the limit of $\iHomR{Y_\bullet}{X}$.
    In particular, this limit exists and hence $\iHomR{-}{X}$ carries the colimit of $Y_\bullet$ to a limit.
\end{proof}

\begin{lem} \label{lem:mapping-objects:ihom-map-of-lax-functor}
    Let $\Cat C^\boxtimes$ and $\Cat D^\otimes$ be two right closed (resp.\ left closed) monoidal categories,
    and $\Phi \colon \Cat C \to \Cat D$ be a lax monoidal functor.
    Then there is a canonical natural transformation 
    \begin{equation*}
        \Phi \iHomRCat{\Cat C}{-}{-} \to \iHomRCat{\Cat D}{\Phi(-)}{\Phi(-)}. 
    \end{equation*}
    (resp. a natural transformation where we replace $\iHomR{-}{-}$ with $\iHomL{-}{-}$).
\end{lem}
\begin{proof}
    By adjunction, it is enough to construct for every $X \in \Cat C$
    a natural transformation
    $$\Phi \iHomRCat{\Cat C}{X}{-} \otimes \Phi(X) \to \Phi,$$
    dinatural in $X$.
    This natural transformation is given by 
    $$\Phi \iHomRCat{\Cat C}{X}{-} \otimes \Phi(X) \to \Phi(\iHomRCat{\Cat C}{X}{-} \boxtimes X) \xrightarrow{\Phi(\mathrm{counit})} \Phi,$$
    using laxness of $\Phi$.
    Since by definition the adjunction datum of $- \boxtimes X \dashv \iHomRCat{\Cat C}{X}{-}$ is itself 
    dinatural in $X$, we see that this natural transformation is dinatural in $X$.
    The proof for $\iHomL{X}{-}$ is analogous.
\end{proof}

\begin{lem} \label{lem:mapping-objects:tensor-hom-adjunction-internal}
    Let $\Cat C^\otimes$ be a left closed, resp.\ right closed monoidal category.
    For all $X, Y, Z \in \Cat C$, there is an equivalence (natural in $X$, $Y$ and $Z$)
    $$\iHomL{X}{\iHomL{Y}{Z}} \cong \iHomL{Y \otimes X}{Z},$$
    resp.\ an equivalence (natural in $X$, $Y$ and $Z$)
    $$\iHomR{X}{\iHomR{Y}{Z}} \cong \iHomR{X \otimes Y}{Z}.$$
\end{lem}
\begin{proof}
    Let $T \in \Cat C$ be an additional test object.
    We have 
    \begin{align*}
        \Map{\Cat C}{T}{\iHomL{X}{\iHomL{Y}{Z}}} 
        &\cong \Map{\Cat C}{X \otimes T}{\iHomL{Y}{Z}} \\
        &\cong \Map{\Cat C}{Y \otimes X \otimes T}{Z} \\
        &\cong \Map{\Cat C}{T}{\iHomL{Y \otimes X}{Z}}.
    \end{align*}
    Here, all three equivalences hold by adjunction and associativity of the tensor product.
    This equivalence is natural in $T$, $X$, $Y$ and $Z$, hence the claim follows from the Yoneda lemma.
    The proof for $\iHomR{-}{-}$ is analogous, using the chain of equivalences 
    \begin{align*}
        \Map{\Cat C}{T}{\iHomR{X}{\iHomR{Y}{Z}}} 
        &\cong \Map{\Cat C}{T \otimes X}{\iHomR{Y}{Z}} \\
        &\cong \Map{\Cat C}{T \otimes X \otimes Y}{Z} \\
        &\cong \Map{\Cat C}{T}{\iHomR{X \otimes Y}{Z}}. \qedhere
    \end{align*}
\end{proof}